\documentclass{amsart}
\usepackage[all]{xy}
\usepackage{tikz, tikz-cd}
 \tikzstyle{int}=[circle, draw,fill=black,outer sep=0,minimum size=3pt, inner sep=0]
  \tikzstyle{ext}=[circle, draw=black,outer sep=0,inner sep=1pt]
\usepackage{latexsym}
\usepackage{amssymb}
\usepackage{amsmath}
\usepackage{amsthm}
\usepackage{amscd}
\usepackage{fancyhdr}
\usepackage{fullpage}
\usepackage{mathrsfs}
\usepackage{color}
\input cyracc.def
\xyoption{arc}

 \newfam\cyrfam

  \font\tencyr=wncyr10

  \font\sevencyr=wncyr7

  \font\fivecyr=wncyr5

  \textfont\cyrfam=\tencyr \scriptfont\cyrfam=\sevencyr

    \scriptscriptfont\cyrfam=\fivecyr

  \newfam\cyifam

  \font\tencyi=wncyi10

  \font\sevencyi=wncyi7

  \font\fivecyi=wncyi5

  \textfont\cyifam=\tencyi \scriptfont\cyifam=\sevencyi

    \scriptscriptfont\cyifam=\fivecyi

\def\id{{\mbox{1 \hskip -7pt 1}}}
\newcommand{\sgn}{{\mathit s  \mathit g\mathit  n}}
 \newcommand{\lon}{\longrightarrow}
 \newcommand{\bu}{\bullet}
 
 \newcommand{\rar}{\rightarrow}
 \newcommand{\hook}{\hookrightarrow}
 
\newcommand{\p}{{\partial}}
\newcommand{\Id}{{\mathrm I\mathrm d}}

 \newcommand{\Z}{{\mathbb Z}}
 \newcommand{\bS}{{\mathbb S}}

 \newcommand{\N}{{\mathbb N}}
 \newcommand{\K}{{\mathbb K}}

\newcommand{\GC}{\mathsf{GC}}

 \newcommand{\ot}{\otimes}

\newcommand{\sC}{{\mathsf C}}

\newcommand{\sG}{{\mathsf G}}
\newcommand{\sP}{{\mathsf P}}

\newcommand{\Lie}{{\mathcal{L}\mathit{ie}}}

\newcommand{\Def}{{\mathsf{Def}}}

\newcommand{\LB}{\mathcal{L}\mathit{ieb}}
\newcommand{\LoB}{\mathcal{L}\mathit{ieb}^\diamond}

\newcommand{\LBcd}{\mathcal{L}\mathit{ieb}_{c,d}}
\newcommand{\wLBd}{\widehat{\mathcal{L}\mathit{ieb}}_{d}}
\newcommand{\wLBcd}{\widehat{\mathcal{L}\mathit{ieb}}_{c,d}}

\newcommand{\HoLBcd}{\mathcal{H}\mathit{olieb}_{c,d}}
\newcommand{\wHoLBd}{\widehat{\mathcal{H}\mathit{olieb}}_{d}}
\newcommand{\wHoLBcd}{\widehat{\mathcal{H}\mathit{olieb}}_{c,d}}
\newcommand{\wHoLBdd}{\widehat{\mathcal{H}\mathit{olieb}}_{d,d}}
\newcommand{\LoBd}{\mathcal{L}\mathit{ieb}_{d}^\diamond}
\newcommand{\LoBcd}{\mathcal{L}\mathit{ieb}_{c,d}^\diamond}
\newcommand{\wLoBd}{\widehat{\mathcal{L}\mathit{ieb}}_{d}^{\Ba{c}\vspace{-1mm}_{\hspace{-2mm}\diamond} \Ea}}
\newcommand{\wLoBcd}{\widehat{\mathcal{L}\mathit{ieb}}_{c,d}^{\Ba{c}\vspace{-1mm}_{\hspace{-2mm}\diamond} \Ea}}
\newcommand{\HoLoBd}{\mathcal{H}\mathit{olieb}_{d}^\diamond}
\newcommand{\HoLoBcd}{\mathcal{H}\mathit{olieb}_{c,d}^\diamond}
\newcommand{\HoLoB}{\mathcal{H}\mathit{olieb}^\diamond}
\newcommand{\HoLB}{\mathcal{H}\mathit{olieb}}
\newcommand{\wHoLoBd}{\widehat{\mathcal{H}\mathit{olieb}}_{d}^{\Ba{c}\vspace{-1mm}_{\hspace{-2mm}\diamond} \Ea}}
\newcommand{\wHoLoBcd}{\widehat{\mathcal{H}\mathit{olieb}}_{c,d}^{\Ba{c}\vspace{-1mm}_{\hspace{-2mm}\diamond} \Ea}}
\newcommand{\wHoLoBdd}{\widehat{\mathcal{H}\mathit{olieb}}_{d,d}^{\Ba{c}\vspace{-1mm}_{\hspace{-2mm}\diamond} \Ea}}
\newcommand{\HoLhBd}{{\mathcal{H}\mathit{olieb}_{d}^\hbar}}
\newcommand{\HoLhBcd}{{\mathcal{H}\mathit{olieb}_{c,d}^\hbar}}
\newcommand{\sPGCd}{{\mathsf{s}\cP \mathsf{GC}}_{d}}
\newcommand{\sPGCcd}{{\mathsf{s}\cP \mathsf{GC}}_{c,d}}
\newcommand{\PGCd}{\cP\mathsf{GC}_{d}}
\newcommand{\PGCcd}{\cP\mathsf{GC}_{c,d}}
\newcommand{\RGCd}{\mathsf{RGC}_d}

\newcommand{\sRGCd}{\mathsf{sRGC}_d}

\newcommand{\dco}{\delta_{\mathrm{coLie}}}

 \newcommand{\Beq}{\begin{equation}}
 \newcommand{\Eeq}{\end{equation}}
 \newcommand{\Beqr}{\begin{eqnarray}}
 \newcommand{\Eeqr}{\end{eqnarray}}
 \newcommand{\Beqrn}{\begin{eqnarray*}}
 \newcommand{\Eeqrn}{\end{eqnarray*}}
 \newcommand{\Ba}{\begin{array}}
 \newcommand{\Ea}{\end{array}}
 \newcommand{\Bi}{\begin{itemize}}
 \newcommand{\Ei}{\end{itemize}}
 \newcommand{\Bc}{\begin{center}}
 \newcommand{\Ec}{\end{center}}
 
 \newcommand{\fg}{{\mathfrak g}}

\newcommand{\ft}{{\mathfrak t}}
\newcommand{\fr}{{\mathfrak r}}

 \newcommand{\f}{{\mathcal O}}
 \newcommand{\cA}{{\mathcal A}}
 \newcommand{\cB}{{\mathcal B}}
 \newcommand{\cC}{{\mathcal C}}
 \newcommand{\caD}{{\mathcal D}}
 \newcommand{\cE}{{\mathcal E}}
 \newcommand{\cF}{{\mathcal F}}
 \newcommand{\cG}{{\mathcal G}}
 \newcommand{\caH}{{\mathcal H}}

 \newcommand{\caL}{{\mathcal L}}
 \newcommand{\cM}{{\mathcal M}}
 
 \newcommand{\cP}{{\mathcal P}}
 
 \newcommand{\cR}{{\mathcal R}}
 \newcommand{\cS}{{\mathcal S}}
 \newcommand{\cT}{{\mathcal T}}

 \newcommand{\cU}{{\mathcal U}}
 
 \newcommand{\cW}{{\mathcal W}}

 \newcommand{\al}{\alpha}
 \newcommand{\be}{\beta}
 \newcommand{\ga}{\gamma}
 
 \newcommand{\Ga}{\Gamma}

 \newcommand{\var}{\varepsilon}
 \newcommand{\la}{\lambda}

 \newcommand{\Img}{{\mathsf I\mathsf m}\, }
 \newcommand{\Hom}{{\mathrm H\mathrm o\mathrm m}}
 
 \newcommand{\sip}{\smallskip}
 \newcommand{\bip}{\bigskip}
 \newcommand{\mip}{\vspace{2.5mm}}
 \newcommand{\conn}{\mathit{conn}}

\theoremstyle{plain}
\swapnumbers
\newtheorem{theorem}{Theorem}[subsection]

\newtheorem{lemma}[theorem]{Lemma}
\newtheorem{proposition}[theorem]{Proposition}

\newtheorem{prop-def}[theorem]{Proposition-definition}

\newtheorem{f-theorem}{Formality Theorem}[section]
\newtheorem{main-theorem}{Main~Theorem}[section]
\newtheorem{section-theorem}{Theorem}[section]
\newtheorem{section-corollary}{Corollary}[section]

\theoremstyle{definition}

\newtheorem{fact-me}{Fact \cite{Me1}}[subsection]

\newcommand{\op}[1]{\mathcal{#1}}

\newcommand{\LieB}{{{\mathcal L}{\mathit i}{\mathit e}\hspace{-0.2mm} {\mathit b}}}
\newcommand{\invLieB}{\LieB^\diamond}
\newcommand{\hoLieB}{{\mathcal{H}\mathit{olieb}}}
\newcommand{\whoLieB}{\widehat{\hoLieB}}

\newcommand{\hoe}{\mathcal{H} \mathit{opois}}
\newcommand{\hoCom}{\mathcal{H} \mathit{ocom}}
\newcommand{\e}{{\cP \mathit{ois}}}
\newcommand{\hoLie}{\mathcal{H} \mathit{olie}}

\newcommand{\GCor}{\mathsf{GC}^{or}}
\newcommand{\RGC}{\mathsf{RGC}}
\newcommand{\RGCxx}[2]{{#1}\mathsf{GC}_{#2}}
\newcommand{\RGCx}[1]{\RGCxx{#1}{d}}

\newcommand{\RGCPn}[1]{\RGCxx{\op P}{#1}}
\newcommand{\SRGC}{\mathsf{sRGC}}
\newcommand{\SRGCxx}[2]{\mathsf{s}{#1}\mathsf{GC}_{#2}}
\newcommand{\SRGCx}[1]{\SRGCxx{#1}{d}}

\newcommand{\SRGCPn}[1]{\SRGCxx{\op P}{#1}}

\newcommand{\HRGCPn}[1]{\mathsf{H}\cP\mathsf{GC}_{#1}}
\newcommand{\RGra}{{\cR \cG ra}}

\newcommand{\RGraphsxx}[2]{{#1}\cG raphs_{#2}}
\newcommand{\RGraphsx}[1]{{#1}\cG raphs_d}

\newcommand{\RGraphsPn}[1]{{\op P}\cG raphs_{#1}}
\newcommand{\Graphs}{\cG raphs}

\newcommand{\fcGC}{\mathsf{fc}\GC}

\newcommand{\gr}{\mathit{gr}}

\newcommand{\mF}{\mathcal{F}}
\newcommand{\Tw}{\mathit{Tw}}

\newcommand{\marked}{{\mathit{marked}}}

\begin{document}

 \sloppy

 \newenvironment{proo}{\begin{trivlist} \item{\sc {Proof.}}}
  {\hfill $\square$ \end{trivlist}}

\long\def\symbolfootnote[#1]#2{\begingroup%
\def\thefootnote{\fnsymbol{footnote}}\footnote[#1]{#2}\endgroup}

 \title{Props of ribbon graphs, involutive Lie bialgebras\\ and moduli spaces of curves
 }

\author{Sergei~Merkulov}
\address{Sergei~Merkulov:  Mathematics Research Unit, University of Luxembourg,  Grand Duchy of Luxembourg}
\email{sergei.merkulov@uni.lu}

\author{Thomas~Willwacher}
\address{Thomas~Willwacher: Institute of Mathematics, University of Zurich, Zurich, Switzerland}
\email{thomas.willwacher@math.uzh.ch}

\thanks{T.W. has been partially supported by the Swiss National Science foundation, grant 200021\_150012, and the SwissMAP NCCR funded by the Swiss National Science foundation. S.M. has been partially supported by the Swedish Vetenskapr\aa det, grant 2012-5478.}

 \begin{abstract} We establish a new and surprisingly strong link between two previously unrelated theories:
  the theory of moduli spaces of curves  $\cM_{g,n}$ (which, according to Penner, is controlled by the ribbon graph complex)  and the homotopy theory of $E_d$ operads (controlled by ordinary graph complexes with no ribbon structure, introduced first by Kontsevich).
 The link between the two goes through a new intermediate {\em stable}\, ribbon graph complex which has roots
 in the deformation theory of quantum $A_\infty$ algebras and the theory of Kontsevich
 compactifications of moduli spaces of curves $\overline{\cM}_{g,n}^K$.

 \sip

Using a new prop of ribbon graphs and the fact that it contains the prop of involutive Lie bialgebras as a subprop we find new algebraic structures on the classical ribbon graph complex computing $H^\bu(\cM_{g,n})$. We use them to prove Comparison Theorems, and in particular to construct a non-trivial map from the ordinary to the ribbon graph cohomology.

\sip

On the technical side, we construct a functor $\f$ from the category of prop(erad)s
to the category of operads. If a properad $\cP$ is in addition equipped with a map from the properad governing Lie bialgebras (or graded versions thereof), then we define a notion of $\cP$-``graph'' complex, of stable $\cP$-graph complex and a certain operad, that is in good cases an $E_d$ operad.
In the ribbon case, this latter operad acts on the deformation complexes of any quantum $A_\infty$-algebra.

We also prove that there is a highly non-trivial, in general, action of the
Grothendieck-Teichm\"uller group $GRT_1$ on the space of so-called {\em non-commutative Poisson structures}\, on any vector space $W$ equipped with a degree $-1$ symplectic form (which interpolate between cyclic $A_\infty$ structures in $W$ and ordinary polynomial Poisson structures on $W$ as an affine space).



\end{abstract}
 \maketitle

{\small
{\small
\tableofcontents
}
}

%
%

{\Large
\section{\bf Introduction}
}

\subsection{On two previously unrelated theories}
The purpose of this paper is to build a bridge between two well developed, but somewhat separate areas of mathematics: (i) the study of the cohomology of the moduli spaces of curves $H(\cM_{g,n})$ and (ii) the study of the homotopy theory of $E_d$ operads, that is, the chain operads of the classical topological operads of $d$-dimensional cubes \cite{BV,M}.
The problem of studying $H(\cM_{g,n})$ is classical and few words need to be said.
Studying $E_d$ operads and their homotopy theory is a more modern endeavor, and less recognized in mainstream mathematics. Let us just mention that in recent decades it has been realized that the $E_d$ operads appear, and often control, many problems in diverse areas of mathematics. For example, the homotopy automorphisms of the $E_2$ operad may be identified with the Grothendieck-Teichm\"uller group, and this fact underlies the appearance of this group in many situations, e.g., in the study of quantum groups, deformation quantization, or the theory of motives through its connection with the motivic Galois group. The higher $E_d$ operads are very important in algebraic topology. In particular they can be used to recognize based $d$-loop spaces and appear prominently in the recently quite popular Goodwillie-Weiss embedding calculus (and the related concept of factorization homology), whose goal  is, simplifying slightly, to equate embedding spaces of manifolds (e.g., knot spaces) to mapping spaces of $E_d$ modules.

\sip

The second goal of this work is to establish some new highly non-trivial pieces of algebraic structure on the totality of the spaces $H(\cM_{g,n})$ (for varying $g,n$).

\sip

More concretely, the space $H(\cM_{g,n})$ may be computed as the cohomology of a combinatorial complex of ribbon graphs introduced by Penner \cite{Pe}.
Here a ribbon graph is a graph together with the data of cyclic orderings of the incident half-edges at each vertex separately.
We denote the ribbon graph complex with $r$ labelled ``punctures" by $\RGC^{labelled}(r)$, a more precise definition will be provided below. Its counterpart with unlabeled punctures we denote by
\[
\RGC = \prod_r \RGC^{labelled}(r)/ \bS_r.
\]
This complex is defined so that its cohomology is connected to the cohomology of the moduli spaces of curves, skew-symmetrized in the punctures.
\Beq\label{1: Cohomology of RBC and M_g,n}
H^k(\RGC) \cong \prod_{\substack{g,n \\ n>0,2-2g<n} }\left( H^{6g-6+3n-k}(\cM_{g,n}, \mathbb{Q}) \otimes \sgn_n \right)/ \bS_n
\oplus
\begin{cases}
\mathbb Q & \text{for $k=1,4,7,\dots$} \\
0 & \text{otherwise}
\end{cases}
\Eeq
(Here $\sgn_n$ is the sign representation of the symmetric group $\bS_n$.) Similarly, one may also define a ribbon graph complex $\RGC_{odd}$ which computes the cohomology of the moduli space of curves, but symmetrized in the punctures instead of anti-symmetrized. (The complex $\RGC_{odd}$ is the one considered by Kontsevich \cite{Ko2}, for a precise definition and the relation to the cohomology of moduli space see below.)

\sip

On the other hand, the deformation theory of the $E_d$ operads is controlled by similar complexes of ordinary (i.e., non-ribbon) graphs\footnote{The superscript $2$ in the symbol $\GC_d^2$ means that we consider graphs with at least bivalent vertices, while the symbol $\GC_d$ is reserved traditionally  for a complex of graphs with at least trivalent vertices.} $\GC_d^2$, that were first defined by Kontsevich \cite{Ko2}. Recently there was a progress in understanding of the cohomology of graph complexes $\GC_d^2$ for even $d$, in particular
it was proven in \cite{Wi} that
$$
H^i(\GC_2^2)=\left\{\Ba{ll} \fg\fr\ft & \mbox{for}\  i=0\\
			   \K     & \mbox{for}\ i=-1 \\
                             0      & \mbox{for}\ i<-1,
                             \Ea\right.
$$
where $\fg\fr\ft$ is the Lie algebra of the prounipotent Grothendieck-Teichm\"uller group $GRT_1$
introduced by Drinfeld in \cite{D2}.

\sip

Previously, little had been known about the relation between the ribbon graph complex $\RGC$ (which is
in fact the $d=0$ member of a family $\RGC_d$ of complexes parameterized by an integer $d\in \Z$, see \S 4) and the graph complexes $\GC_d^2$
 other then the existence of the forgetful map forgetting the cyclic orderings of the half-edges.
Furthermore, we introduce an "intermediate" (between $\RGC_d$ and $\GC_d^2$) complex $\SRGC_d$ of  stable ribbon graphs, i.e., graphs together with the data of cyclic orderings on disjoint subsets of the half-edges incident at each vertex; it originates (see \S 3 and 5) in the deformation theory of so called
{\em quantum $A_\infty$-algebra}\, structures on vector spaces  equipped with a (skew)symmetric scalar product of degree $1-d$.
The complex $\SRGC_d$ is similar (but not identical) to the complexes studied in \cite{Ba1, Ba2, Ha}  in the context of the (co)homology theory of the Kontsevich compactification $\overline{\cM}^K_{g,n}$  \cite{Ko6} of the moduli space of curves. The complexes  $\SRGC_d$ for various $d$ are isomorphic to each other (up to degree shift) so that at the cohomology level it is enough to study $\SRGC:=\SRGC_0$; the same is true
for the family of complexes $\RGC_d$ but in applications it is often useful to keep track of signs and degree shifts coming from different values of the integer parameter $d$ so that we shall work in this paper in the maximal generality keeping the value of $d$ arbitrary.

\sip

Our results about these complexes are two-fold: on the one hand we describe some pieces of natural algebraic structure on the ribbon graph complexes. On the other hand we show how this structure may be used to compare and constrain the cohomology of those complexes.

\subsection{From props of ribbon graphs and of involutive Lie bialgebras to new  algebraic structures on $\RGC$ 
}
We describe natural dg Lie algebra structures $(\RGC,\delta,[-,-])$ and $(\SRGC,\delta,[-,-])$ on the ribbon and, respectively, stable ribbon graph complexes. (For $\GC_d^2$ the existence of the analogous Lie algebra structure has been known.)
Furthermore, we show that there are two natural additional operators $\Delta_1$ an $\Delta_2$ on each of the ribbon and the stable ribbon graph complexes, that anti-commute with the differential and square to zero.

The key step in finding those algebraic structures is the interpretation of the (stable) ribbon graph complexes as deformation complexes of certain (pr)operads. More concretely, we define a properad $\RGra$ of ribbon graphs, that comes equipped with a natural injective map
\[
\invLieB \to \RGra
\]
from the properad governing involutive Lie bialgebras  (often called {\em diamond}\, algebras), and hence with a natural map
$$
\LieB \to \RGra
$$
from the properad of ordinary Lie bialgebras (which factor though the obvious map $\LieB\to \invLieB$ which is identical on all generators). One can consider also a different map $\LieB\xrightarrow{*} \invLieB$ which sends the cobracket generator to zero and the bracket generator to the bracket generator. By composition with the map to $\RGra$ above we obtain a map $\LieB \xrightarrow{*}\RGra$. Then there is an identification
\[
\RGC \cong \Def(\LieB \xrightarrow{*}\RGra)
\]
with the properadic deformation complex, whose definition will be recalled below. It immediately follows that there is a dg Lie algebra structure on $\RGC$ (since there is a canonical one on the right-hand side) and this dg Lie algebra structure can be  described very explicitly.
Now, considering the deformation complex $\Def(\LieB \to \RGra)$ instead, we find that this complex can be identified with the graded vector space $\RGC$, but with a modified differential $\delta +\Delta_1$. This yields the operation $\Delta_1$.\footnote{We learned from A. Caldararu that the operation $\Delta_1$ had essentially been considered earlier by Tom Bridgeland (unpublished).} Considering the deformation complex $\Def(\invLieB \to \RGra)$ similarly leads to the operation $\Delta_2$.
For the stable ribbon graph complex analogous arguments lead to the definition of the analogous algebraic structures.

\sip

 Similar arguments (up to certain degree shifts) yield a dg Lie algebra structure and a deformation, $\delta+\Delta_3$, of the standard differential $\delta$ in the "odd" counterpart $\RGC_{odd}$, the Kontsevich ribbon graph complex.

\sip

We note furthermore that our arguments provide an operad structure on the collection of labelled ribbon graph complexes  $\RGC^{labelled}(-)$. We emphasize that this is not the same as the cyclic operad structure on the compactified moduli spaces of curves. In particular, the composition of two elements with $r_1$ and $r_2$ punctures has $r_1+r_2-1$ punctures. Geometrically, this operation is best explained using Costello's dual interpretation of the ribbon graph complex \cite{Co}.

\subsection{From new algebraic structures to Comparison Theorems for cohomology}

Our first main result computes the homology of the stable ribbon graph complexes.

\begin{theorem}\label{thm:main_RGC}
There are isomorphisms
\[
H(\SRGC) = H(\RGC)\oplus H(\GC_0^2)
\]
and
\[
H(\SRGC_{odd}) = H(\RGC_{odd})\oplus H(\GC_1^2)
\]
\end{theorem}

As a corollary, one may obtain the following statement.
\begin{theorem}\label{thm:main_RGC2}
There are long exact sequences
\begin{equation}\label{equ:longexact}
\cdots \to H^p(\RGC, \delta+\Delta_1) \to H^p(\SRGC,\delta+\Delta_1) \to H^p(\GC_0^2,\delta)\to H^{p+1}(\RGC,\delta+\Delta_1)\to \cdots
\end{equation}
and
\begin{equation}\label{equ:longexact2}
\cdots \to H^p(\RGC_{odd}, \delta+\Delta_3) \to H^p(\SRGC_{odd},\delta+\Delta_3) \to H^p(\GC_1^2,\delta)\to H^{p+1}(\RGC_{odd},\delta+\Delta_3)\to \cdots .
\end{equation}
\end{theorem}

The interesting statement here is that we find, in particular, highly nontrivial connecting homomorphisms
\begin{align}\label{equ:connectinghom}
\rho: H^\bullet(\GC_0^2,\delta)&\to H^{\bullet+1}(\RGC,\delta+\Delta_1)
\\
\rho_{odd}: H^\bullet(\GC_1^2,\delta)&\to H^{\bullet+1}(\RGC_{odd},\delta+\Delta_3)
.
\end{align}
These morphisms can be quite explicitly described and provide a direct map between the ordinary and the ribbon graph cohomology.
In fact, using the fact that the plain graph cohomology $H(\GC_0^2)$ may be (essentially) identified with the homotopy derivations of the properad $\LieB$, and similarly $H(\GC_1^2)$ with the homotopy derivations of the odd analog $\LieB_{odd}$ (see \cite{MW2}), the connecting homomorphism may be most directly understood as induced by the canonical maps
\begin{align*}
\Def(\LieB \xrightarrow{id}\LieB) &\to  \Def(\LieB \xrightarrow{}\RGra)
&\text{and}&&
\Def(\LieB_{odd} \xrightarrow{id}\LieB_{odd}) &\to  \Def(\LieB_{odd} \xrightarrow{}\RGra).
\end{align*}

We conjecture that the connecting homomorphisms are injections. However, for the moment being, we can prove only the following result (we only treat the ``even" case here).
\begin{theorem}\label{thm:connectinghom}
The connecting homomorphism $\rho$ of \eqref{equ:connectinghom} has the following properties:
\begin{itemize}
\item The image of $\rho$ is infinite dimensional.
\item The map $\rho$ sends classes of homogeneous loop order $g$ in $H^\bullet(\GC_0^2,\delta)$ to classes of genus $g$ in $H^{\bullet+1}(\RGC,\delta+\Delta_1)$. (The genus of a ribbon graph is defined to be the same as the genus of the associated punctured surface.)
\item The map $\rho$ sends (ordinary) graph cohomology classes represented by linear combinations of graphs with $k$ edges to ribbon graph cohomology classes with $k+1$ edges.
\end{itemize}
\end{theorem}

We note furthermore that some conjectures about the complex $H(\RGC,\delta +\Delta_1)$ also appear in the forthcoming work \cite{Ca}.


%
%


\subsection{Full stable ribbon graph complex and the deformation theory of quantum $A_\infty$-algebras} We carry out constructions in somewhat larger generality, with the results about moduli spaces outlined above as the most prominent application. In particular, we provide a action of the dg Lie algebra
\[
\RGC^\diamond:=(\RGC[[\hbar]], \delta +\Delta_1+\hbar \Delta_2)
\]
on the space of quantum $A_\infty$ structures on any vector space, thus exhibiting many universal deformations of such objects.

\sip

Furthermore, our methods provide a second proof of the main result of \cite{Wi2} on the relation of the ordinary graph complexes $\GC_d$ to similar complexes of oriented graphs.

\mip

In this paper, we work almost exclusively in the algebraic setting, as one of our goals is to provide an algebraic setup within which one can study the moduli spaces of curves. That said, one should be able to describe the algebraic structures we find in completely geometric language, a task we leave to a future work.

\subsection{Some technical innovations}
We construct a new polydifferential functor $\f$ from the category of props
to the category of operads which has several nice (for our purposes in this paper) properties:

\sip

(a) For any representation $\cP \rar \cE nd_V$ of a dg prop $\cP$ in a dg vector space $V$ there is an
associated representation $\f(\cP) \rar \cE nd_{\odot^\bu V}$ in the graded symmetric algebra $\odot^{\bu}(V)$; elements of $\f(\cP)=:\f\cP$ act on ${\odot^\bu V}$ as polydifferential operators.

\sip

(b) For any dg prop $\cP$, there is a canonical injection of operads $\cC om \rar \f\cP$.  Moreover, any morphism of props $\caL ie_d\rar \cP$
extends to an associated morphism of operads $\cP ois_d\rar \f\cP$, where
$\caL ie_d$ is the prop associated with an operad of degree shifted Lie algebras, and
  $\cP ois_d$ is the operad of Poisson $d$-algebras (which, by definition, has multiplication generator in degree zero while Lie bracket generator in degree $1-d$ so that
$\cP ois_2$ is the same as the operad of Gerstenhaber algebras).
 \sip

 (c) If a dg prop $\cP$ comes equipped with a non-trivial morphism $\LBcd\rar \cP$ from the prop of Lie $(c,d)$-bialgebras
(in which the cobracket generator has degree $1-c$ and the bracket generator has degree $1-d$), there is an associated non-trivial morphism
of dg operads $\caL ie_{c+d} \rar \f_c\cP:=\f(\cP\{c\})$ which gives rise to
\Bi
\item a {\em stable}\, deformation complex
$$
\mathsf{s}\cP\mathsf{GC}_{c,d}:=\Def(\caL ie_{c+d} \rar \f_c\cP)
$$
containing the deformation complex $\cP\mathsf{GC}_{c,d}:=\Def(\LBcd\rar \cP)$ as subcomplex,

\item a {\em twisted}\, operad $\cP\cG raphs_{c,d}$  which has the following properties:
 (i) the dg Lie algebra $ \mathsf{s}\cP\mathsf{GC}_{c,d} $ (and hence $\cP\mathsf{GC}_{c,d}$) acts on $\cP\cG raphs_{c,d}$ by operadic derivations, (ii) there is a canonical morphism
    $$
    \cP ois_{c+d}\rar \cP\cG raphs_{c,d}
    $$
    of operads, and (iii) for any representation $\rho: \cP\rar \cE nd_V$ and any Maurer-Cartan element $m$ in the associated dg Lie algebra $\Def(\caL ie_{c+d} \rar \f_c\cP \rar \cE nd_{\odot^\bu(V[-c])})$
    there is a canonical representation of $\cP\cG raphs_{c,d}$ in the deformation complex of $m$.
\Ei

(d) If a dg prop $\cP$ comes equipped with a non-trivial morphism $s: \LoBcd\rar \cP$ from the prop of involutive Lie $(c,d)$-bialgebras, there is a ``diamond" version of the structures described  above, i.e. there is an associated non-trivial morphism
of dg operads $s^\diamond: \caL ie_{c+d} \rar \f_c\cP[[\hbar]]$ giving rise to (i) a {\em diamond}\, deformation complex
$$
\mathsf{s}\cP\mathsf{GC}_{c,d}^\diamond:=\Def\left(\caL ie_{c+d} \stackrel{s^\diamond}{\lon} \f_c(\cP)[[\hbar]]\right) \supset \cP\mathsf{GC}_{c,d}^\diamond:=\Def(\LoBcd\rar \cP)
$$
 and (ii) a twisted operad $\cP\cG raphs_{c,d}^\diamond=Tw(\f_c\cP[[\hbar]])$ having properties analogous to $\cP\cG raphs_{c,d}$ (but now in the category of continuous $\K[[\hbar]]$ modules).

\sip

This new general construction is powerful enough to reproduce many important operads and graph complexes from the literature (such as, e.g.,  Kontsevich's operad $\cG raphs$ \cite{Ko4}, Kontsevich's graph complex $\mathsf{GC}_d$ \cite{Wi}, or their extensions studied in \cite{MW, MW2, CMW}). In this paper we apply this machinery to a new prop of ribbon graphs $\cR \cG ra_d$
which comes equipped with a canonical morphism from $\LoB_{d,d}$, and which leads us immediately to the graph complexes $\RGC_d$ and $\SRGC_d$
discussed above, and also to two new dg operads $\cR \cG raphs_d$ and $\cR\cG raphs^\diamond_d$ of 2-coloured ribbon graphs; the latter operad is proven to act on the Hochschild complex of an arbitrary {\em quantum}\, $\cA ss_\infty$ algebra (and hence gives us a ribbon analogue
of the Kontsevich-Soibelman operad of braces \cite{KoSo} introduced in the context of ordinary $\cA ss_\infty$ algebras). We prove that the canonical map
$$
    \cP ois_{2d}\rar \cR\cG raphs_d
    $$ 
is a quasi-isomorphism. We also prove that there is a highly non-trivial, in general, action of the
Grothendieck-Teichm\"uller group $GRT_1$ on the space of so-called {\em non-commutative Poisson structures}\, on any vector space $W$ equipped with a degree $-1$ symplectic form (which interpolate between cyclic $A_\infty$ structures in $W$ and ordinary polynomial Poisson structures on $W$ as an affine space).

\subsection{Some notation}
 The set $\{1,2, \ldots, n\}$ is abbreviated to $[n]$;  its group of automorphisms is
denoted by $\bS_n$;
the trivial one-dimensional representation of
 $\bS_n$ is denoted by $\id_n$, while its one dimensional sign representation is
 denoted by $\sgn_n$; we often abbreviate $\sgn_n^d:= \sgn_n^{\ot |d|}$.
The cardinality of a finite set $A$ is denoted by $\# A$.

\sip

We work throughout in the category of $\Z$-graded vector spaces over a field $\K$
of characteristic zero.
If $V=\oplus_{i\in \Z} V^i$ is a graded vector space, then
$V[k]$ stands for the graded vector space with $V[k]^i:=V^{i+k}$ and
and $s^k$ for the associated isomorphism $V\rar V[k]$; for $v\in V^i$ we set $|v|:=i$.
For a pair of graded vector spaces $V_1$ and $V_2$, the symbol $\Hom_i(V_1,V_2)$ stands
for the space of homogeneous linear maps of degree $i$, and
$\Hom(V_1,V_2):=\bigoplus_{i\in \Z}\Hom_i(V_1,V_2)$; for example, $s^k\in \Hom_{-k}(V,V[k])$.

\sip

For a
prop(erad) $\cP$ we denote by $\cP\{k\}$ a prop(erad) which is uniquely defined by
 the following property:
for any graded vector space $V$ a representation
of $\cP\{k\}$ in $V$ is identical to a representation of  $\cP$ in $V[k]$; in particular, one has for an
endomorphism properad $\cE nd_V\{-k\}=\cE nd_{V[k]}$.
 The degree shifted operad of Lie algebras $\caL \mathit{ie}\{d\}$  is denoted by $\caL ie_{d+1}$ while its minimal resolution by $\caH \mathit{olie}_{d+1}$; representations of $\caL ie_{d+1}$ are vector spaces equipped with Lie brackets of degree $-d$.

\sip

For a right (resp., left) module $V$ over a group $G$ we denote by $V_G$ (resp.\
$_G\hspace{-0.5mm}V$)
 the $\K$-vector space of coinvariants:
$V/\{g(v) - v\ |\ v\in V, g\in G\}$ and by $V^G$ (resp.\ $^GV$) the subspace
of invariants: $\{\forall g\in  G\ :\  g(v)=v,\ v\in V\}$. If $G$ is finite, then these
spaces are canonically isomorphic as $char(\K)=0$.

\bip

{\Large
\section{\bf Involutive Lie bialgebras and Kontsevich graph complexes}
}

\mip

\subsection{Lie $n$-bialgebras} A  {\em Lie n-bialgebra}\, is a graded vector space $V$
equipped with linear maps,
$$
\vartriangle: V\rightarrow V\wedge V \ \ \ \mbox{and}\ \ \  [\ , \ ]: \wedge^2 (V[n])
\rightarrow V[n],
$$
such that
\Bi
\item the data $(V,\vartriangle)$ is a Lie coalgebra;
\item the data $(V[n], [\ ,\ ])$ is a Lie algebra;
\item the compatibility condition,
$$
\vartriangle [a, b] = \sum a_1\otimes [a_2, b] +  [a,
b_1]\otimes b_2 - (-1)^{(|a|+n)(|b|+n)}( [b, a_1]\otimes a_2
+ b_1\otimes [b_2, a]),
$$
holds for any $a,b\in V$. Here $\vartriangle a=:\sum a_1\otimes a_2$, $\vartriangle b=:\sum
b_1\otimes b_2$.
\Ei
The case  $n=0$  gives us the ordinary definition of Lie bialgebra \cite{D1}.
The case $n=1$  is if interest because minimal resolutions of Lie 1-bialgebras control
local Poisson geometry \cite{Me1,Me3}. For $n$ even
it makes sense to introduce an {\em involutive Lie
$n$-bialgebra}\, as a Lie $n$-bialgebra $(V, [\ ,\ ], \vartriangle)$ such that the
composition map
$$
\Ba{ccccc}
V & \stackrel{\vartriangle}{\lon} & \Lambda^2V & \stackrel{[\ ,\ ]}{\lon} & V[-n]\\
a & \lon &    \sum a_1\otimes a_2 &\lon & [a_1,a_2]
\Ea
$$
vanishes (for odd $n$ this condition is trivial for symmetry reasons).

\sip

Let $\caL ieb_{n}$
(resp.\ $\caL ieb^\diamond_{n}$) be the properad  of (resp.\ involutive) Lie $n$-bialgebras.
Let us also introduce the notation
\[
 \caL ieb_{c,d} = \caL ieb_{c+d-2}\{1-c\}, \ \ \ \  \caL ieb_{c,d}^\diamond = \caL ieb_{c+d-2}\{1-c\}
\]
for the degree shifted properads of Lie bialgebras and, respectively,  involutive Lie bialgebras, defined such that the cobracket has degree $1-c$, and the bracket degree $1-d$. It is worth emphasizing that
the symbol $\caL ieb_{c,d}^\diamond$ tacitly assumes that $c+d\in 2\Z$, i.e.\ that the numbers $c$ and $d$ have the same parity.

\sip

The properads $\LB_{1,1}$  and $\LoB_{1,1}$ are often denoted in the literature  by $\LB$ and $\LoB$ respectively, and we denote $\LB_{0,1}$ by $\LB_{odd}$,

\subsection{Properad of (involutive) Lie bialgebras.}
Let us furthermore describe explicitly the properads $\caL ieb_{c,d}$ and $\caL ieb_{c,d}^\diamond$ and their minimal resolutions.

\sip

By definition,  $\LBcd$ is a quadratic properad given as the quotient,
$$
\LB_{c,d}:=\cF ree\langle E\rangle/\langle\cR\rangle,
$$
of the free properad generated by an  $\bS$-bimodule $E=\{E(m,n)\}_{m,n\geq 1}$ with
 all $E(m,n)=0$ except
$$
E(2,1):=\id_1\ot \sgn_2^{c}[c-1]=\mbox{span}\left\langle
\Ba{c}\begin{xy}
 <0mm,-0.55mm>*{};<0mm,-2.5mm>*{}**@{-},
 <0.5mm,0.5mm>*{};<2.2mm,2.2mm>*{}**@{-},
 <-0.48mm,0.48mm>*{};<-2.2mm,2.2mm>*{}**@{-},
 <0mm,0mm>*{\circ};<0mm,0mm>*{}**@{},
 <0.5mm,0.5mm>*{};<2.7mm,2.8mm>*{^{_2}}**@{},
 <-0.48mm,0.48mm>*{};<-2.7mm,2.8mm>*{^{_1}}**@{},
 \end{xy}\Ea
=(-1)^{c}
\Ba{c}\begin{xy}
 <0mm,-0.55mm>*{};<0mm,-2.5mm>*{}**@{-},
 <0.5mm,0.5mm>*{};<2.2mm,2.2mm>*{}**@{-},
 <-0.48mm,0.48mm>*{};<-2.2mm,2.2mm>*{}**@{-},
 <0mm,0mm>*{\circ};<0mm,0mm>*{}**@{},
 <0.5mm,0.5mm>*{};<2.7mm,2.8mm>*{^{_1}}**@{},
 <-0.48mm,0.48mm>*{};<-2.7mm,2.8mm>*{^{_2}}**@{},
 \end{xy}\Ea
   \right\rangle
$$
$$
E(1,2):= \sgn_2^{d}\ot \id_1[d-1]=\mbox{span}\left\langle
\Ba{c}\begin{xy}
 <0mm,0.66mm>*{};<0mm,3mm>*{}**@{-},
 <0.39mm,-0.39mm>*{};<2.2mm,-2.2mm>*{}**@{-},
 <-0.35mm,-0.35mm>*{};<-2.2mm,-2.2mm>*{}**@{-},
 <0mm,0mm>*{\circ};<0mm,0mm>*{}**@{},
   <0.39mm,-0.39mm>*{};<2.9mm,-4mm>*{^{_2}}**@{},
   <-0.35mm,-0.35mm>*{};<-2.8mm,-4mm>*{^{_1}}**@{},
\end{xy}\Ea
=(-1)^{d}
\Ba{c}\begin{xy}
 <0mm,0.66mm>*{};<0mm,3mm>*{}**@{-},
 <0.39mm,-0.39mm>*{};<2.2mm,-2.2mm>*{}**@{-},
 <-0.35mm,-0.35mm>*{};<-2.2mm,-2.2mm>*{}**@{-},
 <0mm,0mm>*{\circ};<0mm,0mm>*{}**@{},
   <0.39mm,-0.39mm>*{};<2.9mm,-4mm>*{^{_1}}**@{},
   <-0.35mm,-0.35mm>*{};<-2.8mm,-4mm>*{^{_2}}**@{},
\end{xy}\Ea
\right\rangle
$$
by the ideal generated by the following elements
\Beq\label{R for LieB}
\cR:\left\{
\Ba{c}
\Ba{c}\resizebox{7mm}{!}{
\begin{xy}
 <0mm,0mm>*{\circ};<0mm,0mm>*{}**@{},
 <0mm,-0.49mm>*{};<0mm,-3.0mm>*{}**@{-},
 <0.49mm,0.49mm>*{};<1.9mm,1.9mm>*{}**@{-},
 <-0.5mm,0.5mm>*{};<-1.9mm,1.9mm>*{}**@{-},
 <-2.3mm,2.3mm>*{\circ};<-2.3mm,2.3mm>*{}**@{},
 <-1.8mm,2.8mm>*{};<0mm,4.9mm>*{}**@{-},
 <-2.8mm,2.9mm>*{};<-4.6mm,4.9mm>*{}**@{-},
   <0.49mm,0.49mm>*{};<2.7mm,2.3mm>*{^3}**@{},
   <-1.8mm,2.8mm>*{};<0.4mm,5.3mm>*{^2}**@{},
   <-2.8mm,2.9mm>*{};<-5.1mm,5.3mm>*{^1}**@{},
 \end{xy}}\Ea
 +
\Ba{c}\resizebox{7mm}{!}{\begin{xy}
 <0mm,0mm>*{\circ};<0mm,0mm>*{}**@{},
 <0mm,-0.49mm>*{};<0mm,-3.0mm>*{}**@{-},
 <0.49mm,0.49mm>*{};<1.9mm,1.9mm>*{}**@{-},
 <-0.5mm,0.5mm>*{};<-1.9mm,1.9mm>*{}**@{-},
 <-2.3mm,2.3mm>*{\circ};<-2.3mm,2.3mm>*{}**@{},
 <-1.8mm,2.8mm>*{};<0mm,4.9mm>*{}**@{-},
 <-2.8mm,2.9mm>*{};<-4.6mm,4.9mm>*{}**@{-},
   <0.49mm,0.49mm>*{};<2.7mm,2.3mm>*{^2}**@{},
   <-1.8mm,2.8mm>*{};<0.4mm,5.3mm>*{^1}**@{},
   <-2.8mm,2.9mm>*{};<-5.1mm,5.3mm>*{^3}**@{},
 \end{xy}}\Ea
 +
\Ba{c}\resizebox{7mm}{!}{\begin{xy}
 <0mm,0mm>*{\circ};<0mm,0mm>*{}**@{},
 <0mm,-0.49mm>*{};<0mm,-3.0mm>*{}**@{-},
 <0.49mm,0.49mm>*{};<1.9mm,1.9mm>*{}**@{-},
 <-0.5mm,0.5mm>*{};<-1.9mm,1.9mm>*{}**@{-},
 <-2.3mm,2.3mm>*{\circ};<-2.3mm,2.3mm>*{}**@{},
 <-1.8mm,2.8mm>*{};<0mm,4.9mm>*{}**@{-},
 <-2.8mm,2.9mm>*{};<-4.6mm,4.9mm>*{}**@{-},
   <0.49mm,0.49mm>*{};<2.7mm,2.3mm>*{^1}**@{},
   <-1.8mm,2.8mm>*{};<0.4mm,5.3mm>*{^3}**@{},
   <-2.8mm,2.9mm>*{};<-5.1mm,5.3mm>*{^2}**@{},
 \end{xy}}\Ea
 \ \ , \ \
\Ba{c}\resizebox{8.4mm}{!}{ \begin{xy}
 <0mm,0mm>*{\circ};<0mm,0mm>*{}**@{},
 <0mm,0.69mm>*{};<0mm,3.0mm>*{}**@{-},
 <0.39mm,-0.39mm>*{};<2.4mm,-2.4mm>*{}**@{-},
 <-0.35mm,-0.35mm>*{};<-1.9mm,-1.9mm>*{}**@{-},
 <-2.4mm,-2.4mm>*{\circ};<-2.4mm,-2.4mm>*{}**@{},
 <-2.0mm,-2.8mm>*{};<0mm,-4.9mm>*{}**@{-},
 <-2.8mm,-2.9mm>*{};<-4.7mm,-4.9mm>*{}**@{-},
    <0.39mm,-0.39mm>*{};<3.3mm,-4.0mm>*{^3}**@{},
    <-2.0mm,-2.8mm>*{};<0.5mm,-6.7mm>*{^2}**@{},
    <-2.8mm,-2.9mm>*{};<-5.2mm,-6.7mm>*{^1}**@{},
 \end{xy}}\Ea
 +
\Ba{c}\resizebox{8.4mm}{!}{ \begin{xy}
 <0mm,0mm>*{\circ};<0mm,0mm>*{}**@{},
 <0mm,0.69mm>*{};<0mm,3.0mm>*{}**@{-},
 <0.39mm,-0.39mm>*{};<2.4mm,-2.4mm>*{}**@{-},
 <-0.35mm,-0.35mm>*{};<-1.9mm,-1.9mm>*{}**@{-},
 <-2.4mm,-2.4mm>*{\circ};<-2.4mm,-2.4mm>*{}**@{},
 <-2.0mm,-2.8mm>*{};<0mm,-4.9mm>*{}**@{-},
 <-2.8mm,-2.9mm>*{};<-4.7mm,-4.9mm>*{}**@{-},
    <0.39mm,-0.39mm>*{};<3.3mm,-4.0mm>*{^2}**@{},
    <-2.0mm,-2.8mm>*{};<0.5mm,-6.7mm>*{^1}**@{},
    <-2.8mm,-2.9mm>*{};<-5.2mm,-6.7mm>*{^3}**@{},
 \end{xy}}\Ea
 +
\Ba{c}\resizebox{8.4mm}{!}{ \begin{xy}
 <0mm,0mm>*{\circ};<0mm,0mm>*{}**@{},
 <0mm,0.69mm>*{};<0mm,3.0mm>*{}**@{-},
 <0.39mm,-0.39mm>*{};<2.4mm,-2.4mm>*{}**@{-},
 <-0.35mm,-0.35mm>*{};<-1.9mm,-1.9mm>*{}**@{-},
 <-2.4mm,-2.4mm>*{\circ};<-2.4mm,-2.4mm>*{}**@{},
 <-2.0mm,-2.8mm>*{};<0mm,-4.9mm>*{}**@{-},
 <-2.8mm,-2.9mm>*{};<-4.7mm,-4.9mm>*{}**@{-},
    <0.39mm,-0.39mm>*{};<3.3mm,-4.0mm>*{^1}**@{},
    <-2.0mm,-2.8mm>*{};<0.5mm,-6.7mm>*{^3}**@{},
    <-2.8mm,-2.9mm>*{};<-5.2mm,-6.7mm>*{^2}**@{},
 \end{xy}}\Ea
 \\
 \Ba{c}\resizebox{5mm}{!}{\begin{xy}
 <0mm,2.47mm>*{};<0mm,0.12mm>*{}**@{-},
 <0.5mm,3.5mm>*{};<2.2mm,5.2mm>*{}**@{-},
 <-0.48mm,3.48mm>*{};<-2.2mm,5.2mm>*{}**@{-},
 <0mm,3mm>*{\circ};<0mm,3mm>*{}**@{},
  <0mm,-0.8mm>*{\circ};<0mm,-0.8mm>*{}**@{},
<-0.39mm,-1.2mm>*{};<-2.2mm,-3.5mm>*{}**@{-},
 <0.39mm,-1.2mm>*{};<2.2mm,-3.5mm>*{}**@{-},
     <0.5mm,3.5mm>*{};<2.8mm,5.7mm>*{^2}**@{},
     <-0.48mm,3.48mm>*{};<-2.8mm,5.7mm>*{^1}**@{},
   <0mm,-0.8mm>*{};<-2.7mm,-5.2mm>*{^1}**@{},
   <0mm,-0.8mm>*{};<2.7mm,-5.2mm>*{^2}**@{},
\end{xy}}\Ea
  -
\Ba{c}\resizebox{7mm}{!}{\begin{xy}
 <0mm,-1.3mm>*{};<0mm,-3.5mm>*{}**@{-},
 <0.38mm,-0.2mm>*{};<2.0mm,2.0mm>*{}**@{-},
 <-0.38mm,-0.2mm>*{};<-2.2mm,2.2mm>*{}**@{-},
<0mm,-0.8mm>*{\circ};<0mm,0.8mm>*{}**@{},
 <2.4mm,2.4mm>*{\circ};<2.4mm,2.4mm>*{}**@{},
 <2.77mm,2.0mm>*{};<4.4mm,-0.8mm>*{}**@{-},
 <2.4mm,3mm>*{};<2.4mm,5.2mm>*{}**@{-},
     <0mm,-1.3mm>*{};<0mm,-5.3mm>*{^1}**@{},
     <2.5mm,2.3mm>*{};<5.1mm,-2.6mm>*{^2}**@{},
    <2.4mm,2.5mm>*{};<2.4mm,5.7mm>*{^2}**@{},
    <-0.38mm,-0.2mm>*{};<-2.8mm,2.5mm>*{^1}**@{},
    \end{xy}}\Ea
  - (-1)^{d}
\Ba{c}\resizebox{7mm}{!}{\begin{xy}
 <0mm,-1.3mm>*{};<0mm,-3.5mm>*{}**@{-},
 <0.38mm,-0.2mm>*{};<2.0mm,2.0mm>*{}**@{-},
 <-0.38mm,-0.2mm>*{};<-2.2mm,2.2mm>*{}**@{-},
<0mm,-0.8mm>*{\circ};<0mm,0.8mm>*{}**@{},
 <2.4mm,2.4mm>*{\circ};<2.4mm,2.4mm>*{}**@{},
 <2.77mm,2.0mm>*{};<4.4mm,-0.8mm>*{}**@{-},
 <2.4mm,3mm>*{};<2.4mm,5.2mm>*{}**@{-},
     <0mm,-1.3mm>*{};<0mm,-5.3mm>*{^2}**@{},
     <2.5mm,2.3mm>*{};<5.1mm,-2.6mm>*{^1}**@{},
    <2.4mm,2.5mm>*{};<2.4mm,5.7mm>*{^2}**@{},
    <-0.38mm,-0.2mm>*{};<-2.8mm,2.5mm>*{^1}**@{},
    \end{xy}}\Ea
  - (-1)^{d+c}
\Ba{c}\resizebox{7mm}{!}{\begin{xy}
 <0mm,-1.3mm>*{};<0mm,-3.5mm>*{}**@{-},
 <0.38mm,-0.2mm>*{};<2.0mm,2.0mm>*{}**@{-},
 <-0.38mm,-0.2mm>*{};<-2.2mm,2.2mm>*{}**@{-},
<0mm,-0.8mm>*{\circ};<0mm,0.8mm>*{}**@{},
 <2.4mm,2.4mm>*{\circ};<2.4mm,2.4mm>*{}**@{},
 <2.77mm,2.0mm>*{};<4.4mm,-0.8mm>*{}**@{-},
 <2.4mm,3mm>*{};<2.4mm,5.2mm>*{}**@{-},
     <0mm,-1.3mm>*{};<0mm,-5.3mm>*{^2}**@{},
     <2.5mm,2.3mm>*{};<5.1mm,-2.6mm>*{^1}**@{},
    <2.4mm,2.5mm>*{};<2.4mm,5.7mm>*{^1}**@{},
    <-0.38mm,-0.2mm>*{};<-2.8mm,2.5mm>*{^2}**@{},
    \end{xy}}\Ea
 - (-1)^{c}
\Ba{c}\resizebox{7mm}{!}{\begin{xy}
 <0mm,-1.3mm>*{};<0mm,-3.5mm>*{}**@{-},
 <0.38mm,-0.2mm>*{};<2.0mm,2.0mm>*{}**@{-},
 <-0.38mm,-0.2mm>*{};<-2.2mm,2.2mm>*{}**@{-},
<0mm,-0.8mm>*{\circ};<0mm,0.8mm>*{}**@{},
 <2.4mm,2.4mm>*{\circ};<2.4mm,2.4mm>*{}**@{},
 <2.77mm,2.0mm>*{};<4.4mm,-0.8mm>*{}**@{-},
 <2.4mm,3mm>*{};<2.4mm,5.2mm>*{}**@{-},
     <0mm,-1.3mm>*{};<0mm,-5.3mm>*{^1}**@{},
     <2.5mm,2.3mm>*{};<5.1mm,-2.6mm>*{^2}**@{},
    <2.4mm,2.5mm>*{};<2.4mm,5.7mm>*{^1}**@{},
    <-0.38mm,-0.2mm>*{};<-2.8mm,2.5mm>*{^2}**@{},
    \end{xy}}\Ea
    \Ea
\right.
\Eeq
Note that for $c+d\in 2\Z$ the generators \begin{xy}
 <0mm,-0.55mm>*{};<0mm,-2.5mm>*{}**@{-},
 <0.5mm,0.5mm>*{};<2.2mm,2.2mm>*{}**@{-},
 <-0.48mm,0.48mm>*{};<-2.2mm,2.2mm>*{}**@{-},
 <0mm,0mm>*{\circ};<0mm,0mm>*{}**@{},
 \end{xy} and \begin{xy}
 <0mm,0.66mm>*{};<0mm,3mm>*{}**@{-},
 <0.39mm,-0.39mm>*{};<2.2mm,-2.2mm>*{}**@{-},
 <-0.35mm,-0.35mm>*{};<-2.2mm,-2.2mm>*{}**@{-},
 <0mm,0mm>*{\circ};<0mm,0mm>*{}**@{},
 \end{xy} have the same symmetry properties, while for $c+d\in 2\Z+1$ the opposite ones.

\sip

Similarly,  $\LoBcd$ (with $c+d\in 2\Z$ by default) is a quadratic properad
$
\cF ree\langle E\rangle/\langle\cR_\diamond\rangle
$
generated by the same $\bS$-bimodule $E$ modulo the relations
$$
\cR_\diamond:= \cR \ \bigsqcup
\Ba{c}\resizebox{4mm}{!}
{\xy
 (0,0)*{\circ}="a",
(0,6)*{\circ}="b",
(3,3)*{}="c",
(-3,3)*{}="d",
 (0,9)*{}="b'",
(0,-3)*{}="a'",
\ar@{-} "a";"c" <0pt>
\ar @{-} "a";"d" <0pt>
\ar @{-} "a";"a'" <0pt>
\ar @{-} "b";"c" <0pt>
\ar @{-} "b";"d" <0pt>
\ar @{-} "b";"b'" <0pt>
\endxy}
\Ea
$$

It is clear from the association
$
\vartriangle \leftrightarrow
 \begin{xy}
 <0mm,-0.55mm>*{};<0mm,-2.5mm>*{}**@{-},
 <0.5mm,0.5mm>*{};<2.2mm,2.2mm>*{}**@{-},
 <-0.48mm,0.48mm>*{};<-2.2mm,2.2mm>*{}**@{-},
 <0mm,0mm>*{\circ};<0mm,0mm>*{}**@{},
 \end{xy}$,
 $
[\ , \ ] \leftrightarrow
 \begin{xy}
 <0mm,0.66mm>*{};<0mm,3mm>*{}**@{-},
 <0.39mm,-0.39mm>*{};<2.2mm,-2.2mm>*{}**@{-},
 <-0.35mm,-0.35mm>*{};<-2.2mm,-2.2mm>*{}**@{-},
 <0mm,0mm>*{\circ};<0mm,0mm>*{}**@{},
 \end{xy}
$
that there is a one-to-one correspondence between representations of $\LBcd$ (resp.,
$\LoBcd$) in a finite dimensional space $V$ and (resp., involutive) Lie $(c+d-2)$-bialgebra
structures in $V[c-1]$.

\sip

It was proven in \cite{Ko3,MaVo,V} that the properad $\LBcd$ is Koszul for $d+c\in 2\Z$ and in
 \cite{Me3} for $d+c\in 2\Z+1$ so that its minimal resolution $\HoLBcd$ is
easy to construct. It
is generated by the following (skew)symmetric corollas of degree $1 +c(1-m)+d(1-n)$
$$
\Ba{c}\resizebox{17mm}{!}{\begin{xy}
 <0mm,0mm>*{\circ};<0mm,0mm>*{}**@{},
 <-0.6mm,0.44mm>*{};<-8mm,5mm>*{}**@{-},
 <-0.4mm,0.7mm>*{};<-4.5mm,5mm>*{}**@{-},
 <0mm,0mm>*{};<1mm,5mm>*{\ldots}**@{},
 <0.4mm,0.7mm>*{};<4.5mm,5mm>*{}**@{-},
 <0.6mm,0.44mm>*{};<8mm,5mm>*{}**@{-},
   <0mm,0mm>*{};<-10.5mm,5.9mm>*{^{\sigma(1)}}**@{},
   <0mm,0mm>*{};<-4mm,5.9mm>*{^{\sigma(2)}}**@{},
   <0mm,0mm>*{};<10.0mm,5.9mm>*{^{\sigma(m)}}**@{},
 <-0.6mm,-0.44mm>*{};<-8mm,-5mm>*{}**@{-},
 <-0.4mm,-0.7mm>*{};<-4.5mm,-5mm>*{}**@{-},
 <0mm,0mm>*{};<1mm,-5mm>*{\ldots}**@{},
 <0.4mm,-0.7mm>*{};<4.5mm,-5mm>*{}**@{-},
 <0.6mm,-0.44mm>*{};<8mm,-5mm>*{}**@{-},
   <0mm,0mm>*{};<-10.5mm,-6.9mm>*{^{\tau(1)}}**@{},
   <0mm,0mm>*{};<-4mm,-6.9mm>*{^{\tau(2)}}**@{},
   <0mm,0mm>*{};<10.0mm,-6.9mm>*{^{\tau(n)}}**@{},
 \end{xy}}\Ea
=(-1)^{c|\sigma|+d|\tau|}
\Ba{c}\resizebox{14mm}{!}{\begin{xy}
 <0mm,0mm>*{\circ};<0mm,0mm>*{}**@{},
 <-0.6mm,0.44mm>*{};<-8mm,5mm>*{}**@{-},
 <-0.4mm,0.7mm>*{};<-4.5mm,5mm>*{}**@{-},
 <0mm,0mm>*{};<-1mm,5mm>*{\ldots}**@{},
 <0.4mm,0.7mm>*{};<4.5mm,5mm>*{}**@{-},
 <0.6mm,0.44mm>*{};<8mm,5mm>*{}**@{-},
   <0mm,0mm>*{};<-8.5mm,5.5mm>*{^1}**@{},
   <0mm,0mm>*{};<-5mm,5.5mm>*{^2}**@{},
   <0mm,0mm>*{};<4.5mm,5.5mm>*{^{m\hspace{-0.5mm}-\hspace{-0.5mm}1}}**@{},
   <0mm,0mm>*{};<9.0mm,5.5mm>*{^m}**@{},
 <-0.6mm,-0.44mm>*{};<-8mm,-5mm>*{}**@{-},
 <-0.4mm,-0.7mm>*{};<-4.5mm,-5mm>*{}**@{-},
 <0mm,0mm>*{};<-1mm,-5mm>*{\ldots}**@{},
 <0.4mm,-0.7mm>*{};<4.5mm,-5mm>*{}**@{-},
 <0.6mm,-0.44mm>*{};<8mm,-5mm>*{}**@{-},
   <0mm,0mm>*{};<-8.5mm,-6.9mm>*{^1}**@{},
   <0mm,0mm>*{};<-5mm,-6.9mm>*{^2}**@{},
   <0mm,0mm>*{};<4.5mm,-6.9mm>*{^{n\hspace{-0.5mm}-\hspace{-0.5mm}1}}**@{},
   <0mm,0mm>*{};<9.0mm,-6.9mm>*{^n}**@{},
 \end{xy}}\Ea \ \ \forall \sigma\in \bS_m, \forall\tau\in \bS_n
$$
and has the differential
given on the generators by
\Beq\label{LBk_infty}
\delta
\Ba{c}\resizebox{14mm}{!}{\begin{xy}
 <0mm,0mm>*{\circ};<0mm,0mm>*{}**@{},
 <-0.6mm,0.44mm>*{};<-8mm,5mm>*{}**@{-},
 <-0.4mm,0.7mm>*{};<-4.5mm,5mm>*{}**@{-},
 <0mm,0mm>*{};<-1mm,5mm>*{\ldots}**@{},
 <0.4mm,0.7mm>*{};<4.5mm,5mm>*{}**@{-},
 <0.6mm,0.44mm>*{};<8mm,5mm>*{}**@{-},
   <0mm,0mm>*{};<-8.5mm,5.5mm>*{^1}**@{},
   <0mm,0mm>*{};<-5mm,5.5mm>*{^2}**@{},
   <0mm,0mm>*{};<4.5mm,5.5mm>*{^{m\hspace{-0.5mm}-\hspace{-0.5mm}1}}**@{},
   <0mm,0mm>*{};<9.0mm,5.5mm>*{^m}**@{},
 <-0.6mm,-0.44mm>*{};<-8mm,-5mm>*{}**@{-},
 <-0.4mm,-0.7mm>*{};<-4.5mm,-5mm>*{}**@{-},
 <0mm,0mm>*{};<-1mm,-5mm>*{\ldots}**@{},
 <0.4mm,-0.7mm>*{};<4.5mm,-5mm>*{}**@{-},
 <0.6mm,-0.44mm>*{};<8mm,-5mm>*{}**@{-},
   <0mm,0mm>*{};<-8.5mm,-6.9mm>*{^1}**@{},
   <0mm,0mm>*{};<-5mm,-6.9mm>*{^2}**@{},
   <0mm,0mm>*{};<4.5mm,-6.9mm>*{^{n\hspace{-0.5mm}-\hspace{-0.5mm}1}}**@{},
   <0mm,0mm>*{};<9.0mm,-6.9mm>*{^n}**@{},
 \end{xy}}\Ea
\ \ = \ \
 \sum_{[1,\ldots,m]=I_1\sqcup I_2\atop
 {|I_1|\geq 0, |I_2|\geq 1}}
 \sum_{[1,\ldots,n]=J_1\sqcup J_2\atop
 {|J_1|\geq 1, |J_2|\geq 1}
}\hspace{0mm}
\pm
\Ba{c}\resizebox{22mm}{!}{ \begin{xy}
 <0mm,0mm>*{\circ};<0mm,0mm>*{}**@{},
 <-0.6mm,0.44mm>*{};<-8mm,5mm>*{}**@{-},
 <-0.4mm,0.7mm>*{};<-4.5mm,5mm>*{}**@{-},
 <0mm,0mm>*{};<0mm,5mm>*{\ldots}**@{},
 <0.4mm,0.7mm>*{};<4.5mm,5mm>*{}**@{-},
 <0.6mm,0.44mm>*{};<12.4mm,4.8mm>*{}**@{-},
     <0mm,0mm>*{};<-2mm,7mm>*{\overbrace{\ \ \ \ \ \ \ \ \ \ \ \ }}**@{},
     <0mm,0mm>*{};<-2mm,9mm>*{^{I_1}}**@{},
 <-0.6mm,-0.44mm>*{};<-8mm,-5mm>*{}**@{-},
 <-0.4mm,-0.7mm>*{};<-4.5mm,-5mm>*{}**@{-},
 <0mm,0mm>*{};<-1mm,-5mm>*{\ldots}**@{},
 <0.4mm,-0.7mm>*{};<4.5mm,-5mm>*{}**@{-},
 <0.6mm,-0.44mm>*{};<8mm,-5mm>*{}**@{-},
      <0mm,0mm>*{};<0mm,-7mm>*{\underbrace{\ \ \ \ \ \ \ \ \ \ \ \ \ \ \
      }}**@{},
      <0mm,0mm>*{};<0mm,-10.6mm>*{_{J_1}}**@{},
 <13mm,5mm>*{};<13mm,5mm>*{\circ}**@{},
 <12.6mm,5.44mm>*{};<5mm,10mm>*{}**@{-},
 <12.6mm,5.7mm>*{};<8.5mm,10mm>*{}**@{-},
 <13mm,5mm>*{};<13mm,10mm>*{\ldots}**@{},
 <13.4mm,5.7mm>*{};<16.5mm,10mm>*{}**@{-},
 <13.6mm,5.44mm>*{};<20mm,10mm>*{}**@{-},
      <13mm,5mm>*{};<13mm,12mm>*{\overbrace{\ \ \ \ \ \ \ \ \ \ \ \ \ \ }}**@{},
      <13mm,5mm>*{};<13mm,14mm>*{^{I_2}}**@{},
 <12.4mm,4.3mm>*{};<8mm,0mm>*{}**@{-},
 <12.6mm,4.3mm>*{};<12mm,0mm>*{\ldots}**@{},
 <13.4mm,4.5mm>*{};<16.5mm,0mm>*{}**@{-},
 <13.6mm,4.8mm>*{};<20mm,0mm>*{}**@{-},
     <13mm,5mm>*{};<14.3mm,-2mm>*{\underbrace{\ \ \ \ \ \ \ \ \ \ \ }}**@{},
     <13mm,5mm>*{};<14.3mm,-4.5mm>*{_{J_2}}**@{},
 \end{xy}}\Ea
\Eeq
where the signs on the r.h.s\ are uniquely fixed for $c+d\in 2\Z$ by the fact that they all equal to $+1$ if $ c$ and $d$ are odd integers, and for $c+d\in 2\Z+1$ the signs are given explicitly in
\cite{Me1,Me3}.

\sip

It was proven in \cite{CMW} that the properad $\LoBcd$ is also Koszul and that its minimal resolution $\HoLoBcd$ is
 is a free properad generated
by the following (skew)symmetric corollas of degree $1+c(1-m-a)+d(1-n-a)$
$$
\Ba{c}\resizebox{16mm}{!}{\xy
(-9,-6)*{};
(0,0)*+{a}*\cir{}
**\dir{-};
(-5,-6)*{};
(0,0)*+{a}*\cir{}
**\dir{-};
(9,-6)*{};
(0,0)*+{a}*\cir{}
**\dir{-};
(5,-6)*{};
(0,0)*+{a}*\cir{}
**\dir{-};
(0,-6)*{\ldots};
(-10,-8)*{_1};
(-6,-8)*{_2};
(10,-8)*{_n};
(-9,6)*{};
(0,0)*+{a}*\cir{}
**\dir{-};
(-5,6)*{};
(0,0)*+{a}*\cir{}
**\dir{-};
(9,6)*{};
(0,0)*+{a}*\cir{}
**\dir{-};
(5,6)*{};
(0,0)*+{a}*\cir{}
**\dir{-};
(0,6)*{\ldots};
(-10,8)*{_1};
(-6,8)*{_2};
(10,8)*{_m};
\endxy}\Ea
=(-1)^{(d+1)(\sigma+\tau)}
\Ba{c}\resizebox{20mm}{!}{\xy
(-9,-6)*{};
(0,0)*+{a}*\cir{}
**\dir{-};
(-5,-6)*{};
(0,0)*+{a}*\cir{}
**\dir{-};
(9,-6)*{};
(0,0)*+{a}*\cir{}
**\dir{-};
(5,-6)*{};
(0,0)*+{a}*\cir{}
**\dir{-};
(0,-6)*{\ldots};
(-12,-8)*{_{\tau(1)}};
(-6,-8)*{_{\tau(2)}};
(12,-8)*{_{\tau(n)}};
(-9,6)*{};
(0,0)*+{a}*\cir{}
**\dir{-};
(-5,6)*{};
(0,0)*+{a}*\cir{}
**\dir{-};
(9,6)*{};
(0,0)*+{a}*\cir{}
**\dir{-};
(5,6)*{};
(0,0)*+{a}*\cir{}
**\dir{-};
(0,6)*{\ldots};
(-12,8)*{_{\sigma(1)}};
(-6,8)*{_{\sigma(2)}};
(12,8)*{_{\sigma(m)}};
\endxy}\Ea\ \ \ \forall \sigma\in \bS_m, \forall \tau\in \bS_n,
$$
where $m+n+ a\geq 3$, $m\geq 1$, $n\geq 1$, $a\geq 0$. The differential in
$\HoLoBd$ is given on the generators by
\Beq\label{2: d on Lie inv infty}
\delta
\Ba{c}\resizebox{16mm}{!}{\xy
(-9,-6)*{};
(0,0)*+{a}*\cir{}
**\dir{-};
(-5,-6)*{};
(0,0)*+{a}*\cir{}
**\dir{-};
(9,-6)*{};
(0,0)*+{a}*\cir{}
**\dir{-};
(5,-6)*{};
(0,0)*+{a}*\cir{}
**\dir{-};
(0,-6)*{\ldots};
(-10,-8)*{_1};
(-6,-8)*{_2};
(10,-8)*{_n};
(-9,6)*{};
(0,0)*+{a}*\cir{}
**\dir{-};
(-5,6)*{};
(0,0)*+{a}*\cir{}
**\dir{-};
(9,6)*{};
(0,0)*+{a}*\cir{}
**\dir{-};
(5,6)*{};
(0,0)*+{a}*\cir{}
**\dir{-};
(0,6)*{\ldots};
(-10,8)*{_1};
(-6,8)*{_2};
(10,8)*{_m};
\endxy}\Ea
=
\sum_{l\geq 1}\sum_{a=b+c+l-1}\sum_{[m]=I_1\sqcup I_2\atop
[n]=J_1\sqcup J_2} \pm
\Ba{c}
%
%
\Ba{c}\resizebox{21mm}{!}{\xy
(0,0)*+{b}*\cir{}="b",
(10,10)*+{c}*\cir{}="c",
%
(-9,6)*{}="1",
(-7,6)*{}="2",
(-2,6)*{}="3",
(-3.5,5)*{...},
(-4,-6)*{}="-1",
(-2,-6)*{}="-2",
(4,-6)*{}="-3",
(1,-5)*{...},
(0,-8)*{\underbrace{\ \ \ \ \ \ \ \ }},
(0,-11)*{_{J_1}},
(-6,8)*{\overbrace{ \ \ \ \ \ \ }},
(-6,11)*{_{I_1}},
(6,16)*{}="1'",
(8,16)*{}="2'",
(14,16)*{}="3'",
(11,15)*{...},
(11,6)*{}="-1'",
(16,6)*{}="-2'",
(18,6)*{}="-3'",
(13.5,6)*{...},
(15,4)*{\underbrace{\ \ \ \ \ \ \ }},
(15,1)*{_{J_2}},
(10,18)*{\overbrace{ \ \ \ \ \ \ \ \ }},
(10,21)*{_{I_2}},
%
(0,2)*-{};(8.0,10.0)*-{}
**\crv{(0,10)};
(0.5,1.8)*-{};(8.5,9.0)*-{}
**\crv{(0.4,7)};
(1.5,0.5)*-{};(9.1,8.5)*-{}
**\crv{(5,1)};
(1.7,0.0)*-{};(9.5,8.6)*-{}
**\crv{(6,-1)};
(5,5)*+{...};
\ar @{-} "b";"1" <0pt>
\ar @{-} "b";"2" <0pt>
\ar @{-} "b";"3" <0pt>
\ar @{-} "b";"-1" <0pt>
\ar @{-} "b";"-2" <0pt>
\ar @{-} "b";"-3" <0pt>
\ar @{-} "c";"1'" <0pt>
\ar @{-} "c";"2'" <0pt>
\ar @{-} "c";"3'" <0pt>
\ar @{-} "c";"-1'" <0pt>
\ar @{-} "c";"-2'" <0pt>
\ar @{-} "c";"-3'" <0pt>
\endxy}\Ea
\Ea
\Eeq
where the summation parameter $l$ counts the number of internal edges connecting the two vertices
on the r.h.s., and the signs are  fixed by the fact that they all equal to $+1$ for $c$ and $d$
odd intergers.

\sip

The case $c=d$ is of special interest because of the following basic example.

\subsection{An example of involutive Lie bialgebra}\label{2: subsection on cyclic words}
Let $W$ be a finite dimensional graded vector space equipped over a field $\K$ of
characteristic zero equipped with a degree $1-d$ (not necessarily non-degenerate) pairing,
$$
\Ba{rccc}
\Theta: & W\ot W & \lon & \K[1-d] \\
  &  w_1\ot w_2 & \lon & \Theta(w_1,w_2)
\Ea
$$
satisfying the (skew) symmetry condition,
\Beq\label{2: skewsymmetry on scalar product}
\Theta(w_1,w_2)=(-1)^{d+|w_1||w_2|}\Theta(w_2,w_1).
\Eeq
If $\Theta$ is non-degenerate, then for $d$ odd it makes $W$ into a symplectic vector
space, while for $d$
even it defines an odd symmetric metric on $W$. The associated
vector space of ``cyclic words in $W$",
$$
Cyc^\bu(W):=\sum_{n\geq 0} (W^{\ot n})^{\Z_n},
$$
admits a canonical representation of the operad $\LoB_{d,d}$ given by the following
well-known proposition (see, e.g., \cite{Ch} and references cited there).

\begin{proposition}\label{2: Proposition on inv LieBi in CycW}  The formulae
$$
[(w_1\ot...\ot w_n)_{\Z_n}, (v_1\ot ...\ot v_m)^{\Z_n}]:=\hspace{100mm}
$$
$$
\hspace{20mm} \sum_{i=1}^n\sum_{j=1}^m
 \pm
\Theta(w_i,w_j) (w_1\ot ...\ot  w_{i-1}\ot v_{j+1}\ot ... \ot v_m\ot v_1\ot ... \ot v_{j-1}\ot w_{i+1}\ot\ldots\ot w_n)^{\Z_{n+m-2}}
$$
$$
\vartriangle (w_1\ot\ldots\ot w_n)_{\Z_n}:=\sum_{i\neq j}
\pm\Theta(w_i,w_j)(w_{i+1}\ot ...\ot w_{j-1})^{\Z_{j-i-1}}\bigotimes
(w_{j+1}\ot ...\ot w_{i-1})^{\Z_{n-j+i-1}}    \hspace{15mm}
$$
where $\pm$ stands for standard Koszul sign,
make $Cyc^\bu(W)$ into a $\LoB_{d,d}$-algebra.
\end{proposition}

We shall show in \S~{\ref{4 remark on proof of LieBi on CycW}} below a new short and elementary
proof of this proposition with the help
of a certain properad of ribbon graphs $\cR \cG ra_d$ which comes equipped with the canonical representation
in $Cyc^\bu(W)$.

\subsection{Deformation complexes of (involutive) Lie bialgebras and Moyal type brackets}
\label{2: DefTh invoLieB and Moyal}  According to the general theory \cite{MV}, the set of $\HoLoBcd$-algebra structures in a dg
vector space $V$ can be identified with the set  Maurer-Cartan elements
  of a  graded Lie algebra,
\Beq\label{2: g_V diamond}
\fg_V^\diamond:=\Hom(V,V)[1]\ \oplus \ \Def\left(\HoLoBcd \stackrel{0}{\rar} \cE nd_V\right),
\Eeq
which controls deformations of the zero morphism from $\HoLoBcd$ to the endomorphism
prop  $\cE nd_V$; the summand $\Hom(V,V)[1]$ takes care about deformations of the differential in $V$. As a $\Z$-graded vector space,
\Beqrn
\fg_V^\diamond[-c-d]&=&\prod_{a\geq 0, m,n\geq 1} \Hom_{\bS_m\times \bS_n}\left(\sgn_m^{|c|}\ot\sgn_n^{|d|}[c(m+a-1)+d(n+a-1)],
V^{\ot m}\ot (V^*)^{\ot  n}\right)[-c-d]\\
&=& \prod_{a\geq0, m,n\geq 1} \odot^m(V[-c])\hat{\ot}  \odot^n(V^*[-d])[-(c+d)a]\\
&=& \displaystyle  \widehat{\odot^{\bu\geq 1}}\left(V[-c]\, \oplus\, V^*[-d]\right)[[\hbar]]
\Eeqrn
where $ \hbar$ is a formal  parameter\footnote{For a vector space $W$ we denote by $W[[\hbar]]$
 the vector space of formal power series in $\hbar$ with coefficients in $W$. For later use we denote by $\hbar^m W[[\hbar]]$ the subspace of $W[[\hbar]]$ spanned by series of the form
 $\hbar^m f$ for some $f\in W[[\hbar]]$.}
 of degree $c+d$. Note that $\fg_V^\diamond[-c-d]$
is a graded commutative algebra.
The Lie brackets $\{\ ,\ \}$ in $\fg_V^\diamond$ can be read, according to \cite{MV}, from the formula (\ref{2: d on Lie inv infty}) for the differential
in the minimal resolution $\HoLoBcd$. They are best described as a degree $-(c+d)$ linear map,
$$
\Ba{rccc}
 \{\ ,\ \}: & \fg_V^\diamond[-c-d] \ot \fg_V^\diamond[-c-d]  & \lon & \fg_V^\diamond[-c-d][-c-d]\\
& (f,g) &\lon & \{f,g\}
\Ea
$$
which in turn is best described in terms of a $\K[[\hbar]]$-linear associative non-commutative product $\star_\hbar$
in $\widehat{\odot^{\bu\geq 1}}\left(V[-c] + V^*[-d]\right)[[\hbar]]$ which deforms the standard graded commutative product and is
given on arbitrary polynomials,
$$
v_Iv'_J:=v_{i_1}v_{i_2}\cdots v_{\# I}v'_{j_1}v'_{j_2}\cdots v'_{j_{\# J}} \in \odot^{\# I}(V[-c])\ot
  \odot^{\# J}(V^*[-d]),
$$
as follows
\Beqrn
(v_Iv'_{J}) \star_\hbar (v_Kv'_{L})&:= &(-1)^{|v'_J||v_k|}v_Iv_Kv'_{J}v'_{L} + \\
&& \sum_{k=1}^{\max\{\# I, \# L\}} \frac{\hbar^{k}}{k!}\sum_{s:[k]\hook I\atop
t:[k]\rar L} (-1)^\var \left( \prod_{i=1}^k \langle v_{s(i)}, v'_{t(i)}\rangle \right) v_{I\setminus s([k])} v_K v'_{J} v'_{L\setminus t([k])}
\Eeqrn
where $(-1)^\var$ is the standard Koszul sign, and the second summation runs over all injections
of the set $[k]$ into sets $I$ and $L$. The associativity of the product $\star_\hbar$  implies that the degree $-(c+d)$ brackets
\Beq\label{2: brackets in g_V diamond}
\{f,g\}:= \frac{f \star_\hbar g - (-1)^{|f||g|} g \star_\hbar h}{\hbar}
\Eeq
 make $\fg_V^\diamond$ into a (degree shifted) Lie algebra
which is precisely the one which  controls the deformation theory of the zero morphism
of props  $\HoLoBcd {\rar} \cE nd_V$.
Hence { Maurer-Cartan elements of this Lie algebra, that is,  degree $c+d+1$ elements $\ga$
in  $\displaystyle  \widehat{\odot^{\bu\geq 1}}\left(V[-c]\, \oplus\, V^*[-d]\right)[[\hbar]]$ satisfying the equation
$$
\{\ga,\ga\}=0
$$
are in one-to-one correspondence with $\HoLoBcd$ structures in $V$.}


\sip

As usual for the star products, the limit
$$
\{f,g\}_0:=\lim_{\hbar\rar 0} \frac{f \star_\hbar g - (-1)^{|f||g|} g \star_\hbar h}{\hbar}
$$
 defines a  Poisson structure of degree $-(c+d)$. The Lie algebra
$
 \left(\widehat{\odot^{\bu\geq 1}}\left(V[-c] + V^*[-d]\right), \ \{\ ,\ \}_0\right)
$
is precisely the degree shifted deformation complex  $\fg_V[-c-d]$,
\Beq\label{2: g_V}
\fg_V:=\Hom(V,V)[1]\ \oplus \ \Def\left(\HoLBcd \stackrel{0}{\lon} \cE nd_V\right),
\Eeq
and its MC elements are in 1-1 correspondence with $\HoLBcd$ algebra structures in $V$.

\subsubsection{\bf A compact description of $\HoLoBcd$ algebras}\label{2: deformation theory of inv Lie bialgebras}
The fact that the deformation complex (\ref{2: g_V diamond}) is most naturally described in terms of formal power series in $\hbar$
prompts us to consider properads in the symmetric monoidal category
of topological modules over the graded algebra $\K[[\hbar]]$ with $|\hbar|=c+d$. Consider
the following formal power series in
$\HoLoBcd[[\hbar]]$,
\Beq\label{2: formal power series of generators}
\Ba{c}\resizebox{15mm}{!}{\begin{xy}
 <0mm,0mm>*{\bu};<0mm,0mm>*{}**@{},
 <-0.6mm,0.44mm>*{};<-8mm,5mm>*{}**@{-},
 <-0.4mm,0.7mm>*{};<-4.5mm,5mm>*{}**@{-},
 <0mm,0mm>*{};<-1mm,5mm>*{\ldots}**@{},
 <0.4mm,0.7mm>*{};<4.5mm,5mm>*{}**@{-},
 <0.6mm,0.44mm>*{};<8mm,5mm>*{}**@{-},
   <0mm,0mm>*{};<-8.5mm,5.5mm>*{^1}**@{},
   <0mm,0mm>*{};<-5mm,5.5mm>*{^2}**@{},
   <0mm,0mm>*{};<4.5mm,5.5mm>*{^{m\hspace{-0.5mm}-\hspace{-0.5mm}1}}**@{},
   <0mm,0mm>*{};<9.0mm,5.5mm>*{^m}**@{},
 <-0.6mm,-0.44mm>*{};<-8mm,-5mm>*{}**@{-},
 <-0.4mm,-0.7mm>*{};<-4.5mm,-5mm>*{}**@{-},
 <0mm,0mm>*{};<-1mm,-5mm>*{\ldots}**@{},
 <0.4mm,-0.7mm>*{};<4.5mm,-5mm>*{}**@{-},
 <0.6mm,-0.44mm>*{};<8mm,-5mm>*{}**@{-},
   <0mm,0mm>*{};<-8.5mm,-6.9mm>*{^1}**@{},
   <0mm,0mm>*{};<-5mm,-6.9mm>*{^2}**@{},
   <0mm,0mm>*{};<4.5mm,-6.9mm>*{^{n\hspace{-0.5mm}-\hspace{-0.5mm}1}}**@{},
   <0mm,0mm>*{};<9.0mm,-6.9mm>*{^n}**@{},
 \end{xy}}\Ea
 :=
 \sum_{a\geq 3-m-n\atop
 a\geq 0}^\infty \hbar^{a}
\Ba{c}\resizebox{16mm}{!}{ \xy
(-9,-6)*{};
(0,0)*+{a}*\cir{}
**\dir{-};
(-5,-6)*{};
(0,0)*+{a}*\cir{}
**\dir{-};
(9,-6)*{};
(0,0)*+{a}*\cir{}
**\dir{-};
(5,-6)*{};
(0,0)*+{a}*\cir{}
**\dir{-};
(0,-6)*{\ldots};
(-10,-8)*{_1};
(-6,-8)*{_2};
(10,-8)*{_m};
(-9,6)*{};
(0,0)*+{a}*\cir{}
**\dir{-};
(-5,6)*{};
(0,0)*+{a}*\cir{}
**\dir{-};
(9,6)*{};
(0,0)*+{a}*\cir{}
**\dir{-};
(5,6)*{};
(0,0)*+{a}*\cir{}
**\dir{-};
(0,6)*{\ldots};
(-10,8)*{_1};
(-6,8)*{_2};
(10,8)*{_n};
\endxy}\Ea
\ \ \ \ \ \ \ \ \ \ \ \ \ \mbox{for} \ \ m,n\geq 1.
\Eeq
These are homogeneous elements in  $\HoLoBcd[[\hbar]]$ of degree
$1+c(1-m)+d(1-n)$. Note that
$$
\begin{xy}
 <0mm,0mm>*{\bu};
 <0mm,-0.5mm>*{};<0mm,-3mm>*{}**@{-},
 <0mm,0.5mm>*{};<0mm,3mm>*{}**@{-},
 \end{xy}:= \sum_{a=1}^\infty \hbar^{a}
 \xy
(0,-4)*{};
(0,0)*+{a}*\cir{}
**\dir{-};
(0,4)*{};
(0,0)*+{a}*\cir{}
**\dir{-};
\endxy \in \hbar \HoLoBcd[[\hbar]]
$$
The differential (\ref{2: d on Lie inv infty}) takes now the form of the formal power
series in $\hbar$,
\Beq\label{2: d_hbar}
\delta
\Ba{c}\resizebox{15mm}{!}{\xy
(-9,-6)*{};
(0,0)*+{\bu}
**\dir{-};
(-5,-6)*{};
(0,0)*+{}
**\dir{-};
(9,-6)*{};
(0,0)*+{}
**\dir{-};
(5,-6)*{};
(0,0)*+{}
**\dir{-};
(0,-6)*{\ldots};
(-10,-8)*{_1};
(-6,-8)*{_2};
(10,-8)*{_m};
(-9,6)*{};
(0,0)*+{}
**\dir{-};
(-5,6)*{};
(0,0)*+{}
**\dir{-};
(9,6)*{};
(0,0)*+{}
**\dir{-};
(5,6)*{};
(0,0)*+{}
**\dir{-};
(0,6)*{\ldots};
(-10,8)*{_1};
(-6,8)*{_2};
(10,8)*{_n};
\endxy}\Ea
=
\sum_{l\geq 1}\sum_{[m]=I_1\sqcup I_2\atop
[n]=J_1\sqcup J_2} \hbar^{l-1}(-1)^\var
\Ba{c}
%
%
\Ba{c}\resizebox{22mm}{!}{\xy
(0,0)*+{\bu}="b",
(10,10)*+{\bu}="c",
%
(-9,6)*{}="1",
(-7,6)*{}="2",
(-2,6)*{}="3",
(-3.5,5)*{...},
(-4,-6)*{}="-1",
(-2,-6)*{}="-2",
(4,-6)*{}="-3",
(1,-5)*{...},
(0,-8)*{\underbrace{\ \ \ \ \ \ \ \ }},
(0,-11)*{_{J_1}},
(-6,8)*{\overbrace{ \ \ \ \ \ \ }},
(-6,11)*{_{I_1}},
(6,16)*{}="1'",
(8,16)*{}="2'",
(14,16)*{}="3'",
(11,15)*{...},
(11,6)*{}="-1'",
(16,6)*{}="-2'",
(18,6)*{}="-3'",
(13.5,6)*{...},
(15,4)*{\underbrace{\ \ \ \ \ \ \ }},
(15,1)*{_{J_2}},
(10,18)*{\overbrace{ \ \ \ \ \ \ \ \ }},
(10,21)*{_{I_2}},
%
(0,2)*-{};(8.0,10.0)*-{}
**\crv{(0,10)};
(0.5,1.8)*-{};(8.5,9.0)*-{}
**\crv{(0.4,7)};
(1.5,0.5)*-{};(9.1,8.5)*-{}
**\crv{(5,1)};
(1.7,0.0)*-{};(9.5,8.6)*-{}
**\crv{(6,-1)};
(5,5)*+{...};
\ar @{-} "b";"1" <0pt>
\ar @{-} "b";"2" <0pt>
\ar @{-} "b";"3" <0pt>
\ar @{-} "b";"-1" <0pt>
\ar @{-} "b";"-2" <0pt>
\ar @{-} "b";"-3" <0pt>
\ar @{-} "c";"1'" <0pt>
\ar @{-} "c";"2'" <0pt>
\ar @{-} "c";"3'" <0pt>
\ar @{-} "c";"-1'" <0pt>
\ar @{-} "c";"-2'" <0pt>
\ar @{-} "c";"-3'" <0pt>
\endxy}\Ea
\Ea
\Eeq
where $l$ counts the number of internal edges connecting two vertices on the r.h.s. \mip

Let $\HoLhBcd$ be a dg properad in the category of topological $\K[[\hbar]]$
modules which is freely generated over $\K[[\hbar]]$ by corollas (\ref{2: formal power series of generators}),
and is equipped with differential (\ref{2: d_hbar}). Then we have the following obvious facts
\Bi
\item[(i)] There is a canonical morphism of dg properads
$$
\Ba{rccc}
 &  \HoLhBcd  & \lon & \HoLoBcd[[\hbar]]\vspace{2mm}\\
            &
            \Ba{c}\resizebox{14mm}{!}{\begin{xy}
 <0mm,0mm>*{\bu};<0mm,0mm>*{}**@{},
 <-0.6mm,0.44mm>*{};<-8mm,5mm>*{}**@{-},
 <-0.4mm,0.7mm>*{};<-4.5mm,5mm>*{}**@{-},
 <0mm,0mm>*{};<-1mm,5mm>*{\ldots}**@{},
 <0.4mm,0.7mm>*{};<4.5mm,5mm>*{}**@{-},
 <0.6mm,0.44mm>*{};<8mm,5mm>*{}**@{-},
   <0mm,0mm>*{};<-8.5mm,5.5mm>*{^1}**@{},
   <0mm,0mm>*{};<-5mm,5.5mm>*{^2}**@{},
   <0mm,0mm>*{};<4.5mm,5.5mm>*{^{m\hspace{-0.5mm}-\hspace{-0.5mm}1}}**@{},
   <0mm,0mm>*{};<9.0mm,5.5mm>*{^m}**@{},
 <-0.6mm,-0.44mm>*{};<-8mm,-5mm>*{}**@{-},
 <-0.4mm,-0.7mm>*{};<-4.5mm,-5mm>*{}**@{-},
 <0mm,0mm>*{};<-1mm,-5mm>*{\ldots}**@{},
 <0.4mm,-0.7mm>*{};<4.5mm,-5mm>*{}**@{-},
 <0.6mm,-0.44mm>*{};<8mm,-5mm>*{}**@{-},
   <0mm,0mm>*{};<-8.5mm,-6.9mm>*{^1}**@{},
   <0mm,0mm>*{};<-5mm,-6.9mm>*{^2}**@{},
   <0mm,0mm>*{};<4.5mm,-6.9mm>*{^{n\hspace{-0.5mm}-\hspace{-0.5mm}1}}**@{},
   <0mm,0mm>*{};<9.0mm,-6.9mm>*{^n}**@{},
 \end{xy}}\Ea & \lon &
\displaystyle \sum_{a\geq 3-m-n\atop
 a\geq 0}^\infty \hbar^{a}
\resizebox{15mm}{!}{\mbox{\xy
(-9,-6)*{};
(0,0)*+{a}*\cir{}
**\dir{-};
(-5,-6)*{};
(0,0)*+{a}*\cir{}
**\dir{-};
(9,-6)*{};
(0,0)*+{a}*\cir{}
**\dir{-};
(5,-6)*{};
(0,0)*+{a}*\cir{}
**\dir{-};
(0,-6)*{\ldots};
(-10,-8)*{_1};
(-6,-8)*{_2};
(10,-8)*{_m};
(-9,6)*{};
(0,0)*+{a}*\cir{}
**\dir{-};
(-5,6)*{};
(0,0)*+{a}*\cir{}
**\dir{-};
(9,6)*{};
(0,0)*+{a}*\cir{}
**\dir{-};
(5,6)*{};
(0,0)*+{a}*\cir{}
**\dir{-};
(0,6)*{\ldots};
(-10,8)*{_1};
(-6,8)*{_2};
(10,8)*{_n};
\endxy}}
\Ea
$$

\item[(ii)] There is a 1-1 correspondence between morphisms of dg properads
$
\rho: \HoLoBcd \lon \cP
$
in the category of graded vector spaces over $\K$, and morphisms
of dg properads
$
\rho_\hbar: \HoLhBcd \lon \cP[[\hbar]]
$
in the category of topological $\K[[\hbar]]$-modules satisfying the condition
$
\rho_\hbar\left(\begin{xy}
 <0mm,0mm>*{\bu};
 <0mm,-0.5mm>*{};<0mm,-2.5mm>*{}**@{-},
 <0mm,0.5mm>*{};<0mm,2.5mm>*{}**@{-},
 \end{xy}\right)\in \hbar\cP[[\hbar]]
$
(Let us call such continuous maps $\rho_\hbar$  {\em admissible}).

\item[(iii)] The deformation complex (\ref{2: g_V diamond}) can be equivalently rewritten as $\fg_V^\diamond=\Hom(V,V)[1] \oplus \Def_{adm}\left(\HoLhBcd \stackrel{0}{\lon} \cE nd_V[[\hbar]]\right)$, where the second summand describes deformations of the zero morphism only in the class of admissible morphisms.
\Ei

The properad $\HoLhBcd$ has much fewer generators than $\HoLoBcd$ so that working with the former makes sometimes presentation shorter.

\subsection{An action of the Grothendieck-Teichm\"uller group on (involutive) Lie bialgebras}
There is a highly non-trivial action of the oriented versions of the  Kontsevich graph complexes $\mathsf{GC_{c+d}}$
on the dg  props $\HoLBcd$ and $\HoLoBcd$  which was discovered in \cite{Wi2, MW2}, and which plays a very important role in this paper.

\subsubsection{\bf Kontsevich graph complexes} Let $G_{n,l}$ be the set of directed graphs $\Ga$ with vertex set $V(\Ga)=\{1, \ldots , n\}$
and the set of directed edges $E(\Ga)=\{1, \ldots , k\}$, i.e.\ we assume that
both vertices and edges of $\Ga$ carry numerical labels. There is
a natural right action of the group $\bS_n \times  \bS_l \ltimes
 (\bS_2)^l$ on this set, with $\bS_n$ acting by relabeling the vertices, $\bS_l$ by relabeling the
edges and the $(\bS_2)^l$ via the reversing the directions of the edges.
For each fixed integer $d$, consider a collection of $\bS_n$-modules,
$$
\cG ra_{d}=\left\{\cG ra_d(n):= \prod_{l\geq 0} \K \langle G_{n,l}\rangle \ot_{ \bS_l \ltimes
 (\bS_2)^l}  \sgn_l^{|d|}\ot \sgn_2^{l|d-1|} [l(d-1)]   \right\}_{n\geq 1}
$$
Thus an element of $\cG ra_d(n)$ can be understood as a pair $(\Ga, or)$, where
$\Ga$ is a {\em undirected}\, graph with a fixed bijection $V(\Ga)\rar [n]$,
and $or$ is a choice of {\em orientation}\, in $\Ga$ which depends on the parity of $d$:
for $d$ even $or$ is a total ordering of its edges (up to an even permutation)
while for $d$ odd $or$ is a choice  a  direction on each edge of $\Ga$ (up to its reversing  and the simultaneous multiplication of the graph $\Ga$ by $-1$).
For every $d$ and every graph $\Ga$ there are only two possible orientations, $or$ and $or^{opp}$, in $\Ga$,
and we have a relation $(\Ga, or)=-(\Ga,or^{opp})$. In particular, if $\Ga$ admits an automorphism
which swaps orientations, then  $(\Ga, or)=(\Ga,or^{opp})=-(\Ga,or)=0$. We abbreviate from now on $(\Ga, or)$
to $\Ga$ keeping in mind that some orientation is implicitly chosen, for example
\Beq\label{2: examples of graphs from Gra}
\Ba{c}\resizebox{6mm}{!}{\xy
(0,0)*+{_1}*\cir{}="b",
(0,7)*+{_2}*\cir{}="c",
(6,3.5)*+{_3}*\cir{}="a",
\ar @{-} "b";"c" <0pt>
\ar @{-} "b";"a" <0pt>
\ar @{-} "c";"a" <0pt>
\endxy}
\Ea \in \cG ra_{2k}\ \ , \ \
\Ba{c}\resizebox{6mm}{!}{\xy
(0,0)*+{_1}*\cir{}="b",
(0,7)*+{_2}*\cir{}="c",
(6,3.5)*+{_3}*\cir{}="a",
\ar @{->} "b";"c" <0pt>
\ar @{->} "b";"a" <0pt>
\ar @{->} "c";"a" <0pt>
\endxy}
\Ea = -\Ba{c}\resizebox{6mm}{!}{\xy
(0,0)*+{_1}*\cir{}="b",
(0,7)*+{_2}*\cir{}="c",
(6,3.5)*+{_3}*\cir{}="a",
\ar @{->} "b";"c" <0pt>
\ar @{->} "b";"a" <0pt>
\ar @{<-} "c";"a" <0pt>
\endxy}
\Ea \in \cG ra_{2k+1}
\Eeq
The homological degree of a graph from $\cG ra_{d}$ is given by assigning degree $1-d$ to each edge. For example, all the graphs shown just above have degree $3(1-d)$.

\sip

The $\bS$-module $\cG ra_d$ is an operad. The operadic composition,
$$
\Ba{rccc}
\circ_i: &  \cG ra_d(n) \times \cG ra_d(m) &\lon & \cG ra_d(m+n-1),  \ \ \forall\ i\in [n]\\
         &       (\Ga_1, \Ga_2) &\lon &      \Ga_1\circ_i \Ga_2,
\Ea
$$
is defined by substituting the graph $\Ga_2$ into the $i$-labeled vertex $v_i$ of $\Ga_1$ and taking a sum over  re-attachments of dangling edges (attached before to $v_i$) to vertices of $\Ga_2$
in all possible ways.

\sip

For any operad $\cP=\{\cP(n)\}_{n\geq 1}$  in the category of graded vector spaces,
the linear the map
$$
\Ba{rccc}
[\ ,\ ]:&  \sP \ot \sP & \lon & \sP\\
& (a\in \cP(n), b\in \cP(m)) & \lon &
[a, b]:= \sum_{i=1}^n a\circ_i b - (-1)^{|a||b|}\sum_{i=1}^m b\circ_i a\ \in \cP(m+n-1)
\Ea
$$
makes a graded vector space
$
\sP:= \prod_{n\geq 1}\cP(n)$
into a Lie algebra \cite{KM}; moreover, these brackets induce a Lie algebra structure on the subspace
of invariants
$
\sP^\bS:=  \prod_{n\geq 1}\cP(n)^{\bS_n}$. In particular,
the graded vector space
$$
\mathsf{fGC}_{d}:= \prod_{n\geq 1} \cG ra_{d}(n)^{\bS_n}[d(n-1)]
$$
is a Lie algebra with respect to the above Lie brackets, and as such it can be identified
with the deformation complex $\Def(\caL ie_{d}\stackrel{0}{\rar} \cG ra_{d})$. Hence non-trivial Maurer-Cartan elements $\al$ of $(\mathsf{fGC}_{d}, [\ ,\ ])$ give us non-trivial morphisms of operads
$\al:\caL ie_{d} {\rar} \cG ra_{d}$. One such non-trivial morphism  is easy to detect --- it is given on the generator of $\caL ie_{d}$ by \cite{Wi}
\Beq\label{2: alpha map from Lie to gra}
\al \left(\Ba{c}\begin{xy}
 <0mm,0.66mm>*{};<0mm,3mm>*{}**@{-},
 <0.39mm,-0.39mm>*{};<2.2mm,-2.2mm>*{}**@{-},
 <-0.35mm,-0.35mm>*{};<-2.2mm,-2.2mm>*{}**@{-},
 <0mm,0mm>*{\circ};<0mm,0mm>*{}**@{},
   <0.39mm,-0.39mm>*{};<2.9mm,-4mm>*{^{_2}}**@{},
   <-0.35mm,-0.35mm>*{};<-2.8mm,-4mm>*{^{_1}}**@{},
\end{xy}\Ea\right)=\left\{
\Ba{ll}
\xy
 (0,0)*{\bullet}="a",
(5,0)*{\bu}="b",
\ar @{-} "a";"b" <0pt>
\endxy:=
\Ba{c}\resizebox{6.3mm}{!}{\xy
(0,0)*+{_1}*\cir{}="b",
(8,0)*+{_2}*\cir{}="c",
\ar @{-} "b";"c" <0pt>
\endxy}
\Ea  +
\Ba{c}\resizebox{6.3mm}{!}{\xy
(0,0)*+{_2}*\cir{}="b",
(8,0)*+{_1}*\cir{}="c",
\ar @{-} "b";"c" <0pt>
\endxy}
\Ea & \mathrm{for}\ d\ \mathrm{even}\\
\xy
 (0,0)*{\bullet}="a",
(5,0)*{\bu}="b",
\ar @{->} "a";"b" <0pt>
\endxy:=
\Ba{c}\resizebox{6.3mm}{!}{\xy
(0,0)*+{_1}*\cir{}="b",
(8,0)*+{_2}*\cir{}="c",
\ar @{->} "b";"c" <0pt>
\endxy}
\Ea  -
\Ba{c}\resizebox{7mm}{!}{\xy
(0,0)*+{_2}*\cir{}="b",
(8,0)*+{_1}*\cir{}="c",
\ar @{->} "b";"c" <0pt>
\endxy}
\Ea & \mathrm{for}\ d\ \mathrm{odd}\\
\Ea
\right.
\Eeq
Note that graphs from $\mathsf{fGC}_{d+1}$ have vertices' labels symmetrized (for $d$ even) or skew-symmetrized (for $d$ odd) so that in pictures we can and will omit labels of vertices completely and denote them by black bullets as in the formula above. The above morphism makes
 $(\mathsf{fGC}_{d}, [\ ,\ ])$ into a {\em differential}\, Lie algebra with the differential $\delta:= [\al ,\ ]$. This dg lie algebra has three important subalgebras,
\Bi
\item[(i)] $\mathsf{fcGC}_d \subset \mathsf{fGC}_d$ which is spanned by connected graphs,
\item[(ii)] $\mathsf{GC}_d^2 \subset \mathsf{fcGC}_d$ which is spanned by graphs with at least bivalent vertices, and
\item[(iii)] $\mathsf{GC}_d \subset \mathsf{GC}_d^2$ which is spanned by graphs with at least trivalent vertices.
\Ei
It is obvious that the cohomology $H^\bu(\mathsf{fGC}_{d})$ is completely determined
by $H^\bu(\mathsf{fcGC}_{d})$; it was shown in \cite{Ko2,Wi} that the cohomology
$H^\bu(\mathsf{fcGC}_{d})$ is essentially determined by $H^\bu(\mathsf{GC}_{d})$
as one has the decomposition
$$
H^\bu(\mathsf{fcGC}_{d}) = H^\bu(\mathsf{GC}_{d}^2) = H^\bu(\mathsf{GC}_{d})\ \oplus\ \bigoplus_{j\geq 1\atop j\equiv 2d+1 \mod 4} \K[d-j],
$$
where the summand  $ \K[d-j]$ is generated by the loop-type graph with $j$ binary vertices. Note that the homological degree of graph $\Ga$ from $\mathsf{fGC}_{d}$ is given by
$
|\Ga|=d(\# V(\Ga) -1) + (1-d) \# E(\Ga).
$
It was proven in \cite{Wi} that
$$
H^0(\mathsf{GC}_{2})=\fg\fr\ft
$$
where $\fg\fr\ft$ is the Lie algebra of the Grothendieck-Teichm\"uller group introduced by Drinfeld in the context of deformation quantization of Lie bialgebras \cite{D2}. This group play an important role in many  other areas of mathematics, in particular in the knot theory, in deformation quantizations
of Poisson manifolds \cite{Ko4}  and of Lie bialgebras \cite{EK}, and in the classification theory of solutions of Kashiwara-Vergne problem \cite{AT}.

\subsection{A class of representations of $\cG ra_{d}$ and $ \mathsf{fGC}_{d}$} Let $W$ be a {\em finite}-dimensional vector space equipped with a degree $1-d$ scalar product $\Theta$ satisfying (\ref{2: skewsymmetry on scalar product}).
 Then there is a representation
$$
\Ba{rccc}
\rho: & \cG ra_{d} & \lon & \cE nd_{\odot^\bu W}\\
      & \Ga \in \cG ra_{d}(n) & \lon & \rho_\Ga\in \Hom((\odot^\bu W)^{\ot n}, \odot^\bu W)
      \Ea
$$
given by
$$
\rho_\Ga(f_1 \ot \ldots \ot f_n):=  \mu\left( \prod_{e\in E(\Ga)} \Delta_e (f_1\ot \ldots \ot f_n) \right)
$$
where $\mu: \ot^n( \odot^\bu W) \rar \odot^\bu W$ is the standard multiplication map and, for each edge $e$ connecting vertex labelled by $i\in [n]$ to the vertex labelled by $j\in [n]$, the operator $\Delta_e$ acts only the $i$-th and $j$-th tensor factors $f_i\ot f_j$ as follows
$$
\Ba{rccc}
\Delta_e: & \odot^\bu W \ot \odot^\bu W & \rar & \odot^\bu W\ot \odot^\bu W\\
          & (w_1w_2 \cdots w_p) \ot  (w_1'w_2'\cdots w_q') &\rar & \displaystyle\sum_{k=1}^p\sum_{l=1}^m
 \pm
\Theta(w_k,w_l) w_1\cdots  w_{k-1}w_{k+1}\cdots w_p w'_1\cdots w'_{l-1}w'_{l+1}\cdots w'_q
\Ea
$$
The image of the graph $\xy
 (0,0)*{\bullet}="a",
(5,0)*{\bu}="b",
\ar @{-} "a";"b" <0pt>
\endxy$/ $\xy
 (0,0)*{\bullet}="a",
(5,0)*{\bu}="b",
\ar @{->} "a";"b" <0pt>
\endxy$ under this representations makes the vector space $V:=\odot^\bu W$ into a $\caL ie_{d}$
algebra so that one can consider its Chevalier-Eilenberg complex
$$
CE^\bu(V, V):= \Def( \caL ie_{d}\stackrel{\rho\circ \al}{\lon} \cE nd_V)\simeq  \prod_{n\geq 1} \Hom(\odot^n (V[d]), V[d]).
$$
The representation $\rho$ induces a canonical morphism of dg Lie algebras
$$
 \mathsf{fGC}_{d} \lon CE^\bu(V, V)
$$
which gives us an interpretation of the graph complex  $\mathsf{fGC}_{d}$ as a {\em universal}\, version of $CE^\bu(V, V)$.

\subsection{Oriented graph complexes}\label{2: subsection on operad Gra} Let $G_{n,l}^{or}$ be the subset of $G_{n,l}$
 consisting of directed graphs $\Ga$ with no wheels, that is, directed sequences of edges forming an oriented closed path; let us call such directed graphs {\em oriented}. The $\bS$-module
$$
\cG ra^{or}_{d}=\left\{\cG ra_d(n):= \prod_{l\geq 0} \K \langle G_{n,l}\rangle \ot_{ \bS_l}  \sgn_l^{|d-1|}  \right\}_{n\geq 1}
$$
is an operad which comes with a canonical morphism $\al:\caL ie_{d} {\rar} \cG ra^{or}_{d}$
given by the same formula (\ref{2: alpha map from Lie to gra});
the associated graph complex
$$
\mathsf{fGC}_{d}^{or}:=\Def(\caL ie_{d} \stackrel{\al}{\rar} \cG ra^{or}_{d})
$$
is called the {\em oriented}\, graph complex. It contains a subcomplex $\mathsf{GC}_{d}^{or}$
spanned by connected graphs whose vertices are at least bivalent and which contain no vertices with precisely one input edge and one output edge. It was proven in \cite{Wi2} that that
the cohomology of the oriented graph complex  $\mathsf{GC}_{d+1}^{or}$ is related
 to the cohomology of the ordinary graph complex $\mathsf{fcGC}_d$ as follows,
$$
H^\bu(\mathsf{GC}_{d+1}^{or})=H^\bu(\mathsf{fcGC}_d)=H^\bu(\mathsf{GC}_{d})\ \oplus\ \bigoplus_{j\geq 1\atop j\equiv 2d+1 \mod 4} \K[d-j].
$$
In particular, $H^0( \mathsf{GC}_{3}^{or})=H^0(\mathsf{GC}_2)=\fg\fr\ft$.

\sip

Note that the operad $\cG ra^{or}_{d}$ and the dg Lie algebra $\mathsf{GC}_{d}^{or}$ admit representations in $\odot^\bu W$ and, resp., $CE^\bu(\odot^\bu W,
\odot^\bu W)$ for {\em arbitrary}\, vector spaces $W$ equipped with derree $1-d$ scalar products (\ref{2: skewsymmetry on scalar product}), including {\em infinite}-dimensional ones.

\mip

Consider next a Lie algebra  $\mathsf{GC}_{2d+1}^{or}[[\hbar]]$, where the formal parameter $\hbar$ has homological degree $2d$ and the Lie brackets are defined as the $\K[[\hbar]]$ linear extension
of the standard Lie brackets in $\mathsf{GC}_{2d+1}^{or}$. This Lie algebra admits the following
$MC$ element \cite{CMW}
$$
\Phi_\hbar:= \sum_{k=1}^\infty \hbar^{k-1} \underbrace{
\Ba{c}\resizebox{6mm}{!}  {\xy
(0,5)*{...},
   \ar@/^1pc/(0,0)*{\bullet};(0,10)*{\bullet}
   \ar@/^{-1pc}/(0,0)*{\bullet};(0,10)*{\bullet}
   \ar@/^0.6pc/(0,0)*{\bullet};(0,10)*{\bullet}
   \ar@/^{-0.6pc}/(0,0)*{\bullet};(0,10)*{\bullet}
 \endxy}
 \Ea}_{k\ \mathrm{edges}}
$$
which makes $\mathsf{GC}_{2d+1}^{or}[[\hbar]]$ into a complex with the differential
$$
\Ba{rccc}
\delta_\hbar: & \sG\sC_{2d+1}^{or}[[\hbar]] &\lon & \sG\sC_{2d+1}^{or}[[\hbar]]\\
         &  \Ga &  \lon & \delta_\hbar\Ga:= [\Phi_\hbar,\Ga]_{\mathrm{gra}}
\Ea
$$
 The induced differential  in $\sG\sC_{2d+1}^{or}=\sG\sC_{2d+1}^{or}[[\hbar]]/ \hbar \sG\sC_{2d+1}^{or}[[\hbar]]$ is precisely the
original
differential $\delta$. One has $H^0( \sG\sC_{3}^{or}[[\hbar]],\delta_\hbar)=\fg\fr\ft$ (see \cite{CMW}).

\subsubsection{\bf Actions of oriented graph complexes on genus completions of  $\HoLBcd$ and $\HoLoBcd$}
The properads $\LB_{c,d}$ and $\LoB_{c,d}$ (as well as their minimal resolutions $\HoLBcd$ and $\HoLoBcd$)  are naturally graded by the genus of the graphs. Let us denote by ${\wLBcd}$ and ${\wLoBcd}$ (resp., ${\wHoLBcd}$ and ${\wHoLoBcd}$) their genus completions.
The natural maps ${\wHoLBd} \rar \wLBd$ and ${\wHoLoBd}\rar \wLoBd$ are quasi-isomorphisms.

\sip

There is a natural right action of $\sG\sC_{c+d}^{or}$ on the properad ${\wHoLBcd}$  by continuous properadic derivations \cite{MW2}.
Concretely, for any graph $\Gamma$ one defines the derivation $F(\Gamma)\in \mathrm{Der}({\wHoLBd})$ sending the generating $(m,n)$ corolla of $\wHoLBcd$ to the linear combination of graphs
\Beq \label{equ:def GC action 1}
\left( \Ba{c}\resizebox{14mm}{!}{\begin{xy}
 <0mm,0mm>*{\circ};<0mm,0mm>*{}**@{},
 <-0.6mm,0.44mm>*{};<-8mm,5mm>*{}**@{-},
 <-0.4mm,0.7mm>*{};<-4.5mm,5mm>*{}**@{-},
 <0mm,0mm>*{};<-1mm,5mm>*{\ldots}**@{},
 <0.4mm,0.7mm>*{};<4.5mm,5mm>*{}**@{-},
 <0.6mm,0.44mm>*{};<8mm,5mm>*{}**@{-},
   <0mm,0mm>*{};<-8.5mm,5.5mm>*{^1}**@{},
   <0mm,0mm>*{};<-5mm,5.5mm>*{^2}**@{},
   <0mm,0mm>*{};<4.5mm,5.5mm>*{^{m\hspace{-0.5mm}-\hspace{-0.5mm}1}}**@{},
   <0mm,0mm>*{};<9.0mm,5.5mm>*{^m}**@{},
 <-0.6mm,-0.44mm>*{};<-8mm,-5mm>*{}**@{-},
 <-0.4mm,-0.7mm>*{};<-4.5mm,-5mm>*{}**@{-},
 <0mm,0mm>*{};<-1mm,-5mm>*{\ldots}**@{},
 <0.4mm,-0.7mm>*{};<4.5mm,-5mm>*{}**@{-},
 <0.6mm,-0.44mm>*{};<8mm,-5mm>*{}**@{-},
   <0mm,0mm>*{};<-8.5mm,-6.9mm>*{^1}**@{},
   <0mm,0mm>*{};<-5mm,-6.9mm>*{^2}**@{},
   <0mm,0mm>*{};<4.5mm,-6.9mm>*{^{n\hspace{-0.5mm}-\hspace{-0.5mm}1}}**@{},
   <0mm,0mm>*{};<9.0mm,-6.9mm>*{^n}**@{},
 \end{xy}}\Ea
\right)
\cdot \Gamma=
 \sum
 \Ba{c}\resizebox{10mm}{!}{ \xy
 (-5,7)*{^{_1}},
 (-3,7)*{^{_2}},
 (5.5,7)*{^{_m}},
 (-5,-7.9)*{^{_1}},
 (-3,-7.9)*{^{_2}},
 (5.5,-7.9)*{^{_n}},
(1,4.5)*+{...},
(1,-4.5)*+{...},
(0,0)*+{\Ga}="o",
(-5,6)*{}="1",
(-3,6)*{}="2",
(5,6)*{}="3",
(5,-6)*{}="4",
(-3,-6)*{}="5",
(-5,-6)*{}="6",
\ar @{-} "o";"1" <0pt>
\ar @{-} "o";"2" <0pt>
\ar @{-} "o";"3" <0pt>
\ar @{-} "o";"4" <0pt>
\ar @{-} "o";"5" <0pt>
\ar @{-} "o";"6" <0pt>
\endxy}\Ea
\Eeq
where the sum is taken over all ways of attaching the incoming and outgoing legs such that all vertices are at least trivalent and have at least one incoming and one outgoing edge. This action gives us a morphism of dg Lie algebras  $F : \GC_{c+d+1}^{or}\to \mathrm{Der}(\wHoLBcd)$ which is a quasi-isomorphism up to one class in  $\mathrm{Der}(\wHoLBcd)$ corresponding to the 1-parameter rescaling automorphism
$$
\Ba{c}\resizebox{14mm}{!}{\begin{xy}
 <0mm,0mm>*{\circ};<0mm,0mm>*{}**@{},
 <-0.6mm,0.44mm>*{};<-8mm,5mm>*{}**@{-},
 <-0.4mm,0.7mm>*{};<-4.5mm,5mm>*{}**@{-},
 <0mm,0mm>*{};<-1mm,5mm>*{\ldots}**@{},
 <0.4mm,0.7mm>*{};<4.5mm,5mm>*{}**@{-},
 <0.6mm,0.44mm>*{};<8mm,5mm>*{}**@{-},
   <0mm,0mm>*{};<-8.5mm,5.5mm>*{^1}**@{},
   <0mm,0mm>*{};<-5mm,5.5mm>*{^2}**@{},
   <0mm,0mm>*{};<4.5mm,5.5mm>*{^{m\hspace{-0.5mm}-\hspace{-0.5mm}1}}**@{},
   <0mm,0mm>*{};<9.0mm,5.5mm>*{^m}**@{},
 <-0.6mm,-0.44mm>*{};<-8mm,-5mm>*{}**@{-},
 <-0.4mm,-0.7mm>*{};<-4.5mm,-5mm>*{}**@{-},
 <0mm,0mm>*{};<-1mm,-5mm>*{\ldots}**@{},
 <0.4mm,-0.7mm>*{};<4.5mm,-5mm>*{}**@{-},
 <0.6mm,-0.44mm>*{};<8mm,-5mm>*{}**@{-},
   <0mm,0mm>*{};<-8.5mm,-6.9mm>*{^1}**@{},
   <0mm,0mm>*{};<-5mm,-6.9mm>*{^2}**@{},
   <0mm,0mm>*{};<4.5mm,-6.9mm>*{^{n\hspace{-0.5mm}-\hspace{-0.5mm}1}}**@{},
   <0mm,0mm>*{};<9.0mm,-6.9mm>*{^n}**@{},
 \end{xy}}\Ea
 \lon
 \la^{m+n-2}
 \Ba{c}\resizebox{14mm}{!}{\begin{xy}
 <0mm,0mm>*{\circ};<0mm,0mm>*{}**@{},
 <-0.6mm,0.44mm>*{};<-8mm,5mm>*{}**@{-},
 <-0.4mm,0.7mm>*{};<-4.5mm,5mm>*{}**@{-},
 <0mm,0mm>*{};<-1mm,5mm>*{\ldots}**@{},
 <0.4mm,0.7mm>*{};<4.5mm,5mm>*{}**@{-},
 <0.6mm,0.44mm>*{};<8mm,5mm>*{}**@{-},
   <0mm,0mm>*{};<-8.5mm,5.5mm>*{^1}**@{},
   <0mm,0mm>*{};<-5mm,5.5mm>*{^2}**@{},
   <0mm,0mm>*{};<4.5mm,5.5mm>*{^{m\hspace{-0.5mm}-\hspace{-0.5mm}1}}**@{},
   <0mm,0mm>*{};<9.0mm,5.5mm>*{^m}**@{},
 <-0.6mm,-0.44mm>*{};<-8mm,-5mm>*{}**@{-},
 <-0.4mm,-0.7mm>*{};<-4.5mm,-5mm>*{}**@{-},
 <0mm,0mm>*{};<-1mm,-5mm>*{\ldots}**@{},
 <0.4mm,-0.7mm>*{};<4.5mm,-5mm>*{}**@{-},
 <0.6mm,-0.44mm>*{};<8mm,-5mm>*{}**@{-},
   <0mm,0mm>*{};<-8.5mm,-6.9mm>*{^1}**@{},
   <0mm,0mm>*{};<-5mm,-6.9mm>*{^2}**@{},
   <0mm,0mm>*{};<4.5mm,-6.9mm>*{^{n\hspace{-0.5mm}-\hspace{-0.5mm}1}}**@{},
   <0mm,0mm>*{};<9.0mm,-6.9mm>*{^n}**@{},
 \end{xy}}\Ea
 \ \ \ \ \ \ \ \la\in \K
$$
of  $\wHoLBcd$.  As a complex the continuous derivations $\mathrm{Der}(\wHoLBcd)$ are identical to the deformation complex $\Def(\HoLBcd\rar \wHoLBcd)[1]$,
so that we have a canonical morphism of complexes $ \GC_{c+d+1}^{or}\rar  \Def(\HoLBcd\rar \wHoLBcd)[1]$
which induces \cite{MW2} in turn isomorphisms of cohomology groups,
$$
 H^{\bu+1}\left(\Def(\HoLBcd\rar \wLBcd)\right) \cong
H^\bu(\GC_{c+d+1}^{or}) \ \oplus \ \K \cong H^\bu( \mathsf{GC}_{c+d}^2)\ \oplus \ \K
$$
where the summand $\K$ (placed in degree zero) corresponds to the rescaling automorphism.
Therefore, for any morphism of properads $\wHoLBcd \rar \cP$ we have an associated morphism
of cohomology groups
\Beq\label{2: Holieb to P}
H^\bu( \mathsf{GC}_{c+d}^2)\ \oplus\ \K
\lon  H^{\bu+1}\left(\Def(\HoLBcd\rar \wHoLBcd)\right)
\lon
 H^{\bu+1}\left(\Def(\HoLBcd\rar \cP)\right).
\Eeq


Similarly, the dg Lie algebra $(\GC_{c+d+1}^{or}[[\hbar]], d_\hbar)$ (with $c+d\in 2\Z$)
acts naturally from the right on the properad $\wHoLoBcd$ by continuous properadic derivations. More precisely, let $\Gamma\in \GC_{c+d+1}^{or}$ be a graph.
Then to the element $\hbar^k\Gamma\in \GC_{c+d+1}^{or}[[\hbar]]$ we assign a continuous derivation of $\wHoLoBcd$ whose values on the generators are given by
$$
\left(
\Ba{c}\resizebox{16mm}{!}{\xy
(-9,-6)*{};
(0,0)*+{a}*\cir{}
**\dir{-};
(-5,-6)*{};
(0,0)*+{a}*\cir{}
**\dir{-};
(9,-6)*{};
(0,0)*+{a}*\cir{}
**\dir{-};
(5,-6)*{};
(0,0)*+{a}*\cir{}
**\dir{-};
(0,-6)*{\ldots};
(-10,-8)*{_1};
(-6,-8)*{_2};
(10,-8)*{_n};
(-9,6)*{};
(0,0)*+{a}*\cir{}
**\dir{-};
(-5,6)*{};
(0,0)*+{a}*\cir{}
**\dir{-};
(9,6)*{};
(0,0)*+{a}*\cir{}
**\dir{-};
(5,6)*{};
(0,0)*+{a}*\cir{}
**\dir{-};
(0,6)*{\ldots};
(-10,8)*{_1};
(-6,8)*{_2};
(10,8)*{_m};
\endxy}\Ea
\right)
 \cdot (\hbar^k\Gamma)
=\left\{\Ba{ll}
 \sum
\Ba{c}\resizebox{11mm}{!}{ \xy
 (-5,7)*{^{_1}},
 (-3,7)*{^{_2}},
 (5.5,7)*{^{_m}},
 (-5,-7.9)*{^{_1}},
 (-3,-7.9)*{^{_2}},
 (5.5,-7.9)*{^{_n}},
(1,4.5)*+{...},
(1,-4.5)*+{...},
(0,0)*+{\Ga}="o",
(-5,6)*{}="1",
(-3,6)*{}="2",
(5,6)*{}="3",
(5,-6)*{}="4",
(-3,-6)*{}="5",
(-5,-6)*{}="6",
\ar @{-} "o";"1" <0pt>
\ar @{-} "o";"2" <0pt>
\ar @{-} "o";"3" <0pt>
\ar @{-} "o";"4" <0pt>
\ar @{-} "o";"5" <0pt>
\ar @{-} "o";"6" <0pt>
\endxy}\Ea & \mbox{if}\ k\geq a\\
0   &  \mbox{if}\ k <a\\
\Ea
\right.
$$
where the  sum is over  all graphs obtained by  attaching the external legs to $\Gamma$ in all possible ways and  assigning weights to the vertices in all  ways such that the weights sum to $a-k$ (see \cite{MW2} for full details). By a change of sign the right action may be transformed into a left action and hence we obtain a map of Lie algebras
\[
 F_\hbar \colon \sG\sC_{c+d+1}^{or}[[\hbar]]\to \mathrm{Der}(\wHoLoBcd)\,
\]
which is a quasi-isomorphism \cite{MW2}
up to $\K[[\hbar]]$ span of the derivation corresponding to the 1-parameter automorphism
$$
\Ba{c}\resizebox{16mm}{!}{\xy
(-9,-6)*{};
(0,0)*+{a}*\cir{}
**\dir{-};
(-5,-6)*{};
(0,0)*+{a}*\cir{}
**\dir{-};
(9,-6)*{};
(0,0)*+{a}*\cir{}
**\dir{-};
(5,-6)*{};
(0,0)*+{a}*\cir{}
**\dir{-};
(0,-6)*{\ldots};
(-10,-8)*{_1};
(-6,-8)*{_2};
(10,-8)*{_n};
(-9,6)*{};
(0,0)*+{a}*\cir{}
**\dir{-};
(-5,6)*{};
(0,0)*+{a}*\cir{}
**\dir{-};
(9,6)*{};
(0,0)*+{a}*\cir{}
**\dir{-};
(5,6)*{};
(0,0)*+{a}*\cir{}
**\dir{-};
(0,6)*{\ldots};
(-10,8)*{_1};
(-6,8)*{_2};
(10,8)*{_m};
\endxy}\Ea
\lon
\la^{m+n+2a-2}
\Ba{c}\resizebox{16mm}{!}{\xy
(-9,-6)*{};
(0,0)*+{a}*\cir{}
**\dir{-};
(-5,-6)*{};
(0,0)*+{a}*\cir{}
**\dir{-};
(9,-6)*{};
(0,0)*+{a}*\cir{}
**\dir{-};
(5,-6)*{};
(0,0)*+{a}*\cir{}
**\dir{-};
(0,-6)*{\ldots};
(-10,-8)*{_1};
(-6,-8)*{_2};
(10,-8)*{_n};
(-9,6)*{};
(0,0)*+{a}*\cir{}
**\dir{-};
(-5,6)*{};
(0,0)*+{a}*\cir{}
**\dir{-};
(9,6)*{};
(0,0)*+{a}*\cir{}
**\dir{-};
(5,6)*{};
(0,0)*+{a}*\cir{}
**\dir{-};
(0,6)*{\ldots};
(-10,8)*{_1};
(-6,8)*{_2};
(10,8)*{_m};
\endxy}\Ea
$$
of $\wHoLoBcd$. For any morphism of properads $\wHoLoBcd \rar \cP$ we have an associated morphism
of cohomology groups
\Beq\label{2: Holiebdi to P}
H^\bu( \mathsf{GC}_{c+d+1}^{or}[[\hbar]], \delta_\hbar)\ \oplus\ \K \lon
 H^{\bu+1}\left(\Def(\HoLoBcd\rar \cP)\right).
\Eeq
In \S 5 below we introduce a prop(erad) of ribbon graphs $\cR \cG ra_{d}$ which comes equipped with
a non-trivial map $\wHoLoBdd \rar \cR \cG ra_{d}$ (and hence with a map $\wHoLBdd \rar \cR \cG ra_{d}$ as well as its ``forgetful" versions $\wHoLBcd \rar \cR \cG ra_{d}$)  to which all the above observations apply.

\bip

{\Large
\section{\bf Involutive Lie bialgebras and quantum $A_\infty$ algebras }
}

\mip
\subsection{BV algebras}
A graded commutative algebra $A=\oplus_{i\in \Z} A^i$ is called a {\em BV algebra}
if it comes equipped
with a degree $1$ derivation $\Delta$ of order $\leq 2$,
$$
\Delta=\underbrace{\Delta_1}_{\mathrm order\ 1\ derivation} +
\underbrace{\Delta_2}_{\mathrm order\ 2\ derivation}
$$
such that $\Delta^2=0$. The order 2 part of the derivation makes
$A$ into a $\caL ie_0$ algebra with the bracket
$$
\Ba{rccl}
\{\ ,\ \}: & \odot^2 A & \lon & A\\
         &  (a_1, a_2) & \lon & \{a_1, a_2\}:=(-1)^{|a_1|} \Delta_2(a_1a_2)-
         (-1)^{|a_1|} \Delta_2(a_1)a_2 - a_1\Delta_2(a_2).
         \Ea
$$
while order 1 part of the derivation makes this Lie algebra differential,
$$
\Delta_1\{a_1, a_2\}=\{\Delta_1(a_1), a_2\} - (-1)^{|a_1|}\{a_1,\Delta_1(a_2)\}.
$$
Note that if $A$ is a {\em free}\, graded commutative algebra, i.e\ $A=\odot^\bu W$
for some vector space $W$, then the $BV$ operator $\Delta$ is uniquely determined
by its restriction,
$
\Delta|_{\odot^2 W}: \odot^2 W \lon \odot^{\leq 2}W.
$

\sip

A degree zero element $S$ in a $BV$ algebra $A$ is called a {\em master function}
if it satisfies the equation
$$
\Delta S + \frac{1}{2}\{S,S\}=0.
$$
If the exponent $exp(S)=1+S+ \frac{1}{2}S^2 +\ldots $ makes sense in $A$, then
the above equation can be rewritten equivalently as
$
\Delta(e^S)=0.
$

\subsection{Quantum diamond functions}\label{3: subsect on inv Lie bialg and quantum master functions}
Let $(V, [\ ,\ ])$ be a $\caL ie_{d}$ algebra (so that $V[d-1]$ is an ordinary
  Lie algebra). It is well-known that the Lie brackets in $V$ can be understood as a degree 1 codifferential $\Delta_{_{[\,,\,]}}$ in
  the (completed) graded commutative coalgebra,
  $$
  C_d^\bu(V):=\widehat{\odot^{\bu}}(V[d])
  $$
which gives us also an order 2 derivation of $C_d^\bu(V)$ as a graded commutative algebra
and hence makes $C_d^\bu(V)$ into a $BV$ algebra, and in particular, into a $\caL ie_0$-algebra
 with the  brackets $\{\ ,\ \}$ given as in the subsection above.


\mip

If $(V, [\ ,\ ],\vartriangle)$ is a $\LBcd$-algebra, the  cobracket
$$
\vartriangle: V[d]\lon \odot^2(V[d])[1-c-d]
$$
extends to a derivation $\delta_\vartriangle$
of $C_d^\bu(V)$ as a graded commutative algebra, so that
$$
\hbar\delta_\vartriangle: C_d^\bu(V)[[\hbar]] \lon C_d^\bu(V)[[\hbar]], \ \ \ |\hbar|=c+d,
$$
is a differential in $C_d^\bu(V)[\hbar]]$. The Lie bialgebra
compatibility conditions on $[\ ,\ ]$ and $\vartriangle$ imply
that  $\hbar\delta_\vartriangle$ respects the Lie brackets,
$$
\delta_\vartriangle \{a,b\}= \{\delta_\vartriangle (a) ,b\} - (-1)^{|a|}
\{a, \delta_\vartriangle (b)\}
$$
and hence makes $(C_d^\bu(V)[\hbar]], \{\ ,\ \}, \hbar\delta_\vartriangle )$
into {\em dg}\, $\caL ie_0$ algebra. Note that  $(C_d^\bu(V)[\hbar]], \{\ ,\ \}, \Delta_{_{[\,,\,]}})$
is also a {\em dg}\, $\caL ie_0$ algebra.

\mip

Finally, if $(V, [\ ,\ ], \vartriangle)$ is a $\LoBcd$-algebra, then the above two dg Lie
algebra structures on  $C_d^\bu(V)$ can be combined into a single one,
$$
\left(C_d^\bu(V)[\hbar]], \{\ ,\ \}, \hbar\delta_\vartriangle  + \Delta_{_{[\,,\,]}}\right),
$$
that is, the sum $\Delta:=\hbar\delta_\vartriangle  + \Delta_{_{[\,,\,]}}$ defines a
differential in the Lie algebra $(C_d^\bu(V)[[\hbar]],\{\ ,\ \})$. In fact, this differential is a derivation of
order $\leq 2$ so that the data
$$
\left(C_d^\bu(V)[\hbar]], \hbar\delta_\vartriangle  + \Delta_{[\ ,\ ]}\right),
$$
is a $BV$ algebra.  By analogy to \cite{Ba1,Ba2, Ha, Ha2}, one can define a
{\em quantum master function}\, associated to a $\LoBcd$-algebra  $V$  as a degree
zero element $S\in C_d^\bu(V)[[\hbar]]$ which satisfies  the equation
\Beq\label{3: master eqn for inv Lie bial}
\hbar\delta_\vartriangle  S + \Delta_{_{[\,,\,]}}S + \frac{1}{2}\{S,S\}=0
\Eeq
and lies in the dg Lie subalgebra
$$
 Q_\hbar (C_d^\bu(V)):=\prod_{m\geq 1}\hbar^{m-1}\odot^{m}(V[d])[[\hbar]] \subset  C_d^\bu(V)[[\hbar]]
$$
As $\LoBcd$ algebras are sometimes called {\em diamond}\, algebras, one can call
these solutions {\em quantum diamond functions}\, in order to avoid confusion with the graded commutative Batalin-Vilkovisky formalism.

\subsubsection{\bf Basic example}\label{3: Subsubsect Basic Example} Let $V$ be $Cyc^\bu(W)$ equipped with its standard  $\LoB_{d,d}$-structure
induced from a (skew)symmetric product in a vector space $W$ of degree $1-d$
(see \S {\ref{2: subsection on cyclic words}}). Then we have the following list of associated algebraic structures
\Bi
\item[(i)] {\em Cyclic $A_\infty$ algebra}\, structures in $W$ are in 1-1 correspondence with degree zero
elements $\al$ in $Cyc^{\bu\geq 2}(W)[d]$ satisfying the equation $\{\al,\al\}=0$. We recall that
the {\em  Hochschild complex}\, of a cyclic $A_\infty$ algebra $\al$ is, by definition, the Lie algebra
 $(Cyc^{\bu}(W)[d], \{\ ,\ \})$ equipped with the differential $d_\al:=\{\al,\ \}$.

\item[(ii)] {\em Quantum $A_\infty$ algebra}\, structures in $W$ are, by definition \cite{Ba1},
quantum diamond
functions $S\in Q_\hbar(C_d^\bu(Cyc^{\bu\geq 1}W))$ satisfying an extra boundary condition
 $$
 S|_{\hbar=0}\in Cyc^{\bu\geq 2}W [d].
 $$
According to Serguei Barannikov  \cite{Ba1,Ba2}, quantum $A_\infty$ algebras give rise to (co)homology
 classes in the Kontsevich quotient, $\overline{\cM}_{g, n}^K$ of the Deligne-Mumford
 moduli stack.

 \sip

 The Hochschild complex of a quantum $A_\infty$ algebra $S$ is, by definition, the Lie algebra
 $(Q_\hbar(C_d^\bu(Cyc^{\bu}W)), \{\ ,\ \})$ equipped with the differential
 $$
 d_S:= \hbar\delta_\vartriangle  + \Delta_{_{[\,,\,]}} + \{S,\ \}
 $$

\item[(iii)] One can consider an ``intermediate" object between the two just considered   ---  a degree zero element
$\pi\in \widehat{\odot^\bu}(Cyc^\bu(W))[d]$ satisfying the equation $[\pi,\pi]=0$. For $d=2$ this structure generalizes the classical notion of a {\em Poisson structure}\, on the affine vector space $W$,
and can be called a {\em non-commutative}\, Poisson structure on $W$ as it contains, in general,
cyclic ``words" in $W$ of length $\geq 2$; in fact ordinary polynomial (or formal power series) Poisson structures $\pi$ on $W$ form
a subclass of non-commutative ones characterized by the condition that $\pi$ belongs to the subspace $\widehat{\odot^\bu}(W)[2]$ of
 $\widehat{\odot}^\bu(Cyc^\bu(W))[2]$ spanned by cyclic ``words" of length one. Thus a non-commutative Poisson structure reduces in one extreme case to a cyclic $A_\infty$ structure on $W$
 and in another extreme case to a standard Poisson structure on $W$.
 We prove below that there is a highly non-trivial (in general) action of the Grothendieck-Techm\"uller group $GRT_1$ on the set of non-commutative Poisson structures
on arbitrary $W$ extending the known action of $GRT_1$ on ordinary Poisson structures.
\Ei

\subsection{A reparametrization of quantum diamond functions}\label{3: Subsec on Hamiltons parametr of quaA_infty} In the above notations,
consider a rescaling map of homological degree $c+d$,
$$
\Ba{rccc}
H: &  Q_\hbar (C_d^\bu(V)) & \lon &
\sC_c^\bu(V):=\widehat{\odot^{\bu\geq 1}}(V[-c])[[\hbar]]\\
&
\hbar^{m-1}\odot^{m}(V[d])[[\hbar]] & \stackrel{s^{(1-m)(c+d)}
\hbar^{1-m} }{\lon} &
\odot^{m}(V[-c])[[\hbar]]
\Ea
$$
which multiplies a generating (over $\K[[\hbar]]$) monomial $\hbar^{m-1}v_1\odot v_2\odot\ldots \odot v_m $ in
$Q_\hbar (C_d^\bu(V))$
by $\hbar^{1-m}$ and simultaneously shifts its homological grading by $(1-m)(c+d)$  (so that the total degree of the output equals the total degree
of the input $+\ (c+d)$).

\sip

It is easy to see that, if $(V,[ \ ,\ ], \Delta)$ is a $\LoBcd$-algebra, then $\sC_c^\bu(V)$ is a dg $\caL ie_{c+d}$ algebra with brackets $\{\ , \ \}$ (of degree $1-c-d$) and the differential $\delta_\triangledown+ \hbar \Delta_{[\ , \ ]}$.
The morphism $H$ above extends to isomorphism of dg $\caL ie$ algebras  $Q_\hbar (C_d^\bu(V)) \rar \sC_d^\bu(V)$ and sends, in particular,
quantum diamond functions into degree $c+d$ elements $S$ in
$\sC_d^\bu(V)$
satisfying the equation

\Beq\label{3: rescaled master eqn for inv Lie bial}
\delta_\vartriangle  S + \hbar \Delta_{_{[\,,\,]}}S +  \frac{1}{2}\{S,S\}=0.
\Eeq

Quantum $\cA ss_\infty$ algebras in this parametrization have been studied in
\cite{Ha}; they correspond to solutions $S\in \sC_{c=d}^\bu(Cyc^{\bu\geq 1} W)$ of
(\ref{3: rescaled master eqn for inv Lie bial}) which satisfy a ``boundary" condition
$$
S|_{\hbar=0}\in Cyc^{\bu\geq 2} W[-d]\, \oplus\, \odot^{\bu\geq 2} (Cyc^{\bu\geq 1} W[-d]).
$$
%
%
%
%
%
%
Alastair Hamilton also considered in \cite{Ha}  a more general class of solutions $S$ of (\ref{3: rescaled master eqn for inv Lie bial}) which belong to the full graded commutative algebra
$\widehat{\odot^{\bu}}(V[-d])[2d][[\hbar]]$ and has shown that
under certain boundary conditions (and for $d$ even) they give us homology classes
in the Looijenga  generalization, $\overline{\cM}_{g, n}^L$, of the Kontsevich
moduli space $\overline{\cM}_{g, n}^K$.



\sip

Any quantum diamond function $S\in \sC_c^\bu(V)$ makes the $\caL ie_{c+d}$ algebra $(\sC_d^\bu(V), \{\ ,\ \})$
into a {\em dg}\, $\caL ie_{c+d}$ algebra with the twisted  differential
$$
\delta_S:= \delta_\triangledown + \hbar \Delta_{[\ , \ ]} + \{S,\ \}
$$
called the {\em Hochschild complex}\, of $S$. In the special case $V=Cyc(W)$ and $c=d$ this notion gives us
an equivalent (to the one given \S {\ref{3: Subsubsect Basic Example}}(ii)) definition of the Hochschild complex of a quantum $\cA ss_\infty$ algebra $S$.

\bip

{\Large
\section{\bf Prop of ribbon graphs and
involutive Lie bialgebras}
}

\mip

\subsection{Ribbon graphs} A {\em graph}\, $\Ga$ is, by definition, a triple
$(H(\Ga), \sigma_1, \sim)$ consisting of a finite set  $H(\Ga)$ called the set of
{\em half-edges}, a fixed point free involution $\sigma_1: H(\Ga)\rar H(\Ga)$
whose set of orbits, $E(\Ga):=H(\Ga)/\sigma_1$, is called the set of (unordered)
{\em edges}, and an equivalence relation $\sim$ on $H(\Ga)$ with the corresponding set of
equivalence classes, $V(\Ga):= H(\Ga)/\sim$,
called the set of {\em vertices}. For a vertex $v\in V(\Ga)$ the associated subset
$H(v):=p^{-1}(v)\subset H(\Ga)$, $p: H(\Ga)\rar V(\Ga)$ being the natural projection,
is called the set of half-edges {\em attached to}\, $v$; the number $\# H(v)=:|v|$ is
called the {\em valency}\, of the vertex $v$. Any graph $\Ga$ can be be visualized as
a 1-dimensional $CW$ complex whose $0$-cells are vertices of $\Ga$ and 1-dimensional
cells are edges. Thus it makes sense to talk about {\em connected}\, graphs, and
about {\em connected components}\, of a graph which is not connected.

\sip

A {\em ribbon graph}\, $\Ga$ is, by definition, a triple
$(H(\Ga), \sigma_1, \sigma_0)$ consisting of a finite set  $H(\Ga)$ called the
set of {\em half edges}\,
$H(\Ga)$, a fixed point free involution $\sigma_1: H(\Ga)\rar H(\Ga)$ whose
set of orbits, $E(\Ga):=H(\Ga)/\sigma_1$, is called the set of  {\em edges}, and
a permutation $\sigma_0: H(\Ga)\rar H(\Ga)$. The set of orbits, $V(\Ga):=H(\Ga)/\sigma_0$,
is called the set of {\em vertices} \, of the ribbon graph. Each orbit of $\sigma_0$
comes equipped with an induced cyclic ordering so that
a ribbon graph is just an ordinary graph equipped with a choice of cyclic ordering
on subset of half-edges, $H(v)\subset H(\Ga)$, attached to each vertex $v\in V(\Ga)$.
The orbits of the permutation $\sigma_\infty:= \sigma_0^{-1}\circ \sigma_1$
are called {\em boundaries}\, of the ribbon graph $\Ga$. The set of boundaries of $\Ga$ is
 denoted by $B(\Ga)$.



\sip

The set of  ribbon graphs with $n$ vertices,  $m$ boundaries and $l$ edges is
denoted by $\cR\cB_{m,n}^l$. Every connected ribbon graph $\Ga$ from
 $\cR\cB_{n,m}^l$ can be visualized as a topological 2-surface $S'_\Ga$ with
 $m$ boundary circles and $n$ distinguished points which is obtained from   the $CW$
 realization of $\Ga$ by thickening its
 1-cells into 2-dimensional strips, thickening every vertex $v$ into a closed disk (whose
 center is  a ``distinguished point"), and then gluing the strips to the disks in
 accordance with the given cyclic ordering in $H(v)$. If we also glue a disk
 into each boundary circle of the 2-surface $S'_\Ga$, then we get a {\em closed}\,
 topological $2$-surface $S''_\Ga$ of genus $g= 1+\frac{1}{2}(l-m-n)$. This motivates
 the following definition: for any connected ribbon graph $\Ga$ the number
 \Beq\label{4: genus of a ribbon graph}
 g= 1+\frac{1}{2}\left(\# E(\Ga) - \# V(\Ga)- \# B(\Ga)\right)
 \Eeq
is called the {\em genus}\, of $\Ga$.

\mip

We shall use in this paper a slightly different geometric realization
of a ribbon graph $\Ga \in \cR\cB_{n,m}^l$: in $S'_\Ga$ we always cut out ``small"
open disks with centers at the distinguished points (vertices) and get therefore a
topological 2-surface $S_\Ga$ with $m+n$ boundary circles; $n$ boundary circles
corresponding to vertices are called {\em in}-circles, while
$m$ boundary circles corresponding to elements of $B(\Ga)$ are called
{\em out}-circles. For example,
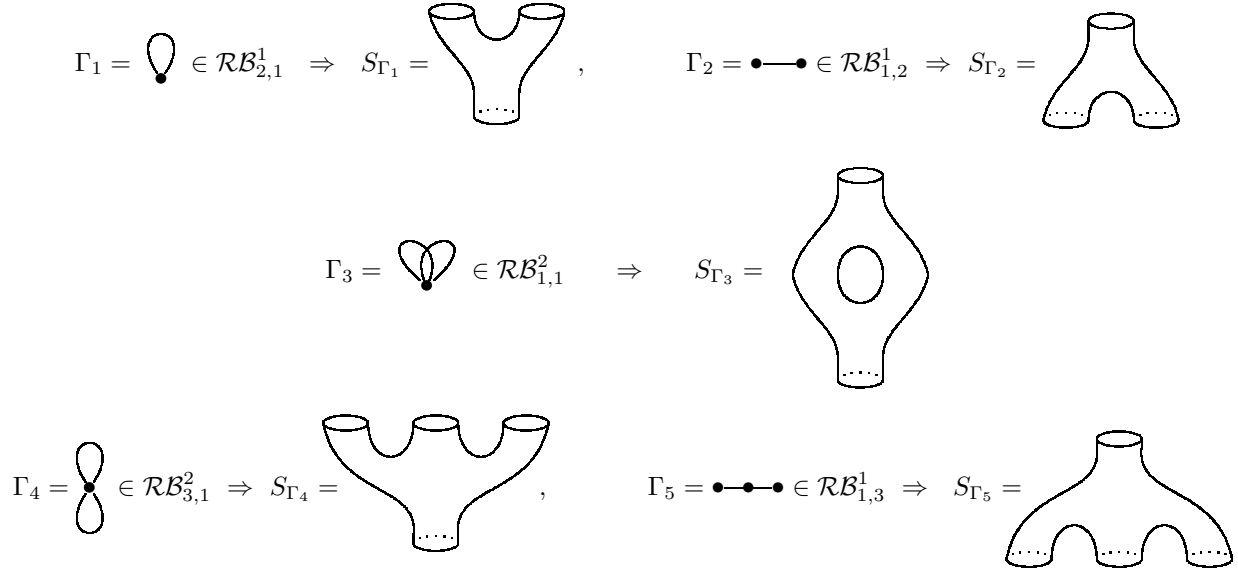
\bip
$$
\ \Ga_1=\xy
(0,-2)*{\bu}="A";
(0,-2)*{\bu}="B";
"A"; "B" **\crv{(6,6) & (-6,6)};
\endxy \hspace{0mm}\in \cR\cB_{2,1}^1
\ \ \Rightarrow \ \ S_{\Ga_1}= \hspace{-4mm}\Ba{c}\mbox{
\xy
(0,-1)*\ellipse(3,1){.};
(0,-1)*\ellipse(3,1)__,=:a(-180){-};
(-3,6)*\ellipse(3,1){-};
(3,6)*\ellipse(3,1){-};
(-3,12)*{}="1"; 
(3,12)*{}="2"; 
(-9,12)*{}="A2";
(9,12)*{}="B2"; 
"1";"2" **\crv{(-3,7) & (3,7)};
(-3,-2)*{}="A";
(3,-2)*{}="B"; 
(-3,1)*{}="A1";
(3,1)*{}="B1"; 
"A";"A1" **\dir{-};
"B";"B1" **\dir{-}; 
"B2";"B1" **\crv{(8,7) & (3,5)};
"A2";"A1" **\crv{(-8,7) & (-3,5)};
\endxy
}\Ea,
\ \ \ \ \ \ \ \ \ \ \
\Ga_2=
\xy
 (0,0)*{\bullet}="a",
(6,0)*{\bu}="b",
\ar @{-} "a";"b" <0pt>
\endxy
\in \cR\cB_{1,2}^1  \  \Rightarrow \
S_{\Ga_2}= \hspace{-3mm}\Ba{c}\mbox{
 \xy
(0,0.6)*\ellipse(3,1){-};
(-3,-6)*\ellipse(3,1){.};
(-3,-6)*\ellipse(3,1)__,=:a(-180){-};
(3,-6)*\ellipse(3,1){.};
(3,-6)*\ellipse(3,1)__,=:a(-180){-};
(-3,-12)*{}="1";
(3,-12)*{}="2";
(-9,-12)*{}="A2";
(9,-12)*{}="B2";
"1";"2" **\crv{(-3,-7) & (3,-7)};
(-3,1)*{}="A";
(3,1)*{}="B";
(-3,-1)*{}="A1";
(3,-1)*{}="B1";
"A";"A1" **\dir{-};
"B";"B1" **\dir{-};
"B2";"B1" **\crv{(8,-7) & (3,-5)};
"A2";"A1" **\crv{(-8,-7) & (-3,-5)};
\endxy
 }
\Ea
$$
\mip
$$
\ \Ga_3=\xy
(0,-2)*{\bu}="A";
(0,-2)*{\bu}="B";
"A"; "B" **\crv{(-4,6) & (10,6)};
"A"; "B" **\crv{(-10,6) & (4,6)};
\endxy\in \cR\cB_{1,1}^2
\ \ \ \ \   \Rightarrow  \ \ \ \ \ S_{\Ga_3}=
\Ba{c}\mbox{
 \xy
(0,0.6)*\ellipse(3,1){-};
(-3,-12)*{}="1";
(3,-12)*{}="2";
(-9,-12)*{}="A2";
(9,-12)*{}="B2";
"1";"2" **\crv{(-3,-7) & (3,-7)};
(-3,1)*{}="A";
(3,1)*{}="B";
(-3,-1)*{}="A1";
(3,-1)*{}="B1";
"A";"A1" **\dir{-};
"B";"B1" **\dir{-};
"B2";"B1" **\crv{(8,-7) & (3,-5)};
"A2";"A1" **\crv{(-8,-7) & (-3,-5)};
\endxy}\\
\ \mbox{\xy
(0,-1)*\ellipse(3,1){.};
(0,-1)*\ellipse(3,1)__,=:a(-180){-};
(-3,12)*{}="1"; 
(3,12)*{}="2"; 
(-9,12)*{}="A2";
(9,12)*{}="B2"; 
"1";"2" **\crv{(-3,7) & (3,7)};
(-3,-2)*{}="A";
(3,-2)*{}="B"; 
(-3,1)*{}="A1";
(3,1)*{}="B1"; 
"A";"A1" **\dir{-};
"B";"B1" **\dir{-}; 
"B2";"B1" **\crv{(8,7) & (3,5)};
"A2";"A1" **\crv{(-8,7) & (-3,5)};
\endxy}
\Ea
$$
\mip
\[
\Ga_4=\hspace{-3mm}\Ba{c}\mbox{\xy
(0,-2)*{\bu}="A";
(0,-2)*{\bu}="B";
"A"; "B" **\crv{(6,6) & (-6,6)};
"A"; "B" **\crv{(6,-10) & (-6,-10)};
\endxy}\Ea \hspace{-2mm}
\in \cR\cB_{3,1}^2 \ \Rightarrow \
S_{\Ga_4}= \hspace{-6mm}\Ba{c}\mbox{
\xy
(0,0)*\ellipse(3,1){.};
(0,0)*\ellipse(3,1)__,=:a(-180){-};
(-6,8)*\ellipse(3,1){-};
(6,8)*\ellipse(3,1){-};
(0,8)*\ellipse(3,1){-};
%
(3,16)*{}="1";
(9,16)*{}="2";
"1";"2" **\crv{(3,10) & (9,10)};
(-3,16)*{}="1";
(-9,16)*{}="2";
"1";"2" **\crv{(-3,10) & (-9,10)};
(-15,16)*{}="A2";
(15,16)*{}="B2";
(-3,0)*{}="A";
(3,0)*{}="B";
(-3,1)*{}="A1";
(3,1)*{}="B1";
"A";"A1" **\dir{-};
"B";"B1" **\dir{-};
"B2";"B1" **\crv{(13,6) & (2,8)};
"A2";"A1" **\crv{(-13,6) & (-2,8)};
\endxy
}
\Ea \hspace{-3mm},
\ \ \ \ \ \ \ \ \ \ \
\Ga_5=
\xy
 (0,0)*{\bullet}="a",
(4,0)*{\bu}="b",
(-4,0)*{\bu}="b'",
\ar @{-} "b'";"b" <0pt>
\endxy
\in \cR\cB_{1,3}^1 \ \Rightarrow\ \
S_{\Ga_5}= \hspace{-6mm}\Ba{c}\mbox{
\xy
(0,0)*\ellipse(3,1){-};
(-6,-8)*\ellipse(3,1){.};
(6,-8)*\ellipse(3,1){.};
(0,-8)*\ellipse(3,1){.};
(-6,-8)*\ellipse(3,1)__,=:a(-180){-};
(6,-8)*\ellipse(3,1)__,=:a(-180){-};
(0,-8)*\ellipse(3,1)__,=:a(180){-};
%
(3,-16)*{}="1";
(9,-16)*{}="2";
"1";"2" **\crv{(3,-10) & (9,-10)};
(-3,-16)*{}="1";
(-9,-16)*{}="2";
"1";"2" **\crv{(-3,-10) & (-9,-10)};
(-15,-16)*{}="A2";
(15,-16)*{}="B2";
(-3,0)*{}="A";
(3,0)*{}="B";
(-3,-1)*{}="A1";
(3,-1)*{}="B1";
"A";"A1" **\dir{-};
"B";"B1" **\dir{-};
"B2";"B1" **\crv{(13,-6) & (2,-8)};
"A2";"A1" **\crv{(-13,-6) & (-2,-8)};
\endxy
}
\Ea
\]
\bip

Therefore it is sometimes useful to represent our ribbon graphs $\Ga$ as 1-dimensional
$CW$ complexes with vertices
blown up into anticlockwise oriented dashed circles:
$$
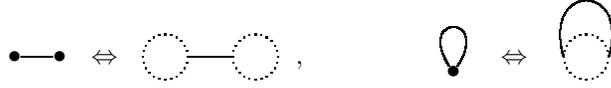
\xy
 (0,0)*{\bullet}="a",
(6,0)*{\bu}="b",
\ar @{-} "a";"b" <0pt>
\endxy\ \ \Leftrightarrow \ \
\xy
 (0,0)*{
\xycircle(3,3){.}};
(12,0)*{
\xycircle(3,3){.}};
 (3,0)*{}="a",
(9,0)*{}="b",
\ar @{-} "a";"b" <0pt>
\endxy\ \ ,\ \ \ \ \ \  \ \ \ \ \ \ \ \
\xy
(0,-2)*{\bu}="A";
(0,-2)*{\bu}="B";
"A"; "B" **\crv{(6,6) & (-6,6)};
\endxy \ \ \Leftrightarrow \ \
\xy
 (0,0)*{
\xycircle(3,3){.}};
(-3,0)*{}="1";
(3,0)*{}="2";
"1";"2" **\crv{(-6,10) & (6,10)};
\endxy
$$

Note that edges attached to a vertex  split the corresponding dashed circle
into a union of dashed intervals, and we call each such interval a {\em corner}\, of the
ribbon graph $\Ga$ under consideration, and denote the set of all corners of $\Ga$ by
 $C(\Ga)$. There are two natural partitions of the set $C(\Ga)$ into disjoint unions of
 subsets,
 $$
 C(\Ga)=\coprod_{v\in V(\Ga)} C(v),\ \ \ \ \ \
 C(\Ga)=\coprod_{b\in B(\Ga)} C(b),\ \ \
 $$
 such that each subset $C(v)$ and $C(b)$ comes equipped with a
canonically induced cyclic ordering.
 It is sometimes helpful to employ this second partition of $C(\Ga)$ to represent each
  boundary $b$ of $\Ga$ as an anticlockwise oriented topological circle of the form
$
 \Ba{c}\resizebox{10mm}{!}{\xy
(6,-8)*{\circlearrowleft},
 (0,0)*{}="a1",
(10,0)*{}="a2",
(13,-3)*{}="a3",
(13,-13)*{}="a4",
(10,-16)*{}="a5",
(0,-16)*{}="a6",
(-3,-13)*{}="a7",
(-3,-3)*{}="a8",
\ar @{-} "a1";"a2" <0pt>
\ar @{.} "a2";"a3" <0pt>
\ar @{-} "a3";"a4" <0pt>
\ar @{.} "a4";"a5" <0pt>
\ar @{-} "a5";"a6" <0pt>
\ar @{.} "a6";"a7" <0pt>
\ar @{-} "a7";"a8" <0pt>
\ar @{.} "a8";"a1" <0pt>
\endxy}\vspace{2mm}\Ea
$
where dashed intervals represent (always different) corners of $\Ga$ corresponding to $b$,
and solid intervals
represent (not necessarily different) edges of $\Ga$ which belong to $b$ when it is viewed
as a cycle of the permutation $\sigma_\infty$. For example, the graph
$\xy
(0,-2)*{\bu}="A";
(0,-2)*{\bu}="B";
"A"; "B" **\crv{(-4,6) & (10,6)};
"A"; "B" **\crv{(-10,6) & (4,6)};
\endxy\in \cR\cB_{1,1}^2$ has $4$ corners belonging to one and the same
boundary, and this boundary  gets represented exactly by the above picture.
This pictorial representation of boundaries makes obvious the following claim-definition
which will be used below.

\subsubsection{\bf Claim}\label{4: claim on edge between corners} { Let
$\Ga\in \cR\cB_{m,n}^l$ be a ribbon graph, and let
$c_1$ and $c_2$ be two (not necessarily different) corners of\, $\Ga$. Let
$e_{c_1,c_2}(\Ga)$ be the ribbon graph obtained from $\Ga$ by attaching a new edge $e$
connecting $c_1$ to $c_2$. Then
 $e_{c_1,c_2}(\Ga)\in \cR\cB_{m-1,n}^{l+1}$ if and only if the corners $c_1$ and
 $c_2$ belong to {\em different}\, boundaries as in the picture
$
 \Ba{c}\resizebox{27mm}{!}{\xy
(20,-1)*{e},
 (0,0)*{}="a1",
(10,0)*{}="a2",
(13,-3)*{}="a3",
(13,-13)*{}="a4",
(10,-16)*{}="a5",
(0,-16)*{}="a6",
(-3,-13)*{}="a7",
(-3,-3)*{}="a8",
 (30,0)*{}="b1",
(40,0)*{}="b2",
(43,-3)*{}="b3",
(43,-13)*{}="b4",
(40,-16)*{}="b5",
(30,-16)*{}="b6",
(27,-13)*{}="b7",
(27,-3)*{}="b8",
(11.5,-1.5)*{}="1";
(28.5,-1.5)*{}="2";
"1";"2" **\crv{(20,3) & (20,3)};
\ar @{-} "a1";"a2" <0pt>
\ar @{.} "a2";"a3" <0pt>
\ar @{-} "a3";"a4" <0pt>
\ar @{.} "a4";"a5" <0pt>
\ar @{-} "a5";"a6" <0pt>
\ar @{.} "a6";"a7" <0pt>
\ar @{-} "a7";"a8" <0pt>
\ar @{.} "a8";"a1" <0pt>
\ar @{-} "b1";"b2" <0pt>
\ar @{.} "b2";"b3" <0pt>
\ar @{-} "b3";"b4" <0pt>
\ar @{.} "b4";"b5" <0pt>
\ar @{-} "b5";"b6" <0pt>
\ar @{.} "b6";"b7" <0pt>
\ar @{-} "b7";"b8" <0pt>
\ar @{.} "b8";"b1" <0pt>
\endxy}\vspace{2mm}\Ea
$,
and  $e_{c_1,c_2}(\Ga)\in \cR\cB_{m+1,n}^{l+1}$ if and only if the corners $c_1$ and
$c_2$ belong to \vspace{-4mm} the {\em same}\, boundary of $\Ga$ as in the picture
$
\vspace{-7mm} \Ba{c}\resizebox{18mm}{!}{\xy
(2.5,5.5)*{e},
 (11.5,-1.5)*{}="1",
(-1.5,-14.5)*{}="2",
 (0,0)*{}="a1",
(10,0)*{}="a2",
(13,-3)*{}="a3",
(13,-13)*{}="a4",
(10,-16)*{}="a5",
(0,-16)*{}="a6",
(-3,-13)*{}="a7",
(-3,-3)*{}="a8",
"1";"2" **\crv{(20,27) & (-27,-20)};
\ar @{-} "a1";"a2" <0pt>
\ar @{.} "a2";"a3" <0pt>
\ar @{-} "a3";"a4" <0pt>
\ar @{.} "a4";"a5" <0pt>
\ar @{-} "a5";"a6" <0pt>
\ar @{.} "a6";"a7" <0pt>
\ar @{-} "a7";"a8" <0pt>
\ar @{.} "a8";"a1" <0pt>
\endxy}\vspace{4mm}\Ea
$
}.
For example, the graph $\Ga_1=\xy
(0,-2)*{\bu}="A";
(0,-2)*{\bu}="B";
"A"; "B" **\crv{(6,6) & (-6,6)};
\endxy \hspace{0mm}\in \cR\cB_{2,1}^1$ has two corners, $C(\Ga_1)=\{c_1,c_2\}$,
which belong to two different boundaries so that
$
e_{c_1,c_2}(\Ga_1)=\xy
(0,-2)*{\bu}="A";
(0,-2)*{\bu}="B";
"A"; "B" **\crv{(-4,6) & (10,6)};
"A"; "B" **\crv{(-10,6) & (4,6)};
\endxy\in \cR\cB_{1,1}^2
$
while
$
e_{c_1,c_1}(\Ga_1)=e_{c_2,c_2}(\Ga_1)=\Ba{c}\mbox{\xy
(0,-2)*{\bu}="A";
(0,-2)*{\bu}="B";
"A"; "B" **\crv{(6,6) & (-6,6)};
"A"; "B" **\crv{(6,-10) & (-6,-10)};
\endxy}\Ea
\in \cR\cB_{3,1}^2
$.

\sip

A (ribbon) graph is called {\em directed}\, if its every edge is equipped
with a choice of flow (``direction"), e.g.
$\xy
 (0,0)*{\bullet}="a",
(6,0)*{\bu}="b",
\ar @{->} "a";"b" <0pt>
\endxy$ and $\xy
(0,-2)*{\bu}="A";
(0,-2)*{\bu}="B";
{\ar@{->} (-1.69,2.4)*{}; (-1.695,2.5)*{}};
"A"; "B" **\crv{(6,6) & (-6,6)};
\endxy$
are directed (ribbon) graphs.

\subsection{Prop $\cR Gra_d$} Let $\cR_{m,n}^l$ be the set of (isomorphism classes of)
directed
ribbon graphs $\Ga$ equipped with bijections $E(\Ga)\rar [l]$, $V(\Ga) \rar [n]=\{1,2,\ldots, n\}$ and $B(\Ga) \rar [\bar{m}]=\{\bar{1},\ldots, \bar{m}\}$.
The permutation group $\bS_l$ acts on elements of $\cR_{m,n}^l$ by changing the orderings
of edges, while the group $\bS_2^{\times l}$ acts by flipping the directions of edges.
For an integer $d\in \Z$, let $sgn_l^{(d)}$ be the one dimensional representation of
the group $P_l:=\bS_l\times \bS_2^{\times l}$
on which (i) $\bS_l$ acts trivially for $d$ odd and by sign for $d$ even, and
(ii) each $\bS_2$ acts trivially for $d$ even and by sign for $d$ odd.
 For every pair of natural numbers
$m,n\in \N$ we define a graded vector space over a field $\K$,
$$
\cR\cG ra_{d}(m,n):=\bigoplus_{l\geq 0} \left(\K\langle \cR_{m,n}^l \rangle \ot_{P_l}
sgn^{(d)}_l\right)[l(d-1)]
$$
Thus a generator $\Ga$ of $\cR\cG ra_d(m,n)$ can be understood as the isomorphism
class of a ribbon graph whose vertices are labelled by elements of $[n]$, boundaries are labelled
by elements of $[m]$, whose every edge is assigned degree $1-d$, and which is equipped with an
{\em orientation}\, which depends on parity of $d$:
\Bi
 \item for even $d$ the orientation of $\Ga$ is defined as a choice
 of ordering of the set of edges $E(\Ga)$ up to a sign action of $\bS_{\# E(\Ga)}$,
\item for odd $d$ the orientation of $\Ga$ is defined as a choice
of a direction on each edge up to a sign action of $\bS_2$.

\Ei
Note that every ribbon graph $\Ga$ has precisely two possible orientations, $or$ and $or^{opp}$,
and $\Ga$ vanishes if it admits an automorphism which changes the orientation.

\sip

{\em A warning}: we sometimes
show neither the labellings of vertices and boundaries nor the choice of an orientation
(e.g.\, directions of edges) in our pictures of elements of $\cR\cG ra_d(m,n)$ below; in such cases some choices of these data are assumed by default.

\subsubsection{\bf Examples} \label{4: example on vanishing (1,1) graph}
(i)
The graph $\xy
(0,-2)*{\bu}="A";
(0,-2)*{\bu}="B";
"A"; "B" **\crv{(-4,6) & (10,6)};
"A"; "B" **\crv{(-10,6) & (4,6)};
\endxy$ equals zero in $\cR\cG ra_{d}(1,1)$ for {\em any}\, $d$ as it admits a changing orientation automorphism.

(ii) The permutation group $\bS_2$ acts on $\xy
 (0,2)*{^1},
(6,2)*{^2},
 (0,0)*{\bullet}="a",
(6,0)*{\bu}="b",
\ar @{-} "a";"b" <0pt>
\endxy\in \cR\cG ra_{d+1}(1,2)$ by changing the labels of the vertices. One has, for
$(12)\in \bS_2$,
$$(12)( \xy
 (0,2)*{^1},
(6,2)*{^2},
 (0,0)*{\bullet}="a",
(6,0)*{\bu}="b",
\ar @{-} "a";"b" <0pt>
\endxy
)
=(-1)^{d} \xy
 (0,2)*{^2},
(6,2)*{^1},
 (0,0)*{\bullet}="a",
(6,0)*{\bu}="b",
\ar @{-} "a";"b" <0pt>
\endxy
$$

(iii) The permutation group $\bS_{\bar{2}}$ acts on  $\xy
 (0.5,1)*{^{\bar{1}}},
(0.5,5)*{^{\bar{2}}},
(0,-2)*{\bu}="A";
(0,-2)*{\bu}="B";
"A"; "B" **\crv{(6,6) & (-6,6)};
\endxy \in \cR\cG ra_{d}(2,1)$ by  flipping the labels $\bar{1}$ and $\bar{2}$ of the boundaries. One has, for
$(\bar{1}\bar{2})\in \bS_{\bar{2}}$,
$$
(\bar{1}\bar{2})(\xy
 (0.5,1)*{^{\bar{1}}},
(0.5,5)*{^{\bar{2}}},
(0,-2)*{\bu}="A";
(0,-2)*{\bu}="B";
"A"; "B" **\crv{(6,6) & (-6,6)};
\endxy
)
=(-1)^{d}\ \xy
 (0.5,1)*{^{\bar{2}}},
(0.5,5)*{^{\bar{1}}},
(0,-2)*{\bu}="A";
(0,-2)*{\bu}="B";
"A"; "B" **\crv{(6,6) & (-6,6)};
\endxy
$$
in $\cR\cG ra_{d}(2,1)$.
\bip

The collection of $\bS$-bimodules,
$$
\cR\cG ra_{d}:=\left\{\cR\cG ra_{d}(m,n)\right\},
$$
forms a prop in the category of graded vector spaces with respect to the
horizontal composition
$$
 \Ba{rccc}
 \circ: & \cR\cG ra_d(m_1,n_1)\ot_\K \cR\cG ra_d(m_2,n_2) &\lon &
 \cR\cG ra_d(m_1+m_2,n_1+n_2)\\
        &    \Ga_2\ot  \Ga_1 & \lon & \Ga_2\sqcup \Ga_1
        \Ea
 $$
defined as the disjoint union of ribbon graphs,
and the vertical composition,
 $$
 \Ba{rccc}
 \circ: & \cR\cG ra_d(p,m)\ot_\K \cR\cG ra_d(m,n) &\lon & \cR\cG ra_d(p,n)\\
        &    \Ga_2\ot  \Ga_1 & \lon & \Ga_2\circ \Ga_1
        \Ea
 $$
which is defined by gluing, for every  $i\in [m]$, the  $i$-th oriented boundary
($out$-circle) of $\Ga_1$,
$$
 \Ba{c}\resizebox{15mm}{!}
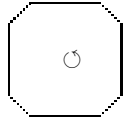
{\xy
(6,-8)*{\circlearrowleft},
 (0,0)*{}="a1",
(10,0)*{}="a2",
(13,-3)*{}="a3",
(13,-13)*{}="a4",
(10,-16)*{}="a5",
(0,-16)*{}="a6",
(-3,-13)*{}="a7",
(-3,-3)*{}="a8",
\ar @{-} "a1";"a2" <0pt>
\ar @{.} "a2";"a3" <0pt>
\ar @{-} "a3";"a4" <0pt>
\ar @{.} "a4";"a5" <0pt>
\ar @{-} "a5";"a6" <0pt>
\ar @{.} "a6";"a7" <0pt>
\ar @{-} "a7";"a8" <0pt>
\ar @{.} "a8";"a1" <0pt>
\endxy}\Ea
$$
to the $i$-th oriented blown-up vertex ($in$-circle) of $\Ga_2$,
$$
 \Ba{c}\resizebox{21mm}{!}
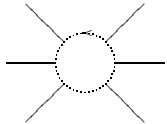
{\xy
 (0,0)*{
\xycircle(6,6){.}};
(-6,0)*{}="1";
(-16,0)*{}="1'";
(4.2,4.2)*{}="2";
(12,12)*{}="2'";
(4.2,-4.2)*{}="3";
(12,-12)*{}="3'";
(6,0)*{}="4";
(16,0)*{}="4'";
(-4.2,4.2)*{}="5";
(-12,12)*{}="5'";
(-4.2,-4.2)*{}="6";
(-12,-12)*{}="6'";
(-0.3,6)*{}="a";
(-0.4,6)*{}="b";
\ar @{-} "1";"1'" <0pt>
\ar @{-} "2";"2'" <0pt>
\ar @{-} "3";"3'" <0pt>
\ar @{-} "4";"4'" <0pt>
\ar @{-} "5";"5'" <0pt>
\ar @{->} "a";"b" <0pt>
\ar @{-} "6";"6'" <0pt>
\endxy}\Ea
$$
 as follows  (cf.\ \cite{TZ}): first we place the
 $out$-circle inside the $in$-one, then we erase the dashed $in$-circle leaving edges
 attached to that circle hanging (for a moment) in the ``air",
 and finally we take a sum over all possible reattachments of the hanging edges to the dashed intervals (corners) of the
$out$-circle  while respecting the cyclic orderings on both sets ---  the
set of hanging edges of the $i$-th vertex of $\Ga_2$ and the set of corners of the
$i$-th boundary of $\Ga_1$,
$$
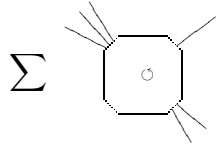
\sum
 \Ba{c}\resizebox{20mm}{!}{\xy
(6,-8)*{\circlearrowleft},
 (0,0)*{}="a1",
(10,0)*{}="a2",
(13,-3)*{}="a3",
(13,-13)*{}="a4",
(10,-16)*{}="a5",
(0,-16)*{}="a6",
(-3,-13)*{}="a7",
(-3,-3)*{}="a8",
(-1,-1)*{}="1";
(-6,7)*{}="1'";
(12,-2)*{}="2";
(20,4)*{}="2'";
(12,-14)*{}="3";
(18,-19)*{}="3'";
(-1.8,-1.8)*{}="4";
(-8,5)*{}="4'";
(-2.5,-2.5)*{}="5";
(-11,2)*{}="5'";
(11,-15)*{}="6";
(15,-22)*{}="6'";
\ar @{-} "1";"1'" <0pt>
\ar @{-} "2";"2'" <0pt>
\ar @{-} "3";"3'" <0pt>
\ar @{-} "4";"4'" <0pt>
\ar @{-} "5";"5'" <0pt>
\ar @{-} "6";"6'" <0pt>
\ar @{-} "a1";"a2" <0pt>
\ar @{.} "a2";"a3" <0pt>
\ar @{-} "a3";"a4" <0pt>
\ar @{.} "a4";"a5" <0pt>
\ar @{-} "a5";"a6" <0pt>
\ar @{.} "a6";"a7" <0pt>
\ar @{-} "a7";"a8" <0pt>
\ar @{.} "a8";"a1" <0pt>
\endxy}
\Ea
$$
Every ribbon graph in this linear combination comes equipped naturally with an induced
orientation, and belongs in fact to $\cR\cG ra_{d}(p,n)$.
It is obvious that this operation satisfies the required axioms for vertical compositions; it is slightly less
obvious that every summand in the resulting linear combination of ribbon graph has
precisely $p$ boundaries, see \S {\ref{4: claim on edge between corners}} above for a hint.
The graph $\bu$ consisting of a single vertex serves as the unit in $\cR \cG ra_{d}$.

\mip

The subspace of $\cR \cG ra_{d}$ spanned by {\em connected}\, ribbon graphs forms a
{\em properad}\, which we denote by the same symbol $\cR \cG ra_{d}$; this should not lead
to confusion as the prop closure of the properad $\cR \cG ra_d$ is precisely the prop
$\cR \cG ra_{d}$; similarly we do not distinguish in notation, for example,
 properad $\LoBd$ and prop
$\LoBd$ as the former completely  determines the latter and vice versa. Thus when we talk about {\em properad}\,  $\cR \cG ra_{d}$ we always mean the subspace of {\em connected}\, ribbon graphs; when we talk about {\em prop}\, $\cR \cG ra_{d}$ we mean the full space of (not necessary connected) ribbon graphs.

\mip

The main motivation for the above definition of the properad $\cR \cG ra_d$
is the following theorem-construction.

\subsubsection{{\bf Theorem}}\label{4: Theorem on repr of Rgra in CycW}
{\em Let $W$ be an arbitrary graded vector space and  $
Cyc^\bu W=\oplus_{n\geq 1} (W^{\ot n})^{\Z_n},
$ the associated space of cyclic words. Then any  (not necessarily non-degenerate)
 pairing $\Theta:  W\ot W  \lon  \K[1-d]$
satisfying (\ref{2: skewsymmetry on scalar product})
gives canonically rise to a representation of prop(erad)s,
$$
\rho_W: \cR \cG ra_{d} \lon \cE nd_{Cyc^\bu W}.
$$
}

\begin{proof} We shall show the argument for $d$ even; the case of $d$ odd is analogous,
and is left to the reader.
 Let
\Beq\label{4: cyclic words cW}
\left\{\cW_i:= (w_{i_1}\ot\ldots\ot w_{{l_i}})_{\Z_{l_i}}\right\}_{ 1\leq i\leq n,\ \ l_i\in \N}\,
\Eeq
be a collection of $n$ cyclic words from $Cyc^\bu(W)$, and let $\Ga\in \cR \cG ra_{d}(m,n)$ be
 a ribbon graph
 with vertices $(v_1,\ldots, v_i,\ldots, v_n)$
and
boundaries $(b_1,\ldots, b_j,\ldots, b_m)$ such that  $\# H(v_i) \leq l_i$,
$\forall i\in [n]$, i.e.\ the valency of each vertex $v_i$ is less than or equal
to the length of the $i$-th cyclic word. A {\em state}\, $s$ on $\Ga$ is, by
definition, an injective morphism of cyclically ordered sets,
$$
s: H(v_i) \lon (w_{i_1}\ot\ldots\ot w_{{l_i}})_{\Z_{l_i}},
$$
that is, an assignment to each half edge $h_i\in H(v_i)$ of an element
$s(h_i)$ from the set  $(w_{i_1}\ot\ldots\ot w_{{l_i}})_{\Z_{l_i}}$
in a way which respects natural cyclic orderings of both sets. To each state $s$ we
associate an element
$$
\Phi_\Ga^s(\cW_1,\ldots, \cW_n)\in \ot^m Cyc^\bu W
$$
as follows:
\Bi
\item[(i)] recall that orbits, $e=(h, \sigma_1(h))$, of the involution
$\sigma_1: H(\Ga)\rar H(\Ga)$ form  the set of edges, $E(\Ga)$, of the graph $\Ga$.
Let us choose arbitrarily a half-edge representative  $h$
of each edge $e=(h, \sigma_1(h))\in E(\Ga)$, and let us denote this set of representatives by $\frac{1}{2}H(\Ga)\simeq E(\Ga)$.
The {\em weight}\, of a state $s$ is, by definition, the number
    $$
    \la_s:=\prod_{h\in  \frac{1}{2}H(\Ga)}\Theta\left(s(h), s(\sigma_1(h))\right).
    $$
\item[(ii)] The complement,  $(w_{i_1},\ldots,  w_{i_{l_p}})\setminus \Img s$,
splits into a  disjoint (cyclically ordered)  union of totally ordered subsets,
$\coprod_{c\in C(v)}I_c$,
    parameterized by the set of corners of the vertex $v$.

\item[(iii)] To each boundary $b_j\in B(\Ga)$ we can now associate
a cyclic word
$$\displaystyle
\cW'_j:=\left(\bigotimes_{c\in C(b_j)} I_c\right)_{Z_{\sum_{c\in C(b_j)}|I_c|}}
$$
where the tensor product is taken along the given cyclic ordering in the set $C(b_q)$,
and then define
$$
\Phi_\Ga^s(\cW_1,\ldots, \cW_n):=(-1)^\sigma \la_s
\cW'_{b_1}\ot\ldots \ot \cW'_{b_m}
$$
where $(-1)^\sigma$ is the standard Koszul sign of the regrouping permutation,
$$
\sigma:\cW_1\ot\ldots \ot \cW_n\lon \prod_{h\in  \frac{1}{2}H(\Ga)}\left(s(h)\ot
s(\sigma_1(h))\right)\ot \cW'_{b_1}\ot\ldots \ot \cW'_{b_m}.
$$
\Ei
Note that $\Phi_\Ga^s(\cW_1,\ldots, \cW_n)$ does not depend on the choice of half-edge representatives of  edges, i.e.\ on the choice of an isomorphism   $\frac{1}{2}H(\Ga)\simeq E(\Ga)$.

\sip

Finally we define a linear map,
$$
\Ba{rccc}
\rho_W: & \cR\cG ra_{d}(m,n) & \lon & \Hom(\ot^n Cyc^\bu W, \ot^m Cyc^\bu W) \\
 & \Ga & \lon & \rho_W(\Ga),
 \Ea
$$
by setting the value of $\rho_W(\Ga)$ on cyclic words (\ref{4: cyclic words cW}) to be
equal to
$$
\rho_W(\Ga)(\cW_1,\ldots, \cW_n):=\left\{\Ba{cl} 0 & \mbox{if}\ \# H(v_p)> l_p\
\mbox{for some}\ p\in [n] \\
\displaystyle \sum_{\mathrm{all\ possiblle}\atop \mathrm{states}\ s} \Phi_\Ga^s(\cW_1,\ldots,
 \cW_n) &  \mbox{otherwise} \Ea\right.
$$
It is now straightforward to check that the map $\rho_W$ respects prop
compositions  in  $\cR \cG ra_d$ and $\cE nd_{Cyc^\bu W}$ (as the prop structure in
  $\cR \cG ra_{d}$  was designed above just to make this claim true).
\end{proof}

\subsubsection{{\bf Theorem} {\em (cf.\ \cite{CS})}}\label{4: Theorem on LieB --> Rgra}
{\em There is a morphism of prop(erad)s,
$$
s^\diamond: \LoB_{d,d} \lon \cR\cG ra_{d}
$$
given on generators as follows,
\Beq\label{4: maps s from LoBd to RGra}
s^\diamond\left(
 \begin{xy}
 <0mm,-0.55mm>*{};<0mm,-2.5mm>*{}**@{-},
 <0.5mm,0.5mm>*{};<2.2mm,2.2mm>*{}**@{-},
 <-0.48mm,0.48mm>*{};<-2.2mm,2.2mm>*{}**@{-},
 <0mm,0mm>*{\circ};<0mm,0mm>*{}**@{},
 \end{xy}\right)=\xy
(0,-2)*{\bu}="A";
(0,-2)*{\bu}="B";
"A"; "B" **\crv{(6,6) & (-6,6)};
\endxy
 \ \ \  , \ \ \
s^\diamond\left(
 \begin{xy}
 <0mm,0.66mm>*{};<0mm,3mm>*{}**@{-},
 <0.39mm,-0.39mm>*{};<2.2mm,-2.2mm>*{}**@{-},
 <-0.35mm,-0.35mm>*{};<-2.2mm,-2.2mm>*{}**@{-},
 <0mm,0mm>*{\circ};<0mm,0mm>*{}**@{},
 \end{xy}\right)=\xy
 (0,0)*{\bullet}="a",
(6,0)*{\bu}="b",
\ar @{-} "a";"b" <0pt>
\endxy
\Eeq
}

\begin{proof} One has to show that the map $s^\diamond$ sends to zero
all the Lie bialgebra relations (\ref{R for LieB})  as well as the involutivity relation
$$
\Ba{c}\resizebox{5.5mm}{!}{
\xy
 (0,0)*{\circ}="a",
(0,6)*{\circ}="b",
(3,3)*{}="c",
(-3,3)*{}="d",
 (0,9)*{}="b'",
(0,-3)*{}="a'",
\ar@{-} "a";"c" <0pt>
\ar @{-} "a";"d" <0pt>
\ar @{-} "a";"a'" <0pt>
\ar @{-} "b";"c" <0pt>
\ar @{-} "b";"d" <0pt>
\ar @{-} "b";"b'" <0pt>
\endxy}
\Ea=0
$$
 among the generators $\begin{xy}
 <0mm,-0.55mm>*{};<0mm,-2.5mm>*{}**@{-},
 <0.5mm,0.5mm>*{};<2.2mm,2.2mm>*{}**@{-},
 <-0.48mm,0.48mm>*{};<-2.2mm,2.2mm>*{}**@{-},
 <0mm,0mm>*{\circ};<0mm,0mm>*{}**@{},
 \end{xy}$ and  $\begin{xy}
 <0mm,0.66mm>*{};<0mm,3mm>*{}**@{-},
 <0.39mm,-0.39mm>*{};<2.2mm,-2.2mm>*{}**@{-},
 <-0.35mm,-0.35mm>*{};<-2.2mm,-2.2mm>*{}**@{-},
 <0mm,0mm>*{\circ};<0mm,0mm>*{}**@{},
 \end{xy}$. The involutivity relation is respected by $\rho$ as
$$
s^\diamond\left(\Ba{c}\resizebox{4.5mm}{!}{\xy
 (0,0)*{\circ}="a",
(0,6)*{\circ}="b",
(3,3)*{}="c",
(-3,3)*{}="d",
 (0,9)*{}="b'",
(0,-3)*{}="a'",
\ar@{-} "a";"c" <0pt>
\ar @{-} "a";"d" <0pt>
\ar @{-} "a";"a'" <0pt>
\ar @{-} "b";"c" <0pt>
\ar @{-} "b";"d" <0pt>
\ar @{-} "b";"b'" <0pt>
\endxy}\Ea\right)=\xy
(0,-2)*{\bu}="A";
(0,-2)*{\bu}="B";
"A"; "B" **\crv{(-4,6) & (10,6)};
"A"; "B" **\crv{(-10,6) & (4,6)};
\endxy
$$
and the latter graph vanishes identically  in $\cR\cG ra_{d}(1,1)$ indeed (see
\S {\ref{4: example on vanishing (1,1) graph}}(i)).
All the other required conditions on $s^\diamond$ can be checked by as elementary calculations
as the above one; we leave the details to the reader.
\end{proof}

\subsubsection{\bf Remark}\label{4 remark on proof of LieBi on CycW}
Composing the above morphism $s^\diamond$ with the representation $\rho_W$ of
Theorem~{\ref{4: Theorem on repr of Rgra in CycW}}, we get a proof of Proposition
~{\ref{2: Proposition on inv LieBi in CycW}}.

\subsection{Ribbon graph complex and universal deformations
of  a class of involutive
Lie bialgebras}\label{3: subsection RGCd^diamond} Using results of \S {\ref{2: deformation theory of inv Lie bialgebras}}
we can now explicitly describe the dg Lie algebra, the {\em  Ribbon Graph Complex},
$$
\mathsf{RGC}_d^\diamond:=\Def\left(\LoB_{d,d} \stackrel{s^\diamond}{\rar} \cR\cG ra_{d}\right)\equiv\Def\left(\HoLoB_{d,d} \stackrel{s'}{\rar} \cR\cG ra_{d}\right)
,
$$
which controls deformations of the following composition of morphisms of dg {\em properads},
$$
s':\HoLoB_{d,d} \stackrel{\mathrm{proj}}{\lon} \LoB_{d,d} \stackrel{s^\diamond}{\lon} \cR \cG ra_{d}.
$$
Here the first arrow is the natural projection and the second arrow is given by
Theorem {\ref{4: Theorem on repr of Rgra in CycW}}. One can understand this deformation
theory as the theory of {\em universal}\,
$\HoLoBd$-deformations of the standard involutive
bialgebra structure in a vector space $W$ equipped with a scalar product $\Theta$ as in \S {\ref{2: subsection on cyclic words}}.

\mip

Let $\hbar$ be a formal variable of homological degree $2d$, then we have a continuous
(in the $\hbar$-adic topology) morphism of dg properads,
$$
s^\hbar: \HoLB_{d,d}^\hbar \stackrel{proj}{\lon} \LoB_{d,d}[[\hbar]] \stackrel{s^\diamond}{\lon}
\cR \cG ra_{d}[[\hbar]],
$$
and an isomorphism of dg Lie algebras,
$$
\mathsf{RGC}^\diamond_d\simeq \Def_{adm}\left(\HoLB_{d,d}^\hbar \stackrel{s^\hbar}{\rar}
\cR\cG ra_{d}[[\hbar]]\right)
$$
where the complex on the r.h.s.\ describes deformations of the map $s^\hbar$ within the
class of all possible {\em admissible}\,  properadic morphisms,
$f: \HoLB_{d,d}^\hbar \stackrel{f}{\rar} \cR\cG ra_{d}[[\hbar]]$, that is, the ones which
satisfy the condition
$$
s^\hbar\left(\begin{xy}
 <0mm,0mm>*{\bu};
 <0mm,-0.5mm>*{};<0mm,-2.5mm>*{}**@{-},
 <0mm,0.5mm>*{};<0mm,2.5mm>*{}**@{-},
 \end{xy}\right)\in \hbar\,\cR\cG ra_{d+1}[[\hbar]].
$$
Therefore, we have a canonical isomorphism of graded vector spaces,
$$
\mathsf{RGC}_{d}^\diamond=\hbar\,\cR\cG ra_{d}(1,1)[[\hbar]]\ \ \oplus\ \
\bigoplus_{m,n\geq 1\atop m+n\geq 3} \left(\cR\cG ra_{d}(m,n)\ot_{\bS_m\times\bS_n}
(sgn_n^{|d|}\ot
sgn_m^{|d|})[[\hbar]]\right)[d(2-m-n)]
$$
i.e.\ a homogeneous element $\Ga$ of $\mathsf{RGC}_{d}^\diamond$ can be understood as a formal
power series
$$
\Ga=\sum_{a\geq 0} \Ga_{a} \hbar^a
$$
where $\Ga_{a}$ is
 a connected {\em oriented}\, ribbon graph  with $m=\# V(\Ga_{a})$ vertices and $n=\# B(\Ga_{a})$
  where
\Bi
\item the orientation $or$ of $\Ga_a$ is defined  a choice of the unit vector\footnote{For a finite set $S$ we denote $\det(S):=\wedge^{\# S}(\K[S])$, where $\K[S]$ is the linear span of $S$ over the field $\K$; we also equip $\det(S)$ with a Euclidian norm.}
in the one-dimensional vector space $\det (E(\Ga_a))$ (for $d$ even)
or $\det (V(\Ga_a))\ot \det(B(\Ga_a)) \bigotimes_{e\in E(\Ga_a)}\det(H(e))$ (for $d$ odd), where $H(e)$ is the set of two half-edges of the given edge $e$; every graph has precisely two orientations, $(\Ga_a, or)$ and $(\Ga_a, -or)$, and we identify
$(\Ga_a, or)=-(\Ga_a, -or)$.

\sip

\item $
|\Ga_{(a)}|= d\left(\# V(\Ga_{(a)})+ \# B(\Ga_{(a)}) -2)+ (1-d)\# E(\Ga_{(a)}\right) -
 2ad= -2(g+a)d + \# E(\Ga)$,
where $g=1-\frac{1}{2}(\# V(\Ga_{a})- \# E(\Ga_a) +\# B(\Ga_{a}))$ is the genus of the compact Riemann
surface associated to $\Ga_a$.

\item $\Ga_{0}=0$\ \ if\ \ $\# V(\Ga_{0})= \# B(\Ga_{0})=1$.
\Ei

The Lie brackets $[\ ,\ ]$ in $\mathsf{RGC}_{d}^\diamond$ can be read off from the
differential
(\ref{2: d_hbar}) along the lines explained in \cite{MV}. Maurer-Cartan elements in
$(\mathsf{RGC}_d^\diamond, [\ ,\ ])$  are in one-to-one correspondence with admissible morphisms of dg props,
$(\HoLhBd,\delta_\hbar) \rar (\cR \cG ra_{d}[[\hbar]],0)$, or, equivalently, with arbitrary
 morphisms
of dg props $(\HoLoBd,\delta) \rar (\cR \cG ra_{d},0)$; here we understand $\cR \cG ra_{d}$
 as a prop with zero differential. In particular, the canonical morphism $s^\diamond$ determined by Theorem~{\ref{4: Theorem on LieB --> Rgra}}
is given by the following Maurer-Cartan element,
$$
\Ga_{\diamond}=  \xy
 (0,0)*{\bullet}="a",
(6,0)*{\bu}="b",
\ar @{-} "a";"b" <0pt>
\endxy\ \  + \ \xy
(0,-2)*{\bu}="A";
(0,-2)*{\bu}="B";
"A"; "B" **\crv{(6,6) & (-6,6)};
\endxy
$$
It makes $\mathsf{RGC}^\diamond_{d}$ into a complex  with the differential
$
\delta_\diamond:=[\Ga_\diamond,\ ]
$. For a connected ribbon graph $\Ga$ one has (up to signs)
\Beqr
\delta_\diamond\Ga &=& \underbrace{\sum_{v\in V(\Ga)}\sum_{H(v)=I_1\sqcup I_2\atop I_1,I_2
\neq \emptyset} \pi_{I_1,I_2}(\Ga)}_{\delta\Ga}
 + \underbrace{\sum_{b\in B(\Ga)}\sum_{c_1,c_2\ in\ C(b)\atop c_1\neq c_2}
 e_{c_1,c_2}(\Ga)}_{\Delta_1\Ga}+\label{4: d_diamond in fRGC}\\
&&+
\underbrace{\hbar\sum_{b_1,b_2\in B(\Ga)\atop b_1\neq b_2} \sum_{c_1\in C(b_1)\atop
c_2\in C(b_2)} e_{c_1,c_2}(\Ga)}_{\Delta_2'\Ga} + \underbrace{\hbar\sum_{v_1,v_2\in
 V(\Ga)\atop v_1\neq v_2} \sum_{c_1\in C(v_1)\atop
c_2\in C(v_2)} \pi_{c_1,c_2}(\Ga)}_{\Delta_2''\Ga} \nonumber
\Eeqr
where
\Bi
\item[-] $e_{c_1,c_2}(\Ga)$ is the graph obtained from $\Ga$ by attaching a new edge
connecting the corner $c_1$ to the corner $c_2$ (see \S {\ref{4: claim on edge between corners}})
\item[-] for a fixed splitting, $H(v)=I_1\sqcup I_2$, of the set of half-edges  attached
to the vertex $v$ into two non-empty subsets
$I_1$ and $I_2$ (each equipped with the induced cyclic ordering), the associated linear
combination of ribbon graphs $\pi_{I_1,I_2}(\Ga)$  is obtained from $\Ga$ by (i) replacing
the vertex $v$
with the 2-vertex graph $\xy
 (-2,0)*{\bullet}="a",
(2,0)*{\bu}="b",
\ar @{-} "a";"b" <0pt>
\endxy$, (ii)  attaching half-edges from $I_1$ (respectively, $I_2$) to one vertex $v_1$
(respectively, another vertex $v_2$)) of the graph  $\xy
 (-2,0)*{\bullet}="a",
(2,0)*{\bu}="b",
\ar @{-} "a";"b" <0pt>
\endxy$, and (iii) taking a sum over possible ways of equipping the sets
$H(v_1)=I_1\sqcup \xy
 (0,-1.5)*{}="a",
(0,1.5)*{}="b",
\ar @{-} "a";"b" <0pt>
\endxy$ and $H(v_2)=I_2\sqcup \xy
 (0,-1.5)*{}="a",
(0,1.5)*{}="b",
\ar @{-} "a";"b" <0pt>
\endxy$ with cyclic structures  which agree with the given cyclic structures on the
subsets $I_1$ and, respectively, $I_2$;
\item[-] for any fixed pair of different vertices, $v_1$ and $v_2$, of $\Ga$,
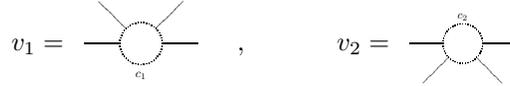
$$
v_1=
 \Ba{c}\resizebox{15mm}{!}{\xy
 (0,0)*{
\xycircle(6,6){.}};
(-6,0)*{}="1";
(-16,0)*{}="1'";
(4.2,4.2)*{}="2";
(12,12)*{}="2'";
(4.2,-4.2)*{}="3";
(12,-12)*{}="3'";
(6,0)*{}="4";
(16,0)*{}="4'";
(-4.2,4.2)*{}="5";
(-12,12)*{}="5'";
(-4.2,-4.2)*{}="6";
(-12,-12)*{}="6'";
(0,-9)*{c_1};
\ar @{-} "1";"1'" <0pt>
\ar @{-} "2";"2'" <0pt>
\ar @{-} "4";"4'" <0pt>
\ar @{-} "5";"5'" <0pt>
\endxy}\Ea\ \ \ , \ \ \ \ \ \ \ \ \ \
v_2=
 \Ba{c}\resizebox{14mm}{!}{\xy
 (0,0)*{
\xycircle(6,6){.}};
(-6,0)*{}="1";
(-16,0)*{}="1'";
(4.2,4.2)*{}="2";
(12,12)*{}="2'";
(4.2,-4.2)*{}="3";
(12,-12)*{}="3'";
(6,0)*{}="4";
(16,0)*{}="4'";
(-4.2,4.2)*{}="5";
(-12,12)*{}="5'";
(-4.2,-4.2)*{}="6";
(-12,-12)*{}="6'";
(0,8)*{c_2};
\ar @{-} "1";"1'" <0pt>
\ar @{-} "3";"3'" <0pt>
\ar @{-} "4";"4'" <0pt>
\ar @{-} "6";"6'" <0pt>
\endxy}\Ea
$$
and for any
fixed corners $c_1$ and $c_2$ belonging to $C(v_1)$ and, respectively, $C(v_2)$,
 the associated graph  $\pi_{c_1,c_2}(\Ga)$ is obtained from $\Ga$ by glueing vertices
 $v_1$ and $v_2$ into a single vertex along the marked corners as shown in the picture,

$$
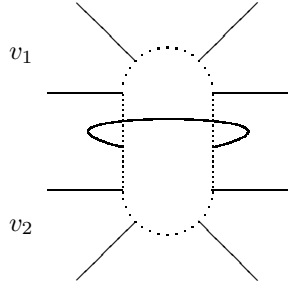
\Ba{c}
\mbox{$v_1  \Ba{c}
\xy
(0,0)*\ellipse(6,6),=:a(180){.};
(-6,0)*{}="1";
(-16,0)*{}="1'";
(4.2,4.2)*{}="2";
(12,12)*{}="2'";
(6,0)*{}="4";
(16,0)*{}="4'";
(-4.2,4.2)*{}="5";
(-12,12)*{}="5'";
\ar @{-} "1";"1'" <0pt>
\ar @{-} "2";"2'" <0pt>
\ar @{-} "4";"4'" <0pt>
\ar @{-} "5";"5'" <0pt>
\endxy
\Ea$}
\vspace{-1mm}
\\
\hspace{3.5mm}
\mbox{$\xy
 (-6,5)*{}="1",
(-6,-8)*{}="1'",
(6,5)*{}="2",
(6,-8)*{}="2'",
(-6,-2)*{}="a";
(6,-2)*{}="b";
"a";"b" **\crv{(-26,3) & (26,3)};
\ar @{.} "1";"1'" <0pt>
\ar @{.} "2";"2'" <0pt>
\endxy$}
\vspace{-3.3mm}
\\
\mbox{$v_2 \Ba{c}
\xy
(0,0)*\ellipse(6,6)__,=:a(-180){.};
(-6,0)*{}="1";
(-16,0)*{}="1'";
(4.2,-4.2)*{}="3";
(12,-12)*{}="3'";
(6,0)*{}="4";
(16,0)*{}="4'";
(-4.2,-4.2)*{}="6";
(-12,-12)*{}="6'";
\ar @{-} "1";"1'" <0pt>
\ar @{-} "3";"3'" <0pt>
\ar @{-} "4";"4'" <0pt>
\ar @{-} "6";"6'" <0pt>
\endxy\Ea$}
\Ea
$$
and then attaching a new loop-like edge to the new vertex as also shown on the picture above.
\Ei

\subsubsection{\bf Remarks}\label{4: Remarks} (i) Let $\sG_{m,n}^l$ be the subspace of $\mathsf{RGC}^\diamond_{d}(m,n)$ spanned (over $\K[[\hbar]])$ by graphs with $l$ edges.
Then for  $\Ga \in  \sG_{m,n}^l$ one has $\delta_\diamond\Ga=\delta\Ga+ \Delta_1\Ga+
 \hbar(\Delta_2'\Ga+\Delta_2'') \Ga$, where $\delta\Ga \in \sG_{m,n+1}^{l+1}$,
$\Delta_1\Ga \in \sG_{m+1,n}^{l+1}$, $\Delta_2'\Ga \in \sG_{m-1,n}^{l+1}$ and
$\Delta_2''\Ga \in \sG_{m,n-1}^{l+1}$. We denote $\Delta_2:= \Delta_2' + \Delta_2''$
so that $\delta_\diamond=\delta +\Delta_1 + \hbar\Delta_2$.

\sip

It follows that operations $\p_1:=\delta + \hbar\Delta_2'$ and
$\p_2:= \Delta_1 + \hbar \Delta_2''$ are also differentials in $\RGC_d^\diamond$.
If we consider a morphism of properads
\[
 s^\diamond_1 : \LoB_{d,d} \lon \RGra_d
\]
such that
\Beq\label{4: map s^diamond_1}
s^\diamond_1\left(
 \begin{xy}
 <0mm,-0.55mm>*{};<0mm,-2.5mm>*{}**@{-},
 <0.5mm,0.5mm>*{};<2.2mm,2.2mm>*{}**@{-},
 <-0.48mm,0.48mm>*{};<-2.2mm,2.2mm>*{}**@{-},
 <0mm,0mm>*{\circ};<0mm,0mm>*{}**@{},
 \end{xy}\right)=0
 \ \ \  , \ \ \
s^\diamond_1\left(
 \begin{xy}
 <0mm,0.66mm>*{};<0mm,3mm>*{}**@{-},
 <0.39mm,-0.39mm>*{};<2.2mm,-2.2mm>*{}**@{-},
 <-0.35mm,-0.35mm>*{};<-2.2mm,-2.2mm>*{}**@{-},
 <0mm,0mm>*{\circ};<0mm,0mm>*{}**@{},
 \end{xy}\right)=\xy
 (0,0)*{\bullet}="a",
(6,0)*{\bu}="b",
\ar @{-} "a";"b" <0pt>
\endxy
\Eeq
The associated  ribbon graph complex
$\Def\left(\LoB_{d,d} \stackrel{s^\diamond_1}{\rar} \cR\cG ra_{d}\right)$ coincides
with $\RGC_d^\diamond$ as a graded vector space but comes equipped with the differential
$\p_1=\delta+ \hbar \Delta_2'$. Similarly, the complex $(\RGC_d,\p_2)$ can be identified with
$\Def\left(\LoB_{d,d} \stackrel{s^\diamond_2}{\rar} \cR\cG ra_{d}\right)$, where $s^\diamond_2$ is a version of the map (\ref{4: map s^diamond_1})
which sends to zero the Lie generator (rather than the coLie one).

\mip

(ii)
The quotient dg Lie algebra $\mathsf{RGC}_d:=\mathsf{RGC}_d^\diamond/\hbar
\mathsf{RGC}_d^\diamond$ controls deformations of the canonical map $s: \LB_{d,d}\rar \cR\cG ra_{d}$ given by the composition  $\LB_{d,d}\rar \LoB_{d,d}\stackrel{s^\diamond}{\rar} \cR\cG ra_{d}$ , or,
equivalently, universal
deformations of the standard Lie bialgebra structure on $Cyc^\bu W$ (with involutivity condition
forgotten). Its differential is given by the first two terms, $\delta + \Delta_1$, in
(\ref{4: d_diamond in fRGC}). This complex plays an important role below. The operator $\Delta_2: \RGC_d\rar \RGC_d$
has degree $1-2d$; it commutes with  $\delta + \Delta_1$
and satisfies $\Delta_2^2=0$.

\sip

\sip

Note  that the deformation complex $\Def\left(\LB_{d,d} \stackrel{s^*}{\rar} \cR\cG ra_{d}\right)$  associated to the
composition map $$
s^*: \LB_{d,d} \lon \LoB_{d,d}\stackrel{s_1^\diamond}{\lon} \cR\cG ra_d
$$
is identical to
the complex $(\RGC_d,\delta)$.

\mip


\mip

(iii) The case $d=0$ is of special interest as the complex $\mathsf{RGC}:= (\mathsf{RGC}_0, \delta)$
computes cohomology of the moduli space $\cM_{g,n}$ of genus $g$ algebraic curves with $n$ punctures
as shown in (\ref{1: Cohomology of RBC and M_g,n}). The complex $\mathsf{RGC}$ splits into a direct sum
$$
\mathsf{RGC}:= \mathsf{RGC}^{\geq 3} \oplus \mathsf{RGC}^{\leq 2}
$$
where  $\mathsf{RGC}^{\geq 3}$ is spanned by ribbon graphs with all vertices at least trivalent
and $\mathsf{RGC}^{\leq 2}$ is spanned by graphs which have at least one vertex of valency $\leq 2$.
The cohomology of the first direct summand is given by (see \cite{Ko3,Ko4,Ko2,LZ, Ha} for detailed definitions, proofs and further references)
$$
H^\bu(\mathsf{RGC}^{\geq 3})=\prod_{\substack{g,n \\ n>0,2-2g<n} }\left( H^{6g-6+3n-k}(\cM_{g,n}, \mathbb{K}) \otimes \sgn_n \right)/ \bS_n
$$
while (cf.\ \cite{Wi})
$$
H^\bu\left(\mathsf{RGC}^{\leq 2}\right)=
\begin{cases}
\mathbb K & \text{for $k=1,4,7,\dots$} \\
0 & \text{otherwise}
\end{cases}
$$
Note that in $d=0$ case all the operators $\delta$, $\Delta_1$ and $\Delta_2$ have degree $1$ so that it makes sense to equip the graded vector space $\mathsf{RGC}$  not only with differentials $\delta$, $\delta+\Delta_1$, but also with $\delta + \Delta_1+ \Delta_2$. By Remark (i) above, the operations $\delta+ \Delta_2'$ and $\Delta_1+ \Delta_2''$ are also differentials in $\mathsf{RGC}$.
%

\sip

It is worth emphasizing that the complex $(\mathsf{RGC},\delta)$
is not identical to the Kontsevich ribbon graph complex in \cite{Ko3} as it computes cohomology
   $H^\bu\left(\cM_{g,n}, \mathbb{Q}\right)$ with numbering of punctures {\em skew}-symmetrized, rather than symmetrized.

\mip

(iv) The complexes  $(\RGCd,\delta)$ with $d$ even (or for $d$ odd) are obviously isomorphic (up to degree shifts) to each other.  It is less evident the isomorphism claim holds true  for {\em all}\, $d$   as the one-dimensional vector spaces  $\det (E(\Ga_a))$ and
 $\det (V(\Ga_a))\ot \det(B(\Ga_a)) \bigotimes_{e\in E(\Ga_a)}\det(H(e))$ (which determine orientations for $d$ even and $d$ odd respectively) are canonically isomorphic to each other --- see Proposition 1 in \cite{CV}.  Hence all the complexes $(\RGCd,\delta)$   compute the (degree shifted versions of) cohomology groups $(H^\bu\left(\cM_{g,n}, \mathbb{K}\right) \otimes \sgn_n )/ \bS_n$.

 \mip

(v) To any ribbon graph from $\mathsf{RGC}_d$ one can associate the genus $g$ of the associated compact Riemann surface, and the differential $\delta$ preserves this numerical characteristic. Therefore the
complex $(\mathsf{RGC}_d, \delta)$ decomposes into a direct sum
$$
\mathsf{RGC}_d= \bigoplus_{g\geq 0} \mathsf{RGC}_d^{(g)}
$$
and one has an explicit relation between complexes for various $d$,
$$
\mathsf{RGC}_d^{(g)}= \mathsf{RGC}^{(g)}[2gd].
$$
%
%
\sip

(vi) Let $^*$ be a map $\cR^l_{m,n}\rar \cR^l_{n,m}$ associating to a ribbon graph $\Ga$
its dual graph $\Ga^*$ defined as follows: represent $\Ga$ as surface with $m$ punctures and $n$ output boundaries, glue punctured disks to each of the $n$ boundaries, and finally draw a ribbon graph $\Ga^*$ on the resulting surface
with $V(\Ga^*)$ identified with punctures of the glued disks and $E(\Ga^*)$ obtained as follows: two vertices from $V(\Ga^*)$ are connected by an edge in if and only if
the associated to these punctures boundaries  of $\Ga$ share a common edge in $\Ga$.
   The map $^*$ induces involutions
of the vector spaces  $\mathsf{RGC}_d^\diamond$ and $\mathsf{RGC}_d$. As
the ribbon graph dual to $\xy
 (0,0)*{\bullet}="a",
(6,0)*{\bu}="b",
\ar @{-} "a";"b" <0pt>
\endxy$ is $\xy
(0,-2)*{\bu}="A";
(0,-2)*{\bu}="B";
"A"; "B" **\crv{(6,6) & (-6,6)};
\endxy$, these involutions commute with the differentials in both complexes $\RGC^\diamond, \delta_\diamond)$ and $(\RGC, \delta + \Delta_1)$.



\bip

\subsubsection{{\bf Variation: Odd case and moduli space of curves with symmetric punctures}}\label{sec:defRGCgeneral}
The family of complexes
$$
(\RGC_{d},\delta) =\Def\left(\LB_{d,d} \stackrel{s^*}{\rar} \cR\cG ra_{d}\right) \ \ \ \forall\ d\in \Z,
$$
defined in \S {\ref{4: Remarks}}(ii) is in fact a subfamily of a larger family of complexes.   Indeed,
for any $c,d\in \Z$ we may define a morphism of properads
\[
 s^* : \LieB_{c,d} \to \RGra_d
\]
such that
$$
s^*\left(
 \begin{xy}
 <0mm,-0.55mm>*{};<0mm,-2.5mm>*{}**@{-},
 <0.5mm,0.5mm>*{};<2.2mm,2.2mm>*{}**@{-},
 <-0.48mm,0.48mm>*{};<-2.2mm,2.2mm>*{}**@{-},
 <0mm,0mm>*{\circ};<0mm,0mm>*{}**@{},
 \end{xy}\right)=0
 \ \ \  , \ \ \
s^*\left(
 \begin{xy}
 <0mm,0.66mm>*{};<0mm,3mm>*{}**@{-},
 <0.39mm,-0.39mm>*{};<2.2mm,-2.2mm>*{}**@{-},
 <-0.35mm,-0.35mm>*{};<-2.2mm,-2.2mm>*{}**@{-},
 <0mm,0mm>*{\circ};<0mm,0mm>*{}**@{},
 \end{xy}\right)=\xy
 (0,0)*{\bullet}="a",
(6,0)*{\bu}="b",
\ar @{-} "a";"b" <0pt>
\endxy
$$
Hence we obtain the ribbon graph complexes
\[
 \RGC_{c,d} := \Def\left(\LieB_{c,d} \stackrel{s^*}{\rar} \cR\cG ra_{d}\right)\equiv\Def\left(\hoLieB_{c,d}  \stackrel{s'}{\rar} \cR\cG ra_{d}\right) \ \ \ \ \ \forall\ c,d\in \Z.
\]
which for $c=d$ coincide with the complexes $(\RGC_{d},\delta)$ considered in \S {\ref{4: Remarks}}(ii). Here $s'$ stands for the composition of the canonical projection $\hoLieB_{c,d}\rar \LB_{c,d}$  with $s^*$.

\sip

In particular, the ribbon graph complex
\[
 \RGC_{odd} := \RGC_{0,1}
\]
is the one originally considered by Kontsevich \cite{Ko6,Ko1,Ko2}; it computes the cohomology of the moduli space of curves with unidentifiable punctures. More concretely, we can split as before
\[
 \RGC_{odd} := \RGC_{odd}^{\leq 2}\oplus \RGC_{odd}^{\geq 3},
\]
where the first summand consists of series of graphs that have at least one vertex of valence $\leq 2$, while the second summand consists of graphs all of whose vertices are at least bivalent.
Then as shown in \cite{Ko2} $H(\RGC_{odd}^{\leq 2})$ is spanned by classes represented by loops of $k$ bivalent vertices, with $k=3,7,11,\dots$, while
\[
H(\RGC_{odd}^{\geq 3})
\cong
\prod_{\substack{g,n \\ n>0,2-2g<n} }\left( H^{4g-4+2n-k}(\cM_{g,n}, \mathbb{K}) \right)/ \bS_n.
.
\]

More generally, all $\RGC_{c,d}$ for $c$, $d$ of the same parity are essentially isomorphic to $\RGC$, up to certain degree shifts, and all  $\RGC_{c,d}$ for $c$, $d$ of opposite parity are essentially isomorphic to $\RGC_{odd}$, again up to certain degree shifts.

\sip

We finally note that apart from the ``trivial'' map $s^*$ more interesting maps from the properad of odd Lie bialgebras may be considered.

{\small


\subsubsection{{\bf Theorem}}\label{4: Theorem on LieB_0,1 --> Rgra_1}
{\em There is a properad map
$$
s^\Theta: \hoLieB_{0,1} \lon \cR\cG ra_{1},
$$
defined on generators such that
\begin{align*}
s^\Theta\left(
\resizebox{14mm}{!}{\begin{xy}
 <0mm,0mm>*{\circ};<0mm,0mm>*{}**@{},
 <-0.6mm,0.44mm>*{};<-8mm,5mm>*{}**@{-},
 <-0.4mm,0.7mm>*{};<-4.5mm,5mm>*{}**@{-},
 <0mm,0mm>*{};<-1mm,5mm>*{\ldots}**@{},
 <0.4mm,0.7mm>*{};<4.5mm,5mm>*{}**@{-},
 <0.6mm,0.44mm>*{};<8mm,5mm>*{}**@{-},
   <0mm,0mm>*{};<-8.5mm,5.5mm>*{^1}**@{},
   <0mm,0mm>*{};<-5mm,5.5mm>*{^2}**@{},
   <0mm,0mm>*{};<4.5mm,5.5mm>*{^{m\hspace{-0.5mm}-\hspace{-0.5mm}1}}**@{},
   <0mm,0mm>*{};<9.0mm,5.5mm>*{^m}**@{},
 <-0.6mm,-0.44mm>*{};<-8mm,-5mm>*{}**@{-},
 <-0.4mm,-0.7mm>*{};<-4.5mm,-5mm>*{}**@{-},
 <0mm,0mm>*{};<-1mm,-5mm>*{\ldots}**@{},
 <0.4mm,-0.7mm>*{};<4.5mm,-5mm>*{}**@{-},
 <0.6mm,-0.44mm>*{};<8mm,-5mm>*{}**@{-},
   <0mm,0mm>*{};<-8.5mm,-6.9mm>*{^1}**@{},
   <0mm,0mm>*{};<-5mm,-6.9mm>*{^2}**@{},
   <0mm,0mm>*{};<4.5mm,-6.9mm>*{^{n\hspace{-0.5mm}-\hspace{-0.5mm}1}}**@{},
   <0mm,0mm>*{};<9.0mm,-6.9mm>*{^n}**@{},
 \end{xy}}
 \right)=
 \begin{cases}
 \frac 1 m
\sum_{\sigma\in\bS_m}
 \underbrace{
\Ba{c}
{\xy
(5,0)*{...},
   \ar@/^1pc/(0,0)*{\bullet};(10,0)*{\bullet}
   \ar@/^{-1pc}/(0,0)*{\bullet};(10,0)*{\bullet}
   \ar@/^0.6pc/(0,0)*{\bullet};(10,0)*{\bullet}
   \ar@/^{-0.6pc}/(0,0)*{\bullet};(10,0)*{\bullet}
 \endxy}
 \Ea}_{m\ \mathrm{edges}}
%
%
&\text{if $n=2$ and $m$ odd}
\\
0 & \text{otherwise}
\end{cases},
\end{align*}
where the notation on the right shall indicate that one sums over all ways of numbering the $m$
 boundaries.}
\begin{proof}
 We have to show that the right-hand side of (the $\hoLieB_{0,1}$-analog of) \eqref{LBk_infty} is zero.
 All ribbon graphs appearing on this right-hand side are easily checked to be of the form
 
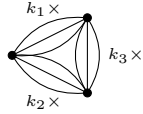
\[
\begin{tikzpicture}[baseline=-.65ex]
\node[int](v) at (0,0){};
\node[int](x) at (1,-.5){};
\node[int](w) at (1,.5){};
\draw (v) edge [bend right] (w) edge [bend left] node[above] {$\scriptstyle k_1\times $} (w) edge (w);
\draw (v) edge [bend left] (x) edge [bend right] node[below] {$\scriptstyle k_2\times $} (x) edge (x);
\draw (w) edge [bend left] node[right] {$\scriptstyle k_3\times $} (x) edge [bend right] (x) edge (x);
\end{tikzpicture}
 \]
with some numbering of the punctures, where (in the notation of \eqref{LBk_infty}) $n=3$ and $m=k_1+k_2+k_3$ is even. Fix some numbering of the punctures for concreteness (which  does not matter by symmetry).
Let us count which terms on the right-hand side of \eqref{LBk_infty} can produce terms such as the above for given $k_1,k_2,k_3$.
The terms that can contribute correspond to $|I_1|$ being equal to one of the $k_j$, which must be odd. Since $m$ is even, either all of the $k_j$ are even (then there is no contribution) or exactly two are odd, say $k_2$ and $k_3$.
Let us first suppose that $|I_1|=k_3$. Then the above contribution to the right-hand side of \eqref{LBk_infty} comes with a combinatorial multiplicity factor $1$, given that the fixed numbering of the punctures uniquely determines how the edges must be attached.
Similarly, the contribution with $|I_1|=k_2$ comes with a factor $-1$, the minus arising from the necessity to once permute the input labels.
If $k_2\neq k_3$, the sum is zero. Otherwise, if $k_2=k_3$ then the term is killed due to the anti-symmetrization on the input labels.
\end{proof}
%

\subsubsection{\bf Corollary}\label{4: Corollary on Holieb to Gra via wheels} {\em There are properad maps
$$
s^W_1:  \hoLieB_{1,0} \lon \cR\cG ra_{1}
$$
and
$$
s^W_2:  \hoLieB_{0,1} \lon \cR\cG ra_{1}\{1\}
$$
defined on the generators as follows
\begin{align*}
s^W_{i}\left(
\resizebox{16mm}{!}{\begin{xy}
 <0mm,0mm>*{\circ};<0mm,0mm>*{}**@{},
 <-0.6mm,0.44mm>*{};<-8mm,5mm>*{}**@{-},
 <-0.4mm,0.7mm>*{};<-4.5mm,5mm>*{}**@{-},
 <0mm,0mm>*{};<-1mm,5mm>*{\ldots}**@{},
 <0.4mm,0.7mm>*{};<4.5mm,5mm>*{}**@{-},
 <0.6mm,0.44mm>*{};<8mm,5mm>*{}**@{-},
   <0mm,0mm>*{};<-8.5mm,5.5mm>*{^1}**@{},
   <0mm,0mm>*{};<-5mm,5.5mm>*{^2}**@{},
   <0mm,0mm>*{};<4.5mm,5.5mm>*{^{m\hspace{-0.5mm}-\hspace{-0.5mm}1}}**@{},
   <0mm,0mm>*{};<9.0mm,5.5mm>*{^m}**@{},
 <-0.6mm,-0.44mm>*{};<-8mm,-5mm>*{}**@{-},
 <-0.4mm,-0.7mm>*{};<-4.5mm,-5mm>*{}**@{-},
 <0mm,0mm>*{};<-1mm,-5mm>*{\ldots}**@{},
 <0.4mm,-0.7mm>*{};<4.5mm,-5mm>*{}**@{-},
 <0.6mm,-0.44mm>*{};<8mm,-5mm>*{}**@{-},
   <0mm,0mm>*{};<-8.5mm,-6.9mm>*{^1}**@{},
   <0mm,0mm>*{};<-5mm,-6.9mm>*{^2}**@{},
   <0mm,0mm>*{};<4.5mm,-6.9mm>*{^{n\hspace{-0.5mm}-\hspace{-0.5mm}1}}**@{},
   <0mm,0mm>*{};<9.0mm,-6.9mm>*{^n}**@{},
 \end{xy}}
 \right)=
 \begin{cases}
 \frac 1 n
\sum_{\sigma\in\bS_n}(-1)^{\la_{i}|\sigma|}
\resizebox{16mm}{!}{\xy
(-12,13)*{^{\sigma(1)}},
(0,17)*{^{\sigma(2)}},
(12,13)*{^{\sigma(3)}},
(-18,0)*{^{\sigma(n)}},
(-12,-14)*{^{\sigma(n-1)}},
(-11,11)*{\bullet}="2",
(0,15)*{\bullet}="3",
(11,11)*{\bullet}="4",
(15,0)*{\bullet}="5",
(11,-11)*{\bullet}="6",
(0,-15)*{\bullet}="7",
(-15,0)*{\bullet}="n",
(-11,-11)*{\bullet}="n-1",
\ar @{->} "2";"3" <0pt>
\ar @{->} "3";"4" <0pt>
\ar @{->} "4";"5" <0pt>
\ar @{->} "5";"6" <0pt>
\ar @{->} "6";"7" <0pt>
\ar @{->} "7";"n-1" <0pt>
\ar @{->} "n-1";"n" <0pt>
\ar @{->} "n";"2" <0pt>
\endxy}
&\text{if $m=2$ and $n$ odd}
\\
0 & \text{otherwise}
\end{cases}
\end{align*}
where $i=1,2$, $\la_1=0$, $\la_2=1$ and $|\sigma|\in \Z_2$ is the parity of the permutation $\sigma$.}

\begin{proof}
The claim about $s_1^W$ follows Theorem {\ref{4: Theorem on LieB_0,1 --> Rgra_1}} upon applying the {\em duality}\, (in sense of inverting inputs and outputs of the properad, see \S{\ref{4: Remarks}}(vi)) functor to its main formula, and noting that the graph dual to the ``$m$-theta" graph on the r.h.s.\ of the map $s^\Theta$ os precisely the graph with $m$ binary vertices.\

The second claim follows from the first one upon noticing that  $\hoLieB_{0,1}= \hoLieB_{0,1}\{1\}$.
\end{proof}

\subsubsection{\bf Corollary}{\em Given any vector space $W$ equipped with a symplectic form  of degree zero.
The associated vector space $Cyc(W)$ comes equipped with a canonical formal power series Poisson structure given by the sum
over all possible odd wheels as shown explicitly  in Corollary \S {\ref{4: Corollary on Holieb to Gra via wheels}}}.

\begin{proof} The prop $ \hoLieB_{1,0}$ is identical to the prop $\mathsf{Lie}^{\hspace{-0.5mm} 1}\mathsf{Bi}_\infty$ from \cite{Me1,Me3} and hence has the property that there is a 1-1 correspondence between representations of $\hoLieB_{1,0}$ in a graded vector space $V$ and formal power series Poisson structures on $V$ with vanishing constant term at $0\in V$. Composing the canonical representation $\cR \cG ra_1\rar \cE nd_{Cyc(W)}$ with the map $s_1^W$ we obtain the required claim.
\end{proof}


\bip


{\Large
\section{\bf Twisting of props by morphisms from $\HoLoBd$\\  and the stable ribbon graph complex}
}

\mip

\subsection{Operads twisted by morphisms from Lie type operads}
For any morphism of operads
$$
f: \caH olie_{d}\rar \f
 $$
from a (degree shifted) operad of strongly homotopy Lie algebras to an arbitrary dg operad $\f$, one can construct  \cite{Wi}
 an associated {\em twisted}\, (by $f$) dg operad,
$$
\cT w\f =\left\{ \cT w\f(n):=\prod_{j\geq 0}\cP(n+j)\ot_{\bS_j} (\K[-d])^{\ot j}\right\}
$$
  which often has interesting cohomology and  applications. For example, this twisting construction gives us almost for gratis dg operads of Kontsevich graphs and the dg operad of braces out of much simpler non-differential operads, and also short proofs of some non-trivial theorems about their representations.
 The differential and operadic compositions in  $\cT w\f$ are explicitly given in \cite{Wi}; one can understand (and even reconstruct) these structures  through the following main property  of $\cT w\f$: given (i) any representation  $\rho: \f\rar \cE nd_V$, (ii) any
 (pro)nilpotent algebra $\mathfrak{n}$,  (iii) any Maurer-Cartan element $m\in V\ot \mathfrak{n}$ with respect to the $\mathfrak{n}$-linear  $\caH olie_{d}$ algebra structure in $V\ot \mathfrak{n}$ induced by the composition $\rho\circ f: \caH olie_{d+1}\rar \cE nd_V$, then there is an induced representation
 $$
 \rho_m: \cT w\f \lon \cE nd_{V\ot \mathfrak{n}}
 $$
  in the $\mathfrak{n}$-module $V\ot \mathfrak{n}$ equipped with the twisted $\mathfrak{n}$-linear differential
 $$
 \delta^{tw} v\ot a:= \left(\delta v + \sum_{n\geq 2}\mu_n(m,\ldots,m,v) \right)a \ \ \ \  \forall v\in V,\ a\in \mathfrak{n},
 $$
 where $\delta$ is the original (untwisted) differential in $V$ and
 the operations $\{\mu_n: \ot^n V \rar V[-d]\}_{n\geq 1}$ are precisely the values of the composition $\rho\circ f$ on the generators of  $\caH olie_{d}$.

 \sip

 The dg operad $\cT w \f$ comes equipped with the natural action  of the dg Lie algebra
 $\Def(\caL ie_{d}\rar \f)$ via operadic derivations \cite{Wi}.

  \sip
  Let us illustrate all notions and claims in the concrete examples.

\subsubsection{\bf Examples}\label{5: Graphs_d+1 example}

(i) The operad $\cG ra_{d}$ (see \S {\ref{2: subsection on operad Gra}}) comes equipped with a non-trivial map
$$
\al': \caH olie_{d}\lon \caL ie_{d} \stackrel{\al}{\lon}  \cG ra_{d}
$$
where the first arrow is the canonical projection and $\al$ is the map defined in (\ref{2: alpha map from Lie to gra}). Hence the operad $\cG ra_{d}$ can be twisted by $\al'$ into a new dg operad
$\cT w\cG ra_{d}$ which we describe explicitly. An element $\Ga\in \cT w \cG ra_{d}(n)$ is  graph $\Ga$ from $\cG ra_{d}(n+j)\ot_{\bS_j}(\K[-d])^{\ot j}$ which has $n+j$ vertices such that $n$ vertices
are labelled by numbers from $[n]$ as in $\cG ra_{d}(n)$ while the  remaining $j$ vertices have labels (skew)symmetrized and hence can be viewed as unlabelled. We call
these unlabelled vertices {\em internal}\, and denote by black bullets, for example
$$
\Ba{c}\resizebox{3mm}{!}{ \xy
 (0,8)*{\bu}="A";
 (0,0)*+{_1}*\frm{o}="B";
 \ar @{-} "A";"B" <0pt>
\endxy} \Ea
\in \cT w\cG ra_{d+1}(1)\ \ \ , \ \ \
\Ba{c}\resizebox{17mm}{!}{ \xy
(-1.5,5)*{}="1",
(1.5,5)*{}="2",
(9,5)*{}="3",
 (0,0)*{\bu}="A";
  (9,0)*{\bu}="O";
 (-6,-10)*+{_1}*\frm{o}="B";
  (6,-10)*+{_2}*\frm{o}="C";
   (14,-10)*+{_3}*\frm{o}="D";
 "A"; "B" **\crv{(-5,-0)}; 
  "A"; "D" **\crv{(5,-0.5)};
   "A"; "B" **\crv{(5,-1)};
  "A"; "C" **\crv{(-5,-7)};
   "A"; "O" **\crv{(5,5)};
\ar @{-} "O";"C" <0pt>
\ar @{-} "O";"D" <0pt>
 \endxy}
 \Ea \in \cT w\cG ra_{d+1}(3)
$$
The set of internal vertices is denote by $V_\bu(\Ga)$; the remaining vertices are called {\em external}\, and denoted by white circles, $V(\Ga)\setminus V_\bu(\Ga)=: V_\circ(\Ga)$.
The operadic composition
$$
\Ba{rccc}
\circ_{i}: & \cT w\cG ra_{d}(n) \ot \cT w\cG ra_{d}(m) &\lon &  \cT w\cG ra_{d}(n+m-1)\\
& \Ga_1\ot \Ga_2 &\lon & \Ga_1 \circ_i \Ga_2
\Ea
$$
is given as in $\cG ra_{d}$, i.e.\ by substituting the graph $\Ga_2$ into the $i$-th external vertex of $\Ga_1$ and taking a sum over all attachments of the resulting hanging edges in $\Ga_1$ to external and internal vertices  of $\Ga_2$.

\sip

The right action of the graph complex $\mathsf{fGC}_{d}$ by derivations is defined by
$$
\Ba{rccc}
\cdot & \cT w\cG ra_{d} \ot \mathsf{fGC}_{d}  & \lon & \cT w \cG ra_{d}\\
    & \Ga \ot \ga & \lon & \Ga\cdot \ga :=\sum_{v\in V_\bu(\Ga)} \Ga\cdot_v \ga
    \Ea
$$
where $\Ga\cdot_v\ga$ is given by  substitution of the graph $\ga$ into the internal vertex $v$  and taking the sum over all possible attachments of the resulting hanging
edges to external and internal vertices  of $\Ga$.

\sip

The operad $\cG ra_{d}$ is non-differential so that the differential $\delta$ in $\cT w\cG ra_{d}$ is completely determined by the map $\al'$, i.e.\ by formulae (\ref{2: alpha map from Lie to gra}). It is given explicitly by \cite{Wi}
$$
\delta\Ga:=
\Ba{c}\resizebox{3mm}{!}{  \xy
 (0,6)*{\bu}="A";
 (0,0)*+{_1}*\frm{o}="B";
 \ar @{-} "A";"B" <0pt>
\endxy} \Ea
  \circ_1 \Ga\ \
- \ \ (-1)^{|\Ga|} \sum_{v\in V_\circ(\Ga)}  \Ga \circ_v
\Ba{c}\resizebox{3mm}{!}{  \xy
 (0,6)*{\bu}="A";
 (0,0)*+{_1}*\frm{o}="B";
 \ar @{-} "A";"B" <0pt>
\endxy} \Ea
\  -\ (-1)^{|\Ga|} \Ga \cdot \left(\xy
 (0,0)*{\bullet}="a",
(5,0)*{\bu}="b",
\ar @{-} "a";"b" <0pt>
\endxy\right)
$$
Let $\cG raphs_{d}$ be a dg suboperad of $\cT w\cG ra_{d}$ spanned by graphs
 with all
internal vertices at least trivalent and with no connected component consisting entirely of internal vertices; it was first introduced by Kontsevich (for $d=2$) in \cite{Ko4}.
\sip

Let $\cP ois_{d}$ be a quadratic operad generated by degree zero
corolla $\Ba{c}\resizebox{7mm}{!}{ \begin{xy}
 <0mm,0.66mm>*{};<0mm,3mm>*{}**@{-},
 <0.39mm,-0.39mm>*{};<2.2mm,-2.2mm>*{}**@{-},
 <-0.35mm,-0.35mm>*{};<-2.2mm,-2.2mm>*{}**@{-},
 <0mm,0mm>*{\circ};<0mm,0mm>*{}**@{},
   <0.39mm,-0.39mm>*{};<2.9mm,-4mm>*{^2}**@{},
   <-0.35mm,-0.35mm>*{};<-2.8mm,-4mm>*{^1}**@{},
\end{xy}}\Ea
=
\Ba{c}\resizebox{7mm}{!}{ \begin{xy}
 <0mm,0.66mm>*{};<0mm,3mm>*{}**@{-},
 <0.39mm,-0.39mm>*{};<2.2mm,-2.2mm>*{}**@{-},
 <-0.35mm,-0.35mm>*{};<-2.2mm,-2.2mm>*{}**@{-},
 <0mm,0mm>*{\circ};<0mm,0mm>*{}**@{},
   <0.39mm,-0.39mm>*{};<2.9mm,-4mm>*{^1}**@{},
   <-0.35mm,-0.35mm>*{};<-2.8mm,-4mm>*{^2}**@{},
\end{xy}}\Ea$
and degree $1-d$ corolla
$\Ba{c}\resizebox{7mm}{!}{ \begin{xy}
 <0mm,0.66mm>*{};<0mm,3mm>*{}**@{-},
 <0.39mm,-0.39mm>*{};<2.2mm,-2.2mm>*{}**@{-},
 <-0.35mm,-0.35mm>*{};<-2.2mm,-2.2mm>*{}**@{-},
 <0mm,0mm>*{\bu};<0mm,0mm>*{}**@{},
   <0.39mm,-0.39mm>*{};<2.9mm,-4mm>*{^2}**@{},
   <-0.35mm,-0.35mm>*{};<-2.8mm,-4mm>*{^1}**@{},
\end{xy}}\Ea
=(-1)^{d}
\Ba{c}\resizebox{7mm}{!}{ \begin{xy}
 <0mm,0.66mm>*{};<0mm,3mm>*{}**@{-},
 <0.39mm,-0.39mm>*{};<2.2mm,-2.2mm>*{}**@{-},
 <-0.35mm,-0.35mm>*{};<-2.2mm,-2.2mm>*{}**@{-},
 <0mm,0mm>*{\bu};<0mm,0mm>*{}**@{},
   <0.39mm,-0.39mm>*{};<2.9mm,-4mm>*{^1}**@{},
   <-0.35mm,-0.35mm>*{};<-2.8mm,-4mm>*{^2}**@{},
\end{xy}}\Ea
 $
which are subject to the following relations
\Beq\label{5: Jacobi identity for Lie operad}
\Ba{c}\resizebox{8mm}{!}{  \begin{xy}
 <0mm,0mm>*{\bu};<0mm,0mm>*{}**@{},
 <0mm,0.69mm>*{};<0mm,3.0mm>*{}**@{-},
 <0.39mm,-0.39mm>*{};<2.4mm,-2.4mm>*{}**@{-},
 <-0.35mm,-0.35mm>*{};<-1.9mm,-1.9mm>*{}**@{-},
 <-2.4mm,-2.4mm>*{\bu};<-2.4mm,-2.4mm>*{}**@{},
 <-2.0mm,-2.8mm>*{};<0mm,-4.9mm>*{}**@{-},
 <-2.8mm,-2.9mm>*{};<-4.7mm,-4.9mm>*{}**@{-},
    <0.39mm,-0.39mm>*{};<3.3mm,-4.0mm>*{^3}**@{},
    <-2.0mm,-2.8mm>*{};<0.5mm,-6.7mm>*{^2}**@{},
    <-2.8mm,-2.9mm>*{};<-5.2mm,-6.7mm>*{^1}**@{},
 \end{xy}}\Ea
\ + \
\Ba{c}\resizebox{8mm}{!}{  \begin{xy}
 <0mm,0mm>*{\bu};<0mm,0mm>*{}**@{},
 <0mm,0.69mm>*{};<0mm,3.0mm>*{}**@{-},
 <0.39mm,-0.39mm>*{};<2.4mm,-2.4mm>*{}**@{-},
 <-0.35mm,-0.35mm>*{};<-1.9mm,-1.9mm>*{}**@{-},
 <-2.4mm,-2.4mm>*{\bu};<-2.4mm,-2.4mm>*{}**@{},
 <-2.0mm,-2.8mm>*{};<0mm,-4.9mm>*{}**@{-},
 <-2.8mm,-2.9mm>*{};<-4.7mm,-4.9mm>*{}**@{-},
    <0.39mm,-0.39mm>*{};<3.3mm,-4.0mm>*{^2}**@{},
    <-2.0mm,-2.8mm>*{};<0.5mm,-6.7mm>*{^1}**@{},
    <-2.8mm,-2.9mm>*{};<-5.2mm,-6.7mm>*{^3}**@{},
 \end{xy}}\Ea
\ + \
 \Ba{c}\resizebox{8mm}{!}{ \begin{xy}
 <0mm,0mm>*{\bu};<0mm,0mm>*{}**@{},
 <0mm,0.69mm>*{};<0mm,3.0mm>*{}**@{-},
 <0.39mm,-0.39mm>*{};<2.4mm,-2.4mm>*{}**@{-},
 <-0.35mm,-0.35mm>*{};<-1.9mm,-1.9mm>*{}**@{-},
 <-2.4mm,-2.4mm>*{\bu};<-2.4mm,-2.4mm>*{}**@{},
 <-2.0mm,-2.8mm>*{};<0mm,-4.9mm>*{}**@{-},
 <-2.8mm,-2.9mm>*{};<-4.7mm,-4.9mm>*{}**@{-},
    <0.39mm,-0.39mm>*{};<3.3mm,-4.0mm>*{^1}**@{},
    <-2.0mm,-2.8mm>*{};<0.5mm,-6.7mm>*{^3}**@{},
    <-2.8mm,-2.9mm>*{};<-5.2mm,-6.7mm>*{^2}**@{},
 \end{xy}}\Ea
\Eeq
and
$$
 \Ba{c}\resizebox{8mm}{!}{ \begin{xy}
 <0mm,0mm>*{\circ};<0mm,0mm>*{}**@{},
 <0mm,0.69mm>*{};<0mm,3.0mm>*{}**@{-},
 <0.39mm,-0.39mm>*{};<2.4mm,-2.4mm>*{}**@{-},
 <-0.35mm,-0.35mm>*{};<-1.9mm,-1.9mm>*{}**@{-},
 <-2.4mm,-2.4mm>*{\circ};<-2.4mm,-2.4mm>*{}**@{},
 <-2.0mm,-2.8mm>*{};<0mm,-4.9mm>*{}**@{-},
 <-2.8mm,-2.9mm>*{};<-4.7mm,-4.9mm>*{}**@{-},
    <0.39mm,-0.39mm>*{};<3.3mm,-4.0mm>*{^3}**@{},
    <-2.0mm,-2.8mm>*{};<0.5mm,-6.7mm>*{^2}**@{},
    <-2.8mm,-2.9mm>*{};<-5.2mm,-6.7mm>*{^1}**@{},
 \end{xy}}\Ea
\ = \
 \Ba{c}\resizebox{8mm}{!}{ \begin{xy}
 <0mm,0mm>*{\circ};<0mm,0mm>*{}**@{},
 <0mm,0.69mm>*{};<0mm,3.0mm>*{}**@{-},
 <0.39mm,-0.39mm>*{};<1.9mm,-1.9mm>*{}**@{-},
 <-0.35mm,-0.35mm>*{};<-1.9mm,-1.9mm>*{}**@{-},
 <2.4mm,-2.4mm>*{\circ};<-2.4mm,-2.4mm>*{}**@{},
 <2.0mm,-2.8mm>*{};<0mm,-4.9mm>*{}**@{-},
 <2.8mm,-2.9mm>*{};<4.7mm,-4.9mm>*{}**@{-},
    <0.39mm,-0.39mm>*{};<-3mm,-4.0mm>*{^1}**@{},
    <-2.0mm,-2.8mm>*{};<0mm,-6.7mm>*{^2}**@{},
    <-2.8mm,-2.9mm>*{};<5.2mm,-6.7mm>*{^3}**@{},
 \end{xy}}\Ea\ \ \ , \ \ \
  \Ba{c}\resizebox{8mm}{!}{ \begin{xy}
 <0mm,0mm>*{\bu};<0mm,0mm>*{}**@{},
 <0mm,0.69mm>*{};<0mm,3.0mm>*{}**@{-},
 <0.39mm,-0.39mm>*{};<2.4mm,-2.4mm>*{}**@{-},
 <-0.35mm,-0.35mm>*{};<-1.9mm,-1.9mm>*{}**@{-},
 <-2.4mm,-2.4mm>*{\circ};<-2.4mm,-2.4mm>*{}**@{},
 <-2.0mm,-2.8mm>*{};<0mm,-4.9mm>*{}**@{-},
 <-2.8mm,-2.9mm>*{};<-4.7mm,-4.9mm>*{}**@{-},
    <0.39mm,-0.39mm>*{};<3.3mm,-4.0mm>*{^3}**@{},
    <-2.0mm,-2.8mm>*{};<0.5mm,-6.7mm>*{^2}**@{},
    <-2.8mm,-2.9mm>*{};<-5.2mm,-6.7mm>*{^1}**@{},
 \end{xy}}\Ea
\ = \
\Ba{c} \resizebox{8mm}{!}{ \begin{xy}
 <0mm,0mm>*{\circ};<0mm,0mm>*{}**@{},
 <0mm,0.69mm>*{};<0mm,3.0mm>*{}**@{-},
 <0.39mm,-0.39mm>*{};<2.4mm,-2.4mm>*{}**@{-},
 <-0.35mm,-0.35mm>*{};<-1.9mm,-1.9mm>*{}**@{-},
 <-2.4mm,-2.4mm>*{\bu};<-2.4mm,-2.4mm>*{}**@{},
 <-2.0mm,-2.8mm>*{};<0mm,-4.9mm>*{}**@{-},
 <-2.8mm,-2.9mm>*{};<-4.7mm,-4.9mm>*{}**@{-},
    <0.39mm,-0.39mm>*{};<3.3mm,-4.0mm>*{^2}**@{},
    <-2.0mm,-2.8mm>*{};<0.5mm,-6.7mm>*{^3}**@{},
    <-2.8mm,-2.9mm>*{};<-5.2mm,-6.7mm>*{^1}**@{},
 \end{xy}}\Ea
\ + \
 \Ba{c}\resizebox{8mm}{!}{ \begin{xy}
 <0mm,0mm>*{\circ};<0mm,0mm>*{}**@{},
 <0mm,0.69mm>*{};<0mm,3.0mm>*{}**@{-},
 <0.39mm,-0.39mm>*{};<2.4mm,-2.4mm>*{}**@{-},
 <-0.35mm,-0.35mm>*{};<-1.9mm,-1.9mm>*{}**@{-},
 <-2.4mm,-2.4mm>*{\bu};<-2.4mm,-2.4mm>*{}**@{},
 <-2.0mm,-2.8mm>*{};<0mm,-4.9mm>*{}**@{-},
 <-2.8mm,-2.9mm>*{};<-4.7mm,-4.9mm>*{}**@{-},
    <0.39mm,-0.39mm>*{};<3.3mm,-4.0mm>*{^1}**@{},
    <-2.0mm,-2.8mm>*{};<0.5mm,-6.7mm>*{^3}**@{},
    <-2.8mm,-2.9mm>*{};<-5.2mm,-6.7mm>*{^2}**@{},
 \end{xy}}\Ea.
$$
The operad $\cP ois_2$ is precisely the well-known operad $e_2$ of Gerstenhaber algebras; the operad $\cP ois_{d}$ for $d\geq 2$ is often called the operad of $(d-1)$-algebras \cite{Ko4} while the case $d=0$ corresponds to the operad of ordinary Poisson algebras. We always assume that $\cP ois_{d}$ is a dg operad with trivial differential.

\sip

We have a canonical morphism of dg operads
$$
\Ba{ccc}
\cP ois_{d} & \lon & \cG raphs_{d}\\
\Ba{c}\resizebox{6mm}{!}{\begin{xy}
 <0mm,0.66mm>*{};<0mm,3mm>*{}**@{-},
 <0.39mm,-0.39mm>*{};<2.2mm,-2.2mm>*{}**@{-},
 <-0.35mm,-0.35mm>*{};<-2.2mm,-2.2mm>*{}**@{-},
 <0mm,0mm>*{\circ};<0mm,0mm>*{}**@{},
   <0.39mm,-0.39mm>*{};<2.9mm,-4mm>*{^2}**@{},
   <-0.35mm,-0.35mm>*{};<-2.8mm,-4mm>*{^1}**@{},
\end{xy}}\Ea & \lon &
\Ba{c}\resizebox{9mm}{!}{ \xy
 (0,0)*+{_1}*\frm{o};
 (7,0)*+{_2}*\frm{o};
\endxy} \Ea\\
\Ba{c}\resizebox{6mm}{!}{\begin{xy}
 <0mm,0.66mm>*{};<0mm,3mm>*{}**@{-},
 <0.39mm,-0.39mm>*{};<2.2mm,-2.2mm>*{}**@{-},
 <-0.35mm,-0.35mm>*{};<-2.2mm,-2.2mm>*{}**@{-},
 <0mm,0mm>*{\bu};<0mm,0mm>*{}**@{},
   <0.39mm,-0.39mm>*{};<2.9mm,-4mm>*{^2}**@{},
   <-0.35mm,-0.35mm>*{};<-2.8mm,-4mm>*{^1}**@{},
\end{xy}}\Ea
&\lon&
\Ba{c} \resizebox{8mm}{!}{\xy
 (7,0)*+{_2}*\frm{o}="B";
 (0,0)*+{_1}*\frm{o}="A";
 \ar @{-} "A";"B" <0pt>
\endxy} \Ea
\Ea
$$
which is proven in \cite{Ko4,LV} to be a quasi-isomorphism.

\mip

(ii) Let $\cR\cT ra_d$ be a suboperad of the properad $\cR\cG ra_d$ spanned by connected
graphs of genus zero, i.e.\ by ribbon (unrooted) trees. There is a
morphism of operads,
$$
\nu: \caL ie_d \lon  \cR\cT ra_d
$$
given on the generators by formula (\ref{2: alpha map from Lie to gra}) so that one can construct an associated twisted dg operad
$$
\cR\cT rees_d:= Tw(\cR\cT ra_d).
$$
Its cohomology is equal \cite{Wa} to the (degree shifted) gravity operad $\cG rav\{d-2\}$ introduced by
Ezra Getzler \cite{Ge}; in this picture the cohomology is generated by the ribbon tree given on the r.h.s.\ of
(\ref{2: alpha map from Lie to gra}) and also by (connected) ribbon trees with  black vertices of valency 3 and white vertices of valency 1.
%
The dg operad $\cR\cT rees_d$ acts naturally on the Hochschild complex (see \S {\ref{3: Subsubsect Basic Example}}(i)) of any cyclic
$A_\infty$ algebra.

\sip

In this section we shall construct a dg operad which admits a natural action on the Hochschild complex of any quantum $A_\infty$ algebra; it is obtained by twisting
of a certain polydifferential operad canonically associated with the full properad $\cR \cG ra_d$ of ribbon graphs.

\subsubsection{\bf Twisting by maps from a Lie type operad}
 We shall use below operads twisted by morphisms from a slightly different Lie type operad.
  Let $\caL ie^\diamond_{d}$ be an operad generated by degree $1$ corolla $\Ba{c}\resizebox{1.7mm}{!}{\begin{xy}
 <0mm,-0.55mm>*{};<0mm,-3mm>*{}**@{-},
 <0mm,0.5mm>*{};<0mm,3mm>*{}**@{-},
 <0mm,0mm>*{\bullet};<0mm,0mm>*{}**@{},
 \end{xy}}\Ea$ and degree $1-d$ corolla
 $\Ba{c}\resizebox{6mm}{!}{\begin{xy}
 <0mm,0.66mm>*{};<0mm,3mm>*{}**@{-},
 <0.39mm,-0.39mm>*{};<2.2mm,-2.2mm>*{}**@{-},
 <-0.35mm,-0.35mm>*{};<-2.2mm,-2.2mm>*{}**@{-},
 <0mm,0mm>*{\bu};<0mm,0mm>*{}**@{},
   <0.39mm,-0.39mm>*{};<2.9mm,-4mm>*{^2}**@{},
   <-0.35mm,-0.35mm>*{};<-2.8mm,-4mm>*{^1}**@{},
\end{xy}}\Ea
=(-1)^{d}
\Ba{c}\resizebox{6mm}{!}{\begin{xy}
 <0mm,0.66mm>*{};<0mm,3mm>*{}**@{-},
 <0.39mm,-0.39mm>*{};<2.2mm,-2.2mm>*{}**@{-},
 <-0.35mm,-0.35mm>*{};<-2.2mm,-2.2mm>*{}**@{-},
 <0mm,0mm>*{\bu};<0mm,0mm>*{}**@{},
   <0.39mm,-0.39mm>*{};<2.9mm,-4mm>*{^1}**@{},
   <-0.35mm,-0.35mm>*{};<-2.8mm,-4mm>*{^2}**@{},
\end{xy}}\Ea
 $
which are subject to relations (\ref{5: Jacobi identity for Lie operad})
and the following ones
\footnote{
The  minimal resolution $\caH o lie^\diamond_{d}$  of the operad $\caL ie^\diamond_{d}$ is  generated  by the (skew)symmetric corollas
$
\Ba{c}\resizebox{10mm}{!}{ \xy
(-7.5,-8.6)*{_{_1}};
(-4.1,-8.6)*{_{_2}};
(9.0,-8.5)*{_{_{k}}};
(0.0,-6)*{...};
(0,5)*{};
(0,0)*+\hbox{$_{{p}}$}*\frm{o}
**\dir{-};
(-4,-7)*{};
(0,0)*+\hbox{$_{{p}}$}*\frm{o}
**\dir{-};
(-7,-7)*{};
(0,0)*+\hbox{$_{{p}}$}*\frm{o}
**\dir{-};
(8,-7)*{};
(0,0)*+\hbox{$_{{p}}$}*\frm{o}
**\dir{-};
(4,-7)*{};
(0,0)*+\hbox{$_{{p}}$}*\frm{o}
**\dir{-};
\endxy}\Ea
%
%
$
of homological degree  $1 - d(k-1 +p)$ and has the  differential \cite{CMW}
$
d\Ba{c}
\resizebox{10mm}{!}{ \xy
(-7.5,-8.6)*{_{_1}};
(-4.1,-8.6)*{_{_2}};
(9.0,-8.5)*{_{_{k}}};
(0.0,-6)*{...};
(0,5)*{};
(0,0)*+\hbox{$_{{p}}$}*\frm{o}
**\dir{-};
(-4,-7)*{};
(0,0)*+\hbox{$_{{p}}$}*\frm{o}
**\dir{-};
(-7,-7)*{};
(0,0)*+\hbox{$_{{p}}$}*\frm{o}
**\dir{-};
(8,-7)*{};
(0,0)*+\hbox{$_{{p}}$}*\frm{o}
**\dir{-};
(4,-7)*{};
(0,0)*+\hbox{$_{{p}}$}*\frm{o}
**\dir{-};
\endxy}\Ea
=
\sum_{p=q+r}\sum_{[k]=I_1\sqcup I_2}
\Ba{c}
%
%
\resizebox{11mm}{!}{ \xy
(0,0)*+{q}*\cir{}="b",
(10,10)*+{r}*\cir{}="c",
%
(-4,-6)*{}="-1",
(-2,-6)*{}="-2",
(4,-6)*{}="-3",
(1,-5)*{...},
(0,-8)*{\underbrace{\ \ \ \ \ \ \ \ }},
(0,-11)*{_{I_1}},
(10,16)*{}="2'",
(11,4)*{}="-1'",
(16,4)*{}="-2'",
(18,4)*{}="-3'",
(13.5,4)*{...},
(15,2)*{\underbrace{\ \ \ \ \ \ \ }},
(15,-1)*{_{I_2}},
%
\ar @{-} "b";"c" <0pt>
\ar @{-} "b";"-1" <0pt>
\ar @{-} "b";"-2" <0pt>
\ar @{-} "b";"-3" <0pt>
%
\ar @{-} "c";"2'" <0pt>
\ar @{-} "c";"-1'" <0pt>
\ar @{-} "c";"-2'" <0pt>
\ar @{-} "c";"-3'" <0pt>
\endxy}
\Ea
$.
Representations, $\rho: \caH o lie^\diamond_{d}\rar \cE nd_V$, of this operad in a dg vector space $(V,d)$
is the same thing as continuous representations of  the  operad $\caH o lie_{d}[[\hbar]]$
in the topological vector space $V[[\hbar]]$ equipped with the differential
$
d + \sum_{p\geq 1}\hbar^p \Delta_p$, $\Delta_p:=\rho \left(\Ba{c}
\resizebox{4mm}{!}{ \xy
(0,5)*{};
(0,0)*+{_p}*\cir{}
**\dir{-};
(0,-5)*{};
(0,0)*+{_p}*\cir{}
**\dir{-};
\endxy}\Ea\right),
$
where the formal parameter $\hbar$ is assumed to have homological degree $d$.}
$$
\Ba{c}\resizebox{1.6mm}{!}{\begin{xy}
 <0mm,0mm>*{};<0mm,-3mm>*{}**@{-},
 <0mm,0mm>*{};<0mm,6mm>*{}**@{-},
 <0mm,0mm>*{\bullet};
 <0mm,3mm>*{\bullet};
 \end{xy}}\Ea=0\ \ ,
  \ \ \
\Ba{c}\begin{xy}
 <0mm,0mm>*{\bu};
 <0mm,0.69mm>*{};<0mm,3.0mm>*{}**@{-},
 <0.39mm,-0.39mm>*{};<2.4mm,-2.4mm>*{}**@{-},
 <-0.35mm,-0.35mm>*{};<-1.9mm,-1.9mm>*{}**@{-},
 <-2.4mm,-2.4mm>*{\bu};
 <-2.4mm,-2.8mm>*{};<-2.4mm,-4.9mm>*{}**@{-};
    <3.3mm,-4.0mm>*{^2};
    <-2.4mm,-6.7mm>*{^1};
 \end{xy}\Ea
\ + \
\Ba{c}\begin{xy}
 <0mm,0mm>*{\bu};
 <0mm,0.69mm>*{};<0mm,3.0mm>*{}**@{-},
 <-0.39mm,-0.39mm>*{};<-2.4mm,-2.4mm>*{}**@{-},
 <0.35mm,-0.35mm>*{};<1.9mm,-1.9mm>*{}**@{-},
 <2.4mm,-2.4mm>*{\bu};
 <2.4mm,-2.8mm>*{};<2.4mm,-4.9mm>*{}**@{-};
    <-3.3mm,-4.0mm>*{^1};
    <2.4mm,-6.7mm>*{^2};
 \end{xy}\Ea
\ + \
 \Ba{c}\begin{xy}
 <-2.4mm,0.7mm>*{\bu};
 <-2.4mm,1.69mm>*{};<-2.4mm,-2.0mm>*{}**@{-},
  <-2.4mm,1mm>*{};<-2.4mm,4mm>*{}**@{-},
 <-2.4mm,-2.4mm>*{\bu};
 <-2.0mm,-2.8mm>*{};<0mm,-4.9mm>*{}**@{-},
 <-2.8mm,-2.9mm>*{};<-4.7mm,-4.9mm>*{}**@{-},
    <-2.0mm,-2.8mm>*{};<0.5mm,-6.7mm>*{^2}**@{},
    <-2.8mm,-2.9mm>*{};<-5.2mm,-6.7mm>*{^1}**@{},
 \end{xy}\Ea=0.
$$
We always understand $\caL ie^\diamond$ as a dg operad with trivial differential.

\sip

Given a morphism of dg operads
$$
f: \caL ie_{d}^\diamond \lon (\f,\delta)
$$
the image $f\left(\begin{xy}
 <0mm,-0.55mm>*{};<0mm,-3mm>*{}**@{-},
 <0mm,0.5mm>*{};<0mm,3mm>*{}**@{-},
 <0mm,0mm>*{\bullet};<0mm,0mm>*{}**@{},
 \end{xy}\right)\in \f(1)$ induces a differential $\delta_\bu$ in $\f$ through the canonocal action of $\f(1)$ on $\f$ as derivations. Moreover, the sum $\delta + \delta_\bu$ is also a differential in $\f$.
The composition
$$
\caL ie_{d} \hook \caL ie_{d}^\diamond \stackrel{f}{\lon} \f
$$
is a morphism of {\em dg}\, operads with $\f$ assumed to have the new differential $\delta + \delta_\bu$.
Applying now the standard twisting construction to the above map from $\caL ie_{d}$ to
$(\f, \delta + \delta_\bu)$ we get a twisted dg operad which we denote by $\cT w^\diamond\f$.
As a non-differential operad $\cT w^\diamond\f$ is identical to $Tw\f$, but it comes equipped with a slightly different differential affected by
the image $f\left(\begin{xy}
 <0mm,-0.55mm>*{};<0mm,-3mm>*{}**@{-},
 <0mm,0.5mm>*{};<0mm,3mm>*{}**@{-},
 <0mm,0mm>*{\bullet};<0mm,0mm>*{}**@{},
 \end{xy}\right)$.


\subsection{Polydifferential operad associated to a prop} In this section we construct
a functor from the category of (augmented) props to the category of operads which has the property
that for any prop $\cP=\{\cP(m,n)\}_{m,n\geq 1}$ and its representation
$$
\rho: \cP \lon \cE nd_V
$$
in a vector space $V$, the associated operad $\f\cP=\{\f\cP(k)\}_{k\geq 1}$
admits an associated representation,
$$
\rho^{poly}: \f\cP \lon \cE nd_{\odot^\bu V}
$$
in the graded commutative algebra $\odot^\bu V$ on which elements $p\in\cP$ act
as  polydifferential operators.

\mip

The idea is simple, and is best expressed in some basis $\{x^\al\}$ in $V$
 (so that $\odot^\bu V\simeq \K[x^\al]$).
Any element $p\in \cP(m,n)$ gets transformed by $\rho$ into a linear map
$$
\Ba{rccc}
\rho(p): & \ot^n V & \lon & \ot^m V\\
         & x^{\al_1}\ot x^{\al_2} \ot\ldots \ot x^{\al_n} & \lon &
 \displaystyle \sum_{\be_1,\be_2,...,\be_m} A^{\al_1\al_2...\al_n}_{\be_1\be_2...,\be_m} x^{\be_1}\ot x^{\be_2} \ot\ldots \ot x^{\be_n}
\Ea
$$
for some $A^{\al_1\al_2...\al_n}_{\be_1\be_2...,\be_m}\in \K$.

\sip

Then, for any partition $[n]=I_1\sqcup \ldots  \sqcup I_k$ of $[n]$ into a disjoint union of (possibly, not all non-empty) subsets, $I_i=\{s_{i_1}, s_{i_2}, \ldots, s_{i_{\# I_i}}\}$, $1\leq i\leq k$, we can associate to $p$ a polydifferential operator
$$
\Ba{rccc}
p^{poly}: & \ot^k (\odot^\bu V) & \lon & \odot^\bu V\\
         & f_1(x)\ot f_2(x) \ot\ldots \ot f_n(x) & \lon &
\displaystyle \sum_{\al_1,...,\al_n\atop
\be_1,..., \be_m}\frac{1}{m!} x^{\be_1}x^{\be_2} \cdots x^{\be_m} A^{\al_{I_1}... \al_{I_k}}_{\be_1\be_2...,\be_m} \frac{\p^{\# I_1} f_1}{\p x^{\al_{I_1}}}\cdots
\frac{\p^{\# I_k} f_k}{\p x^{\al_{I_k}}}
\Ea
$$
where $\al_{I_i}$ for $i\in [k]$ stands for the multindex $\al_{s_{i_1}}\al_{s_{i_2}}\ldots
\al_{s_{\# I_i}}$ and
$$
\frac{\p^{\# I_i} f_i}{\p x^{\al_{I_k}}}:= \left\{\Ba{rl} f_i & \mbox{if}\ \# I_i=0\\
 \frac{\p^{\# I_k} f_k}{ \p x^{\al_{s_{i_2}}}x^{\al_{s_{i_2}}}\ldots \p x^{\al_{s_{\# I_i}}} }
 &  \mbox{if}\ \# I_i\geq 1\\
 \Ea\right.
$$
This association $p\rar p^{poly}$ (for any fixed partition of $[n]$) is independent of the choice of a basis used in the construction. Our purpose is to construct an operad $\f\cP$
 out of the prop $\cP$ together with a linear map
$$
F_{[n]=I_1\sqcup...\sqcup I_k}: \cP(m,n) \lon \f\cP(k)
$$
such that, for any representation $\rho:\cP\rar \cE nd_V$,  the operad $\f\cP$ admits a natural representation $\rho^{poly}$ in
 $\odot^\bu V$ and
$$
p^{poly}= \rho^{poly}\left(F_{[n]=I_1\sqcup...\sqcup I_k}(p)\right).
$$
Note that $\odot^\bu V$ carries a natural representation of the operad $\cC om$ of commutative algebras so that
the latter  can also be incorporated into  $\f\cP$ in the form of operators corresponding, in the above notation,
to the case when all the sets $I_i$ are empty.

\subsubsection{\bf Some notation} Sometimes we understand a prop $\cP$ in the category of graded vector spaces  as a collection of $\bS_m^{op}\times \bS_n$ bimodules $\cP(m,n)$ (with --  solely for simplicity
of some formulae below -- an assumption that $m,n\geq 1$)
and sometimes as a functor from the category of pairs of non-empty finite sets to  the category of graded vector spaces,
$$
\cP=\{\cP(I,J)\}
$$
where each vector space $\cP(I,J)$ is a $S_I^{op}\times S_J$-bimodule; in particular,
$\cP(m,n)= \cP([m],[n])$. The horizontal composition in $\cP$ is denoted by $\circ_H$,
$$
\Ba{rccc}
\circ_H: & \cP(I_1,J_1)\ot \cP(I_2,J_2) & \lon & \cP(I_1\sqcup I_2, J_1\sqcup J_2)\\
         & p_1 \ot p_2 &  \lon &   p_1\circ_H p_2.
\Ea
$$
For any two injections of sets $f: S \hook J_1$ and $g: S\hook I_2$ there is a vertical
composition denoted by $_{f(S)}\circ_{g(S)}$,
$$
\Ba{rccc}
_{f(S)}\circ_{g(S)}: & \cP(I_1,J_1)\ot \cP(I_2,J_2) & \lon & \cP(I_1\sqcup (I_2\setminus i_2(S)), (J_1\setminus j_1(S))\sqcup J_2)\\
         & p_1 \ot p_2 &  \lon &   {p_1}\ _{f(S)}\circ_{g(S)} p_2
\Ea
$$
which glues $S$-labelled (via $f$) outputs of $p_2$ with the corresponding
$S$-labelled (via $g$) inputs of $p_1$.

\subsubsection{\bf Remark} Up to now all our props and operads were unital by default. The functor
$\f$ can be defined for an an arbitrary prop. However, some of our construction in this paper
become nicer if we assume that the props we consider are augmented,
$$
\cP= \K \oplus \bar{\cP},
$$
and apply the functor $\f$ to the augmentation ideal $\bar{\cP}$ only (see the Definition below). This assumption holds true in all applications of the functor $\f$ in this paper.

\sip

 It is worth noting that in some cases this restriction of the functor $\f$ to the augmentation ideal comes at a price of loosing in $\f\cP$ {\em rescaling}\, operators which  can be  useful in applications.

\subsubsection{\bf Definition}\label{5: Definition of OP oprad} Let $\cP$ be an augmented prop. Define a collection of (completed with respect to the filtration by the number of outputs and inputs) $\bS$-modules,
$$
\f\cP(k) :=\cC om(k)\ \oplus \prod_{m,n\geq 1}\bigoplus_{[n]=\ J_1\sqcup ...\sqcup J_k\atop
\# J_1,..., \# J_k\geq 0} \f\cP^m_{J_1,...,J_k}
$$
where
$$
\f\cP^m_{J_1,...,J_k}:=
\id_m \ot_{\bS_m^{op}} \ot \bar{\cP}(m,n)\ot_{S_{J_1}\times ...\times S_{J_k}} \id_{S_{J_1}} \ot \ldots \ot \id_{S_{J_k}}
$$
where $\id_I$ stands for the trivial one-dimensional representation of the permutation group $S_I$.
Thus an element of the summand $ \f\cP^m_{J_1,...,J_k}\subset \f\cP(k)$ is an element of $\cP(m,\# J_1+ \ldots + \# J_k)$ with all its $m$ outputs symmetrized and all its inputs in each bunch $J_s\subset [n]$, $s\in [k]$, also symmetrized. We assume from now on that all legs in each bunch $J_s$ are labelled by the same integer $s$; this defines an action of the group $\bS_k$ on $\f\cP(k)$.
There is a canonical linear map,
\Beq\label{5: m=0 projection}
\pi^m_{J_1,...,J_k}: \bar{\cP}(m,n) \lon \f\cP^m_{J_1,...,J_k}.
\Eeq

\mip

It is often useful to represent elements $p$ of the (non-unital) prop  $\bar{\cP}$ as (decorated) corollas,
$$
p\ \ \sim
\Ba{c}\resizebox{10mm}{!}{ \xy
(0,4.5)*+{...},
(0,-4.5)*+{...},
(0,0)*{\circ}="o",
(-5,5)*{}="1",
(-3,5)*{}="2",
(3,5)*{}="3",
(5,5)*{}="4",
(-3,-5)*{}="5",
(3,-5)*{}="6",
(5,-5)*{}="7",
(-5,-5)*{}="8",
(-5.5,7)*{_1},
(-3,7)*{_2},
(3,6)*{},
(5.9,7)*{m},
(-3,-7)*{_2},
(3,-7)*+{},
(5.9,-7)*{n},
(-5.5,-7)*{_1},
\ar @{-} "o";"1" <0pt>
\ar @{-} "o";"2" <0pt>
\ar @{-} "o";"3" <0pt>
\ar @{-} "o";"4" <0pt>
\ar @{-} "o";"5" <0pt>
\ar @{-} "o";"6" <0pt>
\ar @{-} "o";"7" <0pt>
\ar @{-} "o";"8" <0pt>
\endxy}\Ea \in \bar{\cP}(m,n)
$$
The image of such an element under the projection $\pi^m_{J_1,...,J_k}$ is represented pictorially as the same  corolla whose output legs are decorated by the same symbol 1 (which is omitted in the pictures) and the input legs
decorated with possibly {\em coinciding}\,  indices as in the following picture

$$
\Ba{c}\resizebox{15mm}{!}{\xy
(-9,-6)*{};
(0,0)*{\circ}
**\dir{-};
(-7.5,-6)*{};
(0,0)*{\circ }
**\dir{-};
(-6,-6)*{};
(0,0)*{\circ }
**\dir{-};
(-1,-6)*{};
(0,0)*{\circ }
**\dir{-};
(0,-6)*{};
(0,0)*{\circ }
**\dir{-};
(1,-6)*{};
(0,0)*{\circ }
**\dir{-};
(9,-6)*{};
(0,0)*{\circ }
**\dir{-};
(7.5,-6)*{};
(0,0)*{\circ }
**\dir{-};
(6,-6)*{};
(0,0)*{\circ }
**\dir{-};
(-3,-5)*{...};
(3,-5)*{...};
(-8,-9)*{\underbrace{\ \  }_{J_1}};
(0,-9)*{\underbrace{\ \  }_{J_i}};
(8,-9)*{\underbrace{\ \  }_{J_k}};
(0,9)*{\overbrace{ }^{[m]}};
(-2,6)*{};
(0,0)*{\circ }
**\dir{-};
(-0.7,6)*{};
(0,0)*{\circ }
**\dir{-};
(0.7,6)*{};
(0,0)*{\circ }
**\dir{-};
(2,6)*{};
(0,0)*{\circ }
**\dir{-};
\endxy}\Ea
\ \ \ \ \ \ \ \ \ \ \ \mbox{or} \ \ \ \ \
\Ba{c}\resizebox{15mm}{!}{\xy
(-9,-6)*{};
(0,0)*{\circ }
**\dir{-};
(-7.5,-6)*{};
(0,0)*{\circ }
**\dir{-};
(-6,-6)*{};
(0,0)*{\circ }
**\dir{-};
(-1,-6)*{};
(0,0)*{\circ }
**\dir{-};
(0,-6)*{};
(0,0)*{\circ }
**\dir{-};
(1,-6)*{};
(0,0)*{\circ }
**\dir{-};
(9,-6)*{};
(0,0)*{\circ }
**\dir{-};
(7.5,-6)*{};
(0,0)*{\circ }
**\dir{-};
(6,-6)*{};
(0,0)*{\circ }
**\dir{-};
(-3,-5)*{...};
(3,-5)*{...};
(-9,-7.5)*{_1};
(-7.5,-7.5)*{_1};
(-6,-7.5)*{_1};
(-1.1,-7.5)*{_i};
(1.1,-7.5)*{_i};
(0,-7.5)*{_i};
(7.8,-7.5)*{_k};
(6.3,-7.5)*{_k};
(9.6,-7.5)*{_k};
%
(-2,6)*{};
(0,0)*{\circ }
**\dir{-};
(-0.7,6)*{};
(0,0)*{\circ }
**\dir{-};
(0.7,6)*{};
(0,0)*{\circ }
**\dir{-};
(2,6)*{};
(0,0)*{\circ }
**\dir{-};
\endxy}\Ea \ \  1\leq i\leq k.
$$
Note that some of the sets $J_i$ can be empty so that some of the numbers decorating inputs can have no
legs attached! For example, an element
$
q= \Ba{c}\resizebox{8mm}{!}{\xy
(0,0)*{\circ}="o",
(-5,5)*{}="1",
(-2,5)*{}="2",
(2,5)*{}="3",
(5,5)*{}="4",
(-3,-5)*{}="5",
(3,-5)*{}="6",
(5,-5)*{}="7",
(-5,-5)*{}="8",
(0,-5)*{}="9",
(-5.5,7)*{_1},
(-2,7)*{_2},
(2,7)*{_3},
(5.9,7)*{_4},
(-3,-7)*{_2},
(3.5,-7)*{_4},
(5.9,-7)*{_5},
(-5.5,-7)*{_1},
(0,-7)*{_3},
\ar @{-} "o";"1" <0pt>
\ar @{-} "o";"2" <0pt>
\ar @{-} "o";"3" <0pt>
\ar @{-} "o";"4" <0pt>
\ar @{-} "o";"5" <0pt>
\ar @{-} "o";"6" <0pt>
\ar @{-} "o";"7" <0pt>
\ar @{-} "o";"8" <0pt>
\ar @{-} "o";"9" <0pt>
\endxy}\Ea\in \bar{\cP}(4,5)
$
can generate several different elements in $\f\cP$,
\Beq\label{5: examples of OP}
\Ba{c}\resizebox{10mm}{!}{\xy
(0,0)*{\circ}="o",
(-2.5,5)*{}="1",
(-0.8,5)*{}="2",
(0.8,5)*{}="3",
(2.5,5)*{}="4",
(-3,-5)*{}="5",
(3,-5)*{}="6",
(5,-5)*{}="7",
(-5,-5)*{}="8",
(0,-5)*{}="9",
(-3,-7)*{_1},
(3.5,-7)*{_1},
(5.9,-7)*{_2},
(-5.5,-7)*{_1},
(0,-7)*{_2},
\ar @{-} "o";"1" <0pt>
\ar @{-} "o";"2" <0pt>
\ar @{-} "o";"3" <0pt>
\ar @{-} "o";"4" <0pt>
\ar @{-} "o";"5" <0pt>
\ar @{-} "o";"6" <0pt>
\ar @{-} "o";"7" <0pt>
\ar @{-} "o";"8" <0pt>
\ar @{-} "o";"9" <0pt>
\endxy}\Ea \in \f\cP(2) \ \ \ \ , \ \ \ \
\Ba{c}\resizebox{14mm}{!}{\xy
(0,0)*{\circ}="o",
(-2.5,5)*{}="1",
(-0.8,5)*{}="2",
(0.8,5)*{}="3",
(2.5,5)*{}="4",
(-3,-5)*{}="5",
(3,-5)*{}="6",
(5,-5)*{}="7",
(-5,-5)*{}="8",
(0,-5)*{}="9",
(-3,-7)*{_1},
(3.5,-7)*{_1},
(5.9,-7)*{_2},
(-5.5,-7)*{_1},
(0,-7)*{_2},
(10,-7)*{_3},
(12,-7)*{_4},
\ar @{-} "o";"1" <0pt>
\ar @{-} "o";"2" <0pt>
\ar @{-} "o";"3" <0pt>
\ar @{-} "o";"4" <0pt>
\ar @{-} "o";"5" <0pt>
\ar @{-} "o";"6" <0pt>
\ar @{-} "o";"7" <0pt>
\ar @{-} "o";"8" <0pt>
\ar @{-} "o";"9" <0pt>
\endxy}\Ea \in \f\cP(4)\ \ , \ \ etc.
\Eeq
Often (but not always) it is useful to represent elements of $\f\cP$ not as corollas
decorated by elements from $\cP$ whose legs are labelled by possibly coinciding natural numbers, but
as graphs having two types of vertices:  the small one (with is decorated by an element of $\bar{\cP}$) and new big ones corresponding to inputs of $\f\cP$ and having a ``non-coinciding"  numerical labelling
$$
\Ba{c}\resizebox{15mm}{!}{\xy
(-9,-6)*{};
(0,0)*{\circ }
**\dir{-};
(-7.5,-6)*{};
(0,0)*{\circ }
**\dir{-};
(-6,-6)*{};
(0,0)*{\circ }
**\dir{-};
(-1,-6)*{};
(0,0)*{\circ }
**\dir{-};
(0,-6)*{};
(0,0)*{\circ }
**\dir{-};
(1,-6)*{};
(0,0)*{\circ }
**\dir{-};
(9,-6)*{};
(0,0)*{\circ }
**\dir{-};
(7.5,-6)*{};
(0,0)*{\circ }
**\dir{-};
(6,-6)*{};
(0,0)*{\circ }
**\dir{-};
(-3,-5)*{...};
(3,-5)*{...};
(-9,-7.5)*{_1};
(-7.5,-7.5)*{_1};
(-6,-7.5)*{_1};
(-1.1,-7.5)*{_i};
(1.1,-7.5)*{_i};
(0,-7.5)*{_i};
(7.8,-7.5)*{_k};
(6.3,-7.5)*{_k};
(9.6,-7.5)*{_k};
%
(-2,6)*{};
(0,0)*{\circ }
**\dir{-};
(-0.7,6)*{};
(0,0)*{\circ }
**\dir{-};
(0.7,6)*{};
(0,0)*{\circ }
**\dir{-};
(2,6)*{};
(0,0)*{\circ }
**\dir{-};
\endxy}\Ea\
\ \  {\simeq}
 \ \
\Ba{c}\resizebox{15mm}{!}{\xy
(-9,-6)*{};
(0,0)*{\circ }
**\dir{-};
(-7.5,-6)*{};
(0,0)*{\circ }
**\dir{-};
(-6,-6)*{};
(0,0)*{\circ }
**\dir{-};
(-1,-6)*{};
(0,0)*{\circ }
**\dir{-};
(0,-6)*{};
(0,0)*{\circ }
**\dir{-};
(1,-6)*{};
(0,0)*{\circ }
**\dir{-};
(9,-6)*{};
(0,0)*{\circ }
**\dir{-};
(7.5,-6)*{};
(0,0)*{\circ }
**\dir{-};
(6,-6)*{};
(0,0)*{\circ }
**\dir{-};
(-3,-5)*{...};
(3,-5)*{...};
%
(-2,6)*{};
(0,0)*{\circ }
**\dir{-};
(-0.7,6)*{};
(0,0)*{\circ }
**\dir{-};
(0.7,6)*{};
(0,0)*{\circ }
**\dir{-};
(2,6)*{};
(0,0)*{\circ }
**\dir{-};
(-8.2,-7.9)*+\hbox{${{1}}$}*\frm{o};
(8.2,-7.9)*+\hbox{${{\, k\, }}$}*\frm{o};
(0,-7.9)*+\hbox{${{\, i\, }}$}*\frm{o}
\endxy}\Ea
$$
In this notation elements (\ref{5: examples of OP}) gets represented, respectively, as
$$
\Ba{c}\resizebox{10mm}{!}{\xy
(-2.5,5)*{}="1",
(-0.8,5)*{}="2",
(0.8,5)*{}="3",
(2.5,5)*{}="4",
 (0,0)*{\circ}="A";
 (-6,-10)*+{_1}*\frm{o}="B";
  (6,-10)*+{_2}*\frm{o}="C";
 "A"; "B" **\crv{(-5,-0)}; 
 "A"; "B" **\crv{(-5,-6)};
  "A"; "C" **\crv{(5,-0.5)};
   "A"; "B" **\crv{(5,-1)};
  "A"; "C" **\crv{(-5,-7)};
  \ar @{-} "A";"1" <0pt>
\ar @{-} "A";"2" <0pt>
\ar @{-} "A";"3" <0pt>
\ar @{-} "A";"4" <0pt>
 \endxy}
 \Ea
 \ \ \ \ \ \mbox{and} \ \ \ \ \
\Ba{c}\resizebox{17mm}{!}{ \xy
(-2.5,5)*{}="1",
(-0.8,5)*{}="2",
(0.8,5)*{}="3",
(2.5,5)*{}="4",
 (0,0)*{\circ}="A";
 (-6,-10)*+{_1}*\frm{o}="B";
  (6,-10)*+{_2}*\frm{o}="C";
   (12,-10)*+{_3}*\frm{o};
   (18,-10)*+{_4}*\frm{o};
 "A"; "B" **\crv{(-5,-0)}; 
 "A"; "B" **\crv{(-5,-6)};
  "A"; "C" **\crv{(5,-0.5)};
   "A"; "B" **\crv{(5,-1)};
  "A"; "C" **\crv{(-5,-7)};
  \ar @{-} "A";"1" <0pt>
\ar @{-} "A";"2" <0pt>
\ar @{-} "A";"3" <0pt>
\ar @{-} "A";"4" <0pt>
 \endxy}
 \Ea
$$
while generators of $\cC om(n)\subset \f\cP(n)$ as
$
\Ba{c}\resizebox{15mm}{!}{
\xy
(0,2)*+{_1}*\frm{o};
(6,2)*+{_2}*\frm{o};
(18,2)*+{_n}*\frm{o};
(13,2)*+{\cdots};
\endxy}\Ea
$.

\sip

We denote  $\bar{\cP}(0,n):={\cC om}(n)\simeq\id_n$, $n\geq 1$, and set $\f\cP^0_{J_1,...,J_k}$
to be $\cC om(k)$ for $J_1=\ldots= J_k=\emptyset$ and zero otherwise; we also extend the map
(\ref{5: m=0 projection}) to the value $m=0$
$$
\pi^0_{J_1,...,J_k}: \bar{\cP}(0,n) \lon \f\cP^0_{J_1,...,J_k}
$$
in the obvious way.

\bip

For any $i\in [k]$ consider a map
$$
\Ba{rccc}
\circ_i: & \f\cP(k) \ot \f\cP(l) & \lon & \f\cP(k+l-1)
\Ea
$$
which is given on arbitrary elements $a\in \f\cP^m_{J_1,...,J_k}\subset \f\cP(k)$
and $b\in \f\cP^q_{I_1,...,I_l}\subset \f\cP(l)$
as follows
\Bi
\item[(a)] If $m,q\geq 1$,
\Beqrn
a\circ_i b &:= & \sum_{r=1}^{\max(\# I_i,q)} \sum_{f: [r]\hook I_i\atop
f: [r]\hook [q]} \sum_{J_i\setminus f([r])=T_1\sqcup ...\sqcup T_l\atop
\# T_1,..., \# T_l\geq 0}
\pi^{m+q- r}_{J_1,\ldots, J_{i-1}, I_1\sqcup T_1,...,I_l\sqcup T_l, J_{i+1},..., J_{k}} (a\ _{f([r])}\circ_{g([r])}\ b)   \\
&&
+\ \ \sum_{I_i=T_1\sqcup ...\sqcup T_l\atop
\# T_1,..., \# T_l\geq 0}
\pi^{m+q}_{J_1,\ldots, J_{i-1}, J_1\sqcup T_1,...,J_l\sqcup T_l, J_{i+1},..., J_{k}} (a\ \circ_H\ b)
\Eeqrn
\item[(b)] if $m\geq 1$, $q=0$,
$$
a\circ_i \left(\Ba{c}\resizebox{14mm}{!}{\xy
(0,2)*+{_1}*\frm{o};
(6,2)*+{_2}*\frm{o};
(18,2)*+{_l}*\frm{o};
(13,2)*+{\cdots};
\endxy}\Ea \right):= \sum_{J_i=T_1\sqcup ...\sqcup T_l\atop
\# T_1,..., \# T_l\geq 0}
\pi^{m}_{J_1,\ldots, J_{i-1}, T_1,..., T_l, J_{i+1},..., J_{k}} (a)
$$
\item[(c)] if $m=0$, $q=0$,
$$
\left(\Ba{c}\resizebox{14mm}{!}{\xy
(0,2)*+{_1}*\frm{o};
(6,2)*+{_2}*\frm{o};
(18,2)*+{_k}*\frm{o};
(13,2)*+{\cdots};
\endxy}\Ea \right)\circ_i
\left(\Ba{c}\resizebox{14mm}{!}{\xy
(0,2)*+{_1}*\frm{o};
(6,2)*+{_2}*\frm{o};
(18,2)*+{_l}*\frm{o};
(13,2)*+{\cdots};
\endxy}\Ea \right):=
\Ba{c}\resizebox{14mm}{!}{\xy
(0,2)*+{_1}*\frm{o};
(6,2)*+{_2}*\frm{o};
(19,2)*+{_{_{_{{l+k-1}}}}}*\frm{o};
(11,2)*+{\cdots};
\endxy}\Ea
$$

\Ei
The first group of elements in (a) belongs, for each $r$, to the subspace $\f\cP^{m+q-r}_{J_1,\ldots, J_{i-1}, J_1\sqcup T_1,...,J_l\sqcup T_l, J_{i+l+1},..., J_{k+l-1}}$ of $ \f\cP(k+l-1)$ and is obtained from $a$ and $b$ by (i) taking vertical compositions of along  $r$-labelled outputs of $b$ and the corresponding $r$-labelled inputs of $a$ belonging to the bunch
$J_i$, (ii) symmetrizing over the outputs of $a$ with the remaining $q-r$
outputs of $b$, and (iii) taking the sum over all possible assignments of the remaining
$\# J_i -r$ inputs of $a$ to the input bunches $I_1,..., I_l$ of $b$.
The second group of elements in (a) belongs to the subspace $\f\cP^{m+q}_{J_1,\ldots, J_{i-1}, J_1\sqcup T_1,...,J_l\sqcup T_l, J_{i+l+1},..., J_{k+l-1}}$ of $ \f\cP(k+l-1)$ and is obtained from $a$ and $b$ by taking their horizontal composition, symmetrizing over the all outputs $m+p$ outputs, and taking a sum over all possible assignments of the input legs of $a$ from
the bunch
$J_i$  to the input bunches $I_1,..., I_l$ of $b$.

\sip

Using ``big-circle" notation for inputs one can understand the operadic composition $a\circ_i b$ in a way similar to the description of the operadic composition in the operad $\cG ra_{d+1}$: substitute the decorated corolla $b=\Ba{c}\resizebox{10mm}{!}{\xy
(-9,-6)*{};
(0,0)*{\circ }
**\dir{-};
(-7.5,-6)*{};
(0,0)*{\circ }
**\dir{-};
(-6,-6)*{};
(0,0)*{\circ }
**\dir{-};
(9,-6)*{};
(0,0)*{\circ }
**\dir{-};
(7.5,-6)*{};
(0,0)*{\circ }
**\dir{-};
(6,-6)*{};
(0,0)*{\circ }
**\dir{-};
(0,-5)*{...};
(-2,6)*{};
(0,0)*{\circ }
**\dir{-};
(-0.7,6)*{};
(0,0)*{\circ }
**\dir{-};
(0.7,6)*{};
(0,0)*{\circ }
**\dir{-};
(2,6)*{};
(0,0)*{\circ }
**\dir{-};
(-8.2,-7.9)*+\hbox{${{1}}$}*\frm{o};
(8.2,-7.9)*+\hbox{${{\, l\, }}$}*\frm{o};
\endxy}\Ea
$ into the $i$-th input  vertex of
$a=\Ba{c}\resizebox{12mm}{!}{\xy
(-9,-6)*{};
(0,0)*{\circ }
**\dir{-};
(-7.5,-6)*{};
(0,0)*{\circ }
**\dir{-};
(-6,-6)*{};
(0,0)*{\circ }
**\dir{-};
(-1,-6)*{};
(0,0)*{\circ }
**\dir{-};
(0,-6)*{};
(0,0)*{\circ }
**\dir{-};
(1,-6)*{};
(0,0)*{\circ }
**\dir{-};
(9,-6)*{};
(0,0)*{\circ }
**\dir{-};
(7.5,-6)*{};
(0,0)*{\circ }
**\dir{-};
(6,-6)*{};
(0,0)*{\circ }
**\dir{-};
(-3,-5)*{...};
(3,-5)*{...};
%
(-2,6)*{};
(0,0)*{\circ }
**\dir{-};
(-0.7,6)*{};
(0,0)*{\circ }
**\dir{-};
(0.7,6)*{};
(0,0)*{\circ }
**\dir{-};
(2,6)*{};
(0,0)*{\circ }
**\dir{-};
(-8.2,-7.9)*+\hbox{${{1}}$}*\frm{o};
(8.2,-7.9)*+\hbox{${{\, k\, }}$}*\frm{o};
(0,-7.9)*+\hbox{${{\, i\, }}$}*\frm{o}
\endxy}\Ea$ and then take the sum over all possible ways to attach the hanging edges (previously attached to the $i$-th vertex) to the out-put legs of $b$  and to the input big-circle vertices of $b$. The element  $\xy (0,0)*+{_1}*\frm{o};\endxy$ plays the role of the unit in the operad $\f\cP$.

\subsubsection{\bf Proposition}\label{5: Proposition on OP and its representations} {\em (i) For any augmented prop $\cP$ the associated data $(\f\cP, \circ_i)$
is an operad (called the {\em polydifferential operad associated to $\cP$}).

\sip

(ii) Any representation of the augmented prop $\cP$ in a graded vector space $V$ induces canonically an associated
representation of the operad $\f\cP$ in $\odot^\bu V$.}

\begin{proof} Given an arbitrary representation $\rho: \cP \rar \cE nd_V$. Compositions $\circ_i$
have been defined in such a way that
the morphism of $\bS$-modules $\rho^{poly}: \f\cP\rar \cE nd_{\odot^\bu V}$ given explicitly above satisfies the condition
$$
\rho^{poly}(a\circ_i b)= \rho^{poly}(a)\circ_i \rho^{poly}(b).
$$
Put another way, compositions $\circ_i$ in $\f\cP$ are just combinatorial incarnations of the standard substituition of polydifferential operators acting on $\odot^\bu V$. That these compositions satisfy
axioms of an operad is a straightforward check (which is easiest to do in local coordinates as in the beginning of \S 5.2).
\end{proof}

\subsubsection{\bf Example}\label{5: subsubsect on Examples in OP operad}
 Consider elements $\begin{xy}
 <0mm,0.66mm>*{};<0mm,3mm>*{}**@{-},
 <0.39mm,-0.39mm>*{};<2.2mm,-2.2mm>*{}**@{-},
 <-0.35mm,-0.35mm>*{};<-2.2mm,-2.2mm>*{}**@{-},
 <0mm,0mm>*{\circ};<0mm,0mm>*{}**@{},
   <0mm,0.66mm>*{};<0mm,3.4mm>*{^1}**@{},
   <0.39mm,-0.39mm>*{};<2.9mm,-4mm>*{^2}**@{},
   <-0.35mm,-0.35mm>*{};<-2.8mm,-4mm>*{^1}**@{},
\end{xy}\in \cP(1,2)$ and $\begin{xy}
 <0mm,-0.55mm>*{};<0mm,-2.5mm>*{}**@{-},
 <0.5mm,0.5mm>*{};<2.2mm,2.2mm>*{}**@{-},
 <-0.48mm,0.48mm>*{};<-2.2mm,2.2mm>*{}**@{-},
 <0mm,0mm>*{\circ};<0mm,0mm>*{}**@{},
 <0mm,-0.55mm>*{};<0mm,-3.8mm>*{_1}**@{},
 <0.5mm,0.5mm>*{};<2.7mm,2.8mm>*{^2}**@{},
 <-0.48mm,0.48mm>*{};<-2.7mm,2.8mm>*{^1}**@{},
 \end{xy}\in \cP(2,1)$.
 We can associate to them, e.g., the following three elements in the operad $\f\cP$,
$$
\Ba{c}{\xy
(-3,-3)*{};
(0,0)*{\circ}
**\dir{-};
(3,-3)*{};
(0,0)*{\circ }
**\dir{-};
(0,4)*{};
(0,0)*{\circ }
**\dir{-};
(-3.5,-4.5)*{_1};
(3.5,-4.5)*{_2};
\endxy}\Ea \in \f\cP(2) \  , \ \ \ \
\Ba{c}{\xy
(-0.7,-4)*{};
(-0,0)*{\circ}
**\dir{-};
(0.7,-4)*{};
(0,0)*{\circ }
**\dir{-};
(0,4)*{};
(0,0)*{\circ}
**\dir{-};
(-0.8,-5.6)*{_1};
(0.8,-5.6)*{_1};
\endxy}\Ea
\ , \
\Ba{c}{\xy
(-0.7,4)*{};
(-0,0)*{\circ}
**\dir{-};
(0.7,4)*{};
(0,0)*{\circ}
**\dir{-};
(0,-4)*{};
(0,0)*{\circ }
**\dir{-};
(-0,-5.6)*{_1};
\endxy}\Ea
\in \f\cP(1)
$$
We have in $\f\cP$,
$$
\Ba{c}{\xy
(-3,-3)*{};
(0,0)*{\circ}
**\dir{-};
(3,-3)*{};
(0,0)*{\circ}
**\dir{-};
(0,4)*{};
(0,0)*{\circ}
**\dir{-};
(-3.5,-4.5)*{_1};
(3.5,-4.5)*{_2};
\endxy}
\Ea
\
 \circ_2
 \
\Ba{c}{\xy
(-3,-3)*{};
(0,0)*{\circ}
**\dir{-};
(3,-3)*{};
(0,0)*{\circ }
**\dir{-};
(0,4)*{};
(0,0)*{\circ}
**\dir{-};
(-3.5,-4.5)*{_1};
(3.5,-4.5)*{_2};
\endxy}
\Ea
=
\Ba{c}{\xy
(-3,-3)*{};
(0,0)*{\circ}
**\dir{-};
(4,-4)*{\circ };
(0,0)*{\circ}
**\dir{-};
(4,-4)*{\circ};
(7.5,-7.5)*{}
**\dir{-};
(4,-4)*{\circ};
(1.5,-7.5)*{}
**\dir{-};
(0,4)*{};
(0,0)*{\circ}
**\dir{-};
(-3.5,-4.5)*{_1};
(1.5,-9)*{_2};
(7.5,-9)*{_3};
\endxy}
\Ea\ \ + \ \
\Ba{c}{\xy
(-3,-3)*{};
(0,0)*{\circ}
**\dir{-};
(3,-3)*{};
(0,0)*{\circ}
**\dir{-};
(0,4)*{};
(0,0)*+{\circ}
**\dir{-};
(-3.5,-4.5)*{_1};
(3,-4.5)*{_2};
\endxy}
\Ea
\hspace{-3.5mm}
\Ba{c}{\xy
(-3,-3)*{};
(0,0)*{\circ}
**\dir{-};
(3,-3)*{};
(0,0)*{\circ}
**\dir{-};
(0,4)*{};
(0,0)*{\circ}
**\dir{-};
(-3,-4.5)*{_2};
(3.5,-4.5)*{_3};
\endxy}
\Ea
\ \ + \ \
\Ba{c}{\xy
(-3,-3)*{};
(0,0)*{\circ }
**\dir{-};
(10,-3)*{};
(0,0)*{\circ }
**\dir{-};
(0,4)*{};
(0,0)*{\circ}
**\dir{-};
(-3.5,-4.5)*{_1};
(10,-4.5)*{_3};
\endxy}
\Ea
\hspace{-9.5mm}
\Ba{c}{\xy
(-3,-3)*{};
(0,0)*{\circ }
**\dir{-};
(3,-3)*{};
(0,0)*{\circ }
**\dir{-};
(0,4)*{};
(0,0)*{\circ }
**\dir{-};
(-3,-4.5)*{_2};
(3.5,-4.5)*{_3};
\endxy}
\Ea
$$

$$
\Ba{c}{\xy
(-0.7,-4)*{};
(-0,0)*{\circ }
**\dir{-};
(0.7,-4)*{};
(0,0)*{\circ }
**\dir{-};
(0,4)*{};
(0,0)*{\circ }
**\dir{-};
(-0.8,-6)*{_1};
(0.8,-6)*{_1};
\endxy}\Ea\ \circ_1 \
\Ba{c}{\xy
(-3,-3)*{};
(0,0)*{\circ }
**\dir{-};
(3,-3)*{};
(0,0)*{\circ }
**\dir{-};
(0,4)*{};
(0,0)*{\circ }
**\dir{-};
(-3.5,-4.5)*{_1};
(3.5,-4.5)*{_2};
\endxy}\Ea
=
\Ba{c}{\xy
(0.8,-7.5)*{};
(2,0)*{\circ }
**\dir{-};
(4,-4)*{\circ };
(2,0)*{\circ }
**\dir{-};
(4,-4)*{\circ };
(7.5,-7.5)*{}
**\dir{-};
(4,-4)*{\circ };
(1.5,-7.5)*{}
**\dir{-};
(2,4)*{};
(2,0)*{\circ }
**\dir{-};
(1.0,-9)*{_1};
(2.0,-9)*{_1};
(7.5,-9)*{_2};
\endxy}
\Ea
+
\Ba{c}{\xy
(8.3,-7.5)*{};
(6,0)*{\circ }
**\dir{-};
(4,-4)*{\circ };
(6,0)*{\circ }
**\dir{-};
(4,-4)*{\circ };
(7.5,-7.5)*{}
**\dir{-};
(4,-4)*{\circ };
(1.5,-7.5)*{}
**\dir{-};
(6,4)*{};
(6,0)*{\circ }
**\dir{-};
(8.8,-9)*{_2};
(2.0,-9)*{_1};
(7.2,-9)*{_2};
\endxy}
\Ea
+
\Ba{c}{\xy
(-0.7,-3)*{};
(-0,0)*{\circ }
**\dir{-};
(0.7,-3)*{};
(0,0)*{\circ }
**\dir{-};
(0,4)*{};
(0,0)*{\circ }
**\dir{-};
(-0.8,-4.5)*{_1};
(0.9,-4.5)*{_1};
\endxy}
\Ea
\hspace{-3.5mm}
\Ba{c}{\xy
(-3,-3)*{};
(0,0)*{\circ }
**\dir{-};
(3,-3)*{};
(0,0)*{\circ }
**\dir{-};
(0,4)*{};
(0,0)*{\circ }
**\dir{-};
(-2.8,-4.5)*{_1};
(3.5,-4.5)*{_2};
\endxy}
\Ea
+
\Ba{c}{\xy
(-3,-3)*{};
(0,0)*{\circ }
**\dir{-};
(3,-3)*{};
(0,0)*{\circ }
**\dir{-};
(0,4)*{};
(0,0)*{\circ }
**\dir{-};
(-2.8,-4.5)*{_1};
(3.5,-4.5)*{_2};
\endxy}
\Ea
\hspace{-3.5mm}
\Ba{c}{\xy
(-0.7,-3)*{};
(-0,0)*{\circ }
**\dir{-};
(0.7,-3)*{};
(0,0)*{\circ }
**\dir{-};
(0,4)*{};
(0,0)*{\circ }
**\dir{-};
(-0.7,-4.5)*{_2};
(1.0,-4.5)*{_2};
\endxy}
\Ea
\ \ + \ \
\Ba{c}{\xy
(4,-3)*{};
(0,0)*{\circ }
**\dir{-};
(10,-3)*{};
(0,0)*{\circ }
**\dir{-};
(0,4)*{};
(0,0)*{\circ }
**\dir{-};
(3.9,-4.5)*{_1};
(10,-4.5)*{_2};
\endxy}
\Ea
\hspace{-9.5mm}
\Ba{c}{\xy
(-3,-3)*{};
(0,0)*{\circ }
**\dir{-};
(3,-3)*{};
(0,0)*{\circ }
**\dir{-};
(0,4)*{};
(0,0)*{\circ }
**\dir{-};
(-2.8,-4.5)*{_1};
(3.5,-4.5)*{_2};
\endxy}
\Ea
$$

$$
\Ba{c}{\xy
(-0.7,-4)*{};
(-0,0)*{\circ }
**\dir{-};
(0.7,-4)*{};
(0,0)*{\circ }
**\dir{-};
(0,4)*{};
(0,0)*{\circ }
**\dir{-};
(-0.8,-6)*{_1};
(0.8,-6)*{_1};
\endxy}\Ea\ \circ_1 \
\Ba{c}{\xy
(-0.7,-4)*{};
(-0,0)*{\circ }
**\dir{-};
(0.7,-4)*{};
(0,0)*{\circ }
**\dir{-};
(0,4)*{};
(0,0)*{\circ }
**\dir{-};
(-0.8,-6)*{_1};
(0.8,-6)*{_1};
\endxy}\Ea
=
\Ba{c}{\xy
(1.5,-7.5)*{};
(2,0)*{\circ }
**\dir{-};
(4,-4)*{\circ };
(2,0)*{\circ }
**\dir{-};
(4,-4)*{\circ };
(3.5,-7.5)*{}
**\dir{-};
(4,-4)*{\circ };
(2.3,-7.5)*{}
**\dir{-};
(2,4)*{};
(2,0)*{\circ }
**\dir{-};
(1.0,-9)*{_1};
(2.6,-9)*{_1};
(4,-9)*{_1};
\endxy}
\Ea
\ + \
\Ba{c}{\xy
(-0.7,-4)*{};
(-0,0)*{\circ }
**\dir{-};
(0.7,-4)*{};
(0,0)*{\circ }
**\dir{-};
(0,4)*{};
(0,0)*{\circ }
**\dir{-};
(-0.5,-6)*{_1};
(0.8,-6)*{_1};
\endxy}\Ea
\hspace{-4mm}
\Ba{c}{\xy
(-0.7,-4)*{};
(-0,0)*{\circ }
**\dir{-};
(0.7,-4)*{};
(0,0)*{\circ }
**\dir{-};
(0,4)*{};
(0,0)*{\circ }
**\dir{-};
(-0.5,-6)*{_1};
(0.9,-6)*{_1};
\endxy}\Ea
\ \ \ \ \ , \ \ \ \ \
\Ba{c}{\xy
(-0.7,-4)*{};
(-0,0)*{\circ }
**\dir{-};
(0.7,-4)*{};
(0,0)*{\circ }
**\dir{-};
(0,4)*{};
(0,0)*{\circ }
**\dir{-};
(-0.8,-5.6)*{_1};
(0.8,-5.6)*{_1};
\endxy}\Ea
\ \circ_1 \
\Ba{c}{\xy
(-0.7,4)*{};
(0,0)*{\circ }
**\dir{-};
(0.7,4)*{};
(0,0)*{\circ }
**\dir{-};
(0,-4)*{};
(0,0)*{\circ }
**\dir{-};
(-0,-5.6)*{_1};
\endxy}\Ea
=
\Ba{c}
\xy
 (0,0)*{\circ}="a",
(0,6)*{\circ}="b",
(3,3)*{}="c",
(-3,3)*{}="d",
 (0,9)*{}="b'",
(0,-3)*{}="a'",
\ar@{-} "a";"c" <0pt>
\ar @{-} "a";"d" <0pt>
\ar @{-} "a";"a'" <0pt>
\ar @{-} "b";"c" <0pt>
\ar @{-} "b";"d" <0pt>
\ar @{-} "b";"b'" <0pt>
\endxy
\Ea
\ + \
\begin{xy}
 <0mm,-1.3mm>*{};<0mm,-3.5mm>*{}**@{-},
 <0.38mm,-0.2mm>*{};<2.0mm,2.0mm>*{}**@{-},
 <-0.38mm,-0.2mm>*{};<1.2mm,5.0mm>*{}**@{-},
<0mm,-0.8mm>*{\circ};
 <2.4mm,2.4mm>*{\circ};
 <2.77mm,2.0mm>*{};<1.4mm,-3.5mm>*{}**@{-},
 <2.4mm,3mm>*{};<2.4mm,5.2mm>*{}**@{-},
     <0mm,-1.3mm>*{};<0mm,-5.4mm>*{^1}**@{},
     <2.5mm,2.3mm>*{};<1.6mm,-5.4mm>*{^1}**@{},
    \end{xy}
$$

\subsection{ Polydifferential operads and hypergraphs} The above examples show that the horizontal compositions in a prop $\cP$  play as important role in the definition of the polydifferential operad $\f\cP$ as vertical ones so that to apply the polydifferential functor to a properad (or operad) $\cP$
one has to take first its prop enveloping $\cU\cP$ and then apply $\f$ to the latter; for an augmented properad $\cP$ we understand $\cU\cP$ as $\K \oplus \cU\bar{\cP}$ and define
$$
\f\cP:=\f(\cU{\cP}).
$$
The case of properads is of special interest as elements of $\f\cP$ can be understood now as {\em hypergraphs}, that is, generalizations of graphs in which edges can connect more than two vertices.
For example, elements of a properad $\cP$
$$
p=
\Ba{c}\resizebox{8mm}{!}{ \xy
(0,0)*{\circ}="o",
(-1.5,5)*{}="2",
(1.5,5)*{}="3",
(-1.5,-5)*{}="5",
(1.5,-5)*{}="6",
(4,-5)*{}="7",
(-4,-5)*{}="8",
(-1.5,7)*{_1},
(1.5,7)*{_2},
(-1.5,-7)*{_2},
(1.5,-7)*{_3},
(4,-7)*{_4},
(-4,-7)*{_1},
\ar @{-} "o";"2" <0pt>
\ar @{-} "o";"3" <0pt>
\ar @{-} "o";"5" <0pt>
\ar @{-} "o";"6" <0pt>
\ar @{-} "o";"7" <0pt>
\ar @{-} "o";"8" <0pt>
\endxy}\Ea\in \cP(2,4) \ \ \ \mbox{and} \ \ \ \
q=\Ba{c}\resizebox{7mm}{!}{  \xy
(0,0)*{\circ}="o",
(0,5)*{}="1",
(-3,-5)*{}="2",
(3,-5)*{}="3",
(0,-5)*{}="4",
(0,7)*{_1},
(3.5,-7)*{_3},
(-3.5,-7)*{_1},
(0,-7)*{_2},
\ar @{-} "o";"1" <0pt>
\ar @{-} "o";"2" <0pt>
\ar @{-} "o";"3" <0pt>
\ar @{-} "o";"4" <0pt>
\endxy}\Ea\in \cP(1,3)
$$
generate, for example, an element
$$
\Ba{c}\resizebox{8mm}{!}{  \xy
(-2,0)*{^p};
(0,0)*{\circ}="o",
(-2,5)*{}="2",
(2,5)*{}="3",
(-1.5,-5)*{}="5",
(1.5,-5)*{}="6",
(4,-5)*{}="7",
(-4,-5)*{}="8",
(-2,7)*{_1},
(2,7)*{_2},
(-1.5,-7)*{_2},
(1.5,-7)*{_3},
(4,-7)*{_4},
(-4,-7)*{_1},
\ar @{-} "o";"2" <0pt>
\ar @{-} "o";"3" <0pt>
\ar @{-} "o";"5" <0pt>
\ar @{-} "o";"6" <0pt>
\ar @{-} "o";"7" <0pt>
\ar @{-} "o";"8" <0pt>
\endxy}\Ea
\Ba{c}\resizebox{7.3mm}{!}{ \xy
(2,0)*{^q};
(0,0)*{\circ}="o",
(0,5)*{}="1",
(-3,-5)*{}="2",
(3,-5)*{}="3",
(0,-5)*{}="4",
(0,7)*{_3},
(3.5,-7)*{_7},
(-3.5,-7)*{_5},
(0,-7)*{_6},
\ar @{-} "o";"1" <0pt>
\ar @{-} "o";"2" <0pt>
\ar @{-} "o";"3" <0pt>
\ar @{-} "o";"4" <0pt>
\endxy}\Ea\ \in \cU\cP(3,7)
$$
which in turn gives, for example, rise to an element
$$
\Ba{c}\resizebox{15mm}{!}{  \xy
(-2,0.8)*{^p};
(11,0.8)*{^q};
(-1.5,5)*{}="1",
(1.5,5)*{}="2",
(9,5)*{}="3",
 (0,0)*{\circ}="A";
  (9,0)*{\circ}="O";
 (-6,-10)*+{_1}*\frm{o}="B";
  (6,-10)*+{_2}*\frm{o}="C";
   (14,-10)*+{_3}*\frm{o}="D";
 "A"; "B" **\crv{(-5,-0)}; 
  "A"; "D" **\crv{(5,-0.5)};
   "A"; "B" **\crv{(5,-1)};
  "A"; "C" **\crv{(-5,-7)};
  \ar @{-} "A";"1" <0pt>
\ar @{-} "A";"2" <0pt>
\ar @{-} "O";"C" <0pt>
\ar @{-} "O";"D" <0pt>
\ar @{-} "O";"3" <0pt>
 \endxy}
 \Ea \in \f\cP(3)
$$
which looks like a real graph with vertices of two types --- the small ones which are decorated
by elements of the properad $\cP$ and big ones corresponding to the inputs of the operad $\f\cP$.
In fact it is better to understand this graph as a {\em hypergraph} with  small vertices
playing the role of {\em hyperedges}. With this interpretation of elements of $\f\cP$ one can recover, for example, the operad of graphs  $\cG ra_{d}$ from \S {\ref{2: subsection on operad Gra}} as a particular polydifferential operad (see the following example as well as example {\ref{5: example Konts graph complexes}} below).

\subsubsection{\bf Example}\label{5: Example def of prop T}  Let $\cT_d$ ($\cT$ standing for {\em trivial})  be a properad such that
\[
 \cT_d(m,n)
 =
 \begin{cases}
  \K[d-1] & \text{if $n=2$ and $m=1$} \\
  \K & \text{if $n=1$ and $m=1$} \\
  0 & \text{otherwise}
 \end{cases}.
\]
We define the compositions such that the arity $(1,1)$ element is the properadic unit,
and such that all properadic compositions of the arity $(2,1)$ element are zero.
The properad is clearly augmented, the augmentation being the projection onto the part of arity $(1,1)$.
In this case $\f\cT_d= \cG ra_{d}$.

\subsection{Twisting of properads by morphisms from (involutive) Lie bialgebras in the case $c=d=0$} Let us start for pedagogical reasons with the case $c=d=0$ when all the generators of $\LB_{0,0}/\LoB_{0,0}$ and $\caL ie_0/\caL ie_0^\diamond$ have degree $+1$. Given a morphism of properads
$$
g: \LoB_{0,0} \lon \cP ,
$$
let us denote the images of generators of $\LoB_{0,0}$ under $g$
by the corollas
  $\Ba{c}\resizebox{7mm}{!}{ \xy
(-3,3)*{};
(0,0)*{\circledcirc }
**\dir{-};
(3,3)*{};
(0,0)*{\circledcirc }
**\dir{-};
(0,-4)*{};
(0,0)*{\circledcirc }
**\dir{-};
(-3.5,4.5)*{_1};
(3.5,4.5)*{_2};
\endxy}\Ea$
 and
 $\Ba{c}\resizebox{7mm}{!}{\xy
(-3,-3)*{};
(0,0)*{\circledcirc }
**\dir{-};
(3,-3)*{};
(0,0)*{\circledcirc }
**\dir{-};
(0,4)*{};
(0,0)*{\circledcirc }
**\dir{-};
(-3.5,-4.5)*{_1};
(3.5,-4.5)*{_2};
\endxy}\Ea$
  (some or both of them can, in principle, stand for zero).
 In these notations we have the following statement.

\subsubsection{\bf Lemma}\label{5: Lemma on Lie_0^dimanod to P}
{\em For any a morphism of properads
$
g:\LoB_{0,0} \lon \cP
$
there is an associated morphism of operads,
$$
g^{\diamond}: \caL ie^\diamond_0 \lon \f\cP
$$
given on the generators of by}
$$
g^{\diamond}:\left(\begin{xy}
 <0mm,-0.55mm>*{};<0mm,-3mm>*{}**@{-},
 <0mm,0.5mm>*{};<0mm,3mm>*{}**@{-},
 <0mm,0mm>*{\bullet};<0mm,0mm>*{}**@{},
 \end{xy}\right):= \Ba{c}{\xy
(-0.7,4)*{};
(-0,0)*{\circledcirc }
**\dir{-};
(0.7,4)*{};
(0,0)*{\circledcirc }
**\dir{-};
(0,-4)*{};
(0,0)*{\circledcirc}
**\dir{-};
(-0,-5.6)*{_1};
\endxy}\Ea
 +
 \Ba{c}{\xy
(-0.7,-4)*{};
(-0,0)*{\circledcirc }
**\dir{-};
(0.7,-4)*{};
(0,0)*{\circledcirc }
**\dir{-};
(0,4)*{};
(0,0)*{\circledcirc }
**\dir{-};
(-0.8,-5.6)*{_1};
(0.8,-5.6)*{_1};
\endxy}\Ea
\ \ \ \ , \ \ \ \
g^{\diamond}\left(\Ba{c}\begin{xy}
 <0mm,0.66mm>*{};<0mm,3mm>*{}**@{-},
 <0.39mm,-0.39mm>*{};<2.2mm,-2.2mm>*{}**@{-},
 <-0.35mm,-0.35mm>*{};<-2.2mm,-2.2mm>*{}**@{-},
 <0mm,0mm>*{\bu};<0mm,0mm>*{}**@{},
   <0.39mm,-0.39mm>*{};<2.9mm,-4mm>*{^2}**@{},
   <-0.35mm,-0.35mm>*{};<-2.8mm,-4mm>*{^1}**@{},
\end{xy}\Ea\right) = \Ba{c}{\xy
(-3,-3)*{};
(0,0)*{\circledcirc }
**\dir{-};
(3,-3)*{};
(0,0)*{\circledcirc }
**\dir{-};
(0,4)*{};
(0,0)*{\circledcirc }
**\dir{-};
(-3.5,-4.5)*{_1};
(3.5,-4.5)*{_2};
\endxy}\Ea
$$
\begin{proof} The claim is proven once one checks the following three equations
$$
\left(
 \Ba{c}{\xy
(-0.7,4)*{};
(-0,0)*{\circledcirc }
**\dir{-};
(0.7,4)*{};
(0,0)*{\circledcirc }
**\dir{-};
(0,-4)*{};
(0,0)*{\circledcirc }
**\dir{-};
(-0,-5.6)*{_1};
\endxy}\Ea + \Ba{c}{\xy
(-0.7,-4)*{};
(-0,0)*{\circledcirc }
**\dir{-};
(0.7,-4)*{};
(0,0)*{\circledcirc }
**\dir{-};
(0,4)*{};
(0,0)*{\circledcirc }
**\dir{-};
(-0.8,-5.6)*{_1};
(0.8,-5.6)*{_1};
\endxy}\Ea
\right)
\circ_1
\left(
 \Ba{c}{\xy
(-0.7,4)*{};
(-0,0)*{\circledcirc }
**\dir{-};
(0.7,4)*{};
(0,0)*{\circledcirc }
**\dir{-};
(0,-4)*{};
(0,0)*{\circledcirc }
**\dir{-};
(-0,-5.6)*{_1};
\endxy}\Ea +
\Ba{c}{\xy
(-0.7,-4)*{};
(-0,0)*{\circledcirc }
**\dir{-};
(0.7,-4)*{};
(0,0)*{\circledcirc }
**\dir{-};
(0,4)*{};
(0,0)*{\circledcirc }
**\dir{-};
(-0.8,-5.6)*{_1};
(0.8,-5.6)*{_1};
\endxy}\Ea
\right)
=0,
$$

$$
\left( \Ba{c}{\xy
(-0.7,4)*{};
(-0,0)*{\circledcirc }
**\dir{-};
(0.7,4)*{};
(0,0)*{\circledcirc }
**\dir{-};
(0,-4)*{};
(0,0)*{\circledcirc }
**\dir{-};
(-0,-5.6)*{_1};
\endxy}\Ea + \Ba{c}{\xy
(-0.7,-4)*{};
(-0,0)*{\circledcirc }
**\dir{-};
(0.7,-4)*{};
(0,0)*{\circledcirc }
**\dir{-};
(0,4)*{};
(0,0)*{\circledcirc }
**\dir{-};
(-0.8,-5.6)*{_1};
(0.8,-5.6)*{_1};
\endxy}\Ea\right)
\circ_1
\Ba{c}{\xy
(-3,-3)*{};
(0,0)*{\circledcirc }
**\dir{-};
(3,-3)*{};
(0,0)*{\circledcirc }
**\dir{-};
(0,4)*{};
(0,0)*{\circledcirc }
**\dir{-};
(-3.5,-4.5)*{_1};
(3.5,-4.5)*{_2};
\endxy}\Ea
+
\Ba{c}{\xy
(-3,-3)*{};
(0,0)*{\circledcirc }
**\dir{-};
(3,-3)*{};
(0,0)*{\circledcirc }
**\dir{-};
(0,4)*{};
(0,0)*{\circledcirc }
**\dir{-};
(-3.5,-4.5)*{_1};
(3.5,-4.5)*{_2};
\endxy}\Ea \circ_1
\left(
\Ba{c}{\xy
(-0.7,4)*{};
(-0,0)*{\circledcirc }
**\dir{-};
(0.7,4)*{};
(0,0)*{\circledcirc }
**\dir{-};
(0,-4)*{};
(0,0)*{\circledcirc }
**\dir{-};
(-0,-5.6)*{_1};
\endxy}\Ea
 +
 \Ba{c}{\xy
(-0.7,-4)*{};
(-0,0)*{\circledcirc }
**\dir{-};
(0.7,-4)*{};
(0,0)*{\circledcirc }
**\dir{-};
(0,4)*{};
(0,0)*{\circledcirc }
**\dir{-};
(-0.8,-5.6)*{_1};
(0.8,-5.6)*{_1};
\endxy}\Ea
\right)
\
+
\
\Ba{c}{\xy
(-3,-3)*{};
(0,0)*{\circledcirc }
**\dir{-};
(3,-3)*{};
(0,0)*{\circledcirc }
**\dir{-};
(0,4)*{};
(0,0)*{\circledcirc }
**\dir{-};
(-3.5,-4.5)*{_1};
(3.5,-4.5)*{_2};
\endxy}\Ea \circ_2
\left(
\Ba{c}{\xy
(-0.7,4)*{};
(-0,0)*{\circledcirc }
**\dir{-};
(0.7,4)*{};
(0,0)*{\circledcirc }
**\dir{-};
(0,-4)*{};
(0,0)*{\circledcirc }
**\dir{-};
(-0,-5.6)*{_1};
\endxy}\Ea
 +
 \Ba{c}{\xy
(-0.7,-4)*{};
(-0,0)*{\circledcirc }
**\dir{-};
(0.7,-4)*{};
(0,0)*{\circledcirc }
**\dir{-};
(0,4)*{};
(0,0)*{\circledcirc }
**\dir{-};
(-0.8,-5.6)*{_1};
(0.8,-5.6)*{_1};
\endxy}\Ea
\right)
  = 0
$$
$$
\Ba{c}{\xy
(-3,-3)*{};
(0,0)*{\circledcirc }
**\dir{-};
(4,-4)*{\circledcirc };
(0,0)*{\circledcirc }
**\dir{-};
(4,-4)*{\circledcirc };
(7.5,-7.5)*{}
**\dir{-};
(4,-4)*{\circledcirc };
(1.5,-7.5)*{}
**\dir{-};
(0,4)*{};
(0,0)*{\circledcirc }
**\dir{-};
(-3.5,-4.5)*{_1};
(1.5,-9)*{_2};
(7.5,-9)*{_3};
\endxy}\Ea
+
\Ba{c}{\xy
(-3,-3)*{};
(0,0)*{\circledcirc }
**\dir{-};
(4,-4)*{\circledcirc };
(0,0)*{\circledcirc }
**\dir{-};
(4,-4)*{\circledcirc };
(7.5,-7.5)*{}
**\dir{-};
(4,-4)*{\circledcirc };
(1.5,-7.5)*{}
**\dir{-};
(0,4)*{};
(0,0)*{\circledcirc }
**\dir{-};
(-3.5,-4.5)*{_3};
(1.5,-9)*{_1};
(7.5,-9)*{_2};
\endxy}
\Ea
+
\Ba{c}{\xy
(-3,-3)*{};
(0,0)*{\circledcirc }
**\dir{-};
(4,-4)*{\circledcirc };
(0,0)*{\circledcirc }
**\dir{-};
(4,-4)*{\circledcirc };
(7.5,-7.5)*{}
**\dir{-};
(4,-4)*{\circledcirc };
(1.5,-7.5)*{}
**\dir{-};
(0,4)*{};
(0,0)*{\circledcirc }
**\dir{-};
(-3.5,-4.5)*{_2};
(1.5,-9)*{_3};
(7.5,-9)*{_1};
\endxy}
\Ea
=0.
$$
The last equation follows immediately from the second equation in (\ref{R for LieB}). The first two equations follow from the computations shown in the above Examples, and the relations for generators of $\LoB_{0,0}$.
\end{proof}

\subsubsection{\bf Remark} In fact there is a one-parameter family of morphisms $f^\diamond$ associated
to any morphism $f$ (in the above notations). Let $\hbar$ be a (formal) parameter of degree zero.
Then under the assumptions of the above Lemma there a morphism of operads
$$
g_\hbar^\diamond: \caL ie^\diamond \lon \f\cP[\hbar]
$$
given on generators of $\caL ie^\diamond$ by the following formulae
\Beq\label{5: formulae for F hbar morphism}
g_\hbar^\diamond\left(\begin{xy}
 <0mm,-0.55mm>*{};<0mm,-3mm>*{}**@{-},
 <0mm,0.5mm>*{};<0mm,3mm>*{}**@{-},
 <0mm,0mm>*{\bullet};<0mm,0mm>*{}**@{},
 \end{xy}\right):=  \Ba{c}{\xy
(-0.7,4)*{};
(-0,0)*{\circledcirc }
**\dir{-};
(0.7,4)*{};
(0,0)*{\circledcirc }
**\dir{-};
(0,-4)*{};
(0,0)*{\circledcirc }
**\dir{-};
(-0,-5.6)*{_1};
\endxy}\Ea
+
\hbar
\Ba{c}{\xy
(-0.7,-4)*{};
(-0,0)*{\circledcirc }
**\dir{-};
(0.7,-4)*{};
(0,0)*{\circledcirc }
**\dir{-};
(0,4)*{};
(0,0)*{\circledcirc }
**\dir{-};
(-0.8,-5.6)*{_1};
(0.8,-5.6)*{_1};
\endxy}\Ea
\ \ \ \ , \ \ \ \
f_\hbar^\diamond\left(\Ba{c}\begin{xy}
 <0mm,0.66mm>*{};<0mm,3mm>*{}**@{-},
 <0.39mm,-0.39mm>*{};<2.2mm,-2.2mm>*{}**@{-},
 <-0.35mm,-0.35mm>*{};<-2.2mm,-2.2mm>*{}**@{-},
 <0mm,0mm>*{\bu};<0mm,0mm>*{}**@{},
   <0.39mm,-0.39mm>*{};<2.9mm,-4mm>*{^2}**@{},
   <-0.35mm,-0.35mm>*{};<-2.8mm,-4mm>*{^1}**@{},
\end{xy}\Ea\right) = \Ba{c}{\xy
(-3,-3)*{};
(0,0)*{\circledcirc }
**\dir{-};
(3,-3)*{};
(0,0)*{\circledcirc }
**\dir{-};
(0,4)*{};
(0,0)*{\circledcirc }
**\dir{-};
(-3.5,-4.5)*{_1};
(3.5,-4.5)*{_2};
\endxy}\Ea
\Eeq

\subsubsection{\bf Remark}
Let us next adopt the above lemma/construction for
generic values of the integer parameters $c$ and $d$. 
Denote
$$
\f_{c,d}\cP:= \f(\cP\{c\})
$$
and notice that elements (cf.\ Example \S {\ref{5: subsubsect on Examples in OP operad}})  $\begin{xy}
 <0mm,0.66mm>*{};<0mm,3mm>*{}**@{-},
 <0.39mm,-0.39mm>*{};<2.2mm,-2.2mm>*{}**@{-},
 <-0.35mm,-0.35mm>*{};<-2.2mm,-2.2mm>*{}**@{-},
 <0mm,0mm>*{\circ};<0mm,0mm>*{}**@{},
   <0mm,0.66mm>*{};<0mm,3.4mm>*{^1}**@{},
   <0.39mm,-0.39mm>*{};<2.9mm,-4mm>*{^2}**@{},
   <-0.35mm,-0.35mm>*{};<-2.8mm,-4mm>*{^1}**@{},
\end{xy}\in \cP(1,2)$ and $\begin{xy}
 <0mm,-0.55mm>*{};<0mm,-2.5mm>*{}**@{-},
 <0.5mm,0.5mm>*{};<2.2mm,2.2mm>*{}**@{-},
 <-0.48mm,0.48mm>*{};<-2.2mm,2.2mm>*{}**@{-},
 <0mm,0mm>*{\circ};<0mm,0mm>*{}**@{},
 <0mm,-0.55mm>*{};<0mm,-3.8mm>*{_1}**@{},
 <0.5mm,0.5mm>*{};<2.7mm,2.8mm>*{^2}**@{},
 <-0.48mm,0.48mm>*{};<-2.7mm,2.8mm>*{^1}**@{},
 \end{xy}\in \cP(2,1)$ of homological degrees $1-d$ and, respectively, $1-c$ give us elements,
$$
\Ba{c}{\xy
(-3,-3)*{};
(0,0)*{\circ}
**\dir{-};
(3,-3)*{};
(0,0)*{\circ }
**\dir{-};
(0,4)*{};
(0,0)*{\circ }
**\dir{-};
(-3.5,-4.5)*{_1};
(3.5,-4.5)*{_2};
\endxy}\Ea \in \f_{c,d}\cP(2) \  , \ \ \ \
\Ba{c}{\xy
(-0.7,-4)*{};
(-0,0)*{\circ}
**\dir{-};
(0.7,-4)*{};
(0,0)*{\circ }
**\dir{-};
(0,4)*{};
(0,0)*{\circ}
**\dir{-};
(-0.8,-5.6)*{_1};
(0.8,-5.6)*{_1};
\endxy}\Ea
\ , \
\Ba{c}{\xy
(-0.7,4)*{};
(-0,0)*{\circ}
**\dir{-};
(0.7,4)*{};
(0,0)*{\circ}
**\dir{-};
(0,-4)*{};
(0,0)*{\circ }
**\dir{-};
(-0,-5.6)*{_1};
\endxy}\Ea
\in \f_{c,d}\cP(1)
$$
of homological degrees $1-c-d$, $1-c-d$ and $1$ respectively.
%
%
%

\subsection{Twisting in the Lie bialgebra case}\label{5: Lie case for twisting} Any morphism of dg properads
$$
g: \LBcd \lon (\cP, \delta),
$$
implies by Lemma {\ref{5: Lemma on Lie_0^dimanod to P}} an associated morphism of dg operads
$$
\Ba{rccc}
g^{\diamond}: & \caL ie^\diamond_{c+d} & \lon & (\f_{c,d}\cP,\delta) \\
\Ea
$$
given explicitly by
$$
g^{\diamond}\left( \Ba{c}\begin{xy}
 <0mm,0.66mm>*{};<0mm,3mm>*{}**@{-},
 <0.39mm,-0.39mm>*{};<2.2mm,-2.2mm>*{}**@{-},
 <-0.35mm,-0.35mm>*{};<-2.2mm,-2.2mm>*{}**@{-},
 <0mm,0mm>*{\bu};<0mm,0mm>*{}**@{},
   <0.39mm,-0.39mm>*{};<2.9mm,-4mm>*{^2}**@{},
   <-0.35mm,-0.35mm>*{};<-2.8mm,-4mm>*{^1}**@{},
\end{xy}\Ea\right):=f\left( \Ba{c}\begin{xy}
 <0mm,0.66mm>*{};<0mm,3mm>*{}**@{-},
 <0.39mm,-0.39mm>*{};<2.2mm,-2.2mm>*{}**@{-},
 <-0.35mm,-0.35mm>*{};<-2.2mm,-2.2mm>*{}**@{-},
 <0mm,0mm>*{\circ};<0mm,0mm>*{}**@{},
   <0.39mm,-0.39mm>*{};<2.9mm,-4mm>*{^2}**@{},
   <-0.35mm,-0.35mm>*{};<-2.8mm,-4mm>*{^1}**@{},
\end{xy}\Ea\right)=:
\Ba{c}{\xy
(-3,-3)*{};
(0,0)*{\circledcirc }
**\dir{-};
(3,-3)*{};
(0,0)*{\circledcirc }
**\dir{-};
(0,4)*{};
(0,0)*{\circledcirc }
**\dir{-};
(-3.5,-4.5)*{_1};
(3.5,-4.5)*{_2};
\endxy}\Ea
\ \ \ \  \mbox{and}\ \ \ \
g^{\diamond}\left(\begin{xy}
 <0mm,-0.55mm>*{};<0mm,-3mm>*{}**@{-},
 <0mm,0.5mm>*{};<0mm,3mm>*{}**@{-},
 <0mm,0mm>*{\bullet};<0mm,0mm>*{}**@{},
 \end{xy}\right)
 :=
 f\left(\Ba{c}
 \begin{xy}
 <0mm,-0.55mm>*{};<0mm,-2.5mm>*{}**@{-},
 <0.5mm,0.5mm>*{};<2.2mm,2.2mm>*{}**@{-},
 <-0.48mm,0.48mm>*{};<-2.2mm,2.2mm>*{}**@{-},
 <0mm,0mm>*{\circ};<0mm,0mm>*{}**@{},
 \end{xy}\Ea
 \right)
 =:
  \Ba{c}{\xy
(-0.7,4)*{};
(-0,0)*{\circledcirc }
**\dir{-};
(0.7,4)*{};
(0,0)*{\circledcirc }
**\dir{-};
(0,-4)*{};
(0,0)*{\circledcirc }
**\dir{-};
(-0,-5.6)*{_1};
\endxy}\Ea
$$
The image $g^{\diamond}\left(\begin{xy}
 <0mm,-0.55mm>*{};<0mm,-3mm>*{}**@{-},
 <0mm,0.5mm>*{};<0mm,3mm>*{}**@{-},
 <0mm,0mm>*{\bullet};<0mm,0mm>*{}**@{},
 \end{xy}\right)$ defines a differential $\dco$ in $\f_{c,d}\cP$ through the standard action of $\f_{c,d}\cP(1)$ on the operad  $\f_{c,d}\cP$ by derivations;
 moreover, the sum $\delta+\dco$ is a differential in $\f_{c,d}\cP$
 such that the restriction of $g^\diamond$ to the suboperad $\caL ie_{c+d}$ gives us
 a morphism of {\em dg}\, operads,
\Beq \label{equ:hoLie to OP}
 g': \caL ie_{c+d} \lon (\f_{c,d}\cP,\delta + \dco)
\Eeq

Therefore we obtain two deformation complexes
$$
\PGCcd:=\Def \left( \LBcd \stackrel{g}{\lon} \cP \right)
$$
and
$$
\mathsf{fs}\cP\mathsf{GC}_{c,d}:= \Def\left(\caL ie_{c+d}  \stackrel{g'}{\lon} \f_{c,d}\cP\right)
$$
The latter complex is called the {\em full stable $\cP$-graph complex}. Note that here is an abuse of notation --- both complexes are uniquely determined by a morphism $g:\LB_{c,d}\rar \cP$, not just by the properad $\cP$  --- while the symbol $g$ is omitted from the notation. The associated to the morphism $g^\diamond$ twisted operad is denoted by $f\cP \cG raphs_{c,d}$ ($f$ for {\em full}). Generators of $f\cP \cG raphs_{c,d}$ are given by hypergraphs
$$
\Ba{c}\resizebox{19mm}{!}{  \xy
(-2,0.8)*{^p};
(11,0.8)*{^q};
(-1.5,5)*{}="1",
(1.5,5)*{}="2",
(9,5)*{}="3",
 (0,0)*{\circ}="A";
  (9,0)*{\circ}="O";
    (18,-4)*{\bu}="L";
 (-6,-10)*+{_1}*\frm{o}="B";
  (6,-10)*{\bu}="C";
   (14,-10)*+{_2}*\frm{o}="D";
 "A"; "B" **\crv{(-5,-0)}; 
  "A"; "D" **\crv{(5,-0.5)};
   "A"; "B" **\crv{(5,-1)};
  "A"; "C" **\crv{(-5,-7)};
  \ar @{-} "A";"1" <0pt>
\ar @{-} "A";"2" <0pt>
\ar @{-} "O";"C" <0pt>
\ar @{-} "O";"D" <0pt>
\ar @{-} "O";"L" <0pt>
\ar @{-} "D";"L" <0pt>
\ar @{-} "O";"3" <0pt>
 \endxy}
 \Ea \in f\cP\cG raphs_{c,d}(2)
$$
with hyperedges decorated by elements of $\cP$ and with vertices of two types, external and internal (cf.
Example {\ref{5: Graphs_d+1 example}}), while elements of $\mathsf{fs}\cP\mathsf{RG}_{c,d}$ are given by similar graphs
but with no external vertices, e.g.
$$
\Ba{c}\resizebox{18mm}{!}{  \xy
(-2,0.8)*{^p};
(11,0.8)*{^q};
(-1.5,5)*{}="1",
(1.5,5)*{}="2",
(9,5)*{}="3",
 (0,0)*{\circ}="A";
  (9,0)*{\circ}="O";
 (-6,-10)*{\bu}="B";
  (6,-10)*{\bu}="C";
   (14,-10)*{\bu}="D";
 "A"; "B" **\crv{(-5,-0)}; 
  "A"; "D" **\crv{(5,-0.5)};
   "A"; "B" **\crv{(5,-1)};
  "A"; "C" **\crv{(-5,-7)};
  \ar @{-} "A";"1" <0pt>
\ar @{-} "A";"2" <0pt>
\ar @{-} "O";"C" <0pt>
\ar @{-} "O";"D" <0pt>
\ar @{-} "O";"3" <0pt>
 \endxy}
 \Ea \ \ \ , \ \ \
\Ba{c}\resizebox{19mm}{!}{ \xy
(-2,0.8)*{^{p'}};
(11,0.8)*{^q};
(-1.5,5)*{}="1",
(1.5,5)*{}="2",
(9,5)*{}="3",
 (0,0)*{\circ}="A";
  (9,0)*{\circ}="O";
 (-6,-10)*{\bu}="B";
  (0,-10)*{\bu}="B'";
  (6,-10)*{\bu}="C";
   (14,-10)*{\bu}="D";
 "A"; "B" **\crv{(-5,-0)}; 
  "A"; "B'" **\crv{(-5,-0.5)};
   "A"; "B" **\crv{(5,-1)};
  %
  \ar @{-} "A";"1" <0pt>
\ar @{-} "A";"2" <0pt>
\ar @{-} "O";"C" <0pt>
\ar @{-} "O";"D" <0pt>
\ar @{-} "O";"3" <0pt>
 \endxy}
 \Ea
 \in \mathsf{fs}\cP\mathsf{GC}_{c,d}.
$$
Denote by $\sPGCcd$ the dg Lie subalgebra of  $\mathsf{fs}\cP\mathsf{GC}_{c,d}$ spanned by {\em connected}\, graphs (as the first graph in the picture just above) and call the {\em stable $\cP$-graph complex}. Similarly, we denote by $\cP\cG raphs_{c,d}$ the suboperad of $f\cP\cG raphs_{c,d}$ spanned by connected graphs.

\sip

Note that the map $g'$ above induces a non-trivial morphism of dg operads
\Beq\label{5: e_d to OP_d}
g^{ind}: \cP ois_{c+d} \lon (\f_{c,d}\cP, \delta+ \dco)
\Eeq
and hence, via the inclusion $\f_{c,d}\cP\hook \cP \cG raphs_{c,d}$, a morphism
\Beq\label{5: e_d to Pgraphs_d}
g^{ind}: \cP ois_{c+d} \lon  \cP \cG raphs_{c,d}.
\Eeq

\subsubsection{\bf Proposition}\label{5: Prop on PGCd to sPGCd}
 {\em There  is a canonical monomorphism of dg Lie algebras
\Beq\label{5: i from PGDC to sPGC cd}
i: \PGCcd\lon \sPGCcd
\Eeq
which identifies $\PGCcd$  with the subspace of $\sPGCcd$ spanned by graphs with univalent black vertices.}

\begin{proof} Let us first show that the map $i$ makes sense.
We have
$$
\PGCcd=\prod_{m,n} sgn_m^{|c|} \ot_{\bS_m^{op}} \ot \bar{\cP}(m,n)\ot_{\bS_n} \sgn_n^{|d|}[c(1-m) + d(1-n)].
$$
while
$$
\mathsf{fs}\cP\mathsf{GC}_{c,d}:= \prod_n \f_{c,d}\cP(n) \ot \sgn_n^{|c|+|d|}
[(c+d)(1-n)]
$$
The subspace in $\mathsf{fs}\cP\mathsf{GC}_{c,d}$ spanned by connected graphs with univalent black vertices
is given by
$$
\mathsf{S}:= \prod_n \cS(n) \ot \sgn_n^{|c|+|d|}[(c+d)(1-n)]
$$
where
$$
\cS(n)=\prod_{m}\f(\cP\{c\})^m_{[n]}
$$
is the direct summand in $\f_{c,d}\cP(n)$
associated with the maximal non-empty partition $[n]=\underbrace{1+\ldots + 1}_n$.
As
$$
\f(\cP\{c\})^m_{[n]}=\bar{\cP}\{d\}(m,n)= \sgn_m^{|c|}\ot \bar{\cP}(m,n)\ot \sgn_n^{|c|}[c(n-m)],
$$
we conclude
%
\Beqrn
\mathsf{S}&=& \prod_n  \cS(n) \ot \sgn_n^{|c|+|d|}[(c+d)(1-n)]\\
 & = &\prod_{m,n} \sgn_m^{|c|}\ot \bar{\cP}(m,n) \ot  \sgn_n^{|d|}[c(n-m)+(c+d)(1-n)]\\
 &=& \PGCd
\Eeqrn
as required. It remains to show compatibility of the monomorphism $i$ with differentials which is almost obvious and is left to the reader.
%
%
%
\end{proof}

\subsubsection{\bf Example: Kontsevich graph complexes}\label{5: example Konts graph complexes} Consider  again the properad $\cT_d$ introduced in \S {\ref{5: Example def of prop T}}. It comes equipped with a canonical morphism
\begin{equation}\label{equ:LieB to Td}
 \LieB_{c,d} \to \op T_d
\end{equation}
which sends the Lie bracket generator to the unique generator of $\cT_d$, and the Lie cobracket generator to zero.We have
\begin{align*}
\f_{c,d}(\cT_d) &\cong  \cG ra_{c+d}\\
 \RGCx{\op T_d} &\cong \K
 \\
 \SRGCx{\op T_d} &\cong \fcGC_{c+d} \\
 \RGraphsx{\op T_d} &\cong \Graphs_{c+d}.
\end{align*}

This makes the Kontsevich graph complexes/graph operads special cases of our general construction.

\subsubsection{\bf Example: complexes of marked  graphs}\label{5: Example def of prop T'}
As a second toy example (which will be important later on) let us consider another properad $\op T_{c,d}$ defined such that
\[
 \op T_{c,d}(m,n)
 =
 \begin{cases}
  \K[d-1] & \text{if $m=1$ and $n=2$} \\
  \K[c-1] & \text{if $m=2$ and $n=1$} \\
  \K & \text{if $m=n=1$} \\
  0 & \text{otherwise}
 \end{cases}.
\]
Again we define a generator of the arity $(1,1)$ subspace to be the unit, and define all properadic compositions except those with the unit to be trivial. The properad is clearly augmented.
There is an evident map
$\LieB_{c,d}\to\op T_{c,d}$ given by sending the two generators to the two elements spanning $\op T$.
Again, let us consider the associated graph complexes and twisted dg operad,
\begin{align*}
 \RGCx{\op T_{c,d}} &\cong \K \oplus \K
 \\
 \SRGCx{\op T_{c,d}} &=: \fcGC_{c+d}^\marked \\
 \RGraphsx{\op T_{c,d}} &=: \Graphs_{c+d}^\marked.
\end{align*}
These graph graph complexes (resp.\ dg operad) may be identified with complexes (resp.\ dg operad)  of ordinary graphs, with markings at some of the vertices.
Concretely, the marking signifies the presence of a hyperedge in $\op T'_{c,d}(2,1)$ at the vertex
(there can be at most one such hyperedge at any vertex, otherwise the graph vanishes). The differential in marked complexes/operads acts by splitting the vertices and redistributing edges as in the unmarked case with the only difference that when a marked vertex splits into two, one takes a sum over decorating each new vertex with the mark.
There is an evident commutative diagram of maps of properads,
$$
\xymatrix{
  \LieB_{c,d} \ar[r] \ar[d]_{\equiv} &   \cT_{c,d} \ar[d]\\
 \LieB_{c,d}  \ar[r] & \cT_d
 }
$$
 which induces the following maps of  graph complexes (resp.\ dg operads)
\begin{align*}
 \fcGC_{c+d}^\marked \to \fcGC_{c+d} \\
 \Graphs_{c+d}^\marked \to \Graphs_{c+d}.
\end{align*}

We claim that, homologically, the markings are irrelevant.

\subsubsection{\bf Proposition}{\em
 The map $\Graphs_{c+d}^\marked \to \Graphs_{c+d}$ is a quasi-isomorphism of operads.
 The map
 \[
 \fcGC_{c+d}^\marked \to \fcGC_{c+d} \oplus \K
 \]
 is a quasi-isomorphism of complexes, where the $\K$ on the right stands for the graph with a single vertex which is marked.}

\begin{proof}
 The proof is essentially the same as that of \cite[Proposition 3.4]{Wi}, where it is shown the tadpoles (edges connecting a vertex to itself) can be omitted from the graph complex. Here, one has markings instead of tadpoles, however, the argument and final conclusion stay the same.
\end{proof}

\subsection{Maps from the Poisson operad}\label{sec:}
\newcommand{\thoe}{\widetilde{\hoe}}
We denote by $\hoe_n$ the minimal resolution of the $n$-Poisson operad.
An algebra over $\hoe_n$ is in particular a $\hoLie_n$-algebra, and a $\hoCom$-algebra, with additional (somewhat complicated) homotopies ensuring compatibility between these structures.
In many cases however such a structure simplifies as follows: the $\hoCom$-structure is an honest commutative algebra structure, the $\hoLie_n$-operations are derivations in each slot with respect to the commutative product and all other $\hoe_n$-operations vanish.
Let us denote the quotient of $\hoe_n$ governing such restricted $\hoe_n$-algebras by $\thoe_n$.
Concretely, $\thoe_n$ is generated by $\hoLie_n$ and one binary commutative product generator, such that the $\hoLie_n$ operations are multi-derivations with respect to the commutative product, cf. also the introductory section of \cite{Wi2}.

\sip

Now let $\hoLieB_{c,d}\to \cP$ be a properad map and recall the map \eqref{equ:hoLie to OP},
\[
 \hoLie_{c+d}\to \f_{c,d}\cP.
\]
This map can be easily extended to a map
\[
 \thoe_{c+d}\to \f_{c,d}\cP.
\]
by declaring the commutative product generator to be mapped to the graph with two vertices and no hyperedge.
The multi-derivation relation is satisfied in this case since the $\hoLie_{c+d}$-generators are mapped to graphs with all vertices univalent.
Hence by composition with the projection $\hoe_{c+d}\to \thoe_{c+d}$ we obtain an operad map
\[
  \hoe_{c+d}\to \thoe_{c+d}\to \f_{c,d}\cP.
\]
By functoriality of operadic twisting we obtain a map
\[
  \Tw \hoe_{c+d}\to \Tw \f_{c,d}\cP.
\]
Composing with the canonical map $\hoe_{c+d} \to\Tw \hoe_{c+d}$ (cf. \cite{DW}) we obtain
\[
 \hoe_{c+d}\to \Tw \f_{c,d}\cP.
\]
Unpacking the definition of the twisting functor, one finds that this map again factors through $\thoe_{c+d}$, with the commutative product generator mapped to the graph with two vertices and no edge, while the $k$-ary $\hoLie_{c+d}$-generator gets mapped to a series of graphs with one hyperedge, $k$ univalent external and $l\geq 0$ univalent internal vertices.
Overall, we find the following result.

\begin{proposition}\label{prop:hoen to RGraphs}
 There is a map of operads
 \[
  \hoe_{c+d} \to \thoe_{c+d} \to \RGraphsPn{c,d}\subset \Tw \f_{c,d} \cP
 \]
 extending the canonical map from $\hoLie_{c+d}$, and such that the commutative product generator is sent to the graph with two vertices and no edge.
\end{proposition}

\newcommand{\choLieB}{\widehat{\hoLieB}} 

\subsection{Example}\label{sec:examplePLiecd}
Let us study the constructions of the previous section in a more instructive example. We take for $\cP$ the genus completion $\choLieB_{c,d}$ of $\hoLieB_{c,d}$ with the natural inclusion
\[
 \hoLieB_{c,d} \to \cP.
\]
In this case the resulting $\cP$-graph complexes and $\cP$-graph operads have mostly been studied in the literature.
First, in this case
\[
 \RGCPn{c,d} =\Def(\hoLieB_{c,d} \to \choLieB_{c,d})
\]
is the completed deformation complex of the properad $\hoLieB_{c,d}$. This complex has been studied in \cite{CMW,MW2}.
Elements can be seen as series of {\em oriented}\, graphs (that is, directed graphs without {\em closed}\, paths of directed edges)  with inputs and outputs as the following picture shall illustrate:
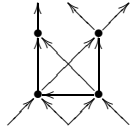
\[
 \resizebox{17mm}{!}{ \xy
(0,0)*{\bu}="d1",
(10,0)*{\bu}="d2",
(-5,-5)*{}="dl",
(5,-5)*{}="dc",
(15,-5)*{}="dr",
(0,10)*{\bu}="u1",
(10,10)*{\bu}="u2",
(5,15)*{}="uc",
(15,15)*{}="ur",
(0,15)*{}="ul",
\ar @{<-} "d1";"d2" <0pt>
\ar @{<-} "d1";"dl" <0pt>
\ar @{<-} "d1";"dc" <0pt>
\ar @{<-} "d2";"dc" <0pt>
\ar @{<-} "d2";"dr" <0pt>
\ar @{<-} "u1";"d1" <0pt>
\ar @{<-} "u1";"d2" <0pt>
\ar @{<-} "u2";"d2" <0pt>
\ar @{<-} "u2";"d1" <0pt>
\ar @{<-} "uc";"u2" <0pt>
\ar @{<-} "ur";"u2" <0pt>
\ar @{<-} "ul";"u1" <0pt>
\endxy}
\]
As remarked in \cite{CMW} and proven in \cite{MW2} there is a map
\begin{equation}\label{equ:GCor to DefLieB}
 \GCor_{c+d+1} \to \RGCPn{c,d}[1],
\end{equation}
given graphically by attaching an arbitrary number of inputs and outputs to a graph in $\GCor_{c+d}$
\[
\Ga \lon  \sum_{m,n\geq 1}
\overbrace{
  \underbrace{
 \Ba{c}\resizebox{9mm}{!}  {\xy
(0,4.5)*+{...},
(0,-4.5)*+{...},
(0,0)*+{\Ga}="o",
(-5,5)*{}="1",
(-3,5)*{}="2",
(3,5)*{}="3",
(5,5)*{}="4",
(-3,-5)*{}="5",
(3,-5)*{}="6",
(5,-5)*{}="7",
(-5,-5)*{}="8",
\ar @{->} "o";"1" <0pt>
\ar @{->} "o";"2" <0pt>
\ar @{->} "o";"3" <0pt>
\ar @{->} "o";"4" <0pt>
\ar @{<-} "o";"5" <0pt>
\ar @{<-} "o";"6" <0pt>
\ar @{<-} "o";"7" <0pt>
\ar @{<-} "o";"8" <0pt>
\endxy}\Ea
}_{n\times}
 }^{m\times}.
\]
and retaining only graphs in which each vertex has at least one input and one output.
As shown in \cite[Theorem 1.2.1]{MW2} the map \eqref{equ:GCor to DefLieB} is a quasi-isomorphism up to a one-dimensional subspace of the cohomology of $\RGCPn{c,d}$ spanned by the following series of graphs
 \begin{equation}\label{equ:extraclass}
  \sum_{m,n}(m+n-2)
  \overbrace{
  \underbrace{
 \Ba{c}\resizebox{9mm}{!}  {\xy
(0,4.5)*+{...},
(0,-4.5)*+{...},
(0,0)*{\bu}="o",
(-5,5)*{}="1",
(-3,5)*{}="2",
(3,5)*{}="3",
(5,5)*{}="4",
(-3,-5)*{}="5",
(3,-5)*{}="6",
(5,-5)*{}="7",
(-5,-5)*{}="8",
\ar @{->} "o";"1" <0pt>
\ar @{->} "o";"2" <0pt>
\ar @{->} "o";"3" <0pt>
\ar @{->} "o";"4" <0pt>
\ar @{<-} "o";"5" <0pt>
\ar @{<-} "o";"6" <0pt>
\ar @{<-} "o";"7" <0pt>
\ar @{<-} "o";"8" <0pt>
\endxy}\Ea
}_{n\times}
 }^{m\times}.
\end{equation}
As shown in \cite{Wi2} (and re-proven in this paper below) we furthermore have $H(\GCor_{c+d+1})\cong H(\GC_{c+d})$, so that in this example the unstable $\cP$-graph cohomology can be understood in terms of the ordinary graph cohomology.

\sip

Next consider the polydifferential operad $\f_{c,d}\cP$ for  $\cP=\hoLieB_{c,d}$. For the present example, elements of $\f\cP_{c,d}(r)$ are series of oriented graphs whose inputs are ``attached'' to $r$ external vertices labelled $1,\dots,r$, as in the following example (where all the edges are directed from the bottom to the top as usually),
\Beq\label{5: graphs from OP for P=Holieb}
\Ba{c}\resizebox{16mm}{!}{  \xy
(-1.5,10)*{}="1",
(1.5,10)*{}="2",
(11,18)*{}="3",
 (0,5)*{\bu}="A";
  (11,13)*{\bu}="O";
    (18,8)*{\bu}="L";
 (-5,-7)*+{_1}*\frm{o}="B";
   (17,-7)*+{_2}*\frm{o}="D";
 "A"; "B" **\crv{(-5,-0)};
  "A"; "D" **\crv{(5,-0.5)};
  \ar @{-} "A";"1" <0pt>
\ar @{-} "A";"2" <0pt>
\ar @{-} "O";"D" <0pt>
\ar @{-} "O";"L" <0pt>
\ar @{-} "D";"L" <0pt>
\ar @{-} "O";"3" <0pt>
\ar @{-} "B";"L" <0pt>
\ar @{-} "B";"O" <0pt>
 \endxy}
 \Ea
\Eeq

The differential of $\f\cP_{c,d}$ has two terms $\delta + d$, with $\delta$ stemming from the differential on $\cP$ and $d$ coming from the map $\hoLieB_{c,d} \to \cP$.
 Combinatorially, the part $\delta$ splits non-external (denoted by black bullets in our pictures) vertices in the oriented graphs
\[
\delta:
\Ba{c}\resizebox{11mm}{!}{\begin{xy}
 <0mm,0mm>*{\bu};<0mm,0mm>*{}**@{},
 <-0.6mm,0.44mm>*{};<-8mm,5mm>*{}**@{-},
 <-0.4mm,0.7mm>*{};<-4.5mm,5mm>*{}**@{-},
 <0mm,5mm>*{\ldots},
 <0.4mm,0.7mm>*{};<4.5mm,5mm>*{}**@{-},
 <0.6mm,0.44mm>*{};<8mm,5mm>*{}**@{-},
     <0mm,7mm>*{\overbrace{\ \ \ \ \ \ \ \ \ \ \ \ \ \ \ \ }},
     <0mm,9mm>*{^m},
 <-0.6mm,-0.44mm>*{};<-8mm,-5mm>*{}**@{-},
 <-0.4mm,-0.7mm>*{};<-4.5mm,-5mm>*{}**@{-},
 <0mm,0mm>*{};<-1mm,-5mm>*{\ldots}**@{},
 <0.4mm,-0.7mm>*{};<4.5mm,-5mm>*{}**@{-},
 <0.6mm,-0.44mm>*{};<8mm,-5mm>*{}**@{-},
   <0mm,-7mm>*{\underbrace{\ \ \ \ \ \ \ \ \ \ \ \ \ \ \ \ }},
     <0mm,-9mm>*{_n},
 \end{xy}}\Ea
\ \ \lon  \ \
 \sum_{[m]=I_1\sqcup I_2\atop
 [n]=J_1\sqcup J_2
}\hspace{0mm}
\pm
\Ba{c}\resizebox{20mm}{!}{ \begin{xy}
 <0mm,0mm>*{\bu},
 <-0.6mm,0.44mm>*{};<-8mm,5mm>*{}**@{-},
 <-0.4mm,0.7mm>*{};<-4.5mm,5mm>*{}**@{-},
 <0mm,0mm>*{};<0mm,5mm>*{\ldots}**@{},
 <0.4mm,0.7mm>*{};<4.5mm,5mm>*{}**@{-},
 <0.6mm,0.44mm>*{};<12.4mm,4.8mm>*{}**@{-},
     <0mm,0mm>*{};<-2mm,7mm>*{\overbrace{\ \ \ \ \ \ \ \ \ \ \ \ }}**@{},
     <0mm,0mm>*{};<-2mm,9mm>*{^{I_1}}**@{},
 <-0.6mm,-0.44mm>*{};<-8mm,-5mm>*{}**@{-},
 <-0.4mm,-0.7mm>*{};<-4.5mm,-5mm>*{}**@{-},
 <0mm,0mm>*{};<-1mm,-5mm>*{\ldots}**@{},
 <0.4mm,-0.7mm>*{};<4.5mm,-5mm>*{}**@{-},
 <0.6mm,-0.44mm>*{};<8mm,-5mm>*{}**@{-},
      <0mm,0mm>*{};<0mm,-7mm>*{\underbrace{\ \ \ \ \ \ \ \ \ \ \ \ \ \ \
      }}**@{},
      <0mm,0mm>*{};<0mm,-10.6mm>*{_{J_1}}**@{},
 <13mm,5mm>*{\bu},
 <12.6mm,5.44mm>*{};<5mm,10mm>*{}**@{-},
 <12.6mm,5.7mm>*{};<8.5mm,10mm>*{}**@{-},
 <13mm,5mm>*{};<13mm,10mm>*{\ldots}**@{},
 <13.4mm,5.7mm>*{};<16.5mm,10mm>*{}**@{-},
 <13.6mm,5.44mm>*{};<20mm,10mm>*{}**@{-},
      <13mm,5mm>*{};<13mm,12mm>*{\overbrace{\ \ \ \ \ \ \ \ \ \ \ \ \ \ }}**@{},
      <13mm,5mm>*{};<13mm,14mm>*{^{I_2}}**@{},
 <12.4mm,4.3mm>*{};<8mm,0mm>*{}**@{-},
 <12.6mm,4.3mm>*{};<12mm,0mm>*{\ldots}**@{},
 <13.4mm,4.5mm>*{};<16.5mm,0mm>*{}**@{-},
 <13.6mm,4.8mm>*{};<20mm,0mm>*{}**@{-},
     <13mm,5mm>*{};<14.3mm,-2mm>*{\underbrace{\ \ \ \ \ \ \ \ \ \ \ }}**@{},
     <13mm,5mm>*{};<14.3mm,-4.5mm>*{_{J_2}}**@{},
 \end{xy}}\Ea
\]
as explained in (\ref{LBk_infty}).
The part $d$ acts by splitting from external vertices and by attaching corollas to outputs
\[
d:
\Ba{c}\resizebox{16mm}{!}{  \xy
(-1.5,10)*{}="1",
(1.5,10)*{}="2",
(11,18)*{}="3",
 (0,5)*{\bu}="A";
  (11,13)*{\bu}="O";
    (18,8)*{\bu}="L";
 (-5,-7)*+{_1}*\frm{o}="B";
   (17,-7)*+{_2}*\frm{o}="D";
 "A"; "B" **\crv{(-5,-0)};
  "A"; "D" **\crv{(5,-0.5)};
  %
  \ar @{-} "A";"1" <0pt>
\ar @{-} "A";"2" <0pt>
\ar @{-} "O";"D" <0pt>
\ar @{-} "O";"L" <0pt>
\ar @{-} "D";"L" <0pt>
\ar @{-} "O";"3" <0pt>
\ar @{-} "B";"L" <0pt>
\ar @{-} "B";"O" <0pt>
 \endxy}
 \Ea
\ \lon \
\sum
\Ba{c}\resizebox{16mm}{!}{  \xy
(-1.5,13)*{}="1",
(1.5,13)*{}="2",
(11,18)*{}="3",
(-9,4)*{}="4",
(-5,4)*{}="5",
<-4.5mm,5.5mm>*{\overbrace{\ \ \ \ \ \ \ \ \  }},
     <-4.5mm,6.9mm>*{^m},
 (0,8)*{\bu}="A";
  (11,13)*{\bu}="O";
    (18,8)*{\bu}="L";
 (-5,-7)*+{_1}*\frm{o}="B";
   (17,-7)*+{_2}*\frm{o}="D";
(-5,0)*{\bu}="E";
 "A"; "E" **\crv{(-5,-0)};
  "A"; "D" **\crv{(5,-0.5)};
  \ar @{-} "A";"1" <0pt>
\ar @{-} "A";"2" <0pt>
\ar @{-} "O";"D" <0pt>
\ar @{-} "O";"L" <0pt>
\ar @{-} "D";"L" <0pt>
\ar @{-} "O";"3" <0pt>
\ar @{-} "B";"L" <0pt>
\ar @{-} "E";"O" <0pt>
\ar @{-} "B";"E" <0pt>
\ar @{-} "E";"4" <0pt>
\ar @{-} "E";"5" <0pt>
 \endxy}
 \Ea
 \ \ +
 \ \
 \sum
 \Ba{c}\resizebox{16mm}{!}{  \xy
(-1.5,10)*{\bu}="1",
(2.5,10)*{}="2",
(11,18)*{}="3",
(-4.5,14)*{}="1'",
(-1.5,14)*{}="2'",
(1.5,14)*{}="3'",
<-1.5mm,15.9mm>*{\overbrace{\ \ \ \ \ \ }},
     <-1.5mm,17.3mm>*{^m},
 (0,5)*{\bu}="A";
  (11,13)*{\bu}="O";
    (18,8)*{\bu}="L";
 (-5,-7)*+{_1}*\frm{o}="B";
   (17,-7)*+{_2}*\frm{o}="D";
 "A"; "B" **\crv{(-5,-0)};
  "A"; "D" **\crv{(5,-0.5)};
  \ar @{-} "A";"1" <0pt>
\ar @{-} "A";"2" <0pt>
\ar @{-} "O";"D" <0pt>
\ar @{-} "O";"L" <0pt>
\ar @{-} "D";"L" <0pt>
\ar @{-} "O";"3" <0pt>
\ar @{-} "B";"L" <0pt>
\ar @{-} "B";"O" <0pt>
\ar @{-} "1'";"1" <0pt>
\ar @{-} "2'";"1" <0pt>
\ar @{-} "3'";"1" <0pt>
 \endxy}\Ea
\]
More precisely, in the r.h.s.\ of this formula
\Bi
\item the first summation symbol stands for the double summation: one summation runs over white vertices of the input graph, and the second summation corresponds to the substitution into, say, $k$-th white vertex  $\xy (0,0)*+{_k}*\frm{o};\endxy$  the sum of corollas
\Beq\label{5: sum of (1,m) corollas}
\sum_{m\geq 2} \Ba{c}\resizebox{11mm}{!}{\begin{xy}
 <0mm,0mm>*{\bu};<0mm,0mm>*{}**@{},
 <-0.6mm,0.44mm>*{};<-8mm,5mm>*{}**@{-},
 <-0.4mm,0.7mm>*{};<-4.5mm,5mm>*{}**@{-},
 <0mm,5mm>*{\ldots},
 <0.4mm,0.7mm>*{};<4.5mm,5mm>*{}**@{-},
 <0.6mm,0.44mm>*{};<8mm,5mm>*{}**@{-},
     <0mm,7mm>*{\overbrace{\ \ \ \ \ \ \ \ \ \ \ \ \ \ \ \ }},
     <0mm,9mm>*{^m},
 <0.0mm,-0.44mm>*{};<0mm,-5mm>*{}**@{-},
 \end{xy}}\Ea
\Eeq
and then summing over all attachments of the ``hanging edges" (i.e.\ the ones attached previously to
 $\xy (0,0)*+{_k}*\frm{o};\endxy$) to the output legs of (\ref{5: sum of (1,m) corollas})
 (so that at least one output leg is hit)
 and to $\xy (0,0)*+{_k}*\frm{o};\endxy$ itself; we assume that the input leg of (\ref{5: sum of (1,m) corollas}) always hits $\xy (0,0)*+{_k}*\frm{o};\endxy$ (see the picture for $k=1$);
\item the second summation symbol stands for the sum over all ways to attach
(\ref{5: sum of (1,m) corollas}) to the output legs of the input graph.
\Ei

\sip

The map
\[
\thoe_{c+d} \to \f_{c,d}\cP
\]
in this case sends the product generator $m$ and the $\hoLie_{c+d}$-generators $\mu_k$ ($k=2,3,\dots$) to the following series of graphs
\[
m\lon \Ba{c}\resizebox{7mm}{!}
{
\xy
(0,1)*+{_1}*\frm{o};
(5,1)*+{_2}*\frm{o};
\endxy}\Ea \ \ \ \ , \ \ \ \
\mu_k \lon
\sum_{m\geq 1}
\Ba{c}\resizebox{14mm}{!}{  \xy
(0,8)*{}="1";
(-2,8)*{}="2";
(2,8)*{}="3";
(2,-4)*{\ldots};
    (0,+3)*{\bu}="L";
 (-8,-5)*+{_1}*\frm{o}="B";
  (-3,-5)*+{_2}*\frm{o}="C";
   (8,-5)*+{_k}*\frm{o}="D";
    <0mm,10mm>*{\overbrace{ \ \ \ \ \ \ \ \ }},
     <0mm,12mm>*{_m},
\ar @{-} "D";"L" <0pt>
\ar @{-} "C";"L" <0pt>
\ar @{-} "B";"L" <0pt>
\ar @{-} "1";"L" <0pt>
\ar @{-} "2";"L" <0pt>
\ar @{-} "3";"L" <0pt>
 \endxy}
 \Ea
\]

\newcommand{\hoqBV}{\mathcal{H}\mathit{oq}\mathcal{BV}} 
\newcommand{\thoqBV}{\widetilde{\hoqBV}}
\newcommand{\qBV}{\mathit{q}\mathcal{BV}} 

More generally, the above map can be extended to a map $\thoqBV_{c+d}\to \f\cP$ (see Appendix \ref{sec:qBV} for the definition of the operad $\thoqBV_{c+d}$) by sending the generators $m$ and $\mu_k^r$ to the series
\begin{equation}\label{equ:hoBVmap}
m\lon \Ba{c}\resizebox{7mm}{!}
{
\xy
(0,1)*+{_1}*\frm{o};
(5,1)*+{_2}*\frm{o};
\endxy}\Ea \ \ \ \ , \ \ \ \
\mu_k^r \lon
\sum_{m\geq 1}
\frac{(m-1)^r}{r!}
\Ba{c}\resizebox{14mm}{!}{  \xy
(0,8)*{}="1";
(-2,8)*{}="2";
(2,8)*{}="3";
(2,-4)*{\ldots};
    (0,+3)*{\bu}="L";
 (-8,-5)*+{_1}*\frm{o}="B";
  (-3,-5)*+{_2}*\frm{o}="C";
   (8,-5)*+{_k}*\frm{o}="D";
    <0mm,10mm>*{\overbrace{ \ \ \ \ \ \ \ \ }},
     <0mm,12mm>*{_m},
\ar @{-} "D";"L" <0pt>
\ar @{-} "C";"L" <0pt>
\ar @{-} "B";"L" <0pt>
\ar @{-} "1";"L" <0pt>
\ar @{-} "2";"L" <0pt>
\ar @{-} "3";"L" <0pt>
 \endxy}
 \Ea
\end{equation}

\begin{lemma}
 The assignment \eqref{equ:hoBVmap} defines a map of operads $\thoqBV_{c+d}\to \f\cP$.
\end{lemma}
\begin{proof}
A straightforward graphical computation we leave to the reader.
\end{proof}

\newcommand{\WiOP}{\cG raphs_{c+d}^{or}}
A slight variant of the operad $\f_{c,d}\cP$ has already been considered in the literature.
In \cite{Wi2} the operad $\WiOP$ is introduced, whose elements are series of oriented graphs as in (\ref{5: examples of OP}), but without output legs, with a similarly defined differential (and with an opposite flow on the edges which is a matter of conventions only).
More concretely, there is a map of operads
\[
 \cG raphs^{or}_{c+d} \to \f_{c,d}\cP
\]
by sending a graph $\Gamma$ to the sum of graphs obtained by attaching output legs to internal vertices, retaining only graphs in which each vertex has at least one outgoing edge.
\[
\Ga \lon  \sum_{m\geq 1}
\overbrace{
 \Ba{c}\resizebox{9mm}{!}  {\xy
(0,4.5)*+{...},
(0,0)*+{\Ga}="o",
(-5,5)*{}="1",
(-3,5)*{}="2",
(3,5)*{}="3",
(5,5)*{}="4",
(-3,-5)*{}="5",
(3,-5)*{}="6",
(5,-5)*{}="7",
(-5,-5)*{}="8",
\ar @{-} "o";"1" <0pt>
\ar @{-} "o";"2" <0pt>
\ar @{-} "o";"3" <0pt>
\ar @{-} "o";"4" <0pt>
\endxy}\Ea
 }^{m\times},
\]
for example,
$$
\Ba{c}\resizebox{16mm}{!}{  \xy
(-1.5,10)*{}="1",
(1.5,10)*{}="2",
(11,18)*{}="3",
 (0,5)*{\bu}="A";
  (11,13)*{\bu}="O";
    (18,8)*{\bu}="L";
 (-5,-7)*+{_1}*\frm{o}="B";
   (17,-7)*+{_2}*\frm{o}="D";
 "A"; "B" **\crv{(-5,-0)};
  "A"; "D" **\crv{(5,-0.5)};
  %
\ar @{-} "O";"D" <0pt>
\ar @{-} "O";"L" <0pt>
\ar @{-} "D";"L" <0pt>
\ar @{-} "B";"L" <0pt>
\ar @{-} "B";"O" <0pt>
 \endxy}
 \Ea
 \
 \lon
 \
 \Ba{c}\resizebox{16mm}{!}{  \xy
(0,10)*{}="1",
(11,18)*{}="3",
 (0,5)*{\bu}="A";
  (11,13)*{\bu}="O";
    (18,8)*{\bu}="L";
 (-5,-7)*+{_1}*\frm{o}="B";
   (17,-7)*+{_2}*\frm{o}="D";
 "A"; "B" **\crv{(-5,-0)};
  "A"; "D" **\crv{(5,-0.5)};
  \ar @{-} "A";"1" <0pt>
\ar @{-} "O";"D" <0pt>
\ar @{-} "O";"L" <0pt>
\ar @{-} "D";"L" <0pt>
\ar @{-} "O";"3" <0pt>
\ar @{-} "B";"L" <0pt>
\ar @{-} "B";"O" <0pt>
 \endxy}
 \Ea
 \
 +
 \
 \Ba{c}\resizebox{16mm}{!}{  \xy
(-1.5,10)*{}="1",
(1.5,10)*{}="2",
(11,18)*{}="3",
 (0,5)*{\bu}="A";
  (11,13)*{\bu}="O";
    (18,8)*{\bu}="L";
 (-5,-7)*+{_1}*\frm{o}="B";
   (17,-7)*+{_2}*\frm{o}="D";
 "A"; "B" **\crv{(-5,-0)};
  "A"; "D" **\crv{(5,-0.5)};
  \ar @{-} "A";"1" <0pt>
\ar @{-} "A";"2" <0pt>
\ar @{-} "O";"D" <0pt>
\ar @{-} "O";"L" <0pt>
\ar @{-} "D";"L" <0pt>
\ar @{-} "O";"3" <0pt>
\ar @{-} "B";"L" <0pt>
\ar @{-} "B";"O" <0pt>
 \endxy}
 \Ea
 \
 +
 \ldots
$$

The map $\thoe_{c+d}\to \f_{c,d}\cP$ factors through the map $\thoe_n\to\WiOP$ from \cite[section 3.5]{Wi2}.

%
It is shown in \cite[Theorem 2]{Wi2} that the map $\thoe_n\to\WiOP$ is a quasi-isomorphism.
By a very similar argument (which we again leave to the reader) one may obtain the following result.

\begin{proposition}\label{prop:qBVOP}
 The map $\hoqBV_{c+d} \to \f_{c,d}\cP$ is a quasi-isomorphism of operads.
\end{proposition}

Proposition {\ref{prop:qBVOP}} in turn may be used to understand the stable $\cP$-graph complex
\[
 \SRGCPn{c,d}\subset \Def(\hoLie_{c+d}\to \f_{c,d}\cP)
\]
in this case.
More concretely, since $\f_{c,d}\cP\simeq \qBV_{c+d}$ we may apply Lemma {\ref{lem:DefLieqBV}} to find that
\begin{equation}\label{equ:SRGCPonedim}
 \Def(\hoLie_{c+d}\to \f\cP) \simeq \Def(\hoLie_{c+d}\to \qBV_{c+d}) \simeq \K.
\end{equation}
Hence $H(\SRGCPn{c,d})$ is one-dimensional, and the one class can be traced back to be represented by the series \eqref{equ:extraclass}, which may be seen as an element of $\SRGCPn{c,d}$ via the canonical map $\RGCPn{c,d}\to \SRGCPn{c,d}$.

\subsection{Twisting in the involutive Lie bialgebra case} Consider a morphism of properads
$g: \LoBcd\rar \cP$. Using Lemma {\ref{5: Lemma on Lie_0^dimanod to P}} we obtain an associated morphism of dg operads
\Beq\label{5: morphism f^diamond to O_dP}
\Ba{rccc}
g^{\diamond}: & \caL ie^\diamond_{c+d} & \lon & \f_{c,d}\cP[[\hbar]] \\
\Ea
\Eeq
given explicitly by
$$
g^{\diamond}\left( \Ba{c}\begin{xy}
 <0mm,0.66mm>*{};<0mm,3mm>*{}**@{-},
 <0.39mm,-0.39mm>*{};<2.2mm,-2.2mm>*{}**@{-},
 <-0.35mm,-0.35mm>*{};<-2.2mm,-2.2mm>*{}**@{-},
 <0mm,0mm>*{\bu};<0mm,0mm>*{}**@{},
   <0.39mm,-0.39mm>*{};<2.9mm,-4mm>*{^2}**@{},
   <-0.35mm,-0.35mm>*{};<-2.8mm,-4mm>*{^1}**@{},
\end{xy}\Ea\right):=f\left( \Ba{c}\begin{xy}
 <0mm,0.66mm>*{};<0mm,3mm>*{}**@{-},
 <0.39mm,-0.39mm>*{};<2.2mm,-2.2mm>*{}**@{-},
 <-0.35mm,-0.35mm>*{};<-2.2mm,-2.2mm>*{}**@{-},
 <0mm,0mm>*{\bu};<0mm,0mm>*{}**@{},
   <0.39mm,-0.39mm>*{};<2.9mm,-4mm>*{^2}**@{},
   <-0.35mm,-0.35mm>*{};<-2.8mm,-4mm>*{^1}**@{},
\end{xy}\Ea\right)=:
\Ba{c}{\xy
(-3,-3)*{};
(0,0)*{\circledcirc }
**\dir{-};
(3,-3)*{};
(0,0)*{\circledcirc }
**\dir{-};
(0,4)*{};
(0,0)*{\circledcirc }
**\dir{-};
(-3.5,-4.5)*{_1};
(3.5,-4.5)*{_2};
\endxy}\Ea
\ \ \ \  \mbox{and}\ \ \ \
f^{\diamond}\left(\begin{xy}
 <0mm,-0.55mm>*{};<0mm,-3mm>*{}**@{-},
 <0mm,0.5mm>*{};<0mm,3mm>*{}**@{-},
 <0mm,0mm>*{\bullet};<0mm,0mm>*{}**@{},
 \end{xy}\right)=\hbar
 \Ba{c}{\xy
(-0.7,-4)*{};
(-0,0)*{\circledcirc }
**\dir{-};
(0.7,-4)*{};
(0,0)*{\circledcirc }
**\dir{-};
(0,4)*{};
(0,0)*{\circledcirc }
**\dir{-};
(-0.8,-5.6)*{_1};
(0.8,-5.6)*{_1};
\endxy}\Ea
+
\Ba{c}{\xy
(-0.7,4)*{};
(-0,0)*{\circledcirc }
**\dir{-};
(0.7,4)*{};
(0,0)*{\circledcirc }
**\dir{-};
(0,-4)*{};
(0,0)*{\circledcirc }
**\dir{-};
(-0,-5.6)*{_1};
\endxy}\Ea
$$
where $\hbar$ is a formal parameter of degree $c+d$.
Hence we obtain a deformation complex
$$
\mathsf{fsPGC}_{c,d}^\diamond:= \Def\left(\caL ie^\diamond_{c+d}  \stackrel{g^\diamond}{\lon} \f\cP_{c,d}[[\hbar]]\right)
$$
and its {\em stable}\, subcomplex $\sPGCcd^\diamond$ spanned by connected graphs. The associated
to the map $g^\diamond$ twisted operad is denoted by $f\cP\cG raphs_{c,d}^\diamond$; its suboperad spanned by connected graphs is denoted by $\cP\cG raphs_{c,d}^\diamond$. The deformation complex (called the {\em  diamond $\cP$-graph complex})
$$
\PGCcd^\diamond:=\Def\left( \LoBcd\stackrel{f}{\lon} \cP\right)
$$
sits inside $\mathsf{fsPGC}_{c,d}^\diamond$ as a subcomplex spanned over $\K[[\hbar]]$ by graphs
with univalent black vertices, i.e.\ there is a canonical monomorphism
$$
i: \PGCcd^\diamond \lon \sPGCcd^\diamond
$$
of dg Lie algebras (cf.\ Proposition {\ref{5: Prop on PGCd to sPGCd}}).

\subsection{Stable $\cP$-graph complex and quantum diamond functions}
By Proposition {\ref{5: Proposition on OP and its representations}}, any representation $\rho: \cP\rar \cE nd_V$ of $\cP$ in a dg vector space $(V,\delta)$ induces an associated polydifferential representation $\rho: \f\cP \rar
\cE nd_{\odot^\bu V}$, or equivalently, a continuous representation
%
%
$$
\rho_{c,d}:  \f_{c,d}\cP[[\hbar]] \rar \cE nd_{\odot^\bu (V[-c])}[[\hbar]]
$$
where $\hbar$ is a formal variable of homological degree $c+d$.
Hence any morphism of properads $g: \LoBd\rar \cP$ induces a composition
$$
\rho^\diamond_{c,d}:  \caL ie_{c+d}^\diamond  \stackrel{g^\diamond}{\lon} \f_{c,d}\cP[[\hbar]]
 \stackrel{\rho_\hbar}{\lon} \cE nd_{\odot^\bu(V[-c])}[[\hbar]].
$$
We conclude that an involutive Lie bialgebra structure in $V$ makes the vector space $\sC^\bu_c(V)=\odot^\bu(V[-c])[[\hbar]]$ into a continuous dg $\caL ie_{c+d}$  algebra with the differential $\hbar\Delta_{[\ ,\ ]} + \delta_\vartriangle:=\rho_{c,d}(\hbar
 \Ba{c}{\xy
(-0.7,-4)*{};
(-0,0)*{\circledcirc }
**\dir{-};
(0.7,-4)*{};
(0,0)*{\circledcirc }
**\dir{-};
(0,4)*{};
(0,0)*{\circledcirc }
**\dir{-};
(-0.8,-5.6)*{_1};
(0.8,-5.6)*{_1};
\endxy}\Ea
+
\Ba{c}{\xy
(-0.7,4)*{};
(-0,0)*{\circledcirc }
**\dir{-};
(0.7,4)*{};
(0,0)*{\circledcirc }
**\dir{-};
(0,-4)*{};
(0,0)*{\circledcirc }
**\dir{-};
(-0,-5.6)*{_1};
\endxy}\Ea)$ and with Lie brackets $\{\ ,\ \}:= \rho_{c,d}(\Ba{c}{\xy
(-3,-3)*{};
(0,0)*{\circledcirc }
**\dir{-};
(3,-3)*{};
(0,0)*{\circledcirc }
**\dir{-};
(0,4)*{};
(0,0)*{\circledcirc }
**\dir{-};
(-3.5,-4.5)*{_1};
(3.5,-4.5)*{_2};
\endxy}\Ea)$ of degree $1-(c+d)$.
%
%
 This structure is precisely the one  we discussed in \S {\ref{3: Subsec on Hamiltons parametr of quaA_infty}} so that its $MC$ elements, that is degree $c+d$ elements $S\in \sC^\bu_d(V)$ satisfying equation (\ref{3: rescaled master eqn for inv Lie bial}) are precisely quantum diamond functions of the given $\LoBcd$-algebra, and any such a quantum diamond function
gives rise to a continuous  representation of the twisted operad $\cP\cG raphs_{c,d}^\diamond$ in the Hochschild
complex of $S$ so that the latter can be considered as an operad
of {\em quantum $\cP$-braces} (cf.\ \cite{KoSo}).

\sip

Let
$$
CE^\bu(\sC_c^\bu(V))=\Def\left(\caL ie^\diamond_{c+d} \stackrel{\rho_{c,d}^\diamond}{\lon} \cE nd_{\sC_c^\bu(V)} \right)    \simeq \prod_n \Hom\left(\odot^n\sC_c^\bu(V), \sC_c^\bu(V)\right)[(c+d)(1-n)].
$$
be the standard Chevalley-Eilenberg complex of the dg $\caL ie_{c+d}$ algebra $(\sC_c^\bu(V), \hbar\Delta_{[\ ,\ ]} + \delta_\vartriangle, \{\ ,\ \})$. The representation $\rho$ induces a canonical morphism
of dg Lie algebras
$$
\rho^\diamond: \RGCPn{c,d}^\diamond \lon CE^\bu(\sC_c^\bu(V))
$$
 and hence a morphism of their cohomology groups,
$$
\rho^\diamond: H^0\left(\RGCPn{c,d}^\diamond \right) \lon H^0\left(CE^\bu(\sC_c^\bu(V))\right).
$$
 Let $\sigma$ be an arbitrary element in $H^0\left(\RGCPn{c,d}^\diamond \right)$ and let
 $\hat{\sigma}$ be a cycle representing $\sigma$ in the $\cP$-graph complex $\RGCPn{c,d}^\diamond$.
 Its image $\rho^\diamond(\hat{\sigma})$ is a degree zero cycle in dg Lie algebra $CE^\bu(\sC_c^\bu(V))$ of co-derivations of the co-commutative coalgebra $\odot^\bu(\sC_c^\bu(V)[c+d])$. Codifferentials in latter coalgebra are precisely $\caH \mathit{olie}_{c+d}$ algebra structures in  $\sC_c^\bu(V)$, so that $\rho^\diamond(\hat{\sigma})$ gives us a $\caH \mathit{olie}_{c+d}$-derivation of the dg Lie algebra
 $(\sC_c^\bu(V), \hbar\Delta_{[\ ,\ ]} + \delta_\vartriangle, \{\ ,\ \})$.
 If the exponent of the adjoint action, $\exp(ad_{\rho^\diamond(\hat{\sigma})})$ happens to converge as a linear automorphism
 of  $CE^\bu(\sC_c^\bu(V))$, it gives us a genuine
 $\caH olie_{c+d}$ automorphism (cf.\ \cite{MW}),
$$
F^\sigma=\left\{F^\sigma_k: \odot^k CE^\bu(\sC_d^\bu(V)) \lon CE^\bu(\sC_c^\bu(V))[(c+d)(1-k)]\right\}_{n\geq 1},
$$
of the dg  $\caL ie_{c+d}$ algebra $(\sC_c^\bu(V), \hbar\Delta_{[\ ,\ ]} + \delta_\vartriangle, \{\ ,\ \})$
with $F^\sigma_1=\Id$.
If the exponent does not converge, one plays a standard trick by introducing a formal parameter $u$ (of homological degree zero) and considering a dg Lie algebra $CE^\bu(\sC_c^\bu(V))[[u]]$. As $\exp({u ad_{\rho^\diamond(\hat{\sigma})}})$ converges,
we obtain from $\hat{\sigma}$  a continuous $\caH olie_{c+d}$-automorphism $F^\sigma=\{F^\sigma_k\}$
of the dg Lie algebra $(\sC_c^\bu(V)[[u]], \hbar\Delta_{[\ ,\ ]} + \delta_\vartriangle, \{\ ,\ \})$.
 Hence for any element $S\in \sC_c^\bu(V)[[u]]$ of degree $c+d$
which satisfies the MC equation
$$
\hbar\Delta_{[\ ,\ ]}S + \delta_\vartriangle S + \frac{1}{2}\{S,S\}=0
$$
the series
$$
S^\sigma:= S + \sum_{k\geq 2}\frac{1}{n!} F^\sigma_n(S, \ldots, S)
$$
gives us another solution of the same MC equation. Thus we can formulate the following proposition (cf.\ \cite{MW}).

\subsubsection{\bf Proposition}\label{5: Action H^0PGCd on diamond functions}
 {\em For any $\cP$-algebra $V$, there is an  action of the Lie algebra $H^0(\PGCcd^\diamond)$ on the differential graded Lie algebra $(\sC_c^\bu(V), \hbar\Delta_{[\ ,\ ]} + \delta_\vartriangle, \{\ ,\ \})$ by $\caH olie_{c+d}$ derivations.

 \sip
 In particular, it follows that there is an action of the group  $\exp(u H^0(\PGCd^\diamond)$ on the set of gauge equivalence classes of quantum diamond functions $S\in \sC_c^\bu(V)[[u]]$.}

\mip

 Finally we note that the stable $\cP$-graph complex $\PGCcd^\diamond$ can be understood as the universal (in the sense of  independence of a particular vector space $V$ and of a particular representation $\rho: \cP\rar \cE nd_V$) incarnation
of the Chevalley-Eilenberg complex $(\sC^\bu_{c}(V), \hbar\Delta_{[\ ,\ ]} + \delta_\vartriangle, \{\ ,\ \})$.

\subsection{Applications to the properad of ribbon graphs} In this paper we are mostly interested in the properad $\cR \cG ra_d$ of ribbon graphs which comes equipped with a canonical morphism
$s^\diamond: \LoB_{d,d} \rar \cR\cG ra_d$ given by (\ref{4: maps s from LoBd to RGra}). The associated ribbon graph complexes
$$
\RGCd=\Def\left(\LB_{d,d} \stackrel{s}{\lon} \cR\cG ra\right)  \ \ \ \mbox{and}\ \ \ \RGCd^\diamond=\Def\left(\LoB_{d,d} \stackrel{s^\diamond}{\lon}  \cR\cG ra_d\right)
$$
were described in detail in \S {\ref{3: subsection RGCd^diamond}}.

\sip

The polydifferential operad $\f_d(\cR \cG ra_d)$ is spanned by ribbon graphs whose vertices
can have coinciding numerical labels while boundaries have no labels at all (because their labels are (skew)symmetrized). The canonical morphism
(\ref{5: morphism f^diamond to O_dP}) is given explicitly by

$$
\Ba{rccc}
f^\diamond:  & \caL ie^\diamond_{2d} & \lon & \f_d(\cR \cG ra_d)[[\hbar]] \\
    & \Ba{c}\begin{xy}
    <0.1mm,-4.5mm>*{_1};
 <0mm,-0.55mm>*{};<0mm,-3mm>*{}**@{-},
 <0mm,0.5mm>*{};<0mm,3mm>*{}**@{-},
 <0mm,0mm>*{\bullet};<0mm,0mm>*{}**@{},
 \end{xy}\Ea
 & \lon &  \Ba{c}
 \mbox{\xy (0,-4)*{_1};
(0,-2)*{\bu}="A";
(0,-2)*{\bu}="B";
"A"; "B" **\crv{(6,6) & (-6,6)};
\endxy}\Ea
+
\hbar
  \Ba{c}\mbox{\xy
(0,-2)*{_1}; (6,-2)*{_1};
 (0,0)*{\bullet}="a",
(6,0)*{\bu}="b",
\ar @{-} "a";"b" <0pt>
\endxy}\Ea
 \\
     & \Ba{c}\begin{xy}
 <0mm,0.66mm>*{};<0mm,3mm>*{}**@{-},
 <0.39mm,-0.39mm>*{};<2.2mm,-2.2mm>*{}**@{-},
 <-0.35mm,-0.35mm>*{};<-2.2mm,-2.2mm>*{}**@{-},
 <0mm,0mm>*{\bu};<0mm,0mm>*{}**@{},
   <0.39mm,-0.39mm>*{};<2.9mm,-4mm>*{^2}**@{},
   <-0.35mm,-0.35mm>*{};<-2.8mm,-4mm>*{^1}**@{},
\end{xy}\Ea
 & \lon &  \Ba{c}\mbox{\xy
(0,-2)*{_1}; (6,-2)*{_2};
 (0,0)*{\bullet}="a",
(6,0)*{\bu}="b",
\ar @{-} "a";"b" <0pt>
\endxy}\Ea
 \\
 \Ea
$$
Denote the composition $\caL ie_{2d} \rar \caL ie^\diamond_{2d} \stackrel{f^\diamond|_{\hbar=0}}{\lon}  \f_d(\cR \cG ra_d)$ by $f'$.
The associated {\em stable ribbon graph complex}
$$
\sRGCd:= \Def\left(\caL ie_{2d} \stackrel{f'}{\lon} \f_d(\cR \cG ra_d)  \right)_{connected}
$$
is spanned by connected ribbon graphs whose boundaries and vertices are unlabelled and, moreover, the vertices are grouped
in clusters; the associated to $f'$ twisted operad is denoted by $\cR\cG raphs_d$; it comes equipped with a non-trivial morphism
$$
\cP ois_{2d} \lon \cR \cG raphs_d.
$$
which is proven below to be a quasi-isomorphism.

\sip

Twisting the morphism $f^\diamond$ as explained in the beginning of this section we obtain
a dg operad
$$
\cR \cG raphs_d^\diamond:= Tw\left(\f_d(\cR \cG ra_d)[[\hbar]]  \right)
$$
which admits a canonical representation in the Hochschild complex (see \S {\ref{3: Subsubsect Basic Example}}) of an arbitrary quantum $\cA ss_\infty$ algebra. Therefore $\cR \cG raphs_d$
gives us a ribbon analogue of the famous {\em braces operad}\, introduced by Kontsevich and Soibelman in \cite{KoSo} in the context of ordinary $\cA ss_\infty$ algebras; we call
$\cR \cG raphs_d$ the dg operad of {\em  ribbon braces}.

\sip

The {\em diamond ribbon graph complex}
$$
\sRGCd^\diamond:= \Def\left(\LoBd \stackrel{f^\diamond}{\lon} \f_d(\cR \cG ra_d)[[\hbar]]\right)_{connected}
$$
acts by derivations on the dg operad $\cR \cG raphs_d^\diamond$; this dg Lie algebra
has the property that its zero-th cohomology group $H^0(\sRGCd)$ acts universally
 on homotopy classes of quantum $\cA ss_\infty$ algebra structures on any given graded vector space $W$ equipped a (skew)symmetric scalar product of degree $1-d$ (see Proposition {\ref{5: Action H^0PGCd on diamond functions}}).


\subsection{Remarks} \label{rem:deltadef}
(i) The differential $\delta^\diamond$ in $\sRGCd^\diamond$ (or,
 in general, in $\sPGCd^\diamond$) consists of three  terms,
$$
\delta^\diamond=\delta + \Delta_1 + \hbar\Delta_2
$$
corresponding, respectively, to the ribbon graphs $\Ba{c}\mbox{\xy
(0,-2)*{_1}; (6,-2)*{_2};
 (0,0)*{\bullet}="a",
(6,0)*{\bu}="b",
\ar @{-} "a";"b" <0pt>
\endxy}\Ea$, $\Ba{c}
 \mbox{\xy (0,-4)*{_1};
(0,-2)*{\bu}="A";
(0,-2)*{\bu}="B";
"A"; "B" **\crv{(6,6) & (-6,6)};
\endxy}\Ea$ and, respectively,
$\hbar
  \Ba{c}\mbox{\xy
(0,-2)*{_1}; (6,-2)*{_1};
 (0,0)*{\bullet}="a",
(6,0)*{\bu}="b",
\ar @{-} "a";"b" <0pt>
\endxy}\Ea$.

(ii) The differential in $\sRGCd$ (or, in general, in $\sPGCd$) is given by the sum $\delta + \Delta_1$.

\mip

(iii) The data $(\sRGCd, \delta)$ is also a dg Lie algebra which can be identified
with $\Def(\caL ie_{2d} \stackrel{f^*}{\rar} \f_d(\cR \cG ra_d))$ where the map $f^*$
is given by the composition
$$
\caL ie_{2d} \hook \caL ie^\diamond_{2d} \stackrel{f^\diamond}{\lon} \f_d(\cR \cG ra_d)
$$
where the first arrow is the natural inclusion of operads. The family of complexes $(\sRGCd, \delta)$
is a subfamily of complexes
$$
\mathsf{sRGC}_{c,d}:=  \Def(\caL ie_{c+d} \stackrel{f_{c,d}^*}{\rar} \f_c(\cR \cG ra_d)),\ \ \ \forall\ c,d\in \Z,
$$
where $f_{c,d}^*$ is the composition
$$
\caL ie_{c+d} \hook \caL ie^\diamond_{c+d} \stackrel{f^\diamond}{\lon} \f_c(\cR \cG ra_d)
$$
\sip

(iv) Note that in $d=0$ case all the operators $\delta$, $\Delta_1$ and $\Delta_2$ in $\sRGCd$ have degree $1$ so that it makes sense to equip the graded vector space $\mathsf{sRGC}:=\mathsf{sRGC}_0$ not only only with differentials $\delta$, $\delta+\Delta_1$, but also with $\delta + \Delta_2$.

\bip



{\Large
\section{\bf The cohomology of the stable complex}
}
\mip

\subsection{Standing assumption and an important point}\label{sec:standing_assumption}
In the following we will work with unital augmented properads $\cP$ together with a map $\hoLieB_{c,d}\to \cP$, that satisfy a certain technical condition.
We assume that the properads $\op P$ we consider come equipped with a descending filtration
\[
 \op P = \mF^0 \op P \supset \mF^1 \op P \supset \mF^2\op P \supset \cdots
\]
satisfying the following properties:
\begin{itemize}
 \item The filtration is complete and compatible with the properad structure.\footnote{We consider the filtration to be compatible with the properad structure if any composition of elements in $\mF^{p_1}\cP, \dots, \mF^{p_k}\cP$ lands inside $\mF^{p_1+\dots+p_k}\cP$.}
 \item Endow $\hoLieB_{c,d}$ with the filtration induced by the grading that places the generator of arity $(p,q)$ into degree $p+q-2$.
 Then we require the map $\hoLieB_{c,d}\to \cP$ to respect the filtrations.
 In particular, by the completeness of the filtration on $\cP$ this means that the map $\hoLieB_{c,d}\to \cP$ extends continuously to a map $\whoLieB_{c,d}\to \cP$.
 \item We have $\op P/\mF^2 \op P=\op T_{c,d}$, cf. \S {\ref{5: Example def of prop T'}}, and the map $\hoLieB_{c,d}\to \op T_{c,d}$ factors as
 \[
  \hoLieB_{c,d} \to \cP \to \cP/\mF^2 \cP = \op T_{c,d}.
 \]
 In particular this means that the Lie bracket generator of $\hoLieB_{c,d}$ is mapped to a non-zero element of $\cP$.
\end{itemize}


\sip

Concretely, we will apply our constructions to three cases:
\begin{itemize}
\item To $\cP=\cT_d$, with $\mF^1\cP$ being the kernel of the augmentation and $\mF^p\cP=0$ for $p\geq 2$.
\item To $\cP$ being the genus completion of $\hoLieB_{c,d}$, with the filtration being obtained from the complete grading that places the generator of arity $(p,q)$ into degree $p+q-2$.
 \item To the ribbon graphs properad $\cP=\RGra_d$. Here the filtration will be simply by the number of edges in the graphs.
\end{itemize}

\sip

Our standing assumptions have some important consequences.
First, the properad $\op P$ is pro-nilpotent. Secondly, we have a natural properad map $\op P\to\op T_{c,d} \to  \op T_d$.
Hence, by functoriality of all the constructions, the graph complexes and operads associated to $\op P$ come equipped with natural maps to the ordinary graph complexes and operads: 
\begin{align}
  \RGCPn{c,d} \to \RGCxx{\op T_d}{c,d} &\cong \K
 \\
 \label{equ:SGRCtoGC}
 \SRGCPn{c,d} \to \SRGCxx{\op T_d}{c,d} &\cong \fcGC_{c+d} \\
 \label{equ:RGraphs to Graphs}
 \RGraphsPn{c,d} \to \RGraphsxx{\op T_d}{c,d} &\cong \Graphs_{c+d}.
\end{align}
We will use these facts extensively below.

\subsection{The main vanishing result}
Note that if the standing assumptions above are satisfied we have natural maps of dg Lie algebras
\begin{equation}\label{equ:almostexact}
 \RGCPn{c,d} \to \SRGCPn{c,d} \to \fcGC_{c+d}.
\end{equation}

The composition is not zero, but close to zero: the kernel of the composition is of codimension 2 and acyclic.

\sip

Our main result, which is a generalization of Theorem {\ref{thm:main_RGC}}, is the following.

\begin{theorem}\label{thm:main}
 Suppose that we have a properad map $\hoLieB_{c,d} \to \cP$ satisfying the standing assumptions of\, \S {\ref{sec:standing_assumption}}.
Then there is a long exact sequence in cohomology
\begin{equation}\label{equ:thmmain_longexact}
\cdots \to H^p(\RGCPn{c,d}) \to H^p(\SRGCPn{c,d}) \to H^p(\fcGC_{c+d}) \to H^{p+1}(\RGCPn{c,d}) \to \cdots .
\end{equation}
The morphisms
\[
 H(\RGCPn{c,d}) \to H(\SRGCPn{c,d}) \to H(\fcGC_{c+d})
\]
are induced by \eqref{equ:almostexact}.
The ``connecting'' homomorphism
\[
H^p(\fcGC_{c+d}) \to H^{p+1}(\RGCPn{c,d})=H^{p+1}(\Def(\hoLieB_{c,d}\to \cP))
\]
is given by the canonical map
\[
H^p(\fcGC_{c+d})\cong H^p(\GC_{c+d+1}^{or})\cong H^{p+1}(\Def(\hoLieB_{c,d} \to \hoLieB_{c,d})) \to H^{p+1}(\Def(\hoLieB_{c,d}\to \cP)),
\]
where for the identification $H^p(\fcGC_{c+d})\cong H^p(\GC_{c+d+1}^{or})$ we use a special isomorphism constructed in the proof.
\end{theorem}


%

\subsection{Proof of Theorem {\ref{thm:main}}}

\subsubsection{\bf First step: $H(\RGraphsPn{c,d})=\e_{c+d}$}
We have seen above that (under our standing assumption) we have natural maps $\hoe_{c+d} \to \RGraphsPn{c,d} \to \Graphs_{c+d}$, cf.\ Proposition {\ref{prop:hoen to RGraphs}} and \eqref{equ:RGraphs to Graphs}.
The key result then is the following.

\subsubsection{\bf Proposition} {\em Under our standing assumptions of section {\ref{sec:standing_assumption}} the maps
\[
\hoe_{c+d} \to \RGraphsPn{c,d} \to \Graphs_{c+d}
\]
are quasi-isomorphisms of operads.}

\mip

It is easy to check that the composition of the two maps in the proposition is the standard quasi-isomorphism $\hoe_{c+d}\to \Graphs_{c+d}$.
Hence the proposition follows by finite dimensionality of $\e_{c+d}$ in each arity if we can show that $H( \RGraphsPn{c,d})=\e_{c+d}$.
To this end we perform an induction on the arity in $ \RGraphsPn{c,d}$. The induction hypothesis states that $H( \RGraphsPn{c,d}(n-1))\cong \e_{c+d}(n-1)$.

\sip

Now endow $\RGraphsPn{c,d}$ with the following filtration.
As in our standing assumption above we have a descending complete filtration $\mF^\bullet\op P$ on $\op P$.
Then we define
\[
\mF^p  \RGraphsPn{c,d}
\]
to consist of series of graphs with $k$ hyperedges in $\mF^{p_1}\op P,\dots , \mF^{p_k}\op P$, 
such that
\[
 p_1+\dots +p_k -k  \geq p.
\]

\subsubsection{\bf Lemma}\label{lem:filtration}
{\em The subspaces $\mF^p \RGraphsPn{c,d}(n)$ define a filtration, i.e., they are subcomplexes.
Furthermore, this filtration is bounded below and complete.
}
\begin{proof}
The filtration is clearly bounded below since the filtration on $\op P$ is trivial for $p\leq 0$ and hence $p_j-1\geq 0$.
It is complete by construction. 
\end{proof}

Then on the associated graded $\gr \RGraphsPn{c,d}$ the only piece of the differential that survives, say $\delta$, is the one creating one internal vertex and one new hyperedge, decorated by the image of the Lie bracket generator under $\hoLieB_{c,d}\to \op P$. Let us call such a hyperedge a \emph{plain edge}.
We claim that $H(\gr  \RGraphsPn{c,d},\delta)\cong \e_{c+d}$.
We endow $\gr  \RGraphsPn{c,d}(n)$ with another very simple decomposition of graded vector spaces (which we may consider as a filtration):
\[
\gr  \RGraphsPn{c,d}(n) =V_0 \oplus V_1 \oplus V_{\geq 1},
\]
where $V_0$ is spanned by graphs such that no hyperedge is connected to external vertex $1$. The space $V_1$ is spanned by graphs for which at external vertex $1$ there is exactly one hyperedge attached (exactly once), and this one hyperedge is a plain edge, i.e., it is decorated by the image of the Lie brachet generator.
The subspace $V_{\geq 1}$ is spanned by all other graphs. The first differential in the spectral sequence is the map
\[
V_1 \leftarrow V_{\geq 1}
\]
obtained by splitting off all (hyper-)edges incident at vertex $1$ to a new internal vertex, connected to $1$ by a plain edge. This is an injection. Furthermore, the cokernel is spanned by graphs such that
\begin{itemize}
\item either $1$ is connected by a plane edge to another external vertex,
\item or  $1$ is connected by a string of $\geq 2$ plain edges to another internal or external vertex.
\end{itemize}

Let us denote the two subspaces of the cokernel generated by these two sorts of graphs $W$ and $W'$, so that the cokernel is $W\oplus W'$. The next differential in the (inner) spectral sequence renders $W'$ acyclic (cf., e.g., the proof of \cite[Proposition 3.4]{Wi}).

For $W$, one uses the same inductive argument leading to the proof that $H(\Graphs_{c+d})\cong \e_{c+d}$ (cf. \cite{LV}), finally concluding that $H( \RGraphsPn{c,d})\cong \e_{c+d}$.
\hfill\qed

\subsubsection{\bf Second step: Analyzing $\Def(\hoe_{c+d}\to  \RGraphsPn{c,d})_\conn$, and deriving the cohomology long exact sequence}
Let us consider the deformation complex
\[
\Def_{operad}(\hoe_{c+d}\to  \RGraphsPn{c,d}).
\]
Elements can be considered as series of graphs in $\RGraphsPn{c,d}$ whose external vertices are organized in  ``clusters''. Consider the ``connected'' subcomplex (spanned by the connected graphs)
\[
\Def_{operad}(\hoe_{c+d}\to  \RGraphsPn{c,d})_\conn \subset \Def_{operad}(\hoe_{c+d}\to  \RGraphsPn{c,d}).
\]
Given Proposition {\ref{prop:hoen to RGraphs}}, we can then conclude that the maps
\[
\Def_{operad}(\hoe_{c+d}\to \e_{c+d} )_\conn \to \Def_{operad}(\hoe_{c+d}\to  \RGraphsPn{c,d})_\conn \to \Def_{operad}(\hoe_{c+d}\to \Graphs_{c+d})_\conn
\]
are quasi-isomorphisms.
Now we may use the result of \cite{Wi} that
\begin{equation}\label{equ:Defhoe2Graphs}
 H(\Def_{operad}(\hoe_{n}\to  \Graphs_n)_\conn) \cong \K[-1] \oplus H(\GC_{n}^2)[-1].
\end{equation}
It then follows that
\begin{equation}\label{equ:Defhoe2RGraphs}
 H(\Def_{operad}(\hoe_{c+d}\to  \RGraphsPn{c,d})_\conn) \cong \K[-1] \oplus H(\GC_{c+d}^2)[-1].
\end{equation}

Now we turn to analyzing the deformation complex $\Def(\hoe_{c+d}\to  \RGraphsPn{c,d})_\conn$, whose cohomology we just computed.
Whenever we consider the deformation complex $\Def(\hoe_{c+d}\to X)$ of a map $\hoe_{c+d}\to X$ factoring through $\e_{c+d}$, the differential (on the deformation complex) has the form
\[
d_{Com} + d_{Lie} + \delta
\]
where $\delta$ is the internal differential induced by the internal differential on $X$, 
$d_{Com}$ is a Harrison type differential and $d_{Lie}$ is the remainder, as in \cite{Wi}.
As in that paper, we may interpret elements of $\Def(\hoe_{c+d}\to  \RGraphsPn{c,d})_\conn$ as elements of  $\RGraphsPn{c,d}$, whose external vertices are organized in clusters, with  $d_{Com}$ keeping the number of clusters the same, but increasing the size of one cluster and $d_{Lie}$ creating one more cluster.

Now there is a natural subcomplex
$$
\HRGCPn{c,d} \subset \Def(\hoe_{c+d}\to  \RGraphsPn{c,d})_\conn
$$
consisting of elements such that all external vertices are in clusters of size one, and each such vertex is connected to exactly one hyperedge (and connected excactly once to the hyperedge). (The notation stands for ``hairy $\cP$-graph complex''. If we replace $\cP$-graphs by ordinary, the complex  $\HRGCPn{c,d}$ would become the hairy graph complex, hence the notation.) Analogously to the case of plain graph complexes we have the following result.

\subsubsection{\bf Proposition} \label{prop:hairy to def}
{\em The inclusion
$$
\HRGCPn{c,d} \to \Def(\hoe_{c+d}\to \RGraphsPn{c,d})_\conn
$$
is a quasi-isomorphism of dg Lie algebras.}
\begin{proof}
One takes a spectral sequence first on the number of internal vertices, plus the number of clusters, so that the differential on the associated graded is $d_{Com}$. But the cohomology of $\Def(\hoe_{c+d}\to \RGraphsPn{c,d})_\conn$ with respect to $d_{Com}$ is just $\HRGCPn{c,d}$.
\end{proof}

Next we analyze the differential on the "hairy $\cP$-graph complex" $\HRGCPn{c,d}$.
It has two pieces
\[
d_{Lie} + \delta.
\]
Again the piece $\delta$ stems from the internal differential on $\RGraphsPn{c,d}$. The piece $d_{Lie}$ adds one or more external vertices, in its own cluster, and connects them by a hyperedge to one internal vertex.
To compute the cohomology we will consider an extended version of the hairy $\cP$-graph complex, which is as a graded vector space obtained by allowing graphs without external vertices, i.e.,
\[
\HRGCPn{c,d}^{ext} = \HRGCPn{c,d} \oplus \SRGCPn{c,d}.
\]
The differential $d_{Lie} + \delta$ naturally extends to this larger vector space, with $\delta$ extending as the differential on $\SRGCPn{c,d}$, and $d_{Lie}$ maps $\SRGCPn{c,d}$ into $\HRGCPn{c,d}$.
Now define $\RGCxx{\overline{\cP}}{c,d}\subset \RGCxx{\cP}{c,d}$ to be the subcomplex which is the kernel of the augmentation $\RGCxx{\cP}{c,d}\to \K$.
Then there is a map
\[
\RGCxx{\overline{\cP}}{c,d} \to \HRGCPn{c,d}^{ext} 
\]
by sending a (non-stable) $\overline{\cP}$-graph to the sum of graphs obtained by coloring its vertices either black or white, i.e., either external or internal. One checks that this map indeed respects the differentials.

\subsubsection{\bf Proposition} \label{prop:RGCtoDef}
{\em The map $\RGCxx{\overline{\cP}}{c,d} \to \HRGCPn{c,d}^{ext}$ above is a quasi-isomorphism.}
\begin{proof}
We endow $\RGCxx{\overline{\cP}}{c,d}$ and $\HRGCPn{c,d}^{ext}$ with the (evident extension of the) descending complete filtration of Lemma {\ref{lem:filtration}}. We claim (say Claim 1) that the above map induces an isomorphism on the cohomology of the associated graded spaces, thus proving the Proposition.
The differential on the associated graded is zero on $\RGCxx{\overline{\cP}}{c,d}$.
The differential on the associated graded of $\HRGCPn{c,d}^{ext}$ acts by the addition of one plain edge, added either by splitting an internal vertex into two internal, an external vertex into an external and an internal vertex, or by attaching a plain edge to a new external vertex to an internal vertex.
To show Claim 1 we endow the associated graded spaces with yet another bounded above descending filtration on the number of internal vertices.
The differential on the associated graded is the piece not adding an internal vertex, so that the graded pieces split according to the number $k$ of internal vertices in graphs.
Combinatorially, the differential adds one new external vertex connected by a plain edge to an internal vertex, and is in particular 0 if $k=0$. Note that the map $\RGCxx{\overline{\cP}}{c,d} \to \HRGCPn{c,d}^{ext}$ induces an isomorphism onto the piece corresponding to $k=0$. It hence suffices to check that the subcomplexes with $k\geq 1$ (i.e., at least one internal vertex) are acyclic. But this is easily established, an explicit homotopy is obtained by removing an external vertex connected by a plain edge to an internal vertex.

\end{proof}

Now consider the short exact sequence
\[
 0\to \HRGCPn{c,d} \to \HRGCPn{c,d}^{ext} \to \SRGCPn{c,d} \to 0.
\]
It induces a corresponding long exact sequence in cohomology
\begin{multline}\label{equ:exact_hairy}
\cdots \to H^p( \HRGCPn{c,d} ) \to H^p( \HRGCPn{c,d}^{ext} ) \to H^p(\SRGCPn{c,d}) \to H^{p+1}( \HRGCPn{c,d} ) \to \cdots \, .
\end{multline}

Note that by Proposition {\ref{prop:hairy to def}} and \eqref{equ:Defhoe2RGraphs} we have
\[
 H(\HRGCPn{c,d}) \cong H(\GC_{c+d}^2)[-1] \oplus \K[-1].
\]
By Proposition {\ref{prop:RGCtoDef}} we have
\[
 H( \HRGCPn{c,d}^{ext} ) \cong H(\RGCxx{\overline{\cP}}{c,d}) \cong H(\RGCPn{c,d})\oplus \K[-1].
\]
By an easy verification we see that the two copies of $\K$ above are mapped to each other in the long exact sequence so that we can omit them to obtain the long exact sequence
\begin{equation}\label{equ:exact_hairy2}
\cdots \to H^{p-1}( \GC_{c+d}^2 ) \to H^p( \RGCPn{c,d} ) \to H^p(\SRGCPn{c,d}) \to H^{p}( \GC_{c+d}^2) \to \cdots \, .
\end{equation}

This shows the existence of the long exact sequence asserted by Theorem {\ref{thm:main}}.
To show this Theorem it hence remains to verify that the morphisms in the long exact sequence indeed agree with the ones stated in the Theorem.
In fact, for the morphisms $H^p( \RGCPn{c,d} ) \to H^p(\SRGCPn{c,d})$ this is obvious. For the other two types of morphisms the verification will be done in the following subsections.



\subsubsection{\bf The simplest example}
Let us apply our findings so far to $\cP=\cT_d$, with the map $\hoLieB_{c,d}\to \cP$ being the projection \eqref{equ:LieB to Td}.
In this case the stable $\cP$-graph complex (essentially) agrees with the ordinary graph complexes as described in section {\ref{5: example Konts graph complexes}}. The unstable $\cP$ graph complex $\RGCPn{c,d}$ is two dimensional and acyclic.

Inserting these data into the long exact sequence \eqref{equ:exact_hairy}, and using Proposition {\ref{prop:hairy to def}} we obtain the long exact sequence
\[
 \cdots \to 0 \to H^{p}(\SRGCPn{c,d} )\cong H^{p}(\fcGC_{c+d}\oplus\K[-1])
 \to H^{p+1}(\HRGCPn{c,d})\cong H^p(\Def(\hoe_{c+d}\to \e_{c+d})_\conn) \to 0 \to \cdots
\]
We hence obtain a rederivation of \eqref{equ:Defhoe2Graphs}, shown in \cite{Wi}.

\sip

There is another interesting consequence of the present example. Clearly, the long exact sequence \eqref{equ:exact_hairy} is functorial in $\hoLieB_{c,d}\to \cP$. But by our standing assumptions any such map fits into a diagram
\[
 \begin{tikzcd}
  \hoLieB_{c,d}\ar{r}\ar{dr} & \cP \ar{d} \\
   & \cT_d
 \end{tikzcd}.
\]
Hence we get a map of the long exact sequences
\[
 \begin{tikzcd}
  \cdots \ar{r} & H^p(\HRGCPn{c,d}) \ar{r}\ar{d} & H^p(\RGCPn{c,d}) \ar{r}\ar{d} & \SRGCPn{c,d} \ar{d} \ar{r}  & \cdots \\
  \cdots \ar{r} & H^p(\GC^2_{c+d}\oplus \K[-1]) \ar{r} & 0 \ar{r} & H^p(\fcGC_{c+d})  \ar{r}  & \cdots
 \end{tikzcd}.
\]

In particular, we can fill in the first missing bit in the proof of Theorem {\ref{thm:main}} and obtain the following.
\begin{lemma}
 The morphisms
 \[
  H^p(\SRGCPn{c,d}) \to H^{p}( \GC_{c+d}^2)
 \]
appearing in the long exact sequence \eqref{equ:exact_hairy2} are induced by the natural projection of complexes
\[
 \SRGCPn{c,d} \to \fcGC_{c+d}
\]
stemming from the projection $\cP\to \cT_d$, together with the identification $H(\fcGC_{c+d})\cong H(\GC_{c+d}^2)$.
\end{lemma}

\subsubsection{\bf Example and a(nother) proof of the main Theorem of \cite{Wi2}}
In this section let us apply the spectral sequence of Theorem {\ref{thm:main}}, which we already costructed, to the properad $\cP$ which is the genus completion of $\hoLieB_{c,d}$ with the map
\[
\hoLieB_{c,d}\to \cP
\]
being taken to be the standard inclusion.

\sip

The (stable and unstable) $\cP$-graph complexes have been discussed in detail in section {\ref{sec:examplePLiecd}}.
The long exact sequence \eqref{equ:thmmain_longexact} in this case reads
\begin{equation*}
 \cdots H^{p-1}(\GC^2_{c+d}) \to \to H^{p-1}(\GCor_{c+d+1}\oplus \K[-1]) \to H^p(\SRGCPn{c,d}) \to H^p(\GC^2_{c+d}) \to \cdots
\end{equation*}
By \eqref{equ:SRGCPonedim} we know that $H(\SRGCPn{c,d})$ is one-dimensional in this case, the single class (of degree 1) being represented by \eqref{equ:extraclass}.
This class is in the image of the map from $H(\RGCPn{c,d})$. Hence we immediately recover the central result of \cite{Wi2}, namely that
\[
 H(\GC^2_{c+d}) \cong  H(\GCor_{c+d+1}).
\]

\subsubsection{\bf The proof of Theorem \ref{thm:main}}
Continuing the example of the previous section, we may use again the functoriality of the long exact sequence. Concretely, any map $\hoLieB_{c,d}\to \cP$ satisfying our standing assumptions fits into the ``tautological'' commutative diagram
\[
 \begin{tikzcd}
  \hoLieB_{c,d} \ar[hookrightarrow]{r} \ar{dr} & \whoLieB_{c,d} \ar{d} \\
  &  \cP
 \end{tikzcd}\, .
\]
Hence functoriality relates the long exact sequence \eqref{equ:exact_hairy2} associated to $\cP$ to that of the previous example
\[
 \begin{tikzcd}
  \cdots \ar{r} & H^{p-1}(\GC^2_{c+d}) \ar{d}{=} \ar[hookrightarrow]{r} & H^{p-1}(\GCor_{c+d+1}\oplus\K[-1])\ar[twoheadrightarrow]{r}\ar{d} & H^p(\K) \ar{r}\ar{d} & \cdots
  \\
  \cdots \ar{r} & H^{p-1}(\GC^2_{c+d})\ar{r} &  H^p(\RGCPn{c,d}) \ar{r} & \SRGCPn{c,d} \ar{r}  & \cdots
 \end{tikzcd}.
\]

In particular, from the left-hand commuting square in this diagram the remaining claim of Theorem {\ref{thm:main}} is evident, and the Theorem thus proven.

%

\bip

{\Large
\section{\bf Proofs of the main Theorems }
}
\mip

\subsection{Proof of Theorem {\ref{thm:main_RGC}} }
We merely apply Theorem {\ref{thm:main}} to $\cP=\RGra$ (respectively, $\cP=\RGra_1$), with the map $\LieB_0\to \RGra$ (respectively, $\LieB_{0,1}\to \RGra_1$) sending the Lie bracket generator to the two-vertex graph and the cobracket generator to zero.
We obtain a long exact sequence of the form \eqref{equ:thmmain_longexact}.
Clearly, in this case the connecting homomorphism $H(\GC_0^2)$ (respectively, from $H(\GC_1^2)$) is zero, since an oriented graph necessarily contains a vertex with (at least) two outputs, to which is assigned the (higher) cobracket generator, i.e., zero.
Theorem {\ref{thm:main_RGC}} hence immediately follows from the exactness of the sequence.
\hfill $\Box$

\sip

The above proof works in fact for arbitrary values of the parameters $c$ and  $d$ so that we  have the following
Theorem (of which Theorem {\ref{thm:main_RGC}} is a special case).

\subsubsection{\bf Theorem} $H(\SRGC_{c,d}) = H(\RGC_{c,d})\oplus H(\GC_{c+d}^2)$

\subsubsection{\bf Corollary: Grothendieck-Teichm\"uller group $GRT_1$ and non-commutative Poisson structures} The case $c=0$, $d=2$ is of special interest as the Lie algebra $\fg\fr\ft_1$
is a Lie subalgebra of $H^0(\GC_2^2)$ and hence, by the above Theorem, of $H^0(\SRGC_{0,2})$.
By Proposition {\ref{5: Action H^0PGCd on diamond functions}}, the Lie algebra
$H^0(\SRGC_{0,2})$ acts by $\caH olie_2$-derivations on the Lie algebra $\odot^\bu (Cyc(W))[[\hbar]]$
equipped with Lie brackets $\{\ ,\ \}$ given explicitly in \S {\ref{2: Proposition on inv LieBi in CycW}}.
As $\fg\fr\ft_1\subset H^0(\SRGC_{0,2})$, we conclude that there is a highly non-trivial, in general, action of the group $GRT_1$  on gauge equivalence classes
of non-commutation Poisson structures $\pi\in u\odot^\bu (Cyc(W))[[\hbar]][[u]]$ (cf.\ \cite{MW}).

\subsection{Proof of Theorem {\ref{thm:main_RGC2}} }
We apply Theorem {\ref{thm:main}} to $\cP=\RGra$ to the map $\LieB_{0,0}\to \RGra$ sending the Lie bracket generator to the graph with two vertices and one edge, and the Lie cobracket generator to the graph with one vertex and one edge, cf. Theorem {\ref{4: Theorem on LieB --> Rgra}}.
This yields the long exact sequence \eqref{equ:longexact}.

Similarly, we may apply Theorem {\ref{thm:main}} to $\cP=\RGra_1$, with the map $\hoLieB_{0,1}\to \RGra_1$ being that of
Theorem {\ref{4: Theorem on LieB_0,1 --> Rgra_1}}. This yields the long exact sequence \eqref{equ:longexact2}.

\subsection{Images of loop classes, and proof of Theorem {\ref{thm:connectinghom}} }
We use the description of the connecting homomorphism of Theorem {\ref{thm:main}} through the oriented graph complex $\GCor_1$. The identification of $H(\GC_0)$ and $H(\GCor_1)$ preserves the loop order grading.
Now consider an oriented graph $\Gamma\in \GCor_1$. We may assume that $\Gamma$ has only bivalent and trivalent vertices, with either exactly two incoming or exactly two outgoing vertices. (If, e.g., there is a vertex with more inputs/outputs the graph is mapped to zero as the images of the corresponding generating morphism of $\hoLieB_0$ in $\RGra_0$ are $0$.)
Let us begin with the source vertices of $\Gamma$ and then compose the images of the bracket and cobracket generators one by one, and record the Euler characteristic of the ribbon graph thus produced.
We can distinguish three types of vertices in $\Gamma$:
\begin{itemize}
\item Vertices with two outputs, no inputs (cobrackets). They create two punctures, one edge and one vertex, the effect on the Euler characteristic is "+2".
\item Vertices with one input and two outputs (cobrackets). They create a puncture and one edge, the effect on the Euler characteristic is zero.
\item Vertices with two inputs (brackets). They destroy a puncture and create an edge, the effect on the Euler characteristic is "-2".
\end{itemize}
So, if we denote the numbers of the above vertices by $v_1$, $v_2$, $v_3$, then the Euler characteristic of all ribbon graphs in the image of $\Gamma$ is $2v_1-2v_3$, and hence the genus is $g=1-v_1+v_3$.
On the other hand, it is not hard to see that the Euler characteristic of $\Gamma$ itself is $v_1-v_3$, and hence $\Gamma$ is of loop order $g$. This shows the second claim of Theorem {\ref{thm:connectinghom}}.

To see the third claim, note that the identification of $H(\GC_0)$ and $H(\GCor_1)$ sends classes of loop order $l$ with $k$ edges to classes with the same loop order, but $k+l$ edges. Now proceeding just as above, we count that the ribbon graphs appearing non-trivially in the image all have $v_1+v_2+v3$ edges (with $v_j$ as above). On the other hand $\Gamma$ itself has $k+l=v_2+2v_3$ edges and loop order $l=1-v_1+v_3$ (as we mentioned above), so that the number of edges of the ribbon graphs in the image is
\[
v_1+v_2+v_3 = v_2+2v_3  - (v_3 - v_1) = k+l-(l-1) = k+1.
\]
Hence the third claim of Theorem {\ref{thm:connectinghom}} is shown.
The first claim follows from the computation of the next subsection.

\subsection{Example computations}
The simplest classes in the graph complexes $\GC_d^2$ are the loop classes $L_r$ consisting of $r$ bivalent vertices, cf. Figure \ref{fig:loops} (left). It is known \cite{Wi2} that under the identification $H(\GC_d^2)\cong \GCor_{d+1}$ these classes correspond to the "zig-zag loop" classes with one more vertex, cf. Figure \ref{fig:loops} (middle). Let us investigate what the images of these classes are (under the map of Theorem {\ref{thm:main}}) in the ribbon graph complexes $(\RGC,\delta+\Delta_1)$ respectively $(\RGC_{odd},\delta+\Delta_3)$.

\begin{proposition}
The torus graphs $T_n\in H(\RGC)$ depicted in Figure \ref{fig:loops} (right) are cocycles and represent non-trivial cohomology classes for $n$ odd.
The image of the loop class $L_{4k+1}\in \GC_0$ (cf. Figure \ref{fig:loops} left) under the connecting homomorphism is the torus class in $(\RGC,\delta+\Delta_1)$ represented by $T_{2k+1}$.
\end{proposition}
\begin{proof}
The statement that the image of $L_{4k+1}$ is $T_{2k+1}$ is a simple direct verification, given that one knows that the loop classes in $\GC_0$ are represented by similar ``zigzag loop'' classes in the oriented graph complex $\GCor_1$, cf. Figure \ref{fig:loops}.

The non-trivial part of the Proposition is that the classes $T_{2k+1}$ are non-trivial (i.e., the cocycles are not exact).
To show this, the easiest way is to construct cycles $c_{2k+1}$ in the (pre-)dual complex which have a non-zero pairing with $T_{2k+1}$. The (pre-)dual complex in this case is the complex of linear combinations of ribbon graphs, with the differential given by the contraction of one edge, plus the removal of one edge as long as this does not alter the genus of the graph.
We will not construct $c_{2k+1}$, but rather construct the formal series $\sum_k c_{2k+1}$, with the understanding that $c_{2k+1}$ is the part with $4k+2$ edges.
Now fix numbers $a_0,a_1,\dots$ such that $a_0=1$ and recursively
\[
a_n = \sum_{j=0}^{n-1} a_j a_{n-j}.
\]
(Concretely $a_n =\frac 1 {n+1}{ 2n \choose n}$ are the Catalan numbers, but this will play no role here.)
We construct $c$ as the sum of all graphs of the following form
\[
c = \sum_{r\geq 1 \text{ odd}}\frac 1 r
\sum_{j_1,\dots,j_r\geq 0}
\sum_{k_1,\dots,k_r\geq 0}
\left( \prod_{i=1}^r a_{j_i} a_{k_i} \right)
  \begin{tikzpicture}[baseline=-.65ex, scale=1.5]
  \node[int] (v0) at (0:1) {};
\node[int] (v2) at (120:1) {};
\node (v4) at (240:1) {$\cdots$};
\draw (v0) edge[thick] node[sloped,anchor=south]{$\scriptstyle 2j_1+1$} (v2) edge[thick] node[sloped,anchor=south]{$\scriptstyle 2j_2+1$} (v4) (v2) edge[thick] node[sloped,anchor=south]{$\scriptstyle 2j_r+1$} (v4);
\draw (v0) to[thick, out=180, in=90, looseness=30] node[sloped,anchor=south]{$\scriptstyle 2k_2+1$} (v0);
\draw (v2) to[thick, out=-70, in=45, looseness=30] node[sloped,anchor=south]{$\scriptstyle 2k_1+1$} (v2);
 \end{tikzpicture}.
\]
Here a fat edge with number $m$ stands for a string of $m$ edges, while a loop with number $m$ stands for $m$ parallel loops.
To check that closedness, note that the differential contracts a string of $2j+1$ edges to a string with $2j$ edges (with coefficient 1), and makes $2k+1$ parallel loops into $2k$ parallel loops.
The $j=0$ piece in particular will result in two bunches of loops to be appended, while the $k=0$ piece results in two strings of edges to be appended. Note that in both cases a resulting bunch of $2K$ loops (respectively, string of $2J$ edges) may be produced in $K$ ways (respectively, $J$ ways), by juxtaposing an $\alpha$-bunch and a $2K-\alpha$-bunch (respectively, an $\alpha$-string and a $2J-\alpha$-string). The coefficients $a_n$ are chosen so that the terms corresponding to $j=0$ cancel those for $k>0$ and those for $j>0$ cancel those for $k=0$.

It is furthermore clear that our class $T_{2k+1}$ has non-zero pairing with $c$ and must hence be a non-trivial class.
\end{proof}

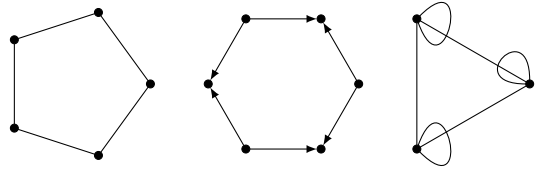
\begin{figure}
\[
 \begin{tikzpicture}[baseline=-.65ex]
  \node[int] (v0) at (0:1) {};
\node[int] (v1) at (72:1) {};
\node[int] (v2) at (144:1) {};
\node[int] (v3) at (216:1) {};
\node[int] (v4) at (288:1) {};
\draw (v0) edge (v1) edge (v4) (v2) edge (v3) edge (v1) (v3) edge (v4);
 \end{tikzpicture}
\quad\quad
 \begin{tikzpicture}[baseline=-.65ex]
  \node[int] (v0) at (0:1) {};
\node[int] (v1) at (60:1) {};
\node[int] (v2) at (120:1) {};
\node[int] (v3) at (180:1) {};
\node[int] (v4) at (240:1) {};
\node[int] (v5) at (300:1) {};
\draw[-latex] (v0) edge (v1) edge (v5) (v2) edge (v1) edge (v3) (v4) edge (v3) edge (v5);
 \end{tikzpicture}
 \quad\quad
  \begin{tikzpicture}[baseline=-.65ex]
  \node[int] (v0) at (0:1) {};
\node[int] (v2) at (120:1) {};
\node[int] (v4) at (240:1) {};
\draw (v0) edge (v2) edge (v4) (v2) edge (v4);
\draw (v0) to[ out=180, in=90, looseness=30] (v0);
\draw (v2) to[ out=-70, in=45, looseness=30] (v2);
\draw (v4) to[ out=70, in=-45, looseness=30] (v4);
 \end{tikzpicture}
 \]
  \caption{\label{fig:loops} A picture of the loop class $L_{5}$ in $\GC_0$ (left), together with the corresponding ``zigzag loop'' class in $\GCor_1$ (middle), and the corresponding ``torus class'' $T_{3}$ (right).}
 \end{figure}

\begin{proposition}
The image in $(\RGC_{odd},\delta+\Delta_1)$ of the loop class $L_{4k-1}$ in $H(\GC_1^2)$ is represented by an infinite sum of graphs  whose leading order term (i.e., the term with fewest edges) is given by the graph $\Theta_{k}$ in Figure \ref{fig:loops2}.
\end{proposition}
\begin{proof}
Again a simple direct verification.
\end{proof}

\begin{figure}
\[
 \begin{tikzpicture}[baseline=-.65ex]
  \node[int] (v0) at (0:1) {};
\node[int] (v1) at (120:1) {};
\node[int] (v2) at (-120:1) {};
\draw (v0) edge (v1) edge (v2) (v2) edge (v1);
 \end{tikzpicture}
\quad\quad
 \begin{tikzpicture}[baseline=-.65ex]
  \node[int] (v0) at (0:1) {};
\node[int] (v1) at (90:1) {};
\node[int] (v2) at (180:1) {};
\node[int] (v3) at (270:1) {};
\draw[-latex] (v0) edge (v1) edge (v3) (v2) edge (v1) edge (v3);
 \end{tikzpicture}
 \quad\quad
  \begin{tikzpicture}[baseline=-.65ex]
  \node[int] (v0) at (-45:1) {};
  \node[int] (v1) at (45:1) {};
\node[int] (v2) at (135:1) {};
\node[int] (v3) at (225:1) {};
\draw (v0) edge (v1) edge[bend left] (v1) edge[bend right] (v1);
\draw (v2) edge (v3) edge[bend left] (v3) edge[bend right] (v3);
\draw (v1) to[ out=180, in=-80, looseness=2] (v2);
\draw (v0) to[ out=100, in=0, looseness=2] (v3);
 \end{tikzpicture}
 \]
  \caption{\label{fig:loops2} A picture of the loop class $L_{3}$ in $\GC_1$ (left), together with the corresponding ``zigzag loop'' class in $\GCor_2$ (middle), and the corresponding ribbon graph $\Theta_{1}$ (right).}
 \end{figure}
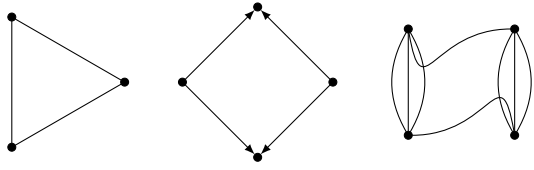

%
%

\bip
\appendix


\newcommand{\DLie}{\mathcal{D}\mathit{lie}} 
\newcommand{\hoDLie}{\mathcal{H}\mathit{o}\caD \mathit{lie}}  

\section{$d$-Quadratic BV operads}\label{sec:qBV}
In this Appendix we will study operads $\qBV_d$ generated by the operads $\e_d$ and one additional square zero operation
$\Delta\in \qBV_d(1)$ that is a derivation with respect to the product and bracket of $\e_d$.
In other words, $\Delta$ satisfies the following compatibility relations with respect to the product $m\in \qBV_d(2)$ and the bracket $\mu\in \qBV_d(2)$:
\begin{align*}
 \Delta^2 &=0
 &
 \Delta m(-,-) &= m(\Delta - , -) + m(-,\Delta -) \\
 \Delta \mu(-,-) &= \mu(\Delta - , -) + (-1)^{n-1}\mu(-,\Delta -) .
\end{align*}
In particular for $d=2$ the operad $\qBV_2$ is the homogeneous associated quadratic operad associated to the Batalin-Vilkovisky operad.
It is known that the operad $\qBV_d$ is Koszul \cite{GCTV},
and we define the minimal resolution
\[
 \hoqBV_d := \Omega(\qBV_d^\vee).
\]
Concretely, as a symmetric sequence $\qBV_d^\vee\cong \e_d^\vee[u]$, where $u$ is a formal variable of degree 0 which stands for the Koszul dual cooperation to $\Delta$. Accordingly, a $\hoqBV_d$-structure on some vector space is determined by operations
\[
 \mu^r_X,
\]
where $r=0,1,2,\dots$ and $X$ ranges over a basis of $\e_d^\vee$.
There is a natural inclusion $\e_d\to \qBV_d$ and hence a map $\hoe_d\to\hoqBV_d$, and a projection (sending $\Delta\mapsto 0$) $\qBV_d\to\e_d$ and hence a map $\hoqBV_d\to \to\hoqBV_d$.

The operad $\qBV_d$ has a suboperad $\DLie_d\subset \qBV_d$ generated by the bracket $\mu$ and the unary operation $\Delta$. This latter operad is also Koszul, the Koszul dual cooperad being identified (as symmetric sequence) with $\DLie_d^\vee\cong \Lie_d^\vee[u]$. Concretely, the minimal resolution
\[
 \hoDLie_d = \Omega(\DLie_d^\vee)
\]
is generated by ($k$-ary) operations $\mu_{k}^r$ for $r\geq 0$, $k\geq 1$ and $k+r\geq 2$.
Pictorially we may represent the operation $\mu_{k}^r$ by the corolla
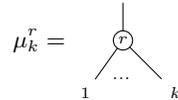
\[
\mu_k^r =
\begin{tikzpicture}[baseline=-.65ex]
 \node[ext] (v) at (0,0) {$\scriptstyle r$};
 \node (w1) at (-.5,-.7) {$\scriptstyle 1$};
 \node (wk) at (.7,-.7) {$\scriptstyle  k$};
 \node at (0,-.5) {$\scriptstyle \cdots$};
 \draw (v) edge +(0,.5) edge (w1) edge (wk);
\end{tikzpicture}
\]
The differential of such an operation is then pictorially written as
\[
\delta:
\begin{tikzpicture}[baseline=-.65ex]
 \node[ext] (v) at (0,0) {$\scriptstyle r$};
 \node (w1) at (-.5,-.7) {$\scriptstyle 1$};
 \node (wk) at (.7,-.7) {$\scriptstyle  k$};
 \node at (0,-.5) {$\scriptstyle \cdots$};
 \draw (v) edge +(0,.5) edge (w1) edge (wk);
\end{tikzpicture}
\mapsto
\sum_{s+t=r}
\sum_{I_1\sqcup I_2 = [k]}
\pm
\begin{tikzpicture}[baseline=-.65ex]
 \node[ext] (v) at (0,.3) {$\scriptstyle s$};
 \node[ext] (w) at (.3,0) {$\scriptstyle t$};
 \node at (-.6,-.6) {$\scriptstyle I_1$};
 \node at (.3,-.8) {$\scriptstyle I_2$};
 \draw (v) edge +(0,.5) edge (w) edge +(-.8,-.7)edge +(-.6,-.7)edge +(-.4,-.7)
  (w) edge +(.3,-.5) edge +(0,-.5)edge +(-.3,-.5);
\end{tikzpicture}
\]

Next, we define the operad $\thoqBV_d$ to be generated by $\hoDLie_d$ and an additional commutative product operation $m$ of arity 2 and degree zero, such that all operations $\mu_k^r$ are multi-derivations with resect to $m$.
Alternatively, we may understand $\thoqBV_d$ as a quotient of $\hoqBV_d$: A $\thoqBV_d$-algebra is a $\hoqBV_d$-algebra
such that the operations $\mu^r_X$ vanish unless (i) $r=0$ and $X\in \e_d^\vee$ is the cobracket (this yields the product operation $m$) or (ii) $X\in \Lie_d^\vee\subset \e_d^\vee$ is a cocommutative coproduct (this yields the operations $\mu_k^r$).

We have a commutative diagram of maps of operads
\[
 \begin{tikzcd}
 \hoLie_d \ar{r} \ar{d} & \hoe_d \ar{r}{\sim} \ar{d} & \thoe_d \ar{d} \\
 \hoDLie_d \ar{r} & \hoqBV_d \ar{r}{\sim} & \thoqBV_d
 \end{tikzcd}.
\]

Finally, let us remark on the deformation complexes associated to the operad $\hoqBV_d$.
First, it is an easy exercise to check that
\[
 \Def(\hoLie_d\to \e_d)
\]
is acyclic. (It follows from symmetry that the ``connected part'' $\Def(\hoLie_d\to \e_d)_\conn$ is 2-dimensional, with non-zero differential.)
By a similar (but slightly more complicated) argument one shows the following.
\begin{lemma}\label{lem:DefLieqBV}
 \[
  H(\Def(\hoLie_d\to \qBV_d))
  \cong
  H(\Def(\hoLie_d\to \qBV_d)_\conn) \cong \K[-1]
 \]
\end{lemma}

\bip

\section{Remark: Properadic twisting and Costello's properad}

Let $\hoLieB_{c,d}\to \cP$ be a properad map. Above, we constructed from this data an operad $\f\cP$, to which we applied the operadic twisting functor to obtain another operad $\cP \cG raphs_{c,d}$. Alternatively, one may also define a notion of {\em properadic}\, twisting.

\sip

To this end, we define a Maurer-Cartan element in a $\hoLieB_{c,d}$-algebra $V$ to be an element $m\in V$ of degree $d$, which satisfies the equations
\begin{equation}\label{equ:properadMC}
\sum_{k\geq 1}\frac 1 {k!} \mu_{l,k}(\underbrace{m,\dots,m}_{k\times}) = 0
\end{equation}
for every $l\geq 1$, where we use the notation $\mu_{l,k}$ to denote the generating $\hoLie_{c,d}$-operations. In particular $\mu_{1,1}$ is the differential in $V$.
Given such a Maurer-Cartan element we may twist the $\hoLieB_{c,d}$ structure to a different one (denoted $\mu_{l,k}^m$) with
\begin{equation}\label{equ:hoLieBtwisting}
\mu_{l,k}^m(v_1,\dots, v_k) = \sum_{k'\geq 1}\frac 1 {k'!} \mu_{l,k+k'}(\underbrace{m,\dots,m}_{k'\times},v_1,\dots,v_k)
\end{equation}
for $v_1,\dots,v_k\in V$. In particular, the differential on $V$ is altered (from $\mu_{1,1}$ to $\mu_{1,1}^m$).

\sip

Now given the properad map $\hoLieB_{c,d}\to \cP$ any $\cP$-algebra $V$ is naturally endowed with a $\hoLieB_{c,d}$ structure. We define the properad $\Tw'\cP$ to be the properad generated by $\cP$ together with an arity $(1,0)$ operation $m$ of degree $d$ satisfying the Maurer-Cartan equations \eqref{equ:properadMC}, and with differential being (formally) the commutator of operations with the arity $(1,1)$ operation $\mu_{1,1}^m$.
The definition is made such that for $V$ a $\cP$-algebra and $m\in V$ a Maurer-Cartan element, the prooperad $\Tw'\cP$ naturally acts on the dg vector space $V^m$, with the additional $(1,0)$ operation acting as the element $m\in V$.
There are natural maps
\[
 \hoLieB_{c,d}\to \Tw' \cP \to \cP,
\]
the former being given on algebras by \eqref{equ:hoLieBtwisting}, while the latter send the arity $(1,0)$ operation to 0.
Finally we define $\Tw\cP$ to be the completion of $\Tw'\cP$ with respect to the filtration by the number of copies of the $(1,0)$-ary operation $m$ occurring.

\sip

If we apply the above twisting construction to the properad map $\hoLieB_{c,d}\to \LieB_{c,d}\xrightarrow{s^*} \RGra_d$ (with the map $s^*$ as in section {\ref{sec:defRGCgeneral}}) we obtain a ribbon graphs properad $\Tw\RGra_d$ whose arity $(p,q)$ operations consist of possibly infinite linear combinations of ribbon graphs with $q$ numbered vertices, $p$ numbered boundary components and an arbitrary number of ``unidentifiable'' vertices, which correspond to copies of the $(1,0)$-ary operation inserted in the corresponding inputs of the underlying operation in $\RGra_d$.
If we consider the deformation complex
\[
\Def(\hoLieB_{c,d} \xrightarrow{s^*} \Tw\RGra_{d})
\]
we obtain a ribbon graph complex with two sorts of vertices. This recovers K. Costello's definition of the ribbon graph complex ``with black and white vertices'' in a purely algebraic manner.

\sip

{\em Example}. For any graded  vector space $W$ equipped with a degree $1-d$ scalar product
satisfying (\ref{2: skewsymmetry on scalar product}) the associated vector space $V= Cyc(W)$
spanned by cyclic words in $W$ carries a natural $\HoLoB_{d,d}$ structure (see \S {\ref{2: subsection on cyclic words}}) with only two non-trivial operations, the Lie bracket $\mu_{1,2}=[\ ,\ ]$
and Lie cobracket $\mu_{2,1}=\triangle$. Hence the equation (\ref{equ:properadMC})
defining Maurer-Cartan elements $m$ (of homological degree $d$) in the $\HoLoB_{d,d}$-algebra  $Cyc(W)$ can be written as
$$
[m,m]=0,\ \ \ \ \triangle m=0.
$$
Solutions of the first equation, $[m,m]=0$, are precisely cyclic $A_\infty$ structures in $W$,
while the second equation puts a constraint on cyclic $A_\infty$ structures. It was noticed in
\cite{Ha2} that cyclic $A_\infty$ structures in $W$ satisfying this extra constraint solve the equation (\ref{3: rescaled master eqn for inv Lie bial}) and hence extend to {\em quantum}\, $A_\infty$ algebra structures in $W$; the conclusion is that Maurer-Cartan elements $m$ in the $\HoLoB_{d,d}$-algebra  $Cyc(W)$ give us cohomology classes in the Kontsevich moduli space $\bar{\cM}^K_{g,n}$. Kontsevich
produced in \cite{Ko5} (see also \cite{Ba2,Ha2}) an infinite family of such Maurer-Cartan elements $m$
in the case $\dim W=1$. Indeed, choose any  integer $n\geq 1$ and equip the one-dimensional space $W:=\K[-1]$ with the scalar product of degree $-2n$ (i.e.\ in the notation of \S {\ref{2: subsection on cyclic words}}, $d=1+2n$)
$$
\Ba{rccc}
\Theta: & W\ot W & \lon &  \K[-2n]\\
        & (a,b)  & \lon & ab
\Ea
$$
Using explicit formulae given in Proposition {\ref{2: Proposition on inv LieBi in CycW}}, it is immediate to see that, for any $x\in W\setminus 0$, the cyclic word
$$
m=\underbrace{x\cdot x \cdot \ldots \cdot x}_{2n+1}
$$
has homological degree $d$ and  satisfies $[m,m]=0$, $\triangle m=0$. According to Kontsevich,
this infinite family of Maurer-Cartan elements generates non-trivial cohomology classes in $H(\cM_{g,n})$ and their linear span is precisely the space of all polynomials in Morita-Miller-Mumford
classes. It was noticed furthermore in \cite{Ba2,Ha2} that these cohomology classes in
$H(\cM_{g,n})$  admit an extension to classes in $H(\bar{\cM}^K_{g,n})$. Therefore the above notion
of a {\em Maurer-Cartan element of a $\HoLB_{c,d}$ algebra}\, admits many non-trivial and useful examples.

\bip

\def\cprime{$'$}

\bip


\end{document}